\documentclass[10pt,aps,prb, preprint,reqno]{amsart}
\usepackage{amsmath}
\usepackage{amssymb}
\usepackage{yfonts}
\usepackage{hyperref}
\usepackage{mathrsfs}
\usepackage {latexsym}
\usepackage{graphics}
\usepackage{txfonts}
\usepackage{breqn}
\usepackage{enumitem}
\allowdisplaybreaks

\newtheorem{theorem}{Theorem}[section]
\newtheorem{remark}{Remark}[section]
\newtheorem{lemma}[theorem]{Lemma}

\newtheorem{proposition}[theorem]{Proposition}

\newtheorem{define}{Definition}[section]

\allowdisplaybreaks 

\begin{document}
\title[Non-uniqueness in law for 3D MHD system]{Non-uniqueness in law of three-dimensional magnetohydrodynamics system forced by random noise}
 
\author{Kazuo Yamazaki}  
\address{Texas Tech University, Department of Mathematics and Statistics, Lubbock, TX, 79409-1042, U.S.A.; Phone: 806-834-6112; Fax: 806-742-1112; E-mail: (kyamazak@ttu.edu)}
\date{}
\maketitle

\begin{abstract}
We prove non-uniqueness in law of the three-dimensional magnetohydrodynamics system that is forced by random noise of an additive and a linear multiplicative type and has viscous and magnetic diffusion, both of which are weaker than a full Laplacian. We apply convex integration to both equations of velocity and magnetic fields in order to obtain the non-uniqueness in law in the class of probabilistically strong solutions.  
\vspace{5mm}

\end{abstract}
\footnote{2010MSC : 35A02; 35R60; 76W05}

\section{Introduction}\label{Introduction}

\subsection{Motivation from physics and mathematics}\label{Motivation from physics and mathematics}
Initiated by Alfv$\acute{\mathrm{e}}$n \cite{A42a} in 1942, the study of magnetohydrodynamics (MHD) concerns the properties of electrically conducting fluids. For example, while fluid turbulence is often investigated through Navier-Stokes (NS) equations, MHD turbulence occurs in laboratory settings such as fusion confinement devices (e.g., reversed field pinch), as well as astrophysical systems (e.g., solar corona) and the conventional system of equations for such a study is that of the MHD. Such hydrodynamic models forced by random noise has a history of more than 60 years (e.g., \cite{BT73, LL56, N65}). Moreover, various forms of dissipation have been suggested in the physics literature; e.g., frictional dissipation by atmospheric scientists \cite{PBH00}. Fractional Laplacian defined via Fourier transform $\mathcal{F}$ as $(-\Delta)^{m} f (x) \triangleq \sum_{k \in \mathbb{Z}^{n}} \lvert  k \rvert^{2m} \mathcal{F} (f) (k) e^{ik\cdot x}$ for $x \in \mathbb{T}^{n} \triangleq [-\pi, \pi]^{n}$ also appears naturally in the models from geophysics such as surface quasi-geostrophic equations (e.g., \cite{C02}). The purpose of this manuscript is to prove non-uniqueness in law of the three-dimensional MHD system forced by two prototypical forms of random noise: additive and linear multiplicative. The forms of viscous and magnetic diffusion will be fractional Laplacians with both powers allowed to be arbitrarily close to, but strictly smaller than, one. 

\subsection{Previous works}\label{Previous works}
Let us define $nD$ as ``$n$-dimensional,'' $\partial_{t} \triangleq \frac{\partial}{\partial t}$ and write components of any vector with super-indices. We denote by $u: \mathbb{R}_{+} \times \mathbb{T}^{3} \mapsto \mathbb{R}^{3}$ and $b: \mathbb{R}_{+} \times \mathbb{T}^{3} \mapsto \mathbb{R}^{3}$ the velocity and magnetic vector fields, $\pi: \mathbb{R}_{+} \times \mathbb{T}^{3} \mapsto \mathbb{R}$ the pressure scalar field, and viscous and magnetic diffusivity constants as $\nu_{1}, \nu_{2} \geq 0$, respectively. While we mostly focus on the case the spatial domain is $\mathbb{T}^{n}$, some of our discussions can be readily extended to the case of $\mathbb{R}^{n}$. Under such notations, the generalized MHD system reads 
\begin{subequations}\label{deterministic GMHD} 
\begin{align}
& \partial_{t} u + (u\cdot \nabla)u + \nabla \pi + \nu_{1} (-\Delta)^{m_{1}} u = (b\cdot\nabla) b, \hspace{2mm} \nabla\cdot u = 0, \hspace{3mm} t > 0, \\
& \partial_{t} b + (u\cdot\nabla)b + \nu_{2} (-\Delta)^{m_{2}} b = (b\cdot \nabla)u, \hspace{10mm} \nabla \cdot b = 0, \hspace{3mm} t > 0, 
\end{align}
\end{subequations} 
with initial condition denoted by $(u^{\text{in}}, b^{\text{in}})(x) \triangleq (u,b)(0,x)$. We refer to the case $\nu_{1}, \nu_{2} > 0$ and $m_{1} = m_{2} = 1$ as the classical MHD system whereas the case $\nu_{1} = \nu_{2} = 0$ the ideal MHD system; furthermore, considering $b\equiv 0$ reduces \eqref{deterministic GMHD} to the NS equations, and additionally taking $\nu_{1} = 0$ recovers the Euler equations. 
\begin{define}
(E.g., \cite[Def. 3.5-3.6]{BV19b}) In the diffusive case $\nu_{1}, \nu_{2}  >0$, $(u,b)$ such that 
\begin{equation*}
u \in C_{\text{weak}}^{0} ([0,T]; L^{2} (\mathbb{T}^{n})) \cap L^{2} ([0,T]; \dot{H}^{m_{1}} (\mathbb{T}^{n})), b \in C_{\text{weak}}^{0} ([0,T]; L^{2} (\mathbb{T}^{n})) \cap L^{2} ([0,T]; \dot{H}^{m_{2}} (\mathbb{T}^{n})) 
\end{equation*}
is called a Leray-Hopf weak solution to \eqref{deterministic GMHD} over $[0,T]$ if for any $t \in [0, T]$ they are weakly divergence-free and mean-zero, and satisfy \eqref{deterministic GMHD} distributionally and the energy inequality 
\begin{equation}\label{estimate 34}
\frac{1}{2} ( \lVert u (t) \rVert_{L_{x}^{2}}^{2} + \lVert b(t) \rVert_{L_{x}^{2}}^{2}) + \int_{0}^{t} \nu_{1} \lVert u\rVert_{\dot{H}_{x}^{m_{1}}}^{2} + \nu_{2} \lVert b \rVert_{\dot{H}_{x}^{m_{2}}}^{2} ds \leq \frac{1}{2} (\lVert u^{\text{in}} \rVert_{L_{x}^{2}}^{2} + \lVert b^{\text{in}} \rVert_{L_{x}^{2}}^{2}). 
\end{equation} 
Moreover, $(u,b)$ such that $u, b \in C_{t}^{0}L_{x}^{2}$ is called a weak solution to \eqref{deterministic GMHD} over $[0,T]$ if for any $t \in [0,T]$ they are weakly divergence-free and mean-zero, and satisfy \eqref{deterministic GMHD} distributionally. Analogous statements can be made in the inviscid case (cf. \cite[Def. 3.1]{BV19b}). 
\end{define} 
Mathematical analysis on the classical MHD system was pioneered by Duvaut and Lions \cite{DL72} and fundamental results such as the global existence of a Leray-Hopf weak solution in case $n \in \{2,3\}$, and its uniqueness in case $n = 2$ are well-known (e.g., \cite[The. 3.1]{ST83}). The investigation of the generalized NS equations was initiated by Lions in \cite[Rem. 8.1]{L59}, followed by \cite[Rem. 6.11]{L69} that already claimed the uniqueness of its Leray-Hopf weak solution when $m_{1} \geq \frac{1}{2} + \frac{n}{4}$. This range of $m_{1} \in [\frac{1}{2} + \frac{n}{4}, \infty)$ corresponds to the $L^{2}(\mathbb{T}^{n})$-critical and $L^{2}(\mathbb{T}^{n})$-subcritical regime of the NS equations considering the rescaling property that if $(u,\pi)(t,x)$ solves the generalized NS equations, then so does $(u_{\lambda}, \pi_{\lambda})(t,x) \triangleq (\lambda^{2m_{1} - 1} u, \lambda^{4m_{1} -2} \pi) (\lambda^{2m_{1}}t, \lambda x)$; we call the complement $(0, \frac{1}{2} + \frac{n}{4})$ the $L^{2}(\mathbb{T}^{n})$-supercritical regime. Analogous results for the MHD system were obtained by Wu \cite{W03}. We also mention logarithmic improvements beyond the critical threshold that was initiated by Tao \cite{T09} and extended by Barbato et al. \cite{BMR14} for the NS equations (see \cite{W11, Y18} for the case of the MHD system). Next, let us consider the stochastic generalized MHD system:
\begin{subequations}\label{stochastic GMHD}
\begin{align}
&du + [\nu_{1} (-\Delta)^{m_{1}} u + \text{div} (u\otimes u - b \otimes b) + \nabla \pi] dt = F_{1}(u) dB_{1}, \hspace{2mm} \nabla\cdot u = 0, \hspace{5mm} t > 0, \label{estimate 35}\\
&db + [\nu_{2} (-\Delta)^{m_{2}} b + \text{div} (u\otimes b - b\otimes u)]dt = F_{2} (b) dB_{2}, \hspace{10mm} \nabla\cdot b = 0, \hspace{5mm} t > 0, \label{estimate 36}
\end{align}
\end{subequations} 
where either $F_{k} \equiv 1$ and both $B_{k}$ for $k \in \{1,2\}$ are $G_{k}G_{k}^{\ast}$-Wiener processes on a probability space $(\Omega, \mathcal{F}, \textbf{P})$ with $G_{k}$ to be described subsequently or $F_{k}(x) = x$ and $B_{k}$ for $k \in \{1,2\}$ are both $\mathbb{R}$-valued Wiener processes; here we indicated an adjoint operator by an asterisk. In the former case, Flandoli and Romito \cite{FR08} proved the existence of a Leray-Hopf weak solution to the 3D stochastic classical NS equations (see \cite[Def. 3.3]{FR08} and \cite{GRZ09}); analogous result for the 3D stochastic classical MHD system was proven by the author \cite{Y19}. 
\begin{define}
If for any solution $(u, b, B_{1}, B_{2})$ and $(\tilde{u}, \tilde{b}, \tilde{B}_{1}, \tilde{B}_{2})$ to \eqref{stochastic GMHD} with same initial distributions, defined potentially on different filtered probability spaces, $\mathcal{L} (u,b) = \mathcal{L} (\tilde{u}, \tilde{b})$; i.e., they have same probability laws, then uniqueness in law holds for \eqref{stochastic GMHD}. If for any solutions $(u, b, B_{1}, B_{2})$ and $(\tilde{u}, \tilde{b}, B_{1}, B_{2})$ with common initial condition defined on same probability space, $(u,b)(t) = (\tilde{u}, \tilde{b})(t)$ for all $t$ $\textbf{P}$-almost surely ($\textbf{P}$-a.s.), then path-wise uniqueness holds for \eqref{stochastic GMHD}. Moreover, if a solution is adapted to the canonical left-continuous filtration generated by $(B_{1}, B_{2})$ and augmented by all the negligible sets, then it is a probabilistically strong solution. While uniqueness in law does not imply path-wise uniqueness (e.g., \cite[Exa. 2.2]{C03}), Yamada-Watanabe theorem states that path-wise uniqueness implies uniqueness in law, and path-wise uniqueness and existence of a probabilistically weak solution together imply existence of a probabilistically  strong solution. Conversely, Cherny's theorem \cite[The. 3.2]{C03} states that existence of a probabilistically  strong solution and uniqueness in law together imply path-wise uniqueness. 
\end{define} 
The stochastic classical MHD system has caught much attention: existence of global weak solution in case of an additive noise and multiplicative noise if $n=3$, along with path-wise uniqueness if $n=2$ as long as noise is Lipschitz \cite{S10, SS99}; ergodicity in case of an additive noise if $n =2$ \cite{BD07}, large deviation principle if $n =2$ \cite{CM10} (see also \cite{S21} on tamed stochastic MHD system). 
\begin{remark}\label{Remark 1.1}
At the level of path-wise uniqueness, the results in the stochastic and deterministic cases were more or less comparable. Hence, there has been considerable effort in the community of stochastic partial differential equations (PDEs) in fluid mechanics to prove uniqueness in law in the 3D case (e.g., \cite[p 878--879]{DD03}). Due to Cherny's theorem, this can imply path-wise uniqueness if the existence of a probabilistically strong solution can additionally be shown. However, as pointed out by Flandoli \cite[p. 84]{F08}, the latter result was absent in the literature for a long time. Indeed, the Leray-Hopf weak solution of the 3D stochastic classical NS equations that was constructed in \cite{FR08} was a probabilistically weak solution. Without path-wise uniqueness in hand, this result cannot imply the existence of a probabilistically strong solution via Yamada-Watanabe theorem. Working on the case of an additive noise path-wise will not solve this issue either because the standard approach of obtaining uniform bounds and relying on compactness to deduce a convergent subsequence will return a limit that depends on the fixed path. 
\end{remark} 
\noindent Next, let us discuss the convex integration technique, that was particularly extended to a probabilistic setting to remarkably prove not only non-uniqueness in law, which immediately implies path-wise non-uniqueness, but also the existence of a strong solution to the 3D stochastic classical NS equations in \cite{HZZ19} by Hofmanov$\acute{\mathrm{a}}$ et al., although at the level of analytically weak solution, not a Leray-Hopf weak solution. 

The origin of convex integration is accredited to the work of Nash \cite{N54} concerning isometric embeddings; it was Gromov who considered its work as part of homotopy-principle and established convex integration technique in \cite[Par. 2.4]{G86}.  M$\ddot{\mathrm{u}}$ller and $\acute{\mathrm{S}}$ver$\acute{\mathrm{a}}$k extended convex integration to Lipschitz mappings and obtained some unexpected solutions to certain PDEs \cite{MS98, MS03}. Effort to further advance convex integration technique was fueled by the famous open problem of Onsager's conjecture \cite{O49}, of which positive direction was solved in 1994 \cite{CET94, E94} but its negative direction, specifically that for any $\alpha \in [0, \frac{1}{3})$ there exists a solution $u(t) \in C^{\alpha}(\mathbb{T}^{3})$ for all $t$ to the 3D Euler equations that fails to conserve energy, had remained open for more than a decade. First, De Lellis and Sz$\acute{\mathrm{e}}$kelyhidi Jr. \cite{DS09} refined convex integration and proved the existence of a solution $u \in L_{t,x}^{\infty}$ to the $nD$ Euler equations with compact support for any $n \in \mathbb{N}\setminus \{1\}$, effectively extending the results of Scheffer \cite{S93} and Shnirelman \cite{S97} which proved same result but with regularity in $L_{t,x}^{2}$ for $n = 2$. This result was subsequently improved to the regularity level of $C_{t,x}$ by De Lellis and Sz$\acute{\mathrm{e}}$kelyhidi Jr. \cite{DS13},  $C_{t,x}^{\alpha}$ for $\alpha < \frac{1}{5}$ by Buckmaster et al. \cite{BDIS15}, and finally $C_{t,x}^{\alpha}$ for $\alpha < \frac{1}{3}$ by Isett \cite{I18}. By extending techniques from \cite{BDIS15, I18}, the authors in \cite{CDD18, D19} proved non-uniqueness of Leray-Hopf weak solution to the 3D generalized NS equations with $m_{1} < \frac{1}{5}$ and then $m_{1} < \frac{1}{3}$, respectively. Although an application of the convex integration to the classical NS equations was believed to be infeasible due to the viscous diffusion, by an addition of a new ingredient of  intermittency, Buckmaster and Vicol \cite{BV19a} proved the non-uniqueness of weak solution to the 3D classical NS equations. This result was extended to the full $L^{2}(\mathbb{T}^{3})$-supercritical regime; i.e., for any $m_{1} < \frac{5}{4}$ by Buckmaster et al. \cite{BCV18} and Luo and Titi \cite{LT20}. Concerning the MHD system, inspired by Taylor's conjecture \cite{T74}, Faraco et al. \cite{FLS21} adapted the approach of \cite{DS09} on ideal 3D MHD system and proved the existence of infinitely many bounded solutions with compact support in space-time that violate conservation of total energy and cross helicity but preserve magnetic helicity. Independently, Beekie et al. \cite{BBV21} employed the approach of \cite{BV19a} also on ideal 3D MHD system and proved that there exists $\beta > 0$ such that there are weak solutions $u,b \in C_{t}\dot{H}_{x}^{\beta}$ that do not conserve magnetic helicity and their total energy and cross helicity are non-trivial non-constant functions of time.  

Far-reaching consequences of convex integration have included the stochastic case. Using techniques from \cite{DS10}, Chiodaroli et al. \cite{CFF19} and Breit et al. \cite{BFH20} proved path-wise non-uniqueness of certain stochastic Euler equations (see also \cite{HZZ20}). Partially inspired by the ideas from \cite{BV19b}, Hofmanov$\acute{\mathrm{a}}$ et al. \cite{HZZ19} proved non-uniqueness in law of the 3D stochastic classical NS equations forced by additive or linear multiplicative noise. This inspired the author to extend to various models: 3D stochastic generalized NS equations with $m_{1} \in (\frac{13}{20}, \frac{5}{4})$ in \cite{Y20a} and $m_{1} \in (0, \frac{1}{2})$ in \cite{Y21c}; 2D stochastic generalized NS equations with $m_{1} \in (0,1)$ in \cite{Y20c}. 

Let us explain some of the motivation of this manuscript. The results of non-uniqueness in law for the 2D and 3D stochastic generalized NS equations in \cite{Y20a, Y20c} were successfully extended to the Boussinesq system by the author in \cite{Y21a} which we recall here for convenience. We denote by $\theta: \mathbb{R}_{+} \times \mathbb{T}^{n} \mapsto \mathbb{R}$ the temperature scalar field so that $nD$ stochastic generalized Boussinesq system reads 
\begin{subequations}\label{stochastic generalized Boussinesq}
\begin{align}
&du + [\nu_{1} (-\Delta)^{m_{1}} u + \text{div} (u\otimes u) + \nabla \pi] dt = \theta e^{n} dt + F_{1}(u) dB_{1}, \hspace{2mm} \nabla\cdot u = 0, \hspace{5mm} t > 0,  \label{estimate 10}\\
&d\theta + [-\nu_{2}\Delta \theta + \text{div} (u \theta)]dt = F_{2} (\theta) dB_{2},  \hspace{50mm} t > 0; \label{estimate 11}
\end{align}
\end{subequations} 
due to a technical reason that is explained in \cite[Rem. 2.1]{Y21a}, we chose $m_{2} = 1$ in \eqref{estimate 11} in comparison to \eqref{estimate 36}. For simplicity, let us consider the case $n = 2$ and an additive noise; i.e., $F_{k} \equiv 1$ and $B_{k}$ for $k \in \{1,2\}$ are certain $G_{k}G_{k}^{\ast}$-Wiener processes on a probability space $(\Omega, \mathcal{F}, \textbf{P})$ with $(\mathcal{F}_{t})_{t\geq 0}$ as the canonical left-continuous filtration generated by $(B_{1}, B_{2})$ augmented by all the $\textbf{P}$-negligible sets. Then, summarizing the work in \cite{Y21a} very briefly, at the crucial step of convex integration, one can consider a pair of Ornstein-Uhlenbeck processes $(z_{1}, z_{2})$ that satisfies 
\begin{subequations}\label{estimate 395}
\begin{align}
& dz_{1} + \nu_{1}(-\Delta)^{m_{1}} z_{1} dt + \nabla \pi_{1} dt = dB_{1}, \hspace{3mm} \nabla\cdot z_{1} = 0, \hspace{5mm}  t > 0, \hspace{2mm} z_{1}(0, x) \equiv 0,   \\
& dz_{2} - \nu_{2}\Delta z_{2} dt = dB_{2}, \hspace{43mm} t > 0, \hspace{2mm} z_{2}(0,x) \equiv 0, 
\end{align}
\end{subequations} 
(see \cite[Equ. (31)]{Y21a}), and define $v \triangleq u - z_{1}$ which will lead to the following stochastic Boussineq-Reynolds system: for $q \in \mathbb{N}_{0} \triangleq \mathbb{N} \cup \{0\}$, 
\begin{subequations}\label{estimate 1}
\begin{align}
& \partial_{t} v_{q} + \nu_{1}(-\Delta)^{m_{1}} v_{q} + \text{div} ((v_{q} + z_{1}) \otimes (v_{q} + z_{1})) + \nabla \pi_{q} = \theta_{q} e^{2} + \text{div} \mathring{R}_{q}, \hspace{1mm} \nabla\cdot v_{q} = 0,  \label{estimate 2}\\
& d\theta_{q} + [-\nu_{2}\Delta \theta_{q} + \text{div} ((v_{q} + z_{1} ) \theta_{q} )] dt = dB_{2}, \label{estimate 3}
\end{align}
\end{subequations} 
where $\mathring{R}_{q}$ will be a trace-free symmetric matrix (see \cite[Equ. (45)]{Y21a}). The critical point here is that one only has to add the Reynolds stress term on the equation of $v_{q}$ and not $\theta_{q}$; i.e., it suffices to perform convex integration only on the velocity vector field and not temperature scalar field. Indeed, this is because one can explicitly construct $(v_{q}, \theta_{q})$ at level $q = 0$ that solves \eqref{estimate 1} and satisfies certain inductive estimates (see \cite[Pro. 4.7]{Y21a} for details), assume the inductive hypothesis at the level $q$, construct only $v_{q+1}$ explicitly, substitute such $v_{q+1}$ into \eqref{estimate 3} at level $q+1$, observe that the resulting equation is a linear stochastic PDE and thus deduce $\theta_{q+1}$ uniquely for any given initial condition $\theta(0,x)$, substitute this $\theta_{q+1}$ this time into \eqref{estimate 2} which determines $\mathring{R}_{q+1}$ uniquely, resume the convex integration on only \eqref{estimate 2} following the previous works \cite{Y20a, Y20c} on the NS equations and prove that such $v_{q+1}$ particularly satisfies 
\begin{equation}\label{estimate 4} 
\lVert v_{q+1} (t) - v_{q}(t) \rVert_{L_{x}^{2}} \leq M_{0}(t)^{\frac{1}{2}} \delta_{q+1}^{\frac{1}{2}}
\end{equation} 
where $M_{0}(t)$ is a certain function and $\delta_{q+1} \searrow 0$ as $q\nearrow + \infty$ (see \eqref{estimate 93}-\eqref{estimate 94}). Concerning the limiting solution as $q\nearrow + \infty$, \eqref{estimate 4} can allow one to show that $\{v_{q}\}_{q=0}^{\infty}$ is Cauchy in $C_{t}\dot{H}^{\gamma} (\mathbb{T}^{2})$ for sufficiently small $\gamma$ (see \cite[Equ. (63)]{Y21a}). Remarkably, it turns out that \eqref{estimate 4} can deduce similar Cauchy-ness for $\{\theta_{q}\}_{q=0}^{\infty}$. Indeed, \eqref{estimate 3} shows that $\theta_{q+1} - \theta_{q}$ satisfies 
\begin{equation}\label{estimate 5}
\partial_{t} (\theta_{q+1} - \theta_{q}) - \nu_{2}\Delta (\theta_{q+1} - \theta_{q}) + (v_{q+1} + z_{1}) \cdot \nabla (\theta_{q+1} - \theta_{q}) + (v_{q+1} - v_{q}) \cdot \nabla \theta_{q} = 0,  
\end{equation} 
where the noise fortunately canceled out because it is only additive and hence $L^{2}(\mathbb{T}^{2})$-inner products with $\theta_{q+1} - \theta_{q}$ lead to  
\begin{equation}\label{estimate 6}
\frac{1}{2} \partial_{t} \lVert \theta_{q+1} - \theta_{q} \rVert_{L_{x}^{2}}^{2} + \nu_{2} \lVert \theta_{q+1} - \theta_{q} \rVert_{\dot{H}_{x}^{1}}^{2}  = \int_{\mathbb{T}^{n}} (v_{q+1} - v_{q}) \cdot \nabla (\theta_{q+1} - \theta_{q}) \theta_{q}dx 
\end{equation} 
(see \cite[Equ. (97)-(98)]{Y21a}). Due to $\nu_{2}\lVert \theta_{q+1} - \theta_{q} \rVert_{\dot{H}_{x}^{1}}^{2}$ from diffusion that can handle $\nabla (\theta_{q+1} - \theta_{q})$ on the right hand side of \eqref{estimate 6}, using the well-known $L_{\omega}^{p}C_{t}L_{x}^{p}$-estimate of $\theta_{q}$ for any $p \in [1, \infty)$ and \eqref{estimate 4} allows one to prove that $\{\theta_{q}\}_{q =0}^{\infty}$ is Cauchy in $\cap_{p \in [1,\infty)} L_{\omega}^{p} C_{t} L_{x}^{p} \cap L_{\omega}^{p} L_{t}^{2} \dot{H}_{x}^{1}$ and deduce a limiting solution $\lim_{q\to\infty} \theta_{q}  \triangleq \theta \in \cap_{p \in [1,\infty)} L_{\omega}^{p} C_{t} L_{x}^{p} \cap L_{\omega}^{p} L_{t}^{2} \dot{H}_{x}^{1}$ which is $(\mathcal{F}_{t})_{t\geq 0}$-adapted because each $\theta_{q}$ is $(\mathcal{F}_{t})_{t\geq 0}$-adapted (see \cite[Proof of The. 2.1]{Y21a} for details). Analogous attempt on the magnetic field would be to consider  
\begin{equation}\label{estimate 39}
db_{q} + [- \nu_{2} \Delta b_{q} + \text{div} ((v_{q} + z_{1}) \otimes b_{q} - b_{q} \otimes (v_{q} + z_{1})] dt = dB_{2} 
\end{equation} 
which leads to 
\begin{align*}
& \partial_{t} (b_{q+1} - b_{q}) - \nu_{2}\Delta (b_{q+1} - b_{q}) + (v_{q+1} + z_{1}) \cdot \nabla (b_{q+1} - b_{q}) + (v_{q+1} - v_{q}) \cdot \nabla b_{q} \\
& \hspace{21mm} - (b_{q+1} \cdot \nabla) (v_{q+1} - v_{q}) - (b_{q+1} - b_{q}) \cdot \nabla (v_{q}+ z_{1}) = 0 
\end{align*} 
and therefore, $L^{2}(\mathbb{T}^{2})$-inner products with $b_{q+1} - b_{q}$ give 
\begin{align}
& \frac{1}{2} \partial_{t} \lVert b_{q+1} - b_{q} \rVert_{L_{x}^{2}}^{2} + \nu_{2}\lVert b_{q+1} - b_{q} \rVert_{\dot{H}_{x}^{1}}^{2} = - \int_{\mathbb{T}^{2}} (v_{q+1} - v_{q}) \cdot \nabla b_{q} \cdot (b_{q+1} - b_{q}) \label{estimate 7} \\
& \hspace{27mm} - (b_{q+1} \cdot \nabla) (v_{q+1} - v_{q}) \cdot (b_{q+1} - b_{q}) - (b_{q+1} - b_{q}) \cdot \nabla v_{q} \cdot (b_{q+1} - b_{q}) dx,  \nonumber 
\end{align} 
and this estimate is simply too difficult to close with the only help from \eqref{estimate 4} due to a lack of $L_{\omega}^{p}C_{t}L_{x}^{p}$-estimate of $b_{q}$ for all $p \in [1, \infty)$ in contrast to $\theta_{q}$. One might attempt to obtain $L_{\omega}^{p}C_{t}L_{x}^{p}$-estimate of $b_{q}$ for some $p \in [1, \infty)$ but will quickly realize that it requires coupling with an estimate on $v_{q}$ in contrast to the case of $\theta_{q}$ in the Boussinesq system and that brings about many troubles due to the presence of the Reynolds stress $\mathring{R}_{q}$ in the equation of $v_{q}$. This difficulty seems to be absent in the deterministic case because, in pursuit of merely proving non-uniqueness at the regularity level of a weak solution, Cauchy-ness is unnecessary, although certainly sufficient. E.g., in the deterministic case, one may consider 
\begin{subequations}
\begin{align}
& \partial_{t} u_{q} + (u_{q} \cdot \nabla) u_{q} + \nabla \pi_{q} + \nu_{1}(-\Delta)^{m_{1}}u_{q} = (b_{q} \cdot \nabla ) b_{q}, \hspace{3mm} \nabla\cdot u_{q} = 0, \label{estimate 8}\\
& \partial_{t} b_{q} + (u_{q} \cdot \nabla) b_{q} - \nu_{2}\Delta b_{q} = (b_{q} \cdot \nabla) u_{q} + \mathring{R}_{q}, \hspace{13mm} \nabla\cdot b_{q} = 0, \label{estimate 9}
\end{align}
\end{subequations} 
explicitly construct $b_{q+1}$ via convex integration, substitute this $b_{q+1}$ to \eqref{estimate 8} at level $q+1$, consider it as essentially just the NS equations with an external force, use the classical compactness argument to deduce a convergent subsequence such that its limit $u_{q+1}$ solves \eqref{estimate 8}, substitute this $u_{q+1}$ back into \eqref{estimate 9} at level $q+1$, thereby determine $\mathring{R}_{q+1}$ in \eqref{estimate 9}, and try to complete the convex integration scheme. Alas, as elaborated in Remark \ref{Remark 1.1}, such a process of going through compactness argument to deduce a convergent subsequence is not desirable in the stochastic case because $(\mathcal{F}_{t})_{t\geq 0}$-adaptedness becomes lost and the resulting solution is probabilistically weak, not probabilistically strong. 
 
We are now convinced that in order to attain non-uniqueness result for the stochastic generalized MHD system \eqref{stochastic GMHD}, we must employ convex integration on both equations of the velocity and the magnetic fields. Thus, we turn to the approach of \cite{BBV21, FLS21}. An immediate difficulty is that those works were on the ideal MHD system. However, there is hope considering that some convex integration schemes on the Euler equations such as \cite{BDIS15} and \cite[Sec. 5]{BV19b} were extended to the viscous case, although significantly weaker than a full Laplacian, in \cite{CDD18, D19, Y21c}. Indeed, the author actually was able to prove non-uniqueness in law of the stochastic generalized MHD system \eqref{stochastic GMHD} if $\nu_{1} > 0, \nu_{2} > 0, m_{1}, m_{2} \in (0, \frac{3}{4})$ by adapting the approach of \cite{BBV21}. After such computations were completed, the author was informed that Chen and Liu \cite{CL21} made a similar comment very recently that the proof of \cite{BBV21} can be extended to the case $\nu_{1} \geq 0, m_{1} \in [0, \frac{3}{4})$ but $\nu_{2} = 0$, proved non-uniqueness of a certain 3D deterministic elastodynamics system and claimed similar result for 3D deterministic MHD system with $\nu_{1} \geq 0, m_{1} \in [0, 1)$ but $\nu_{2} = 0$. Because we believe that the proof of extending the convex integration scheme in \cite{BBV21} to the case $\nu_{1} > 0, \nu_{2} > 0, m_{1}, m_{2} \in (0, \frac{3}{4})$ is of independent interest mathematically, we will leave its main idea in the Appendix. The convex integration scheme within \cite{CL21} involves high-regularity estimate of its Reynolds stress; see ``$\lVert R_{q}^{i} \rVert_{C_{x,t}^{1}} + \lVert R_{q}^{v} \rVert_{C_{x,t}^{1}} \leq \lambda_{q}^{10}$'' in \cite[Equ. (2.7)]{CL21} and it was explained in previous works (e.g., \cite[Rem. 1.2]{Y20c} and \cite[Sec. 1.2]{Y21c}) that such estimate seems difficult in the stochastic setting because an analogous Reynolds stress for us will involve Brownian motion that has no regularity $C_{t}^{\alpha}$ for $\alpha \geq \frac{1}{2}$; see e.g., $z_{k}$ in \eqref{estimate 76}  that has temporal regularity of $C_{t}^{\frac{1}{2} - \delta}$ for $\delta > 0$ and the presence of such $z_{k}$, $k \in \{1,2\}$, in \eqref{estimate 426}. Relying on techniques from previous works such as \cite{HZZ19, Y20a, Y20c, Y21a}, we will overcome this difficulty and prove non-uniqueness in law of the 3D stochastic generalized MHD system with $\nu_{1}, \nu_{2}  > 0, m_{1}, m_{2} \in (0,1)$ at the level of probabilistically strong solution and thereby effectively show that the strategy of proving its path-wise uniqueness via a combination of uniqueness in law and existence of a strong solution and relying on Cherny's theorem fails.   

\section{Statement of main results}\label{Section 2}
Hereafter, for simplicity we assume $\nu_{1} = \nu_{2} = 1$. Our first set of results Theorems \ref{Theorem 2.1}-\ref{Theorem 2.2} concerns the additive noise case. The precise conditions on $G_{k}$ in Theorems \ref{Theorem 2.1}-\ref{Theorem 2.2} will be given in Section \ref{Subsection 3.1}. 
\begin{theorem}\label{Theorem 2.1} 
Suppose that $m_{k} \in (0,1), F_{k} \equiv 1,B_{k}$ is a $G_{k}G_{k}^{\ast}$-Wiener process and 
\begin{equation}\label{estimate 18}
Tr ((-\Delta)^{\frac{5}{2} - m_{k} + 2\sigma}  G_{k}G_{k}^{\ast}) < \infty 
\end{equation} 
for some $\sigma > 0$ for both $k \in \{1,2\}$. Then, given $T> 0, K > 1$, and $\kappa \in (0,1)$, there exist $\gamma \in (0,1)$ and a $\textbf{P}$-a.s. strictly positive stopping time $\mathfrak{t}$ such that 
\begin{equation}\label{estimate 19}
\textbf{P} ( \{ \mathfrak{t} \geq T \}) > \kappa 
\end{equation}
and the following is additionally satisfied. There exists a pair of $\{\mathcal{F}_{t}\}_{t\geq 0}$-adapted processes $(u,b)$ that is a weak solution of \eqref{stochastic GMHD} starting from a deterministic initial condition $(u^{\text{in}}, b^{\text{in}})$, satisfies 
\begin{equation}\label{estimate 20}
\text{esssup}_{\omega \in \Omega} \lVert u(\omega) \rVert_{C_{\mathfrak{t}} \dot{H}_{x}^{\gamma}} < \infty, \hspace{3mm} \text{esssup}_{\omega \in \Omega} \lVert b(\omega) \rVert_{C_{\mathfrak{t}} \dot{H}_{x}^{\gamma}} < \infty,
\end{equation}
and on a set $\{\mathfrak{t} \geq T \}$, 
\begin{equation}\label{estimate 21}
 \lVert b(T) \rVert_{L_{x}^{2}} > K [ ( \lVert u^{\text{in}} \rVert_{L_{x}^{2}} + \lVert b^{\text{in}} \rVert_{L_{x}^{2}}) + \sum_{k=1}^{2} \sqrt{ T Tr ( G_{k}G_{k}^{\ast} )}].
\end{equation}
\end{theorem} 

\begin{theorem}\label{Theorem 2.2} 
Suppose that $m_{k} \in (0,1), F_{k} \equiv 1$, $B_{k}$ is a $G_{k}G_{k}^{\ast}$-Wiener process, and \eqref{estimate 18} holds for some $\sigma > 0$ for both $k \in \{1,2\}$. Then non-uniqueness in law for \eqref{stochastic GMHD} holds on $[0,\infty)$. Moreover, for all $T> 0$ fixed, non-uniqueness in law holds for \eqref{stochastic GMHD} on $[0,T]$. 
\end{theorem} 

Our second set of results concerns the linear multiplicative noise case. Hofmanov$\acute{\mathrm{a}}$ et al. in \cite{HZZ19} provided a very nice approach in the case of the NS equations such that analogous computations, some of them being identical, to the additive case can imply the desired result in the linear multiplicative case. Surprisingly, new difficulties arise in the case of the MHD system. Let us sketch its ideas and continue furthermore in Remark \ref{Remark 5.1}. 
\begin{remark}\label{Remark 2.1}
In short, in the linear multiplicative case, there is a well-known transformation that can turn stochastic PDEs to random PDEs; i.e., 
\begin{equation}\label{estimate 40}
v \triangleq \Upsilon_{1}^{-1} u \text{ where } \Upsilon_{1} \triangleq e^{B_{1}} \hspace{1mm} \text{ and } \hspace{1mm}  \Xi \triangleq \Upsilon_{2}^{-1}b \text{ where } \Upsilon_{2} \triangleq e^{B_{2}} 
\end{equation}
that turns \eqref{stochastic GMHD} with $F_{k} (x) = x$ and $B_{k}$ being an $\mathbb{R}$-valued Wiener process, $k \in \{1,2\}$, to 
\begin{subequations} 
\begin{align}
& \partial_{t} v + \frac{1}{2} v + (-\Delta)^{m_{1}} v + \text{div} ( \Upsilon_{1} v \otimes v - \Upsilon_{1}^{-1} \Upsilon_{2}^{2} \Xi \otimes \Xi) + \Upsilon_{1}^{-1} \nabla \pi = 0,\label{estimate 22} \\
& \partial_{t} \Xi + \frac{1}{2} \Xi + (-\Delta)^{m_{2}} \Xi+ \text{div} (\Upsilon_{1} v \otimes \Xi - \Upsilon_{1}\Xi \otimes v) = 0. \label{estimate 23}
\end{align}
\end{subequations}
Now the corresponding nonlinear term for the NS equations is only $\text{div} (\Upsilon_{1} v \otimes v)$. One of the most technical terms in the Reynolds stress estimate within convex integration is an oscillation term and those corresponding to \eqref{estimate 22}-\eqref{estimate 23} are of the form 
\begin{subequations}
\begin{align}
& \text{div} ( \Upsilon_{1,l} w_{q+1}^{p} \otimes w_{q+1}^{p} - \Upsilon_{1,l}^{-1} \Upsilon_{2,l}^{2} d_{q+1}^{p} \otimes d_{q+1}^{p} + \mathring{R}_{l}^{v}) + \partial_{t} w_{q+1}^{t} , \label{estimate 24}\\
& \text{div} (\Upsilon_{1,l} w_{q+1}^{p} \otimes d_{q+1}^{p} - \Upsilon_{1,l} d_{q+1}^{p} \otimes w_{q+1}^{p} + \mathring{R}_{l}^{\Xi}) + \partial_{t} d_{q+1}^{t}, \label{estimate 25}
\end{align}
\end{subequations}
(see \eqref{estimate 396} and \eqref{estimate 397}) where we extended $\Upsilon_{k}$ from \eqref{estimate 40} to $t < 0$ by the value at $t = 0$ and mollified it to obtain 
\begin{equation}\label{estimate 292}
\Upsilon_{k,l} \triangleq \Upsilon_{k} \ast_{t} \vartheta_{l}, \hspace{2mm} k \in \{1,2\}
\end{equation}
with $\{\vartheta_{l} \}_{l > 0}$, specifically $\vartheta_{l} (\cdot) \triangleq \frac{1}{l} \vartheta (\frac{\cdot}{l})$, being a family of standard mollifiers on $\mathbb{R}$ with mass one and compact support in $\mathbb{R}_{+}$ (see \eqref{estimate 291}, \eqref{estimate 334}, and  \eqref{estimate 331} for definitions of others such as $w_{q+1}^{p}, d_{q+1}^{p}, \mathring{R}_{l}^{v}, \mathring{R}_{l}^{\Theta}, w_{q+1}^{t}$, and $d_{q+1}^{t}$). In the case of the NS equations, which corresponds only to $\text{div} (\Upsilon_{1,l} w_{q+1}^{p} \otimes w_{q+1}^{p} + \mathring{R}_{l}^{v}) + \partial_{t} w_{q+1}^{t}$, the oscillation term in the additive case was actually $\text{div} (w_{q+1}^{p} \otimes w_{q+1}^{p} + \mathring{R}_{l}^{v}) + \partial_{t} w_{q+1}^{t}$ (see \cite[Equ. (4.48)]{HZZ19}) where $w_{q+1}^{p} = \sum_{\xi \in \Lambda} a_{\xi} W_{\xi}$ and thus Hofmanov$\acute{\mathrm{a}}$ et al. strategically defined $w_{q+1}^{p}$ in the linear multiplicative case as $\sum_{\xi \in \Lambda} \bar{a}_{\xi} W_{\xi}$ where $\bar{a}_{\xi} \triangleq \Upsilon_{1,l}^{-\frac{1}{2}} a_{\xi}$ (see \cite[Equ. (6.20)]{HZZ19}). This effectively reduces the oscillation term in the linear multiplicative case to the oscillation term in the additive case, specifically $\Upsilon_{1,l} (\sum_{\xi \in \Lambda} \bar{a}_{\xi} W_{\xi}) \otimes (\sum_{\xi \in \Lambda}\bar{a}_{\xi} W_{\xi})$ to $(\sum_{\xi \in \Lambda} a_{\xi} W_{\xi}) \otimes (\sum_{\xi \in \Lambda} a_{\xi} W_{\xi})$, and makes their estimates in the additive case  become directly applicable in the linear multiplicative case. A glance at \eqref{estimate 24}-\eqref{estimate 25} shows that there is no way to reduce all four nonlinear terms therein to the corresponding oscillation terms in its additive case. Indeed, in order to cancel out $\Upsilon_{1,l}$ in $\Upsilon_{1,l} w_{q+1}^{p} \otimes w_{q+1}^{p}$ of \eqref{estimate 24}, we will have to ask $w_{q+1}^{p}$ to eliminate $\Upsilon_{1,l}^{-\frac{1}{2}}$. With that fixed, in order to cancel out $\Upsilon_{1,l}$ in $\Upsilon_{1,l} w_{q+1}^{p} \otimes d_{q+1}^{p}$ and $\Upsilon_{1,l} d_{q+1}^{p} \otimes w_{q+1}^{p}$ in \eqref{estimate 25}, we will have to ask $d_{q+1}^{p}$ to eliminate $\Upsilon_{1,l}^{-\frac{1}{2}}$ as well. However, with such a choice on $w_{q+1}^{p}$ and $d_{q+1}^{p}$, $\Upsilon_{1,l}^{-1} \Upsilon_{2,l}^{2} $ within $\Upsilon_{1,l}^{-1} \Upsilon_{2,l}^{2} d_{q+1}^{p} \otimes d_{q+1}^{p}$ of \eqref{estimate 24} does not cancel out; in fact, it becomes $\Upsilon_{1,l}^{-2} \Upsilon_{2,l}^{2}$. One easy way to get around this problem is to assume that $\Upsilon_{1} \equiv \Upsilon_{2}$ so that $\Upsilon_{1,l}^{-2} \Upsilon_{2,l}^{2}$ cancels out by itself; unfortunately, that will require compromising to $B_{1} \equiv B_{2}$. 
\end{remark}
We overcome this difficulty to achieve the desired result when $B_{1}$ is not identically equal to $B_{2}$; with details to be described in Remark \ref{Remark 5.1}, we now present our second set of results. 

\begin{theorem}\label{Theorem 2.3} 
Suppose that $F_{k}(x) = x$ and $B_{k}$ is an $\mathbb{R}$-valued Wiener process on $(\Omega, \mathcal{F}, \textbf{P})$ for both $k \in \{1,2\}$. Then, given $T > 0, K > 1$, and $\kappa \in (0,1)$, there exist $\gamma \in (0,1)$ and a $\textbf{P}$-a.s. strictly positive stopping time $\mathfrak{t}$ such that \eqref{estimate 19} holds and the following is additionally satisfied. There exists a pair of $\{\mathcal{F}_{t}\}_{t\geq 0}$-adapted processes $(u,b)$ that is a weak solution to \eqref{stochastic GMHD} starting from a deterministic initial condition $(u^{\text{in}}, b^{\text{in}})$, satisfies \eqref{estimate 20}, and on a set $\{\mathfrak{t} \geq T \}$, 
\begin{equation}\label{estimate 26}
\lVert b(T) \rVert_{L_{x}^{2}} > K e^{\frac{T}{2}} (\lVert u^{\text{in}} \rVert_{L_{x}^{2}} + \lVert b^{\text{in}} \rVert_{L_{x}^{2}}). 
\end{equation} 
\end{theorem} 

\begin{theorem}\label{Theorem 2.4} 
Suppose that $F_{k}(x) = x$ and $B_{k}$ is an $\mathbb{R}$-valued Wiener process on $(\Omega, \mathcal{F}, \textbf{P})$ for both $k \in \{1,2\}$. Then non-uniqueness in law holds for \eqref{stochastic GMHD} on $[0, \infty)$. Moreover, for all $T> 0$ fixed, non-uniqueness in law holds for \eqref{stochastic GMHD} on $[0,T]$. 
\end{theorem} 

\begin{remark}
To the best of the author's knowledge, this is the first result of non-uniqueness in law, first non-uniqueness even path-wise actually, and the first construction of a probabilistically strong solution to the stochastic MHD system. We point out that because the MHD system with zero magnetic field reduces to the NS equations, non-uniqueness results of the NS equations already imply that of the MHD system if one takes zero magnetic field; in contrast, our non-uniqueness results hold with a non-zero magnetic field (see \eqref{estimate 21} and \eqref{estimate 26}). Very recently, \cite{HZZ21} gave different results concerning non-uniqueness in law of the 3D stochastic classical NS equations, partially inspired by \cite{BMS21}. We can pursue their type of results for the MHD system as well; such results, even if attained, will not imply Theorems \ref{Theorem 2.1}-\ref{Theorem 2.4}, as pointed out in \cite[p. 3]{HZZ19}. Finally, it is a natural question to ask if Theorems \ref{Theorem 2.1}-\ref{Theorem 2.4} can be improved to $m_{1}, m_{2} < \frac{5}{4}$ as in \cite{Y20a} on the stochastic generalized NS equations; at the time of writing this manuscript, the author was not able to achieve this task. 
\end{remark} 

In what follows, we provide a minimum amount of notations and setups of convex integration in Section \ref{Preliminaries}, prove Theorems \ref{Theorem 2.1}-\ref{Theorem 2.4} in Sections \ref{Section 4}-\ref{Section 5}, and provide further preliminaries in the Appendix for convenience. Our proof is inspired by many previous works,  especially \cite{BBV21, CL21, HZZ19, Y21a}, and there are some similarities to the work on 2D deterministic generalized NS equations by Luo and Qu \cite{LQ20} and its stochastic counterpart in \cite{Y20c}. 

\section{Preliminaries}\label{Preliminaries}
\subsection{Notations and assumptions}\label{Subsection 3.1}
We write $A \lesssim_{a,b} B$ and $A \approx_{a,b} B$ to imply that $A \leq C(a,b) B$ and $A = C(a,b)B$ for some constant $C = C(a,b) \geq 0$, respectively. We also write $A \overset{(\cdot)}{\lesssim}B$ to indicate that this inequality is due to an equation $(\cdot)$. We define $\mathbb{P} \triangleq \text{Id} - \nabla \Delta^{-1} \text{div}$ as the Leray projection onto the space of divergence-free vector fields, and $\mathbb{P}_{\leq r}$ to be the Fourier operator with a Fourier symbol of $1_{\lvert k \rvert \leq r} (k)$ and $\mathbb{P}_{> r} \triangleq \text{Id} - \mathbb{P}_{\leq r}$. Let us denote a tensor product by $\otimes$ while trace-free tensor products by $\mathring{\otimes}$. We write for $p \in [1, \infty]$, 
\begin{equation}\label{estimate 42}
\lVert g \rVert_{L^{p}} \triangleq \lVert g \rVert_{L_{t}^{\infty} L_{x}^{p}} \text{ and } \lVert g \rVert_{C_{t,x}^{N}} \triangleq \sum_{0 \leq k + \lvert \alpha \rvert \leq N} \lVert \partial_{t}^{k} D^{\alpha} g \rVert_{L^{\infty}},
\end{equation} 
where $k \in \mathbb{N}_{0}$ and $\alpha$ is a multi-index. We define $L_{\sigma}^{2} \triangleq \{f  \in L^{2}(\mathbb{T}^{3}): \nabla\cdot f = 0, \int_{\mathbb{T}^{3}} f dx = 0\}$. For any Polish space $H$, we write $\mathcal{B}(H)$ to denote the $\sigma$-algebra of Borel sets in $H$. We denote a mathematical expectation with respect to (w.r.t.) any probability measure $P$ by $\mathbb{E}^{P}$. We denote by $\langle \cdot, \cdot \rangle$ the $L^{2}(\mathbb{T}^{3})$-inner product while $\langle \langle A, B \rangle \rangle$ a quadratic variation of $A$ and $B$, and $\langle \langle A \rangle \rangle \triangleq \langle \langle A, A \rangle \rangle$. For $t \geq 0$, we let 
\begin{equation}\label{estimate 43}
\Omega_{t} \triangleq \{ (f,g): f, g \in C([t, \infty); H^{-3} (\mathbb{T}^{3})) \cap L_{\text{loc}}^{\infty} ([t, \infty); L_{\sigma}^{2}) \}. 
\end{equation} 
We define $\xi \triangleq (\xi_{1}, \xi_{2}): \Omega_{0} \mapsto H^{-3} (\mathbb{T}^{3}) \times H^{-3} (\mathbb{T}^{3})$ to be the canonical process by $\xi_{t}(\omega) \triangleq \omega(t)$. We denote by $\mathcal{P}(\Omega_{0})$ the set of all probability measures on $(\Omega_{0}, \mathcal{B})$ where $\mathcal{B}$ is the Borel $\sigma$-algebra of $\Omega_{0}$ from the topology of locally uniform convergence on $\Omega_{0}$. We equip $\Omega_{t}$ with Borel $\sigma$-algebra $\mathcal{B}^{t} \triangleq \sigma \{ \xi(s): s \geq t \}$, and additionally define for $t \geq 0$, $\mathcal{B}_{t}^{0} \triangleq \sigma \{\xi(s): s \leq t\}$ and $\mathcal{B}_{t} \triangleq \cap_{s > t} \mathcal{B}_{s}^{0}$. For any Hilbert spaces $U_{1}$ and $U_{2}$, we denote by $L_{2} (U_{k}, L_{\sigma}^{2}), k \in \{1,2 \}$, the spaces of all Hilbert-Schmidt operators from $U_{k}$ to $L_{\sigma}^{2}$ with norms $\lVert \cdot \rVert_{L_{2}(U_{k}, L_{\sigma}^{2})}$. We impose on $G_{k}: L_{\sigma}^{2} \mapsto L_{2} (U_{k}, L_{\sigma}^{2})$ to be $\mathcal{B}(L_{\sigma}^{2})/ \mathcal{B}(L_{2}(U_{k}, L_{\sigma}^{2}))$-measurable and satisfy for all $\phi, \psi_{j}, \psi \in C^{\infty} (\mathbb{T}^{3}) \cap L_{\sigma}^{2}$ such that $\lim_{j\to\infty} \lVert \psi_{j} - \psi \rVert_{L_{x}^{2}} = 0$
\begin{equation}
\lVert G_{k} (\phi) \rVert_{L_{2} (U_{k}, L_{\sigma}^{2})} \leq C (1+ \lVert \phi \rVert_{L_{x}^{2}}), \hspace{3mm} \lim_{j\to\infty} \lVert G_{k} (\psi_{j})^{\ast} \phi - G_{k} (\psi)^{\ast} \phi \rVert_{U_{k}} = 0. 
\end{equation} 
Finally, we assume the existence of Hilbert spaces $\tilde{U}_{1}, \tilde{U}_{2}$ such that the embeddings of $U_{k} \hookrightarrow \tilde{U}_{k}, k \in \{1,2\}$, are Hilbert-Schmidt. Let us define 
\begin{equation}\label{estimate 398}  
\bar{\Omega} \triangleq \prod_{k=1}^{2} C( [0,\infty); H^{-3} (\mathbb{T}^{3}) \times \tilde{U}_{k}) \cap L_{\text{loc}}^{\infty} ([0,\infty); L_{\sigma}^{2} \times \tilde{U}_{k})
\end{equation} 
and $\mathcal{P}(\bar{\Omega})$ to be the set of all probability measures on $(\bar{\Omega}, \bar{\mathcal{B}})$ where $\bar{\mathcal{B}}$ is the Borel $\sigma$-algebra of $\bar{\Omega}$. Analogously, we define  the canonical process on $\bar{\Omega}$ to be $(\xi, \zeta): \bar{\Omega} \mapsto \prod_{k=1}^{2} H^{-3} (\mathbb{T}^{3}) \times \tilde{U}_{k}$ by $(\xi, \zeta)_{t} (\omega) \triangleq \omega(t)$. Finally, we also define for $t \geq 0$, 
\begin{equation}
\bar{\mathcal{B}}^{t} \triangleq \sigma \{ (\xi, \zeta) (s): s \geq t \}, \hspace{1mm} \bar{\mathcal{B}}_{t}^{0} \triangleq \sigma\{ (\xi, \zeta)(s): s \leq t \}, \hspace{1mm}\text{ and } \hspace{1mm} \bar{\mathcal{B}}_{t} \triangleq \cap_{s > t} \bar{\mathcal{B}}_{s}^{0}.
\end{equation} 
 
\subsection{Convex integration}\label{Subsection 3.2}
Convex integration scheme on MHD system requires the following two geometric lemmas from \cite[Pro. 2.2-2.3]{CL21}, originally from \cite[Lem. 4.1-4.2]{BBV21}.  
\begin{lemma}\label{Lemma 3.1}
\rm{(\cite[Pro. 2.3]{CL21}, cf. \cite[Lem. 4.1]{BBV21})} There exists a set $\Lambda_{\Xi} \subset \mathbb{S}^{2} \cap \mathbb{Q}^{3}$ consisting of vectors $\xi$ with associated orthonormal basis $(\xi, \xi_{1}, \xi_{2}), \epsilon_{\Xi} > 0$, and smooth positive functions $\gamma_{\xi}: B_{\epsilon_{\Xi}} (0) \mapsto \mathbb{R}$, where $B_{\epsilon_{\Xi}}(0)$ is the ball of radius $\epsilon_{\Xi}$ centered at 0 in the space of $3\times 3$ skew-symmetric matrices, such that for $A \in B_{\epsilon_{\Xi}} (0)$, the following identity holds:
\begin{equation}\label{estimate 44}
A = \sum_{\xi \in \Lambda_{\Xi}} (\gamma_{\xi} (A))^{2} (\xi \otimes \xi_{2} - \xi_{2} \otimes \xi).
\end{equation}
\end{lemma} 
 
\begin{lemma}\label{Lemma 3.2}
\rm{(\cite[Pro. 2.2]{CL21}, cf. \cite[Lem. 4.2]{BBV21})} There exists a set $\Lambda_{v} \subset \mathbb{S}^{2} \cap \mathbb{Q}^{3}$ consisting of vectors $\xi$ with associated orthonormal basis $(\xi, \xi_{1}, \xi_{2}), \epsilon_{v} > 0$, and smooth positive functions $\gamma_{\xi}: B_{\epsilon_{v}} (\text{Id}) \mapsto \mathbb{R}$, where $B_{\epsilon_{v}}(\text{Id})$ is the ball of radius $\epsilon_{v}$ centered at the identity in the space of $3\times 3$ symmetric matrices, such that for $A \in B_{\epsilon_{v}}(\text{Id})$, the following identity holds:
\begin{equation}\label{estimate 45}
A = \sum_{\xi \in \Lambda_{v}} (\gamma_{\xi} (A))^{2} (\xi \otimes \xi). 
\end{equation} 
\end{lemma}

\begin{remark}\label{Remark 3.1}
\rm{(cf. \cite[Rem. 2.4]{CL21}, \cite[Rem. 4.3-4.4]{BBV21})} We can choose $\Lambda_{\Xi}$ and $\Lambda_{v}$ so that $\Lambda_{\Xi}\cap \Lambda_{v} = \emptyset$  and for $\xi \neq \xi'$, their orthonormal bases satisfy $\xi_{1} \neq \xi_{1}'$. For convenience, let us set 
\begin{equation}\label{estimate 46}
\Lambda \triangleq \Lambda_{\Xi} \cup \Lambda_{v}; 
\end{equation} 
we also note that there exists $N_{\Lambda} \in \mathbb{N}$ such that 
\begin{equation}\label{estimate 47}
\{ N_{\Lambda} \xi, N_{\Lambda} \xi_{1}, N_{\Lambda} \xi_{2} \} \subset N_{\Lambda} \mathbb{S}^{2} \cap \mathbb{Z}^{3}. 
\end{equation} 
Finally, we denote by $M_{\ast}$ a universal geometric constant such that 
\begin{equation}\label{estimate 48}
\sum_{\xi \in \Lambda_{\Xi}} \lVert \gamma_{\xi} \rVert_{C^{1}(B_{\epsilon_{\Xi}} (0))} + \sum_{\xi \in \Lambda_{v}}  \lVert \gamma_{\xi} \rVert_{C^{1}(B_{\epsilon_{v}}( \text{Id} ))}  \leq M_{\ast}.
\end{equation} 
\end{remark}
Next, we describe the intermittent flow from \cite{CL21}. We let $\Psi: \mathbb{R} \mapsto \mathbb{R}$ be a smooth cutoff function supported on $[-1, 1]$. We assume that it is normalized in such a way that 
\begin{equation}\label{estimate 49}
\phi \triangleq - \frac{d^{2}}{(dx)^{2}} \Psi \text{ satisfies } \int_{\mathbb{R}} \phi^{2} (x) dx = 2 \pi. 
\end{equation} 
For parameters 
\begin{equation}\label{estimate 50}
0 < \sigma \ll r \ll 1 
\end{equation} 
to be specified shortly, we define the rescaled functions: 
\begin{equation}\label{estimate 51}
\phi_{r} (x) \triangleq \frac{1}{r^{\frac{1}{2}}} \phi(\frac{x}{r}), \hspace{1mm} \phi_{\sigma} (x)\triangleq \frac{1}{\sigma^{\frac{1}{2}}} \phi (\frac{x}{\sigma}), \hspace{1mm} \text{ and } \hspace{1mm} \Psi_{\sigma}(x) \triangleq \frac{1}{\sigma^{\frac{1}{2}}} \Psi( \frac{x}{\sigma}). 
\end{equation} 
We periodize these functions so that we can view the resulting functions, which we continue to denote respectively as $\phi_{r}, \phi_{\sigma}$, and $\Psi_{\sigma}$, as functions defined on $\mathbb{T}$. Then we fix a parameter $\lambda$ such that
\begin{equation}\label{estimate 52}
\lambda \sigma \in \mathbb{N}, 
\end{equation} 
as well as a large time-oscillation parameter 
\begin{equation}\label{estimate 53}
\mu \gg \sigma^{-1},
\end{equation} 
and define for every $\xi \in \Lambda$ 
\begin{subequations}\label{estimate 422}
\begin{align}
& \phi_{\xi} (t,x) \triangleq \phi_{\xi, r, \sigma, \lambda, \mu} (t,x) \triangleq \phi_{r} (\lambda \sigma N_{\Lambda} (\xi \cdot x + \mu t)), \label{estimate 54}\\
& \varphi_{\xi} (x) \triangleq \phi_{\xi, \sigma, \lambda} (x) \triangleq \phi_{\sigma} (\lambda \sigma N_{\Lambda} \xi_{1} \cdot x), \label{estimate 55}\\
& \Psi_{\xi} (x) \triangleq \Psi_{\xi, \sigma, \lambda} (x) \triangleq \Psi_{\sigma}(\lambda \sigma N_{\Lambda} \xi_{1} \cdot x), \label{estimate 56}
\end{align}
\end{subequations}
which are $(\mathbb{T} / \lambda \sigma)^{3}$-periodic. Due to \eqref{estimate 49} and \eqref{estimate 51}, they satisfy the identities of 
\begin{equation}\label{estimate 57}
- \Delta \Psi_{\xi} (x) = \lambda^{2} N_{\Lambda}^{2} \varphi_{\xi} (x), \hspace{1mm} \xi \cdot \nabla \phi_{\xi} = \mu^{-1} \partial_{t} \phi_{\xi}, \text{ and } \mathbb{P}_{=0} (\phi_{\xi}^{2} \varphi_{\xi}^{2}) = \frac{1}{(2\pi)^{3}} \int_{\mathbb{T}^{3}} \phi_{\xi}^{2} \varphi_{\xi}^{2} dx = 1 
\end{equation}
and the following estimates. 
\begin{lemma}
\rm{(\cite[Lem. 2.5]{CL21}, cf. \cite[Lem. 5.1-5.2]{BBV21})}  For any $p \in [1,\infty]$, $M , N \in \mathbb{N}$, and $\xi \neq \xi'$, the following estimates hold: 
\begin{subequations}\label{estimate 175}
\begin{align}
& \lVert \nabla^{M} \partial_{t}^{N} \phi_{\xi} \rVert_{C_{t}L_{x}^{p}} \lesssim (\lambda \sigma)^{M + N} r^{\frac{1}{p} - \frac{1}{2} - M - N} \mu^{N}, \label{estimate 58} \\
& \lVert \nabla^{M} \varphi_{\xi} \rVert_{L_{x}^{p}} + \lVert \nabla^{M} \Psi_{\xi} \rVert_{L_{x}^{p}} \lesssim \lambda^{M} \sigma^{\frac{1}{p} - \frac{1}{2}}, \label{estimate 59} \\
& \lVert \nabla^{M} (\phi_{\xi} \varphi_{\xi}) \rVert_{C_{t}L_{x}^{p}} + \lVert \nabla^{M} (\phi_{\xi} \Psi_{\xi} ) \rVert_{C_{t}L_{x}^{p}} \lesssim \lambda^{M} r^{\frac{1}{p} - \frac{1}{2}} \sigma^{\frac{1}{p} - \frac{1}{2}}, \label{estimate 60}\\ 
& \lVert \phi_{\xi} \varphi_{\xi} \phi_{\xi'} \varphi_{\xi'} \rVert_{C_{t}L_{x}^{p}} \lesssim \sigma^{\frac{2}{p} - 1} r^{-1},  \label{estimate 174}
\end{align}
\end{subequations}
where the implicit constants only depend on $p, N$, and $M$.  
\end{lemma} 

\section{Proofs of Theorems \ref{Theorem 2.1}-\ref{Theorem 2.2}}\label{Section 4}

We first give a definition of a solution to \eqref{stochastic GMHD}. 
\begin{define}\label{Definition 4.1}
Fix any $\gamma \in (0,1)$. Let $s \geq 0$ and $\xi^{\text{in}} = (\xi_{1}^{\text{in}}, \xi_{2}^{\text{in}}) \in L_{\sigma}^{2} \times L_{\sigma}^{2}$. Then $P \in \mathcal{P} (\Omega_{0})$ is a martingale solution to \eqref{stochastic GMHD} with initial condition $\xi^{\text{in}}$ at initial time $s$ if 
\begin{enumerate}
\item [] (M1) $P (\{ \xi(t) = \xi^{\text{in}} \hspace{1mm} \forall \hspace{1mm} t \in [0,s] \}) = 1$ and for all $l \in \mathbb{N}$, 
\begin{equation}
P ( \{ \xi \in \Omega_{0}: \int_{0}^{l} \sum_{k=1}^{2} \lVert G_{k} (\xi_{k} (r)) \rVert_{L_{2} (U_{k}, L_{\sigma}^{2})}^{2} dr < \infty \} ) = 1, 
\end{equation} 
\item [] (M2)  for every $\psi_{i} = (\psi_{i}^{1},\psi_{i}^{2}) \in (C^{\infty} (\mathbb{T}^{3}) \cap L_{\sigma}^{2})^{2}$ and $t \geq s$, the processes 
\begin{subequations}
\begin{align}
M_{1, t,s}^{i} \triangleq& \langle \xi_{1}(t) - \xi_{1}(s), \psi_{i}^{1} \rangle + \int_{s}^{t} \langle \text{div} ( \xi_{1}  \otimes \xi_{1} - \xi_{2} \otimes \xi_{2})(r) + (-\Delta)^{m_{1}} \xi_{1} (r), \psi_{i}^{1} \rangle dr,  \label{estimate 61} \\
M_{2,t,s}^{i} \triangleq& \langle \xi_{2} (t) - \xi_{2} (s), \psi_{i}^{2} \rangle + \int_{s}^{t} \langle \text{div} ( \xi_{1} \otimes \xi_{2} - \xi_{2} \otimes \xi_{1} )(r) + (-\Delta)^{m_{2}} \xi_{2} (r), \psi_{i}^{2} \rangle dr, \label{estimate 62} 
\end{align}
\end{subequations} 
are continuous, square-integrable $(\mathcal{B}_{t})_{t\geq s}$-martingale under $P$ such that $\langle \langle M_{k, t,s}^{i} \rangle \rangle = \int_{s}^{t} \lVert G_{k} (\xi_{k} (r))^{\ast} \psi_{i}^{k} \rVert_{U_{k}}^{2} dr$ for both $k \in \{1,2\}$, 
\item [] (M3) for any $q \in \mathbb{N}$ there exists a function $t \mapsto C_{t,q} \in \mathbb{R}_{+}$ for all $t \geq s$ such that
\begin{align}
&\mathbb{E}^{P} [ \sup_{r \in [0,t]} \lVert \xi_{1} (r) \rVert_{L_{x}^{2}}^{2q} + \int_{s}^{t} \lVert \xi_{1} \rVert_{\dot{H}_{x}^{\gamma}}^{2} dr \nonumber \\
& \hspace{5mm} + \sup_{r \in [0,t]} \lVert \xi_{2} (r) \rVert_{L_{x}^{2}}^{2q} + \int_{s}^{t} \lVert \xi_{2} \rVert_{\dot{H}_{x}^{\gamma}}^{2} dr] \leq C_{t,q} (1+ \lVert \xi_{1}^{\text{in}} \rVert_{L_{x}^{2}}^{2q} + \lVert \xi_{2}^{\text{in}} \rVert_{L_{x}^{2}}^{2q}). \label{estimate 63}
\end{align}
\end{enumerate} 
The set of all such martingale solutions with common constant $C_{t,q}$ in \eqref{estimate 63} for every $q \in \mathbb{N}$ and $t \geq s$ will be denoted by $\mathcal{C} ( s, \xi^{\text{in}}, \{C_{t,q} \}_{q\in \mathbb{N}, t \geq s} )$. 
\end{define} 
\begin{define}\label{Definition 4.2}
Fix any $\gamma \in (0,1)$. Let $s \geq 0$, $\xi^{\text{in}} = (\xi_{1}^{\text{in}}, \xi_{2}^{\text{in}}) \in L_{\sigma}^{2} \times L_{\sigma}^{2}$ and $\tau \geq s$ be a stopping time of $(\mathcal{B}_{t})_{t\geq s}$. Define the space of trajectories stopped at time $\tau$ by 
\begin{equation}\label{estimate 64} 
\Omega_{0, \tau} \triangleq \{ \omega( \cdot \wedge \tau(\omega)) : \omega \in \Omega_{0} \}
\end{equation} 
and $\mathcal{B}_{\tau}$ to be the $\sigma$-field associated to $\tau$. Then $P \in \mathcal{P} (\Omega_{0,\tau})$ is a martingale solution to \eqref{stochastic GMHD} on $[s, \tau]$ with initial condition $\xi^{\text{in}}$ at initial time $s$ if 
\begin{enumerate}
\item [] (M1) $P ( \{ \xi(t)= \xi^{\text{in}} \hspace{1mm} \forall \hspace{1mm} t \in [0,s]\}) = 1$ and for all $l \in \mathbb{N}$, 
\begin{equation}\label{estimate 65} 
P ( \{ \xi \in \Omega_{0}: \int_{0}^{l \wedge \tau} \sum_{k=1}^{2} \lVert G_{k} (\xi_{k} (r)) \rVert_{L_{2} (U_{k}, L_{\sigma}^{2})}^{2} dr < \infty \}) = 1, 
\end{equation} 
\item [] (M2) for every $\psi_{i} = (\psi_{i}^{1}, \psi_{i}^{2}) \in (C^{\infty} (\mathbb{T}^{3} ) \cap L_{\sigma}^{2} )^{2}$ and $t\geq s$, the processes 
\begin{subequations}
\begin{align}
M_{1, t\wedge \tau, s}^{i} \triangleq& \langle \xi_{1} (t\wedge \tau) - \xi_{1}^{\text{in}}, \psi_{i}^{1} \rangle \nonumber\\
& \hspace{10mm} + \int_{s}^{t \wedge \tau} \langle \text{div} (\xi_{1} \otimes \xi_{1} - \xi_{2} \otimes \xi_{2} )(r) + (-\Delta)^{m_{1}} \xi_{1} (r), \psi_{i}^{1} \rangle dr, \label{estimate 66} \\
M_{2, t \wedge \tau, s}^{i} \triangleq& \langle \xi_{2} (t \wedge \tau) - \xi_{2}^{\text{in}}, \psi_{i}^{2} \rangle \nonumber\\
&  \hspace{10mm} + \int_{s}^{t \wedge \tau} \langle \text{div} ( \xi_{1} \otimes \xi_{2} - \xi_{2} \otimes \xi_{1} ) (r) + (-\Delta)^{m_{2}} \xi_{2} (r), \psi_{i}^{2} \rangle dr, \label{estimate 67} 
\end{align}
\end{subequations} 
are continuous, square-integrable $(\mathcal{B}_{t})_{t\geq s}$-martingales under $P$ such that $\langle \langle M_{k, t \wedge \tau, s}^{i} \rangle \rangle$ $= \int_{s}^{t \wedge \tau} \lVert G_{k} (\xi_{k} (r))^{\ast} \psi_{i}^{k} \rVert_{U_{k}}^{2} dr$ for both $k \in \{1,2\}$, 
\item [] (M3) for any $q \in \mathbb{N}$, there exists a function $t \mapsto C_{t,q} \in \mathbb{R}_{+}$ for all $t \geq s$ such that  
\begin{align}
&\mathbb{E}^{P} [ \sup_{r \in [0, t \wedge \tau]} \lVert \xi_{1} (r) \rVert_{L_{x}^{2}}^{2q} + \int_{s}^{t \wedge \tau} \lVert \xi_{1} (r) \rVert_{\dot{H}_{x}^{\gamma}}^{2} dr \nonumber \\
& \hspace{10mm} + \sup_{r \in [0,t \wedge \tau]} \lVert \xi_{2} (r) \rVert_{L_{x}^{2}}^{2q} + \int_{s}^{t \wedge \tau} \lVert \xi_{2} (r) \rVert_{\dot{H}_{x}^{\gamma}}^{2} dr] \leq C_{t,q} (1+ \lVert \xi_{1}^{\text{in}} \rVert_{L_{x}^{2}}^{2q} + \lVert \xi_{2}^{\text{in}} \rVert_{L_{x}^{2}}^{2q}).  \label{estimate 68} 
\end{align} 
\end{enumerate} 
\end{define} 

The following result concerns the existence and stability of solution to \eqref{stochastic GMHD}; the proof of existence may be found in \cite[Sec. 3]{Y19} while the proof of stability is a straight-forward modification of analogous results such as \cite[The. 3.1]{HZZ19} and \cite[Pro. 4.1]{Y21a}:  
\begin{proposition}\label{Proposition 4.1}
For any $(s, \xi^{\text{in}}) \in [0,\infty) \times (L_{\sigma}^{2})^{2}$, there exists a martingale solution $P \in \mathcal{P} (\Omega_{0})$ to \eqref{stochastic GMHD} with initial condition $\xi^{\text{in}}$ at initial time $s$ that satisfies Definition \ref{Definition 4.1}. Moreover, if there exists a family $\{ (s_{l}, \xi_{l} ) \}_{l \in \mathbb{N}} \subset [0,\infty) \times (L_{\sigma}^{2})^{2}$ such that $\lim_{l\to\infty} \lVert (s_{l}, \xi_{l}) - (s, \xi^{\text{in}}) \rVert_{\mathbb{R} \times L_{\sigma}^{2} \times L_{\sigma}^{2}} = 0$ and $P_{l} \in \mathcal{C} (s_{l}, \xi_{l}, \{ C_{t,q} \}_{q \in \mathbb{N}, t \geq s_{l}} )$ is the martingale solution corresponding to $(s_{l}, \xi_{l})$, then there exists a subsequence $\{ P_{l_{k}} \}_{k \in \mathbb{N}}$ that converges weakly to some $P \in \mathcal{C} ( s, \xi^{\text{in}}, \{C_{t,q} \}_{q\in \mathbb{N}, t\geq s} )$. 
\end{proposition} 
Proposition \ref{Proposition 4.1} leads to the following two results, of which proofs may be found in \cite[Pro. 3.2 and 3.4]{HZZ19}: 
\begin{lemma}\label{Lemma 4.2}
(\cite[Pro. 3.2]{HZZ19}) Let $\tau$ be a bounded stopping time of $(\mathcal{B}_{t})_{t\geq 0}$. Then for every $\omega \in \Omega_{0}$, there exists $Q_{\omega} \triangleq \delta_{\omega} \otimes_{\tau(\omega)} R_{\tau(\omega), \xi(\tau(\omega), \omega)} \in \mathcal{P} (\Omega_{0})$ with $\delta_{\omega}$ being a point-mass at $\omega$ such that 
\begin{subequations}
\begin{align} 
& Q_{\omega} ( \{ \omega' \in \Omega_{0}:\hspace{0.5mm}  \xi(t, \omega') = \omega(t) \hspace{1mm} \forall \hspace{1mm} t \in [0, \tau(\omega) ] \}) = 1, \label{estimate 69}  \\
& Q_{\omega}(A) = R_{\tau (\omega), \xi(\tau(\omega), \omega)} (A) \hspace{1mm} \forall \hspace{1mm} A \in \mathcal{B}^{\tau(\omega)}, \label{estimate 70}
\end{align}
\end{subequations} 
where $R_{\tau(\omega), \xi(\tau(\omega), \omega)} \in \mathcal{P}(\Omega_{0})$ is a martingale solution to \eqref{stochastic GMHD} with initial condition $\xi(\tau(\omega), \omega)$ at initial time $\tau(\omega)$, and the mapping $\omega \mapsto Q_{\omega}(B)$ is $\mathcal{B}_{\tau}$-measurable for every $B \in \mathcal{B}$. 
\end{lemma} 

\begin{lemma}\label{Lemma 4.3}
(\cite[Pro. 3.4]{HZZ19}) Let $\tau$ be a bounded stopping time of $(\mathcal{B}_{t})_{t\geq 0}$, $\xi^{\text{in}} \in L_{\sigma}^{2} \times L_{\sigma}^{2}$, and $P$ be a martingale solution to \eqref{stochastic GMHD} on $[0,\tau]$ with initial condition $\xi^{\text{in}}$ at initial time 0 that satisfies Definition \ref{Definition 4.2}. Suppose that there exists a Borel set $\mathcal{N} \subset \Omega_{0,\tau}$ such that $P(\mathcal{N}) = 0$ and $Q_{\omega}$ from Lemma \ref{Lemma 4.2} satisfies for every $\omega \in \Omega_{0} \setminus \mathcal{N}$ 
\begin{equation}\label{estimate 71}
Q_{\omega} (\{\omega' \in \Omega_{0}:\hspace{0.5mm}  \tau(\omega') = \tau(\omega) \}) = 1. 
\end{equation} 
Then the probability measure $P \otimes_{\tau}R \in \mathcal{P}(\Omega_{0})$ defined by 
\begin{equation}\label{estimate 72} 
P\otimes_{\tau} R (\cdot) \triangleq \int_{\Omega_{0}} Q_{\omega} (\cdot) P(d\omega) 
\end{equation} 
satisfies $P \otimes_{\tau}R \rvert_{\Omega_{0,\tau}} = P \rvert_{\Omega_{0,\tau}}$ and it is a martingale solution to \eqref{stochastic GMHD} on $[0,\infty)$ with initial condition $\xi^{\text{in}}$ at initial time 0. 
\end{lemma} 

Now, similarly to \eqref{estimate 395} we consider 
\begin{subequations}\label{z}
\begin{align}
& dz_{1} + (-\Delta)^{m_{1}} z_{1} dt + \nabla \pi_{1} dt = dB_{1}, \hspace{3mm} \nabla\cdot z_{1} = 0 \text{ for } t > 0, \hspace{2mm} z_{1}(0, x) \equiv 0, \label{estimate 37}  \\
& dz_{2} + (-\Delta)^{m_{2}} z_{2} dt = dB_{2}, \hspace{15mm} \nabla\cdot z_{2} = 0  \text{ for } t > 0, \hspace{2mm} z_{2}(0,x) \equiv 0, \label{estimate 38}
\end{align}
\end{subequations}
and 
\begin{subequations}\label{estimate 399}
\begin{align}
& \partial_{t} v + (-\Delta)^{m_{1}} v + \text{div} ((v+ z_{1}) \otimes (v+ z_{1}) - (\Xi + z_{2}) \otimes (\Xi + z_{2}) ) + \nabla \pi_{2} = 0,  \nabla\cdot v = 0, \label{estimate 73} \\
& \partial_{t} \Xi + (-\Delta)^{m_{2}} \Xi + \text{div} ((v+ z_{1}) \otimes (\Xi + z_{2}) - (\Xi + z_{2}) \otimes (v+ z_{1} )) = 0,  \hspace{7mm} \nabla\cdot \Xi = 0,  \label{estimate 74}
\end{align}
\end{subequations} 
so that $(u, b) = (v+ z_{1}, \Xi + z_{2})$ solves \eqref{stochastic GMHD} with $\pi \triangleq \pi_{1} + \pi_{2}$. We fix $G_{k}G_{k}^{\ast}$-Wiener process $B_{k}$ on $(\Omega, \mathcal{F}, \textbf{P})$ for both $k \in \{1,2\}$ with $(\mathcal{F}_{t})_{t\geq 0}$ as the canonical filtration of $(B_{1}, B_{2})$ augmented by all the $\textbf{P}$-negligible sets. We see from \eqref{z} that 
\begin{equation}\label{estimate 75}
z_{1} (t) = \int_{0}^{t} e^{- (t-r) (-\Delta)^{m_{1}}} \mathbb{P} dB_{1}(r), \hspace{3mm} z_{2}(t) = \int_{0}^{t} e^{- (t-r) (-\Delta)^{m_{2}}} dB_{2}(r), 
\end{equation}  
where $e^{- (-\Delta)^{m_{k}}}$ are semigroups generated by $-(-\Delta)^{m_{k}}$ for $k \in \{1,2\}$. The following proposition concerning regularity of $z_{k},k \in \{1,2\}$, are consequences of \cite[Pro. 4.4]{Y21a} and \eqref{estimate 18}: 
\begin{proposition}\label{Proposition 4.4} 
For all $\delta \in (0, \frac{1}{2}), T > 0,$ and $l \in \mathbb{N}$, 
\begin{equation}\label{estimate 76} 
\sum_{k=1}^{2} \mathbb{E}^{\textbf{P}} [ \lVert z_{k} \rVert_{C_{T} \dot{H}_{x}^{\frac{5+ \sigma}{2}}}^{l} + \lVert z_{k} \rVert_{C_{T}^{\frac{1}{2} - \delta} \dot{H}_{x}^{\frac{3+ \sigma}{2}}}^{l} ] < \infty. 
\end{equation} 
\end{proposition} 
Next, for every $\omega = (\omega_{1}, \omega_{2}) \in \Omega_{0}$ we define 
\begin{subequations}
\begin{align}
M_{1, t, 0}^{\omega} \triangleq& \omega_{1}(t) - \omega_{1}(0) + \int_{0}^{t} \mathbb{P} \text{div} ( \omega_{1} \otimes \omega_{1} - \omega_{2} \otimes \omega_{2})(r) + (-\Delta)^{m_{1}} \omega_{1} (r) dr, \label{estimate 77} \\
M_{2, t,0}^{\omega} \triangleq& \omega_{2}(t) - \omega_{2} (0)  + \int_{0}^{t} \text{div} ( \omega_{1} \otimes \omega_{2} - \omega_{2} \otimes \omega_{1} )(r) + (-\Delta)^{m_{2}} \omega_{2} (r) dr, \label{estimate 78} 
\end{align}
\end{subequations} 
and 
\begin{subequations}
\begin{align}
Z_{1}^{\omega}(t) \triangleq M_{1, t,0}^{\omega} - \int_{0}^{t} \mathbb{P} (-\Delta)^{m_{1}} e^{- (t-r) (-\Delta)^{m_{1}}} M_{1,r,0}^{\omega} dr, \label{estimate 79}\\
Z_{2}^{\omega} (t) \triangleq M_{2,t,0}^{\omega} - \int_{0}^{t} (-\Delta)^{m_{2}} e^{- (t-r) (-\Delta)^{m_{2}}} M_{2,r,0}^{\omega} dr. \label{estimate 80}
\end{align}
\end{subequations} 
If $P$ is a martingale solution to \eqref{stochastic GMHD}, then the mappings $\omega \mapsto M_{k, t,0}^{\omega}$ for both $k \in \{1,2\}$ are $G_{k}G_{k}^{\ast}$-Wiener processes under $P$ and 
\begin{equation}\label{estimate 81}
Z_{1}(t) = \int_{0}^{t} e^{- (t-r) (-\Delta)^{m_{1}}} \mathbb{P} dM_{1,r,0} \hspace{1mm} \text{ and } \hspace{1mm}  Z_{2} (t) = \int_{0}^{t} e^{- (t-r) (-\Delta)^{m_{2}}} dM_{2,r,0}. 
\end{equation} 
Because $M_{k, t,0}^{\omega}$ is a $G_{k}G_{k}^{\ast}$-Wiener process under $P$ for both $k \in \{1,2\}$, Proposition \ref{Proposition 4.4} gives for all $\delta \in (0, \frac{1}{2}), T > 0$, 
\begin{equation}\label{estimate 82} 
Z_{k} \in C_{T} \dot{H}_{x}^{\frac{5+ \sigma}{2}} \cap C_{T}^{\frac{1}{2} - \delta} \dot{H}_{x}^{\frac{3+ \sigma}{2}} \hspace{2mm} \textbf{P}\text{-a.s.} 
\end{equation} 
Now for the Sobolev constant $C_{S} > 0$ such that $\lVert f \rVert_{L^{\infty} (\mathbb{T}^{3})} \leq C_{S} \lVert f \rVert_{\dot{H}^{\frac{3+ \sigma}{2}}(\mathbb{T}^{3})}$ for all $f \in \dot{H}^{\frac{3+ \sigma}{2}} (\mathbb{T}^{3})$ such that $\int_{\mathbb{T}^{3}} f dx = 0$, we define for $\omega \in \Omega_{0}$  
\begin{align}
\tau_{L}^{\lambda} &(\omega) \triangleq \inf\{t \geq 0: C_{S} \max_{k=1,2} \lVert Z_{k}^{\omega} (t) \rVert_{\dot{H}_{x}^{\frac{5+\sigma}{2}}} > (L - \frac{1}{\lambda})^{\frac{1}{4}} \} \nonumber \\
& \wedge \inf \{t \geq 0: C_{S} \max_{k=1,2} \lVert Z_{k}^{\omega} \rVert_{C_{t}^{\frac{1}{2} - \delta} \dot{H}_{x}^{\frac{3+\sigma}{2}}} > (L - \frac{1}{\lambda})^{\frac{1}{2}} \} \wedge L \hspace{1mm} \text{ and } \hspace{1mm} \tau_{L}(\omega) \triangleq \lim_{\lambda \to \infty} \tau_{L}^{\lambda} (\omega) \label{estimate 83} 
\end{align} 
so that $(\tau_{L}^{\lambda})_{\lambda \in \mathbb{N}}$ is non-decreasing in $\lambda$. It follows from \cite[Lem. 3.5]{HZZ19} that $\tau_{L}$ is a $(\mathcal{B}_{t})_{t\geq 0}$-stopping time. We define for $L > 1$ and $\delta \in (0, \frac{1}{12})$, 
\begin{align}
T_{L} \triangleq& \inf\{t \geq 0: C_{S} \max_{k=1,2} \lVert z_{k} (t) \rVert_{\dot{H}_{x}^{\frac{5+ \sigma}{2}}} \geq L^{\frac{1}{4}} \} \nonumber\\
& \wedge \inf\{t \geq 0: C_{S} \max_{k=1,2} \lVert z_{k} \rVert_{C_{t}^{\frac{1}{2} - 2 \delta} \dot{H}_{x}^{\frac{3+ \sigma}{2}}} \geq L^{\frac{1}{2}} \} \wedge L\label{estimate 84}
\end{align}
and realize that $T_{L} > 0$ and $\lim_{L \to\infty} T_{L} = + \infty$ $\textbf{P}$-a.s. due to Proposition \ref{Proposition 4.4}. The stopping time $\mathfrak{t}$ in Theorem \ref{Theorem 2.1} is $T_{L}$ for $L$ sufficiently large. Next, we assume Theorem \ref{Theorem 2.1} on a probability space $(\Omega, \mathcal{F}, (\mathcal{F}_{t})_{t\geq 0}, \textbf{P})$ and denote by $P$ the law of the solution $(u,b)$ constructed from Theorem \ref{Theorem 2.1}, i.e., $P= \mathcal{L}(u,b)$. 
\begin{proposition}\label{Proposition 4.5}
\rm{(cf. \cite[Pro. 3.7]{HZZ19})} Let $\tau_{L}$ be defined by \eqref{estimate 83}. Then $P = \mathcal{L} (u,b)$ is a martingale solution over $[0, \tau_{L}]$ according to Definition \ref{Definition 4.2}. 
\end{proposition}
\begin{proof}[Proof of Proposition \ref{Proposition 4.5}]
The proof is a straight-forward modification of the proofs of \cite[Pro. 3.7]{HZZ19} (\cite[Pro. 4.5]{Y21a} in case of a system of equations) that particularly shows (see \cite[Equ. (3.15)]{HZZ19} and \cite[Equ. (315)]{Y21a}) 
\begin{equation}\label{estimate 86}
\tau_{L} (u,b) = T_{L} \hspace{3mm} \textbf{P}\text{-a.s.}
\end{equation}
\end{proof}
\begin{proposition}\label{Proposition 4.6}
\rm{(cf. \cite[Pro. 3.8]{HZZ19})} Let $\tau_{L}$ be defined by \eqref{estimate 83} and $P = \mathcal{L}(u,b)$ constructed from Theorem \ref{Theorem 2.1}. Then $P \otimes_{\tau_{L}} R$ defined by \eqref{estimate 72} is a martingale solution on $[0,\infty)$ according to Definition \ref{Definition 4.1}. 
\end{proposition}
\begin{proof}[Proof of Proposition \ref{Proposition 4.6}]
The proof is a straight-forward modification of the proofs of \cite[Pro. 3.8]{HZZ19} (\cite[Pro. 4.6]{Y21a} in case of a system of equations) that particular shows that there exists a $P$-measurable set $\mathcal{N} \subset \Omega_{0}$ such that $P(\mathcal{N}) = 0$ and for all $\omega \in \Omega_{0} \setminus \mathcal{N}$, 
\begin{equation}\label{estimate 87}
Q_{\omega} ( \{ \omega' \in \Omega_{0}: \tau_{L} (\omega') = \tau_{L} (\omega) \}) = 1 
\end{equation}
(see \cite[Equ. (3.19)]{HZZ19} and \cite[Equ. (324)]{Y21a}). 
\end{proof}

We re now ready to prove Theorem \ref{Theorem 2.2} assuming Theorem \ref{Theorem 2.1}. 
\begin{proof}[Proof of Theorem \ref{Theorem 2.2}]
We fix $T > 0$ arbitrarily, $K > 1$, and $\kappa \in (0,1)$ such that $\kappa K^{2} \geq 1$, rely on Theorem \ref{Theorem 2.1} and Proposition \ref{Proposition 4.6} to deduce the existence of $L > 1$ and a martingale solution $P \otimes_{\tau_{L}} R$ to \eqref{stochastic GMHD} on $[0, \infty)$ such that $P \otimes_{\tau_{L}} R = P$ on $[0, \tau_{L}]$ where $P =  \mathcal{L} (u,b)$ for the solution $(u,b)$ constructed in Theorem \ref{Theorem 2.1}. Hence, $P \otimes_{\tau_{L}} R$ starts with a deterministic initial condition $\xi^{\text{in}} = (u^{\text{in}}, b^{\text{in}})$ from the proof of Theorem \ref{Theorem 2.1} and satisfies 
\begin{equation}\label{estimate 85}
P \otimes_{\tau_{L} } R ( \{ \tau_{L} \geq T \}) \overset{\eqref{estimate 72} \eqref{estimate 87}}{=} P(\{\tau_{L} \geq T \})  \overset{\eqref{estimate 86}}{=} \textbf{P} ( \{ T_{L} \geq T \}) \overset{\eqref{estimate 19}}{>} \kappa. 
\end{equation} 
This implies 
\begin{equation}\label{estimate 89}
\mathbb{E}^{P \otimes_{\tau_{L} R}} [ \lVert \xi_{2}(T) \rVert_{L_{x}^{2}}^{2} ] \overset{\eqref{estimate 21} \eqref{estimate 85}}{>} \kappa K^{2} [ \lVert u^{\text{in}} \rVert_{L_{x}^{2}}^{ 2} + \lVert b^{\text{in}} \rVert_{L_{x}^{2}} + T \sum_{k=1}^{2} Tr(G_{k}G_{k}^{\ast})]. 
\end{equation} 
On the other hand, it is well-known that a classical Galerkin approximation (e.g., \cite{FR08, Y19}) and \cite[The. 4.2.4]{Z12} in case of a fractional Laplacian) can give another martingale solution $\Theta$ to \eqref{stochastic GMHD} such that 
\begin{equation}\label{estimate 88}
\mathbb{E}^{\Theta} [ \lVert \xi (T) \rVert_{L_{x}^{2}}^{2} ] \leq \lVert \xi^{\text{in}} \rVert_{L_{x}^{2}}^{2} + T \sum_{k=1}^{2} Tr (G_{k}G_{k}^{\ast}). 
\end{equation} 
\end{proof}

Considering \eqref{estimate 73}-\eqref{estimate 74}, for $q \in \mathbb{N}_{0}$ we aim to construct a solution $(v_{q}, \Xi_{q}, \mathring{R}_{q}^{v}, \mathring{R}_{q}^{\Xi})$ to 
\begin{subequations}\label{estimate 104}
\begin{align}
& \partial_{t} v_{q} + (-\Delta)^{m_{1}} v_{q} + \text{div} (( v_{q} + z_{1} ) \otimes (v_{q} + z_{1}) - (\Xi_{q} + z_{2}) \otimes (\Xi_{q} + z_{2} )) + \nabla \pi_{q} = \text{div} \mathring{R}_{q}^{v}, \label{estimate 90} \\
& \partial_{t} \Xi_{q} + (-\Delta)^{m_{2}} \Xi_{q} + \text{div} ((v_{q} + z_{1}) \otimes (\Xi_{q} + z_{2}) - (\Xi_{q} + z_{2}) \otimes (v_{q} + z_{1})) = \text{div} \mathring{R}_{q}^{\Xi}, \label{estimate 91} \\
& \nabla\cdot v_{q} = 0, \hspace{1mm} \nabla\cdot \Xi_{q} = 0, \label{estimate 92} 
\end{align}
\end{subequations} 
where $\mathring{R}_{q}^{v}$ is a symmetric trace-free matrix and $\mathring{R}_{q}^{\Xi}$ a skew-symmetric matrix, called the Reynolds stress and the magnetic Reynolds stress, respectively. We see that $\pi_{q}$ can be deduced as 
\begin{equation*}
\pi_{q} = (-\Delta)^{-1} \text{divdiv} ( ( v_{q} + z_{1}) \otimes (v_{q} + z_{1}) - (\Xi_{q} + z_{2}) \otimes (\Xi_{q} + z_{2}) - \mathring{R}_{q}^{v} )
\end{equation*} 
with $\int_{\mathbb{T}^{3}} \pi_{q} dx = 0$. For any $a \in \mathbb{N}, b \in \mathbb{N}, \beta \in (0,1)$, and $L \geq 1$ to be specified subsequently, we define 
\begin{equation}\label{estimate 93}
\lambda_{q} \triangleq a^{b^{q}}, \hspace{2mm} \delta_{q} \triangleq \lambda_{q}^{-2\beta}, 
\end{equation} 
and 
\begin{equation}\label{estimate 94}
M_{0}(t) \triangleq L^{4} e^{4Lt}. 
\end{equation} 
We see from \eqref{estimate 84} that for any $\delta \in (0, \frac{1}{12}), t \in [0, T_{L}]$, and both $k \in \{1,2\}$, 
\begin{equation}\label{estimate 95}
\lVert z_{k} (t) \rVert_{L_{x}^{\infty}} \leq L^{\frac{1}{4}}, \hspace{3mm} \lVert z_{k} (t) \rVert_{\dot{W}_{x}^{1,\infty}} \leq L^{\frac{1}{4}}, \hspace{1mm} \text{ and } \hspace{1mm} \lVert z_{k} \rVert_{C_{t}^{\frac{1}{2} - 2 \delta} L_{x}^{\infty}} \leq L^{\frac{1}{2}} 
\end{equation} 
by the definition of $C_{S}$. We see that if $b \geq 2$ and 
\begin{equation}\label{estimate 96}
a^{\beta b} > (2\pi)^{3} + 1, 
\end{equation} 
which we will assume hereafter, then $\sum_{1 \leq \iota \leq q} \delta_{\iota}^{\frac{1}{2}} < \frac{1}{2}$ for any $q \in \mathbb{N}$ which guarantees the second inequalities in \eqref{estimate 97}-\eqref{estimate 98} below. In fact, to obtain only $\sum_{1 \leq \iota \leq q} \delta_{\iota}^{\frac{1}{2}} < \frac{1}{2}$ for any $q \in \mathbb{N}$, we only need $a^{\beta b} > 3$; however, we will need \eqref{estimate 96} subsequently in the computation of \eqref{estimate 125}. We set the convention that $\sum_{1 \leq \iota \leq 0} \triangleq 0$. We fix a universal constant $c_{v} > 0$ determined by \eqref{estimate 156}; for this fixed $c_{v} > 0$, let $c_{\Xi}> 0$ be another universal constant determined by \eqref{estimate 157}. For such fixed $c_{v}, c_{\Xi} > 0$, we assume the following bounds over $t \in [0, T_{L}]$ inductively: 
\begin{subequations}\label{estimate 105}
\begin{align}
& \lVert v_{q} \rVert_{C_{t}L_{x}^{2}} \leq M_{0}(t)^{\frac{1}{2}} (1+ \sum_{1 \leq \iota \leq q} \delta_{\iota}^{\frac{1}{2}}) \leq 2 M_{0}(t)^{\frac{1}{2}}, \label{estimate 97}\\
& \lVert \Xi_{q} \rVert_{C_{t}L_{x}^{2}} \leq M_{0}(t)^{\frac{1}{2}} (1+ \sum_{1 \leq \iota \leq q} \delta_{\iota}^{\frac{1}{2}}) \leq 2 M_{0}(t)^{\frac{1}{2}}, \label{estimate 98}\\
& \lVert v_{q} \rVert_{C_{t,x}^{1}} \leq M_{0}(t)^{\frac{1}{2}} \lambda_{q}^{4}, \hspace{9mm} \lVert \Xi_{q} \rVert_{C_{t,x}^{1}} \leq M_{0}(t)^{\frac{1}{2}} \lambda_{q}^{4}, \label{estimate 99}\\
& \lVert \mathring{R}_{q}^{v} \rVert_{C_{t}L_{x}^{1}} \leq c_{v} M_{0}(t) \delta_{q+1}, \hspace{3mm} \lVert \mathring{R}_{q}^{\Xi} \rVert_{C_{t}L_{x}^{1}} \leq c_{\Xi} M_{0}(t) \delta_{q+1}. \label{estimate 100}
\end{align}
\end{subequations} 
The operators $\mathcal{R}$ and $\mathcal{R}^{\Xi}$ in the following statement are defined in Lemma \ref{divergence inverse operator}. 
\begin{proposition}\label{Proposition 4.7}
Let 
\begin{equation}\label{estimate 101}
v_{0}(t,x) \triangleq \frac{L^{2} e^{2L t}}{(2\pi)^{\frac{3}{2}}} 
\begin{pmatrix}
\sin(x^{3}) \\
0\\
0
\end{pmatrix} \hspace{2mm} 
\text{ and } \hspace{2mm} 
\Xi_{0}(t,x) \triangleq \frac{L^{2} e^{2L t}}{(2\pi)^{3}} 
\begin{pmatrix}
\sin(x^{3})\\
\cos(x^{3}) \\
0 
\end{pmatrix}.  
\end{equation}
Then, together with
\begin{subequations} 
\begin{align}
\mathring{R}_{0}^{v} (t,x) &\triangleq \frac{2L^{3} e^{2L t}}{(2\pi)^{\frac{3}{2}}} 
\begin{pmatrix}
0 & 0 & - \cos(x^{3}) \\
0 & 0 & 0 \\
-\cos(x^{3}) & 0 & 0 
\end{pmatrix} \label{estimate 102}  \\
&+ ( \mathcal{R} (-\Delta)^{m_{1}} v_{0} + v_{0} \mathring{\otimes} z_{1}+  z_{1} \mathring{\otimes} v_{0} + z_{1} \mathring{\otimes} z_{1} - \Xi_{0}\mathring{\otimes} z_{2} - z_{2} \mathring{\otimes} \Xi_{0} - z_{2} \mathring{\otimes} z_{2}) (t,x), \nonumber \\
\mathring{R}_{0}^{\Xi} (t,x) &\triangleq \frac{2L^{3} e^{2L t}}{(2\pi)^{3}} 
\begin{pmatrix}
0 & 0 & -\cos(x^{3}) \\
0 & 0 & \sin(x^{3}) \\
\cos(x^{3}) & -\sin(x^{3}) & 0 
\end{pmatrix}   \label{estimate 103} \\
&+ (\mathcal{R}^{\Xi} (-\Delta)^{m_{2}} \Xi_{0} + v_{0} \otimes z_{2} + z_{1} \otimes \Xi_{0} + z_{1} \otimes z_{2} - \Xi_{0} \otimes z_{1} - z_{2} \otimes v_{0} - z_{2} \otimes z_{1}) (t,x), \nonumber 
\end{align} 
\end{subequations}
$(v_{0}, \Xi_{0})$ satisfy \eqref{estimate 104} and \eqref{estimate 105} at level $q= 0$ provided 
\begin{subequations} \label{estimate 106}
\begin{align}
&(40)^{\frac{4}{3}} < L, \label{estimate 106a}\\
& ((2\pi)^{3} + 1)^{2} 20 \pi^{\frac{3}{2}} \max\{ \frac{1}{c_{v}}, \frac{1}{c_{\Xi}} \} < a^{2\beta b} 20 \pi^{\frac{3}{2}} \max \{ \frac{1}{c_{v}}, \frac{1}{c_{\Xi}} \} \leq L \leq \frac{ (2\pi)^{\frac{3}{2}} a^{4} -2}{2}, \label{estimate 106b}
\end{align}
\end{subequations} 
where the inequality $((2\pi)^{3} + 1)^{2} < a^{2\beta b}$ in \eqref{estimate 106b} is assumed to justify \eqref{estimate 96}. Finally, $v_{0}(0,x), \Xi_{0}(0,x), \mathring{R}_{0}^{v}(0,x)$, and $\mathring{R}_{0}^{\Xi}(0,x)$ are all deterministic. 
\end{proposition} 

\begin{proof}[Proof of Proposition \ref{Proposition 4.7}]
First, $v_{0}$ and $\Xi_{0}$ are both mean-zero and divergence-free; thus, $\mathcal{R}(-\Delta)^{m_{1}} v_{0}$ and $\mathcal{R}^{\Xi} (-\Delta)^{m_{2}} \Xi_{0}$ are well-defined. Moreover, as $\mathcal{R} (-\Delta)^{m_{1}} v_{0}$ is trace-free and symmetric by Lemma \ref{divergence inverse operator}, we see that $\mathring{R}_{0}^{v}$ is trace-free and symmetric. Similarly, $\mathcal{R}^{\Xi} (-\Delta)^{m_{2}} \Xi_{0}$ is skew-symmetric by Lemma \ref{divergence inverse operator} and this implies that $\mathring{R}_{0}^{\Xi}$ is skew-symmetric. Moreover, it can be readily verified that \eqref{estimate 104} is satisfied with $\pi_{0} \triangleq \frac{2}{3} v_{0} \cdot z_{1} + \frac{ \lvert z_{1} \rvert^{2}}{3} - \frac{2}{3} \Xi_{0} \cdot z_{2} - \frac{ \lvert z_{2} \rvert^{2}}{3}$. Next, we compute 
\begin{equation}\label{estimate 109}
\lVert v_{0}(t) \rVert_{L_{x}^{2}}  = \frac{ M_{0}(t)^{\frac{1}{2}}}{\sqrt{2}} \leq M_{0}(t)^{\frac{1}{2}} \hspace{1mm} \text{ and } \hspace{1mm} \lVert \Xi_{0}(t) \rVert_{L_{x}^{2}} = \frac{M_{0}(t)^{\frac{1}{2}}}{(2\pi)^{\frac{3}{2}}} \leq M_{0}(t)^{\frac{1}{2}}, 
\end{equation} 
which verifies \eqref{estimate 97}-\eqref{estimate 98} at level $q= 0$. We also compute 
\begin{subequations}
\begin{align}
&\lVert v_{0} \rVert_{C_{t,x}^{1}} = \frac{L^{2} e^{2L t} 2(L+1)}{(2\pi)^{\frac{3}{2}}} \overset{\eqref{estimate 106b}}{\leq} M_{0}(t)^{\frac{1}{2}} \lambda_{0}^{4},\\
& \lVert \Xi_{0} \rVert_{C_{t,x}^{1}} = \frac{ L^{2} e^{2L t} (1 + \sqrt{2} + 2L)}{(2\pi)^{3}} \overset{\eqref{estimate 106b}}{\leq} M_{0}(t)^{\frac{1}{2}} \lambda_{0}^{4}, 
\end{align}
\end{subequations}
which verifies \eqref{estimate 99} at level $q=0$. Next, we verify \eqref{estimate 100} at level $q=0$. First, we can compute directly from \eqref{estimate 102} 
\begin{align}
& \lVert \mathring{R}_{0}^{v} (t) \rVert_{L_{x}^{1}} \leq 16 \pi^{\frac{1}{2}} L^{3} e^{2L t}+ \lVert \mathcal{R} (-\Delta)^{m_{1}} v_{0}(t) \rVert_{L_{x}^{1}} \nonumber \\
&\hspace{5mm} + (\lVert v_{0} \mathring{\otimes} z_{1} \rVert_{L_{x}^{1}} + \lVert z_{1} \mathring{\otimes} v_{0} \rVert_{L_{x}^{1}} + \lVert \Xi_{0} \mathring{\otimes} z_{2} \rVert_{L_{x}^{1}} + \lVert z_{2} \mathring{\otimes} \Xi_{0} \rVert_{L_{x}^{1}} + \lVert z_{1} \mathring{\otimes} z_{1} \rVert_{L_{x}^{1}} + \lVert z_{2} \mathring{\otimes} z_{2} \rVert_{L_{x}^{1}})(t). \label{estimate 108}
\end{align} 
We use the fact that $\Delta v_{0} = -v_{0}$ and directly compute using \eqref{estimate 107} and \eqref{estimate 109} 
\begin{equation}\label{estimate 110}
\lVert \mathcal{R} (-\Delta)^{m_{1}} v_{0}(t) \rVert_{L_{x}^{1}}  \leq ( 2\pi)^{\frac{3}{2}} 18 \lVert (-\Delta)^{m_{1} - \frac{1}{2}} v_{0}(t) \rVert_{L_{x}^{2}} \leq  36 \sqrt{2} \pi^{\frac{3}{2}} \lVert v_{0}(t) \rVert_{L^{2}} =  36 \pi^{\frac{3}{2}} M_{0}(t)^{\frac{1}{2}}. 
\end{equation} 
Next, straight-forward computations using H$\ddot{\mathrm{o}}$lder's inequalities give us 
\begin{align}
& (\lVert v_{0} \mathring{\otimes} z_{1} \rVert_{L_{x}^{1}} + \lVert z_{1} \mathring{\otimes} v_{0} \rVert_{L_{x}^{1}} + \lVert \Xi_{0} \mathring{\otimes} z_{2} \rVert_{L_{x}^{1}} + \lVert z_{2} \mathring{\otimes} \Xi_{0} \rVert_{L_{x}^{1}} + \lVert z_{1} \mathring{\otimes} z_{1} \rVert_{L_{x}^{1}} + \lVert z_{2} \mathring{\otimes} z_{2} \rVert_{L_{x}^{1}})(t)   \nonumber\\
\overset{\eqref{estimate 109} \eqref{estimate 95} }{\leq}&  40 M_{0}(t)^{\frac{1}{2}} L^{\frac{1}{4}} \pi^{\frac{3}{2}} + 20 M_{0}(t)^{\frac{1}{2}} L^{\frac{1}{4}} + 160 \pi^{3} L^{\frac{1}{2}}.  \label{estimate 111}
\end{align}
Applying \eqref{estimate 110}-\eqref{estimate 111} to \eqref{estimate 108} gives us 
\begin{align*}
\lVert \mathring{R}_{0}^{v} (t) \rVert_{L_{x}^{1}} \overset{\eqref{estimate 110} \eqref{estimate 111} \eqref{estimate 108} \eqref{estimate 106a}}{\leq} 20 \pi^{\frac{3}{2}} L^{3} e^{2L t} \overset{\eqref{estimate 106b}}{\leq}  c_{v} M_{0}(t) \delta_{1}. 
\end{align*} 
Similarly, we compute from \eqref{estimate 103}  
\begin{align}
& \lVert \mathring{R}_{0}^{\Xi} (t) \rVert_{L_{x}^{1}} \leq \frac{16 L^{3} e^{2L t}}{\pi} + \lVert \mathcal{R}^{\Xi} (-\Delta)^{m_{2}} \Xi_{0}(t) \rVert_{L_{x}^{1}}  \label{estimate 113}\\
& \hspace{5mm} + (\lVert v_{0} \otimes z_{2} \rVert_{L_{x}^{1}} + \lVert z_{1} \otimes \Xi_{0} \rVert_{L_{x}^{1}} + \lVert z_{1} \otimes z_{2} \rVert_{L_{x}^{1}} + \lVert \Xi_{0} \otimes z_{1} \rVert_{L_{x}^{1}} + \lVert z_{2} \otimes v_{0} \rVert_{L_{x}^{1}} + \lVert z_{2} \otimes z_{1} \rVert_{L_{x}^{1}}) (t). \nonumber 
\end{align} 
As $\Delta \Xi_{0} = -\Xi_{0}$, similarly to \eqref{estimate 110} we can estimate via \eqref{estimate 112} and \eqref{estimate 109}, 
\begin{equation}\label{estimate 114} 
 \lVert \mathcal{R}^{\Xi} (-\Delta)^{m_{2}} \Xi_{0}(t) \rVert_{L_{x}^{1}} 
\leq  6 (2\pi)^{\frac{3}{2}} \lVert (-\Delta)^{m_{2} - \frac{1}{2}} \Xi_{0}(t) \rVert_{L_{x}^{2}} \leq 6 (2\pi)^{\frac{3}{2}} \lVert \Xi_{0}(t) \rVert_{L_{x}^{2}}  =  6 M_{0}(t)^{\frac{1}{2}}. 
\end{equation}
Finally, we estimate 
\begin{align}
& (\lVert v_{0} \otimes z_{2} \rVert_{L_{x}^{1}} + \lVert z_{1} \otimes \Xi_{0} \rVert_{L_{x}^{1}} + \lVert z_{1} \otimes z_{2} \rVert_{L_{x}^{1}} + \lVert \Xi_{0} \otimes z_{1} \rVert_{L_{x}^{1}} + \lVert z_{2} \otimes v_{0} \rVert_{L_{x}^{1}} + \lVert z_{2} \otimes z_{1} \rVert_{L_{x}^{1}})(t) \nonumber \\
\overset{\eqref{estimate 109}\eqref{estimate 95}}{\leq}& 36 \pi^{\frac{3}{2}} M_{0}(t)^{\frac{1}{2}} L^{\frac{1}{4}} + 18 L^{\frac{1}{4}} M_{0}(t)^{\frac{1}{2}} + 144 \pi^{3} L^{\frac{1}{2}}. \label{estimate 115}
\end{align} 
Thus, applying \eqref{estimate 114}-\eqref{estimate 115} to \eqref{estimate 113} gives us 
\begin{align}
\lVert \mathring{R}_{0}^{\Xi} (t) \rVert_{L_{x}^{1}} \overset{\eqref{estimate 113} \eqref{estimate 114} \eqref{estimate 115} \eqref{estimate 106a}}{\leq} 20 \pi^{\frac{3}{2}} M_{0}(t) L^{-1} \overset{\eqref{estimate 106b}}{\leq}  c_{\Xi} M_{0}(t) \delta_{1}. 
\end{align} 
Finally, it is clear that $v_{0} (0,x)$ and $\Xi_{0}(0,x)$ are both deterministic. As $z_{1}(0,x) \equiv z_{2} (0,x) \equiv 0$ by \eqref{estimate 37}-\eqref{estimate 38}, we see that $\mathring{R}_{0}^{v} (0,x)$ and $\mathring{R}_{0}^{\Xi} (0,x)$ are both deterministic. 
\end{proof}

\begin{proposition}\label{Proposition 4.8}
Let $L$ satisfy 
\begin{equation}\label{estimate 116}
\max \{ (40)^{\frac{4}{3}}, (( 2\pi)^{3} + 1)^{2} 20 \pi^{\frac{3}{2}} \max \{ \frac{1}{c_{v}}, \frac{1}{c_{\Xi}} \} \} < L 
\end{equation} 
and suppose that $(v_{q}, \Xi_{q}, \mathring{R}_{q}^{v}, \mathring{R}_{q}^{\Xi})$ are $(\mathcal{F}_{t})_{t\geq 0}$-adapted processes that solve \eqref{estimate 104} and satisfy \eqref{estimate 105}. Then there exist a choice of parameters $a, b,$ and $\beta$ such that \eqref{estimate 106} is fulfilled and $(\mathcal{F}_{t})_{t\geq 0}$-adapted processes $(v_{q+1}, \Xi_{q+1}, \mathring{R}_{q+1}^{v}, \mathring{R}_{q+1}^{\Xi})$ that satisfy \eqref{estimate 104} and \eqref{estimate 105} at level $q+1$, and for all $t \in [0, T_{L}]$, 
\begin{equation}\label{estimate 117}
\lVert v_{q+1}(t) - v_{q}(t) \rVert_{L_{x}^{2}} \leq M_{0}(t)^{\frac{1}{2}} \delta_{q+1}^{\frac{1}{2}} \hspace{1mm} \text{ and } \hspace{1mm} \lVert \Xi_{q+1}(t) - \Xi_{q}(t) \rVert_{L_{x}^{2}} \leq M_{0}(t)^{\frac{1}{2}} \delta_{q+1}^{\frac{1}{2}}. 
\end{equation} 
Finally, if $(v_{q}, \Xi_{q}, \mathring{R}_{q}^{v}, \mathring{R}_{q}^{\Xi}) (0,x)$ are deterministic, then so are $(v_{q+1}, \Xi_{q+1}, \mathring{R}_{q+1}^{v}, \mathring{R}_{q+1}^{\Xi})(0,x)$. 
\end{proposition}

Taking Proposition \ref{Proposition 4.8} for granted, we are ready to prove Theorem \ref{Theorem 2.1}.
\begin{proof}[Proof of Theorem \ref{Theorem 2.1}]
Given any $T> 0, K > 1$, and $\kappa \in (0,1)$, starting from $(v_{0}, \Xi_{0}, \mathring{R}_{0}^{v}, \mathring{R}_{0}^{\Xi})$ in Proposition \ref{Proposition 4.7}, by taking $L > 0$ sufficiently large that satisfies \eqref{estimate 116}, Proposition \ref{Proposition 4.8} gives us $(v_{q}, \Xi_{q}, \mathring{R}_{q}^{v}, \mathring{R}_{q}^{\Xi})$ for all $q \in \mathbb{N}$ that satisfy \eqref{estimate 104}, \eqref{estimate 105}, and \eqref{estimate 117}. Then, for all $\gamma \in (0, \frac{\beta}{4+ \beta})$ and $t \in [0, T_{L}]$, by Gagliardo-Nirenberg's inequality and the fact that $b^{q+1} \geq b(q+1)$ for all $q \geq 0$ and $b\geq 2$, we can compute 
\begin{align}
& \sum_{q \geq 0} \lVert v_{q+1}(t) - v_{q}(t) \rVert_{\dot{H}_{x}^{\gamma}} + \lVert \Xi_{q+1} (t) - \Xi_{q} (t) \rVert_{\dot{H}_{x}^{\gamma}} \label{estimate 118}   \\
\overset{\eqref{estimate 117} \eqref{estimate 99}}{\lesssim}& \sum_{q \geq 0} (M_{0}(t)^{\frac{1}{2}} \delta_{q+1}^{\frac{1}{2}})^{1-\gamma} (M_{0}(t)^{\frac{1}{2}} \lambda_{q+1}^{4})^{\gamma} \lesssim M_{0}(t)^{\frac{1}{2}} \sum_{q \geq 0} a^{b (q+1) [-\beta (1- \gamma ) + 4 \gamma ]} \lesssim M_{0}(L)^{\frac{1}{2}}. \nonumber 
\end{align} 
Thus, $\{v_{q} \}_{q=0}$ and $ \{\Xi_{q}\}_{q = 0}$ are both Cauchy in $C([0, T_{L}]; \dot{H}^{\gamma} (\mathbb{T}^{3}))$ and thus we deduce the limiting processes $\lim_{q\to \infty} v_{q} \triangleq v$ and $\lim_{q\to\infty} \Xi_{q} \triangleq \Xi$ both in $C([0, T_{L} ]; \dot{H}^{\gamma} (\mathbb{T}^{3}))$ for which there exists a deterministic constant $C_{L} > 0$ such that 
\begin{equation}\label{estimate 119}
\lVert v \rVert_{C([0, T_{L} ]; \dot{H}_{x}^{\gamma})} + \lVert\Xi \rVert_{C([0, T_{L} ] ; \dot{H}_{x}^{\gamma} )} \leq C_{L}. 
\end{equation}
Because $(v_{q}, \Xi_{q})$ are $(\mathcal{F}_{t})_{t\geq 0}$-adapted due to Propositions \ref{Proposition 4.7}-\ref{Proposition 4.8}, we see that $(v, \Xi)$ are $(\mathcal{F}_{t})_{t\geq 0}$-adapted. Additionally, because for all $t \in [0, T_{L}]$, $\lVert \mathring{R}_{q}^{v} \rVert_{C_{t}L_{x}^{1}}, \lVert \mathring{R}_{q}^{\Xi} \rVert_{C_{t}L_{x}^{1}} \to 0$ as $q\to \infty$ due to \eqref{estimate 100}, $(v, \Xi)$ is a weak solution to \eqref{estimate 399}. Consequently from \eqref{z} and \eqref{estimate 399}, we see that $(u,b) = (v + z_{1}, \Xi + z_{2})$ solves \eqref{stochastic GMHD}. We see that $v(0,x) = u^{\text{in}}(x)$ and $\Xi (0,x) = b^{\text{in}} (x)$ as $z_{1} (0,x) \equiv 0$ and $z_{2}(0,x) \equiv 0$ due to \eqref{z} and choose $L = L(T, K, c_{v}, c_{\Xi}, \lVert u^{\text{in}} \rVert_{L_{x}^{2}}, \lVert b^{\text{in}} \rVert_{L_{x}^{2}}) > 0$ that satisfies \eqref{estimate 116} to be larger if necessary to satisfy 
\begin{subequations}\label{estimate 400}
\begin{align}
& ( \frac{1}{(2\pi)^{3}} + \frac{1}{(2\pi)^{\frac{3}{2}}}) + \frac{1}{L} < ( \frac{1}{(2\pi)^{\frac{3}{2}}} - \frac{1}{(2\pi)^{3}})e^{LT}, \nonumber\\
& \hspace{3mm} \text{ which is equivalent to } (( \frac{1}{(2\pi)^{3}} + \frac{1}{(2\pi)^{\frac{3}{2}}}) L^{2} + L) e^{LT} < (\frac{1}{(2\pi)^{\frac{3}{2}}} - \frac{1}{(2\pi)^{3}}) M_{0}(T)^{\frac{1}{2}}, \label{estimate 120}\\
& K \lVert u^{\text{in}} \rVert_{L_{x}^{2}} + L^{\frac{1}{4}} (2\pi)^{\frac{3}{2}} + K \sum_{k=1}^{2} \sqrt{ T Tr (G_{k}G_{k}^{\ast} )} \leq (e^{LT} -K) \lVert b^{\text{in}} \rVert_{L_{x}^{2}} + L e^{LT}. \label{estimate 121} 
\end{align}
\end{subequations}
Now because $\lim_{L\to\infty} T_{L} = + \infty$ $\textbf{P}$-a.s., for the fixed $T > 0$ and $\kappa > 0$, increasing $L$ sufficiently larger if necessary allows us to obtain \eqref{estimate 19}. Next, we see from \eqref{estimate 75} that $z_{1}(t)$ and $z_{2}(t)$ are both $(\mathcal{F}_{t})_{t\geq 0}$-adapted. As we already verified that $(v, \Xi)$ are $(\mathcal{F}_{t})_{t\geq 0}$-adapted, we see that $(u,b)$ are $(\mathcal{F}_{t})_{t\geq 0}$-adapted. Moreover, \eqref{estimate 95} and \eqref{estimate 119} imply \eqref{estimate 20}. Next, as $b^{q+1} \geq b(q+1)$ for all $q \in \mathbb{N}$ if $b \geq 2$, we can compute for all $t \in [0, T_{L}]$ 
\begin{equation}\label{estimate 125}
\lVert \Xi(t) - \Xi_{0}(t) \rVert_{L_{x}^{2}} \overset{\eqref{estimate 117}}{\leq} M_{0}(t)^{\frac{1}{2}} \sum_{q\geq 0} \delta_{q+1}^{\frac{1}{2}} \overset{\eqref{estimate 106}}{<} M_{0}(t)^{\frac{1}{2}} \frac{1}{(2\pi)^{3}}. 
\end{equation}
This allows us to compute  
\begin{align}
(\lVert \Xi(0) \rVert_{L_{x}^{2}} + L) e^{LT} \leq& ( \lVert \Xi(0) - \Xi_{0}(0) \rVert_{L_{x}^{2}} + \lVert \Xi_{0} (0) \rVert_{L_{x}^{2}} + L) e^{LT} \nonumber\\
\overset{\eqref{estimate 125}  \eqref{estimate 109}}{\leq}& ( \frac{M_{0}(0)^{\frac{1}{2}} }{(2\pi)^{3}}+ \frac{ M_{0}(0)^{\frac{1}{2}}}{(2\pi)^{\frac{3}{2}}} + L) e^{LT} \overset{\eqref{estimate 120}}{<} ( \frac{1}{(2\pi)^{\frac{3}{2}}} - \frac{1}{(2\pi)^{3}}) M_{0}(T)^{\frac{1}{2}}  \nonumber \\
\overset{\eqref{estimate 109} \eqref{estimate 125}}{<}& \lVert \Xi_{0}(T) \rVert_{L_{x}^{2}} - \lVert \Xi(T) - \Xi_{0}(T) \rVert_{L_{x}^{2}} \leq \lVert \Xi(T) \rVert_{L_{x}^{2}}. \label{estimate 127}
\end{align}
Therefore, on $\{T_{L} \geq T \}$, we are now able to deduce \eqref{estimate 21} as follows: 
\begin{align}
 \lVert b(T) \rVert_{L_{x}^{2}} \geq& \lVert \Xi (T) \rVert_{L_{x}^{2}} - \lVert z_{2}(T) \rVert_{L_{x}^{2}} \overset{ \eqref{estimate 127}}{>}(\lVert \Xi(0) \rVert_{L_{x}^{2}} + L) e^{LT} - \lVert z_{2}(T) \rVert_{L_{x}^{\infty}} (2\pi)^{\frac{3}{2}} \\
\overset{\eqref{estimate 95}}{\geq}& ( \lVert b^{\text{in}} \rVert_{L_{x}^{2}} + L) e^{LT} - L^{\frac{1}{4}} (2\pi)^{\frac{3}{2}} 
\overset{\eqref{estimate 121} }{\geq} K [\lVert u^{\text{in}} \rVert_{L_{x}^{2}} + \lVert b^{\text{in}} \rVert_{L_{x}^{2}} + \sum_{k=1}^{2} \sqrt{T Tr (G_{k}G_{k}^{\ast} )}].  \nonumber
\end{align}
Finally, because $v_{0} (0,x)$ and $\Xi_{0} (0,x)$ from Proposition \ref{Proposition 4.7} are deterministic, Proposition \ref{Proposition 4.8} implies that $v(0,x)$ and $\Xi(0,x)$ remain deterministic; as $z_{1}(0,x) \equiv z_{2}(0,x) \equiv 0$ by \eqref{z}, we conclude that $u^{\text{in}}$ and $b^{\text{in}}$ are both deterministic. 
\end{proof}

\subsection{Proof of Proposition \ref{Proposition 4.8}}
\subsubsection{Choice of parameters}\label{Subsection 4.1} 
We define 
\begin{equation}\label{estimate 128}
m_{1}^{\ast} \triangleq 
\begin{cases}
0 & \text{ if } m_{1} \in (0, \frac{1}{2}], \\
2m_{1} - 1 & \text{ if } m_{1} \in (\frac{1}{2}, 1), 
\end{cases} \hspace{2mm} 
\text{ and } \hspace{2mm} 
m_{2}^{\ast} \triangleq 
\begin{cases}
0 & \text{ if } m_{2} \in (0, \frac{1}{2}], \\
2m_{2} - 1 & \text{ if } m_{2} \in (\frac{1}{2}, 1), 
\end{cases}
\end{equation} 
from which it follows that $m_{1}^{\ast}, m_{2}^{\ast} \in [0, 1)$. We fix 
\begin{equation}\label{eta}
\eta \in \mathbb{Q}_{+} \cap ( \frac{1- \max \{m_{1}^{\ast}, m_{2}^{\ast} \}}{16}, \frac{1- \max \{m_{1}^{\ast}, m_{2}^{\ast} \}}{8} ] \subset (0, \frac{1}{8}].
\end{equation} 
Let us fix 
\begin{equation}\label{alpha}
\alpha \triangleq \frac{ 1- \max\{m_{1}^{\ast}, m_{2}^{\ast} \}}{1600}. 
\end{equation}
We let ``$\lambda$'' in the convex integration from Section \ref{Subsection 3.2} to be $\lambda_{q+1}$ and choose
\begin{equation}\label{sigma, r, mu}
\sigma \triangleq \lambda_{q+1}^{2\eta - 1}, \hspace{1mm} r \triangleq \lambda_{q+1}^{6\eta - 1}, \hspace{1mm} \text{ and } \hspace{1mm} \mu \triangleq \lambda_{q+1}^{1-\eta}. 
\end{equation}
As $\eta \in (0, \frac{1}{8}]$, we see that this choice satisfies \eqref{estimate 50} and \eqref{estimate 53}. Using the fact that $\eta \in \mathbb{Q}_{+}$, we can also fulfill \eqref{estimate 52} by appropriately choosing 
\begin{equation}\label{b}
b \in \{ \iota \in \mathbb{N}: \iota > \frac{16}{\alpha} \} 
\end{equation}
as long as $a \in \mathbb{N}$. In fact, we will take $a \in 2 \mathbb{N}$ so that $\lambda_{q+1} \sigma \in 2 \mathbb{N}$ because we will apply Lemma \ref{Lemma 6.3} with ``$k \in \mathbb{N}$'' of its hypothesis being $\frac{\lambda_{q+1} \sigma}{2}$. For convenience, we also choose 
\begin{equation}\label{l}
l \triangleq \lambda_{q+1}^{ - \frac{3\alpha}{2}} \lambda_{q}^{-2},  
\end{equation}
from which these useful estimates follow:
\begin{equation}\label{estimate 130}
l \lambda_{q}^{4} \overset{\eqref{b} \eqref{l}}{<} \lambda_{q+1}^{-\alpha}, \hspace{2mm} L \lesssim l^{-1}, \hspace{2mm}  \text{ and } \hspace{2mm} l^{-1} \overset{\eqref{b}}{<} \lambda_{q+1}^{\frac{13 \alpha}{8}}.
\end{equation}
Concerning \eqref{estimate 106}, due to the hypothesis \eqref{estimate 116}, we can choose $a \in 2\mathbb{N}$ sufficiently large so that the last inequality in \eqref{estimate 106b} is satisfied while $\beta > 0$ sufficiently small so that the first and second inequalities in \eqref{estimate 106b} hold. Thus, we consider such $m_{1}^{\ast}, m_{2}^{\ast}, \eta, \alpha, \sigma, r, \mu, b,$ and $l$ fixed, preserving our freedom to take $a \in 2\mathbb{N}$ larger and $\beta > 0$ smaller as needed, at minimum guaranteeing 
\begin{equation}\label{estimate 129}
16 \beta b < \alpha. 
\end{equation} 
\subsubsection{Mollification} 
In addition to the mollifiers $\{ \vartheta_{l}\}_{l}$ from Section \ref{Section 2}, we let $\{ \varrho_{l}\}_{l > 0}$, specifically $\varrho_{l} (\cdot) \triangleq  l^{-3} \varrho(\frac{\cdot}{l})$, be a family of standard mollifiers on $\mathbb{R}^{3}$ with mass one. We extend $v_{q}, \Xi_{q}, \mathring{R}_{q}^{v}, \mathring{R}_{q}^{\Xi}$, and $z_{k}$ for $k \in \{1,2\}$ to $t < 0$ by their respective values at $t =0 $ and mollify in space and time to obtain 
\begin{subequations}\label{estimate 291}
\begin{align}
&v_{l} \triangleq (v_{q} \ast_{x} \varrho_{l}) \ast_{t} \vartheta_{l}, \hspace{5mm} \Xi_{l} \triangleq (\Xi_{q}\ast_{x} \varrho_{l}) \ast_{t} \vartheta_{l}, \\
&\mathring{R}_{l}^{v} \triangleq (\mathring{R}_{q}^{v} \ast_{x} \varrho_{l}) \ast_{t} \vartheta_{l}, \hspace{3mm} \mathring{R}_{l}^{\Xi} \triangleq (\mathring{R}_{q}^{v} \ast_{x} \varrho_{l})\ast_{t} \vartheta_{l}, \hspace{3mm} z_{k,l} \triangleq (z_{k} \ast_{x} \varrho_{l} ) \ast_{t} \vartheta_{l}. 
\end{align}
\end{subequations}
One can directly verify from \eqref{estimate 104} that the corresponding mollified system of equations is 
\begin{subequations}\label{estimate 204}
\begin{align}
& \partial_{t} v_{l} + (-\Delta)^{m_{1}} v_{l} + \text{div} ((v_{l} + z_{1,l} ) \otimes (v_{l} + z_{1,l} ) - (\Xi_{l} + z_{2,l}) \otimes (\Xi_{l} + z_{2,l} )) + \nabla \pi_{l}  \nonumber \\
& \hspace{55mm} = \text{div} ( \mathring{R}_{l}^{v} + R_{\text{com1}}^{v}), \hspace{5mm} \nabla \cdot v_{l} = 0,  \\
& \partial_{t} \Xi_{l} + (-\Delta)^{m_{2}} \Xi_{l} + \text{div} (( v_{l} + z_{1,l} ) \otimes (\Xi_{l} + z_{2,l}) - (\Xi_{l} + z_{2,l} ) \otimes (v_{l} + z_{1,l} )) \nonumber \\
& \hspace{55mm} = \text{div} ( \mathring{R}_{l}^{\Xi} + R_{\text{com1}}^{\Xi}), \hspace{4mm} \nabla \cdot \Xi_{l} = 0, 
\end{align}
\end{subequations} 
where 
\begin{subequations}\label{estimate 235}
\begin{align}
\pi_{l} \triangleq& \pi_{q} \ast_{x} \varrho_{l} \ast_{t} \vartheta_{l} - \frac{1}{3} (\lvert v_{l} + z_{1,l} \rvert^{2} - \lvert \Xi_{l} + z_{2,l} \rvert^{2})   \nonumber\\
& \hspace{18mm} + \frac{1}{3} (\lvert v_{q} + z_{1} \rvert^{2} - \lvert \Xi_{q} + z_{2} \rvert^{2}) \ast_{x} \varrho_{l} \ast_{t} \vartheta_{l},  \\
R_{\text{com1}}^{v} \triangleq& (v_{l} + z_{1,l} )\mathring{\otimes} (v_{l} + z_{1,l}) - (\Xi_{l} + z_{2,l}) \mathring{\otimes} (\Xi_{l} + z_{2,l}) \nonumber\\
& \hspace{10mm} - ((v_{q} + z_{1}) \mathring{\otimes} (v_{q} +z_{1}) - (\Xi_{q} + z_{2}) \mathring{\otimes} (\Xi_{q} + z_{2} )) \ast_{x} \varrho_{l} \ast_{t} \vartheta_{l}, \\
R_{\text{com1}}^{\Xi} \triangleq & (v_{l} + z_{1,l}) \otimes (\Xi_{l} + z_{2,l}) - (\Xi_{l} + z_{2,l}) \otimes (v_{l} + z_{1,l}) \nonumber \\
& \hspace{10mm} - ((v_{q} + z_{1}) \otimes (\Xi_{q} + z_{2} ) - (\Xi_{q} + z_{2}) \otimes (v_{q} + z_{1} )) \ast_{x} \varrho_{l} \ast_{t} \vartheta_{l}. 
\end{align}
\end{subequations}
Next, we can estimate 
\begin{equation}\label{estimate 293}
\lVert R_{\text{com1}}^{v} \rVert_{C_{t}L_{x}^{1}} \leq \sum_{k=1}^{8} I_{k} 
\end{equation} 
where 
\begin{subequations}\label{estimate 294}
\begin{align}
I_{1} \triangleq& \lVert v_{l} \mathring{\otimes} v_{l} - (v_{q} \mathring{\otimes} v_{q}) \ast_{x} \varrho_{l} \ast_{t} \vartheta_{l} \rVert_{C_{t}L_{x}^{1}}, \hspace{5mm}  I_{2} \triangleq \lVert z_{1,l} \mathring{\otimes} z_{1,l} - (z_{1} \mathring{\otimes} z_{1}) \ast_{x} \varrho_{l}\ast_{t}\vartheta_{l} \rVert_{C_{t}L_{x}^{1}}, \\
I_{3} \triangleq& \lVert v_{l} \mathring{\otimes} z_{1,l} - (v_{q} \mathring{\otimes} z_{1}) \ast_{x} \varrho_{l} \ast_{t} \vartheta_{l} \rVert_{C_{t}L_{x}^{1}},  \hspace{4mm} 
I_{4} \triangleq \lVert z_{1,l} \mathring{\otimes} v_{l} - (z_{1} \mathring{\otimes} v_{q}) \ast_{x} \varrho_{l} \ast_{t} \vartheta_{l} \rVert_{C_{t}L_{x}^{1}}, \\
I_{5} \triangleq& \lVert \Xi_{l} \mathring{\otimes} \Xi_{l} - (\Xi_{q} \mathring{\otimes} \Xi_{q}) \ast_{x} \varrho_{l} \ast_{t} \vartheta_{l} \rVert_{C_{t}L_{x}^{1}},  \hspace{3mm} 
I_{6} \triangleq \lVert z_{2,l} \mathring{\otimes} z_{2,l} - (z_{2} \mathring{\otimes} z_{2}) \ast_{x} \varrho_{l} \ast_{t} \vartheta_{l} \rVert_{C_{t}L_{x}^{1}}, \\
I_{7} \triangleq& \lVert \Xi_{l} \mathring{\otimes} z_{2,l} - (\Xi_{q} \mathring{\otimes} z_{2}) \ast_{x} \varrho_{l} \ast_{t} \vartheta_{l} \rVert_{C_{t}L_{x}^{1}}, \hspace{3mm} 
I_{8} \triangleq \lVert z_{2,l} \mathring{\otimes} \Xi_{l} - (z_{2} \mathring{\otimes} \Xi_{q}) \ast_{x} \varrho_{l} \ast_{t} \vartheta_{l} \rVert_{C_{t}L_{x}^{1}}. 
\end{align}
\end{subequations}
Further applying Minkowski's inequalities such as 
\begin{align*}
I_{1} +I_{5} 
\lesssim& \lVert v_{l} - v_{q} \rVert_{C_{t}L_{x}^{\infty}} \lVert v_{q} \rVert_{C_{t}L_{x}^{1}} + \lVert \Xi_{l} - \Xi_{q} \rVert_{C_{t}L_{x}^{\infty}} \lVert \Xi_{q} \rVert_{C_{t}L_{x}^{1}} \\
&+ \lVert v_{q} \mathring{\otimes} v_{q} - (v_{q} \mathring{\otimes} v_{q}) \ast_{x} \varrho_{l} \ast_{t} \vartheta_{l} \rVert_{C_{t}L_{x}^{1}} + \lVert \Xi_{q} \mathring{\otimes} \Xi_{q} - (\Xi_{q} \mathring{\otimes} \Xi_{q}) \ast_{x} \varrho_{l} \ast_{t} \vartheta_{l} \rVert_{C_{t}L_{x}^{1}}, 
\end{align*}
and applying standard mollifier estimates (e.g., \cite[Lem. 1]{CDS12}) lead to 
\begin{align}
\lVert R_{\text{com1}}^{v} \rVert_{C_{t}L_{x}^{1}} 
\lesssim& l ( \lVert v_{q} \rVert_{C_{t,x}^{1}} + \lVert \Xi_{q} \rVert_{C_{t,x}^{1}} + \sum_{k=1}^{2} \lVert z_{k} \rVert_{C_{t}C_{x}^{1}}) ( \lVert v_{q} \rVert_{C_{t}L_{x}^{2}} + \lVert \Xi_{q} \rVert_{C_{t}L_{x}^{2}} + \sum_{k=1}^{2} \lVert z_{k} \rVert_{C_{t,x}}) \nonumber \\
&+ l^{\frac{1}{2} - 2\delta} \sum_{k=1}^{2} \lVert z_{k} \rVert_{C_{t}^{\frac{1}{2} - 2 \delta} C_{x}} ( \sum_{k=1}^{2} \lVert z_{k} \rVert_{C_{t,x}} + \lVert v_{q} \rVert_{C_{t} L_{x}^{2}} + \lVert \Xi_{q} \rVert_{C_{t}L_{x}^{2}}).  \label{estimate 247}
\end{align}
Similar applications of Minkowski's inequalities and standard mollifier estimates give 
\begin{align}
\lVert R_{\text{com1}}^{\Xi} \rVert_{C_{t}L_{x}^{1}} \leq& \sum_{k=1}^{8} II_{k}  \nonumber \\
\lesssim& l ( \lVert v_{q} \rVert_{C_{t,x}^{1}} + \lVert \Xi_{q} \rVert_{C_{t,x}^{1}} + \sum_{k=1}^{2} \lVert z_{k} \rVert_{C_{t}C_{x}^{1}}) ( \lVert v_{q} \rVert_{C_{t}L_{x}^{2}} + \lVert \Xi_{q} \rVert_{C_{t}L_{x}^{2}} + \sum_{k=1}^{2} \lVert z_{k} \rVert_{C_{t,x}})  \nonumber \\
&+ l^{\frac{1}{2} - 2 \delta} \sum_{k=1}^{2} \lVert z_{k} \rVert_{C_{t}^{\frac{1}{2} - 2 \delta} C_{x}} ( \sum_{k=1}^{2} \lVert z_{k} \rVert_{C_{t,x}} + \lVert v_{q} \rVert_{C_{t}L_{x}^{2}}+ \lVert \Xi_{q} \rVert_{C_{t}L_{x}^{2}}), \label{estimate 248}
\end{align}
where we split $\lVert R_{\text{com1}}^{\Xi} \rVert_{C_{t}L_{x}^{1}}$ to  
\begin{subequations}\label{estimate 295}
\begin{align}
II_{1} \triangleq& \lVert v_{l} \otimes \Xi_{l} - (v_{q} \otimes \Xi_{q}) \ast_{x} \varrho_{l} \ast_{t} \vartheta_{l} \rVert_{C_{t}L_{x}^{1}}, \\
II_{2} \triangleq& \lVert z_{1,l} \otimes z_{2,l} - (z_{1} \otimes z_{2}) \ast_{x} \varrho_{l} \ast_{t} \vartheta_{l} \rVert_{C_{t}L_{x}^{1}}, \\
II_{3} \triangleq& \lVert v_{l} \otimes z_{2,l} - (v_{q} \otimes z_{2}) \ast_{x} \varrho_{l} \ast_{t} \vartheta_{l} \rVert_{C_{t}L_{x}^{1}}, \\
II_{4} \triangleq& \lVert z_{1,l} \otimes \Xi_{l} - (z_{1} \otimes \Xi_{q}) \ast_{x} \varrho_{l} \ast_{t} \vartheta_{l} \rVert_{C_{t}L_{x}^{1}}, \\
II_{5} \triangleq& \lVert \Xi_{l} \otimes v_{l} - (\Xi_{q} \otimes v_{q}) \ast_{x} \varrho_{l} \ast_{t} \vartheta_{l} \rVert_{C_{t}L_{x}^{1}}, \\
II_{6} \triangleq& \lVert z_{2,l} \otimes z_{1,l} - (z_{2} \otimes z_{1}) \ast_{x} \varrho_{l} \ast_{t} \vartheta_{l} \rVert_{C_{t}L_{x}^{1}}, \\
II_{7} \triangleq& \lVert \Xi_{l} \otimes z_{1,l} - (\Xi_{q} \otimes z_{1}) \ast_{x} \varrho_{l} \ast_{t} \vartheta_{l} \rVert_{C_{t}L_{x}^{1}}, \\
II_{8} \triangleq& \lVert z_{2,l} \otimes v_{l} - (z_{2} \otimes v_{q}) \ast_{x} \varrho_{l} \ast_{t} \vartheta_{l} \rVert_{C_{t}L_{x}^{1}}. 
\end{align}
\end{subequations}
For $a \in 2 \mathbb{N}$ sufficiently large, we can also estimate by Young's inequality for convolution and the fact that mollifiers have unit mass 
\begin{subequations}\label{estimate 301}
\begin{align}
& \lVert v_{q} - v_{l} \rVert_{C_{t}L_{x}^{2}} + \lVert \Xi_{q} - \Xi_{l} \rVert_{C_{t}L_{x}^{2}}\overset{\eqref{estimate 99}}{\lesssim} l M_{0}(t)^{\frac{1}{2}} \lambda_{q}^{4}  \overset{\eqref{estimate 130}}{\lesssim} M_{0}(t)^{\frac{1}{2}} \lambda_{q+1}^{-\alpha} \overset{\eqref{estimate 129}}{\ll} M_{0}(t)^{\frac{1}{2}} \delta_{q+1}^{\frac{1}{2}}, \label{estimate 199} \\
& \lVert v_{l} \rVert_{C_{t}L_{x}^{2}} \leq \lVert v_{q} \rVert_{C_{t}L_{x}^{2}} \leq M_{0}(t)^{\frac{1}{2}} (1+ \sum_{1 \leq \iota \leq q} \delta_{\iota}^{\frac{1}{2}}), \label{estimate 200}\\
& \lVert \Xi_{l} \rVert_{C_{t}L_{x}^{2}}\leq\lVert \Xi_{q}\rVert_{C_{t}L_{x}^{2}}\leq M_{0}(t)^{\frac{1}{2}}(1+\sum_{1\leq \iota\leq q} \delta_{\iota}^{\frac{1}{2}}). \label{estimate 201}
\end{align}
\end{subequations}
\subsubsection{Perturbation} 
We let $\chi: [0,\infty) \mapsto \mathbb{R}$ be a smooth function such that 
\begin{equation}\label{estimate 131}
\chi(z) \triangleq 
\begin{cases}
1 & \text{ if } z \in [0, 1],\\
z & \text{ if } z \geq 2, 
\end{cases}
\end{equation} 
and $z \leq 2 \chi(z) \leq 4 z$ for $z \in (1,2)$. We define for $t \in [0, T_{L}]$, 
\begin{equation}\label{estimate 133}
\rho_{\Xi} (t, x) \triangleq 2 \delta_{q+1} \epsilon_{\Xi}^{-1} c_{\Xi} M_{0}(t) \chi ( \frac{ \lvert \mathring{R}_{l}^{\Xi} (t ,x) \rvert}{c_{\Xi} \delta_{q+1} M_{0}(t)} ), 
\end{equation} 
where $\epsilon_{\Xi} > 0$ is from Lemma \ref{Lemma 3.1} and $c_{\Xi} > 0$ is from \eqref{estimate 100}. It follows from \eqref{estimate 131} that 
\begin{equation}\label{estimate 132}
 \lvert \frac{ \mathring{R}_{l}^{\Xi} (t, x) \rvert}{\rho_{\Xi} (t,x)} \rvert = \lvert \frac{ \mathring{R}_{l}^{\Xi} (t,x)}{2 \delta_{q+1} \epsilon_{\Xi}^{-1} c_{\Xi} M_{0}(t) \chi ((c_{\Xi} \delta_{q+1} M_{0}(t))^{-1} \lvert \mathring{R}_{l}^{\Xi} (t, x) \rvert )} \rvert \leq \epsilon_{\Xi}.
\end{equation} 
Moreover, for all $p \in [1,\infty)$ we have the following estimate: 
\begin{equation}\label{estimate 134}
\lVert \rho_{\Xi} \rVert_{C_{t}L_{x}^{p}} 
\overset{\eqref{estimate 133}\eqref{estimate 131} }{\leq}12 \epsilon_{\Xi}^{-1} ((8 \pi^{3})^{\frac{1}{p}} \delta_{q+1} c_{\Xi} M_{0} (t) + \lVert \mathring{R}_{l}^{\Xi} \rVert_{C_{t}L_{x}^{p}}).
\end{equation} 
Using $W^{4,1} (\mathbb{T}^{3}) \hookrightarrow L^{\infty} (\mathbb{T}^{3})$ and chain rule estimates such as \cite[Equ. (130)]{BDIS15} we can attain the following bounds: 
\begin{subequations}\label{estimate 394}
\begin{align}
& \lVert \rho_{\Xi} \rVert_{C_{t,x}} 
\overset{\eqref{estimate 134}}{\leq}12 \epsilon_{\Xi}^{-1} (\delta_{q+1} c_{\Xi} M_{0}(t) + \lVert \mathring{R}_{l}^{\Xi} \rVert_{C_{t}L_{x}^{\infty}})\overset{\eqref{estimate 100}}{\lesssim}  l^{-4} M_{0}(t) \delta_{q+1}, \label{estimate 135}\\
&  \lVert \rho_{\Xi} \rVert_{C_{t}C_{x}^{j}} \lesssim \delta_{q+1} \epsilon_{\Xi}^{-1} c_{\Xi} \lVert M_{0}(s) [ (c_{\Xi} \delta_{q+1} M_{0}(s))^{-1} \lVert \mathring{R}_{l}^{\Xi} (s) \rVert_{C_{x}^{j}} + (c_{\Xi} \delta_{q+1} M_{0}(s))^{-j} \lVert \mathring{R}_{l}^{\Xi} (s) \rVert_{C_{x}^{1}}^{j}] \rVert_{C_{t}} \nonumber\\
& \hspace{70mm}  \overset{\eqref{estimate 100}}{\lesssim} \delta_{q+1} M_{0}(t) l^{-5j} \hspace{2mm} \forall \hspace{1mm} j \geq 1, \label{estimate 136} \\
& \lVert \rho_{\Xi} \rVert_{C_{t}^{1}C_{x}^{j}} \lesssim l^{-5j - 5} M_{0}(t) \delta_{q+1} \hspace{2mm} \hspace{61mm} \forall \hspace{1mm}  j \geq 0, \label{estimate 137} \\
& \lVert \rho_{\Xi} \rVert_{C_{t}^{2}C_{x}} \overset{\eqref{estimate 135} \eqref{estimate 100}}{\lesssim} l^{-10} M_{0}(t) \delta_{q+1}. \label{estimate 408}
\end{align}
\end{subequations} 
In particular, to deduce \eqref{estimate 137} one can use the fact that $\partial_{t} M_{0}(t) = 4L M_{0}(t)$ by \eqref{estimate 94}  so that 
\begin{equation}\label{estimate 409}
\partial_{t}\rho_{\Xi} \overset{\eqref{estimate 133}}{=} 4L \rho_{\Xi} + 2 \epsilon_{\Xi}^{-1} \chi'  ( (c_{\Xi} \delta_{q+1} M_{0}(t))^{-1} \lvert \mathring{R}_{l}^{\Xi} (t,x) \rvert ) (\partial_{t} \lvert \mathring{R}_{l}^{\Xi} \rvert - \lvert \mathring{R}_{l}^{\Xi}\rvert  4L ), 
\end{equation}
which also allows us to compute $\partial_{t}^{2} \rho_{\Xi}$ to deduce \eqref{estimate 408} as well. Immediate consequences of \eqref{estimate 135}-\eqref{estimate 137} using 
\begin{equation}\label{estimate 144}
\rho_{\Xi}(t) \overset{\eqref{estimate 133}\eqref{estimate 131}}{\geq} \delta_{q+1} \epsilon_{\Xi}^{-1}c_{\Xi} M_{0}(t), 
\end{equation} 
include the following estimates:
\begin{subequations}\label{estimate 421}
\begin{align}
& \lVert \rho_{\Xi}^{\frac{1}{2}} \rVert_{C_{t,x}} \overset{\eqref{estimate 135}}{\lesssim} \delta_{q+1}^{\frac{1}{2}} l^{-2} M_{0}(t)^{\frac{1}{2}}, \label{estimate 141}\\
& \lVert \rho_{\Xi}^{\frac{1}{2}} \rVert_{C_{t}C_{x}^{j}} \overset{\eqref{estimate 144} \eqref{estimate 136}}{\lesssim} \delta_{q+1}^{\frac{1}{2}} l^{-5j} M_{0}(t)^{\frac{1}{2}} \hspace{27mm} \forall \hspace{1mm} j \geq 1, \label{estimate 142}\\
& \lVert \rho_{\Xi}^{\frac{1}{2}} \rVert_{C_{t}^{1}C_{x}^{j}} \overset{\eqref{estimate 144} \eqref{estimate 136} \eqref{estimate 137}}{\lesssim} \delta_{q+1}^{\frac{1}{2}} l^{-5j - 5} M_{0}(t)^{\frac{1}{2}} \hspace{15mm}\forall \hspace{1mm} j \geq 0, \label{estimate 143} 
\end{align}
\end{subequations} 
where we used \cite[Equ. (130)]{BDIS15} in \eqref{estimate 142}-\eqref{estimate 143}. Now we define the magnetic amplitude functions
\begin{equation}\label{estimate 138}
a_{\xi} (t,x) \triangleq \rho_{\Xi}^{\frac{1}{2}}(t,x) \gamma_{\xi} ( - \frac{ \mathring{R}_{l}^{\Xi}(t,x)}{\rho_{\Xi}(t,x)})  \hspace{10mm} \forall \hspace{1mm}  \xi \in \Lambda_{\Xi} 
\end{equation} 
where $\gamma_{\xi}$ is that of Lemma \ref{Lemma 3.1} and $- \frac{\mathring{R}_{l}^{\Xi}}{\rho_{\Xi}} \in B_{\epsilon_{\Xi}}(0)$ due to \eqref{estimate 132}. It follows that 
\begin{equation}\label{estimate 139}
\sum_{\xi \in \Lambda_{\Xi}} a_{\xi}^{2} \mathbb{P}_{=0} (\phi_{\xi}^{2} \varphi_{\xi}^{2}) (\xi \otimes \xi_{2} - \xi_{2} \otimes \xi) 
\overset{\eqref{estimate 57}}{=} \sum_{\xi \in \Lambda_{\Xi}} a_{\xi}^{2} (\xi \otimes \xi_{2} - \xi_{2} \otimes \xi)  
\overset{\eqref{estimate 138} \eqref{estimate 44}}{=} - \mathring{R}_{l}^{\Xi}.
\end{equation} 
This leads to an identity of 
\begin{equation}\label{estimate 207}
 \sum_{\xi \in \Lambda_{\Xi}} a_{\xi}^{2} \phi_{\xi}^{2} \varphi_{\xi}^{2}  (\xi \otimes \xi_{2} - \xi_{2} \otimes \xi) + \mathring{R}_{l}^{\Xi} \overset{\eqref{estimate 139}}{=} \sum_{\xi \in \Lambda_{\Xi}} a_{\xi}^{2} \mathbb{P}_{\neq 0} (\phi_{\xi}^{2}\varphi_{\xi}^{2}) (\xi \otimes \xi_{2} - \xi_{2} \otimes \xi). 
\end{equation} 
Let us estimate for all $\xi \in \Lambda_{\Xi}, t \in [0, T_{L}]$ by taking $c_{\Xi}$ sufficiently small so that 
\begin{equation}\label{estimate 157}
c_{\Xi} \leq \min\{ \frac{c_{v} \epsilon_{\Xi}}{ 12 (8 \pi^{3} + 1) M_{\ast}^{2} \lvert \Lambda_{\Xi} \rvert }, \frac{ \epsilon_{\Xi}}{12 (8\pi^{3} + 1) M_{\ast}^{2} 9 C_{\ast}^{2} (8\pi^{3}) \lvert \Lambda_{\Xi} \rvert^{2}  } \}
\end{equation} 
for $M_{\ast}$ from \eqref{estimate 48}, and using the fact that mollifiers have mass one, 
\begin{align}
&\lVert a_{\xi} \rVert_{C_{t}L_{x}^{2}}  
\overset{\eqref{estimate 138}\eqref{estimate 132}}{\leq} \lVert \rho_{\Xi} \rVert_{C_{t}L_{x}^{1}}^{\frac{1}{2}} \lVert \gamma_{\xi} \rVert_{C(B_{\epsilon_{\Xi}}(0))}  \label{estimate 146}\\
\overset{\eqref{estimate 134}}{\leq}& (12 \epsilon_{\Xi}^{-1} (8 \pi^{3} \delta_{q+1} c_{\Xi} M_{0}(t) + \lVert \mathring{R}_{l}^{\Xi} \rVert_{C_{t}L_{x}^{1}})^{\frac{1}{2}} M_{\ast}   \overset{\eqref{estimate 105}}{\leq} \min \{ ( \frac{c_{v} }{\lvert \Lambda_{\Xi} \rvert} )^{\frac{1}{2}}, \frac{1}{3 C_{\ast} (8\pi^{3})^{\frac{1}{2}}  \lvert \Lambda_{\Xi} \rvert  } \} \delta_{q+1}^{\frac{1}{2}} M_{0}(t)^{\frac{1}{2}}  \nonumber 
\end{align}
for $C_{\ast}$ from Lemma \ref{Lemma 6.2}. The bound by $ ( \frac{c_{v}}{\lvert \Lambda_{\Xi} \rvert})^{\frac{1}{2}} \delta_{q+1}^{\frac{1}{2}} M_{0}(t)^{\frac{1}{2}}$ will be subsequently needed in \eqref{estimate 147} and the bound by $\frac{1}{ 3  C_{\ast} (8\pi^{3})^{\frac{1}{2}} \lvert \Lambda_{\Xi} \rvert } \delta_{q+1}^{\frac{1}{2}} M_{0}(t)^{\frac{1}{2}}$ in \eqref{estimate 401}. Next, we can estimate 
\begin{equation}\label{estimate 150}
\lVert a_{\xi} \rVert_{C_{t}C_{x}^{j}} \lesssim \delta_{q+1}^{\frac{1}{2}} l^{-5j-2} M_{0}(t)^{\frac{1}{2}} \hspace{3mm} \forall \hspace{1mm} j \geq 0, \xi \in \Lambda_{\Xi}; 
\end{equation} 
the case $j = 0$ can be proven immediately as 
\begin{align*}
\lVert a_{\xi} \rVert_{C_{t,x}} \overset{\eqref{estimate 138}\eqref{estimate 132}}{\leq} \lVert \rho_{\Xi}^{\frac{1}{2}} \rVert_{C_{t,x}} \lVert \gamma_{\xi} \rVert_{C(B_{\epsilon_{\Xi}} (0))}  
\overset{\eqref{estimate 141} \eqref{estimate 48}}{\lesssim}& \delta_{q+1}^{\frac{1}{2}} l^{-2} M_{0}(t)^{\frac{1}{2}}  
\end{align*}
whereas the case $j \geq 1$ may be handled using \cite[Equ. (130)]{BDIS15}, \eqref{estimate 144}, \eqref{estimate 394}, \eqref{estimate 421}, and \eqref{estimate 132}. Next, we can compute using \eqref{estimate 48} and \eqref{estimate 421}
\begin{equation}\label{estimate 151}
\lVert a_{\xi} \rVert_{C_{t}^{1}C_{x}^{j}} \lesssim \delta_{q+1}^{\frac{1}{2}} l^{-5j - 5} M_{0}(t)^{\frac{1}{2}} + \delta_{q+1}^{\frac{1}{2}} l^{-2} M_{0}(t)^{\frac{1}{2}} \lVert \gamma_{\xi} ( - \frac{ \mathring{R}_{l}^{\Xi}}{\rho_{\Xi}}) \rVert_{C_{t}^{1}C_{x}^{j}} \hspace{3mm} \forall \hspace{1mm} j \geq 0, \xi \in \Lambda_{\Xi}, 
\end{equation} 
which will lead to via \eqref{estimate 132}, \eqref{estimate 144}, \eqref{estimate 100}, and \eqref{estimate 394},  
\begin{equation}\label{estimate 309}
\lVert a_{\xi} \rVert_{C_{t}^{1}C_{x}^{j}} \lesssim \delta_{q+1}^{\frac{1}{2}} l^{-7-5j} M_{0}(t)^{\frac{1}{2}} \hspace{5mm} \forall \hspace{1mm} j \in \{0,1,2\}, \xi \in \Lambda_{\Xi}. 
\end{equation}  
We can also compute $\partial_{t}^{2} a_{\xi}$ directly from \eqref{estimate 138} and estimate 
\begin{equation}\label{estimate 406}
\lVert a_{\xi} \rVert_{C_{t}^{2}C_{x}} \overset{\eqref{estimate 100} \eqref{estimate 132} \eqref{estimate 394} \eqref{estimate 144} \eqref{estimate 141} }{\lesssim} l^{-12} \delta_{q+1}^{\frac{1}{2}} M_{0}(t)^{\frac{1}{2}} \hspace{3mm} \forall \hspace{1mm} \xi \in \Lambda_{\Xi}. 
\end{equation} 
Next, the geometric lemma for velocity allows us to control matrices in a neighborhood of an identity rather than the origin. Moreover, due to extra self-interacting terms, $w_{q+1}^{p}$ will need more wave vectors than $d_{q+1}^{p}$, to be defined in \eqref{estimate 166}. For this purpose we define 
\begin{equation}\label{estimate 145} 
\mathring{G}^{\Xi} \triangleq \sum_{\xi \in \Lambda_{\Xi}} a_{\xi}^{2} (\xi\otimes \xi - \xi_{2} \otimes \xi_{2}). 
\end{equation} 
We will need the following estimates of $\mathring{G}^{\Xi}$: 
\begin{subequations}\label{estimate 403}
\begin{align}
& \lVert \mathring{G}^{\Xi} \rVert_{C_{t}L_{x}^{1}} \overset{\eqref{estimate 145}}{\leq} 6 \sum_{\xi \in\Lambda_{\Xi}} \lVert a_{\xi} \rVert_{C_{t}L_{x}^{2}}^{2} \overset{\eqref{estimate 146}}{\leq} 6 c_{v} \delta_{q+1}M_{0}(t), \label{estimate 147}\\
& \lVert \mathring{G}^{\Xi} \rVert_{C_{t}C_{x}^{j}} \overset{\eqref{estimate 145}}{\lesssim}  \sum_{\xi \in \Lambda_{\Xi}} \lVert a_{\xi} \rVert_{C_{t,x}} \lVert a_{\xi} \rVert_{C_{t}C_{x}^{j}} \overset{\eqref{estimate 150}}{\lesssim} \delta_{q+1} l^{-5j-4} M_{0}(t) \hspace{7mm} \forall \hspace{1mm} j\geq 0, \label{estimate 148}\\
&\lVert \mathring{G}^{\Xi} \rVert_{C_{t}^{1}C_{x}^{j}}  \overset{\eqref{estimate 150} \eqref{estimate 309}}{\lesssim} \delta_{q+1} l^{-9-5j} M_{0}(t) \hspace{35mm} \forall \hspace{1mm} j \in \{0,1,2\}, \label{estimate 149}\\
& \lVert \mathring{G}^{\Xi} \rVert_{C_{t}^{2}C_{x}} \overset{\eqref{estimate 150}}{\lesssim} l^{-14} \delta_{q+1} M_{0}(t).  \label{estimate 411}
\end{align}
\end{subequations} 
Next, we define $\rho_{v}$ and the associated velocity amplitude function as 
\begin{subequations}
\begin{align}
\rho_{v} (t,x) \triangleq& 2\epsilon_{v}^{-1} c_{v} \delta_{q+1} M_{0}(t) \chi \left( \frac{ \lvert \mathring{R}_{l}^{v} (t,x) + \mathring{G}^{\Xi} (t,x) \rvert }{c_{v} \delta_{q+1} M_{0}(t)} \right), \label{estimate 152}\\
a_{\xi} (t,x) \triangleq& \rho_{v}^{\frac{1}{2}} (t,x) \gamma_{\xi} \left( \text{Id} - \frac{ \mathring{R}_{l}^{v}(t,x) + \mathring{G}^{\Xi}(t,x)}{\rho_{v}(t,x)} \right) \hspace{5mm} \forall \hspace{1mm} \xi \in \Lambda_{v}, \label{estimate 153} 
\end{align}
\end{subequations} 
for $\gamma_{\xi}$ from Lemma \ref{Lemma 3.2}; we recall from Remark \ref{Remark 3.1} that $\Lambda_{v}\cap \Lambda_{\Xi} =\emptyset$ so that there is no discrepancy in the definitions of $a_{\xi}$ for $\xi \in \Lambda_{v}$ and $\Lambda_{\Xi}$. Similarly to \eqref{estimate 132}, we can verify that 
\begin{equation}\label{estimate 155}
\lvert \frac{ \mathring{R}_{l}^{v} (t,x) + \mathring{G}^{\Xi} (t,x)}{\rho_{v} (t,x) } \rvert 
\overset{\eqref{estimate 152}}{=} \lvert \frac{ \mathring{R}_{l}^{v} (t,x) + \mathring{G}^{\Xi} (t,x)}{2\epsilon_{v}^{-1} c_{v} \delta_{q+1} M_{0}(t) \chi \left( \frac{ \lvert \mathring{R}_{l}^{v} (t,x) + \mathring{G}^{\Xi} (t,x)\rvert }{c_{v} \delta_{q+1} M_{0}(t)} \right) } \rvert \overset{\eqref{estimate 131}}{\leq} \epsilon_{v}. 
\end{equation} 
We can estimate for ay $p \in [1,\infty)$, 
\begin{equation}\label{estimate 154}
\lVert \rho_{v} \rVert_{C_{t}L_{x}^{p}} \overset{\eqref{estimate 152} \eqref{estimate 131}}{\leq} 8 \epsilon_{v}^{-1} (c_{v} \delta_{q+1} M_{0}(t) (8\pi^{3})^{\frac{1}{p}} + \lVert \mathring{R}_{l}^{v} + \mathring{G}^{\Xi} \rVert_{C_{t}L_{x}^{p}}).
\end{equation} 
Due to \eqref{estimate 155}, for $c_{v}> 0$ sufficiently small that 
\begin{equation}\label{estimate 156}
c_{v} \leq \frac{ \epsilon_{v}}{72  C_{\ast}^{2} (8\pi^{3} + 7) 8\pi^{3} M_{\ast}^{2} \lvert \Lambda_{v} \rvert^{2} }, 
\end{equation} 
where $C_{\ast}$ is from Lemma \ref{Lemma 6.2} and $M_{\ast}$ is from \eqref{estimate 48} , we can estimate for $\xi \in \Lambda_{v}$ using the fact that mollifiers have mass one, 
\begin{align}
\lVert a_{\xi} \rVert_{C_{t}L_{x}^{2}} \overset{\eqref{estimate 155} \eqref{estimate 154}\eqref{estimate 48} }{\leq}& [8\epsilon_{v}^{-1} ( c_{v} \delta_{q+1} M_{0}(t) 8\pi^{3} + \lVert \mathring{R}_{l}^{v} + \mathring{G}^{\Xi} \rVert_{C_{t}L_{x}^{1}})]^{\frac{1}{2}} M_{\ast} \nonumber \\
\overset{\eqref{estimate 100} \eqref{estimate 147}}{\leq}& M_{\ast} (8 \epsilon_{v}^{-1})^{\frac{1}{2}}   ( c_{v} \delta_{q+1} M_{0}(t) 8 \pi^{3} + c_{v} M_{0}(t) \delta_{q+1} + 6 c_{v} \delta_{q+1} M_{0}(t))^{\frac{1}{2}} \nonumber \\
\overset{\eqref{estimate 156}}{\leq}& \frac{ \delta_{q+1}^{\frac{1}{2}} M_{0}(t)^{\frac{1}{2}}}{3  C_{\ast} (8\pi^{3})^{\frac{1}{2}} \lvert \Lambda_{v} \rvert }. \label{estimate 169}
\end{align}
We emphasize that the smallness of $c_{v}$ only depends on $M_{\ast}$ and $\Lambda_{v}$; thus, the fact that the smallness of $c_{\Xi}$ depended on $c_{v}$ in \eqref{estimate 157} does not make this a circulatory argument. Next, we estimate 
\begin{subequations}\label{estimate 404}
\begin{align}
&\lVert \rho_{v} \rVert_{C_{t,x}} \overset{\eqref{estimate 154}}{\lesssim} \delta_{q+1}M_{0}(t) + \lVert \mathring{R}_{l}^{v} \rVert_{C_{t}W_{x}^{4,1}} + \lVert \mathring{G}^{\Xi} \rVert_{C_{t,x}} \overset{ \eqref{estimate 148} \eqref{estimate 100} }{\lesssim}  \delta_{q+1} l^{-4} M_{0}(t), \label{estimate 158}\\
&\lVert \rho_{v} \rVert_{C_{t}C_{x}^{j}} \overset{\eqref{estimate 152} \eqref{estimate 148}\eqref{estimate 100}}{\lesssim}  \delta_{q+1} l^{-9j} M_{0}(t) \hspace{6mm} \forall \hspace{1mm} j \geq 1, 
\label{estimate 159}\\
&\lVert \rho_{v} \rVert_{C_{t}^{1}C_{x}^{j}} \overset{\eqref{estimate 148} \eqref{estimate 149} \eqref{estimate 100}}{\lesssim}  l^{-9j-9} M_{0}(t) \delta_{q+1} \hspace{3mm} \forall \hspace{1mm} j \in \{0,1,2\}, \label{estimate 160}  \\
& \lVert \rho_{v} \rVert_{C_{t}^{2}C_{x}} \overset{\eqref{estimate 148} \eqref{estimate 100} \eqref{estimate 160}}{\lesssim} \delta_{q+1} M_{0}(t) l^{-18}, \label{estimate 412}
\end{align}
\end{subequations}
where to verify \eqref{estimate 160} one can directly compute $\partial_{t} \rho_{v}$ and use the fact that $\partial_{t} M_{0}(t) = 4L M_{0}(t)$, from which the estimate of \eqref{estimate 412} also follows. The estimates \eqref{estimate 158}-\eqref{estimate 160}, along with the lower bound on $\rho_{v}$
\begin{equation}\label{estimate 323}
\rho_{v}(t) \overset{\eqref{estimate 152} \eqref{estimate 131}}{\geq} \delta_{q+1} \epsilon_{v}^{-1}c_{v} M_{0}(t), 
\end{equation}  
which is analogous to \eqref{estimate 144}, lead to the following estimates via \cite[Equ. (130)]{BDIS15}:
\begin{subequations}
\begin{align}
& \lVert \rho_{v}^{\frac{1}{2}} \rVert_{C_{t,x}} \overset{\eqref{estimate 158}}{\lesssim} \delta_{q+1}^{\frac{1}{2}} l^{-2} M_{0}(t)^{\frac{1}{2}}, \label{estimate 161}\\
& \lVert \rho_{v}^{\frac{1}{2}} \rVert_{C_{t}C_{x}^{j}} \overset{\eqref{estimate 159}}{\lesssim} \delta_{q+1}^{\frac{1}{2}} M_{0}(t)^{\frac{1}{2}} l^{-9j} \hspace{6mm} \forall \hspace{1mm} j \geq 0, \label{estimate 162}\\
& \lVert \rho_{v}^{\frac{1}{2}} \rVert_{C_{t}^{1} C_{x}^{j}} \overset{\eqref{estimate 160}}{\lesssim} \delta_{q+1}^{\frac{1}{2}}  M_{0}(t)^{\frac{1}{2}} l^{-9j-9} \hspace{3mm} \forall \hspace{1mm} j \in \{0,1,2\}. \label{estimate 163}
\end{align} 
\end{subequations} 
With these estimates in hand, we can now obtain 
\begin{subequations}\label{estimate 187}
\begin{align}
&  \lVert a_{\xi} \rVert_{C_{t}C_{x}^{j}} \overset{\eqref{estimate 161} \eqref{estimate 162}}{\lesssim} \delta_{q+1}^{\frac{1}{2}} M_{0}(t)^{\frac{1}{2}} l^{-9j-2} \hspace{4mm} \forall \hspace{1mm} j \geq 0, \hspace{9mm} \xi \in \Lambda_{v}, \label{estimate 164}\\
& \lVert a_{\xi} \rVert_{C_{t}^{1}C_{x}^{j}} \lesssim \delta_{q+1}^{\frac{1}{2}} M_{0}(t)^{\frac{1}{2}} l^{-9j - 11} \hspace{19mm} \forall \hspace{1mm} j \in \{0,1,2\}, \xi \in \Lambda_{v}, \label{estimate 165} 
\end{align}
\end{subequations} 
where for \eqref{estimate 165} one can estimate 
\begin{align*}
 \lVert a_{\xi} \rVert_{C_{t}^{1}C_{x}^{j}} \overset{\eqref{estimate 163} \eqref{estimate 48} \eqref{estimate 161}}{\lesssim} l^{-9j-9} M_{0}(t)^{\frac{1}{2}} \delta_{q+1}^{\frac{1}{2}} + \delta_{q+1}^{\frac{1}{2}} l^{-2} M_{0}(t)^{\frac{1}{2}} \lVert \gamma_{\xi} (\text{Id} - \frac{ \mathring{R}_{l}^{v} + \mathring{G}^{\Xi}}{\rho_{v}} ) \rVert_{C_{t}^{1}C_{x}^{j}}
\end{align*}
and then $\lVert\gamma_{\xi} (\text{Id} - \frac{ \mathring{R}_{l}^{v} + \mathring{G}^{\Xi}}{\rho_{v}} ) \rVert_{C_{t}^{1}C_{x}^{j}} \lesssim l^{-9j - 9}$ using \cite[Equ. (130)]{BDIS15}.  Furthermore, we can directly compute $\partial_{t}^{2} a_{\xi}$ from \eqref{estimate 153} and estimate 
\begin{equation}\label{estimate 413}
 \lVert a_{\xi} \rVert_{C_{t}^{2}C_{x}} \overset{\eqref{estimate 100} \eqref{estimate 403}\eqref{estimate 155} \eqref{estimate 404} \eqref{estimate 323} }{\lesssim} l^{-18} M_{0}(t)^{\frac{1}{2}} \delta_{q+1}^{\frac{1}{2}}. 
\end{equation}
Finally, we have the following identity
\begin{align}
& \sum_{\xi \in \Lambda_{v}} a_{\xi}^{2} \phi_{\xi}^{2} \varphi_{\xi}^{2} (\xi \otimes \xi) + \mathring{R}_{l}^{v} + \mathring{G}^{\Xi}  \nonumber \\
\overset{\eqref{estimate 57} \eqref{estimate 153} }{=}& \sum_{\xi \in \Lambda_{v}} \rho_{v} \gamma_{\xi}^{2} \left(\text{Id} - \frac{ \mathring{R}_{l}^{v} + \mathring{G}^{\Xi}}{\rho_{v}} \right) (\xi \otimes \xi)  + \sum_{\xi \in \Lambda_{v}} a_{\xi}^{2} \mathbb{P}_{\neq 0} (\phi_{\xi}^{2} \varphi_{\xi}^{2}) (\xi \otimes \xi) + \mathring{R}_{l}^{v} + \mathring{G}^{\Xi} \nonumber \\
\overset{\eqref{estimate 45}}{=}& \rho_{v} \text{Id} + \sum_{\xi \in \Lambda_{v}} a_{\xi}^{2} \mathbb{P}_{\neq 0} (\phi_{\xi}^{2} \varphi_{\xi}^{2}) (\xi \otimes \xi).  \label{estimate 216}
\end{align} 
Now we define the principal part of the perturbation as 
\begin{equation}\label{estimate 166}
w_{q+1}^{p} \triangleq \sum_{\xi \in \Lambda} a_{\xi} \phi_{\xi} \varphi_{\xi} \xi, \hspace{3mm} d_{q+1}^{p} \triangleq \sum_{\xi \in \Lambda_{\Xi}} a_{\xi} \phi_{\xi} \varphi_{\xi} \xi_{2}. 
\end{equation} 
Because neither $w_{q+1}^{p}$ or $d_{q+1}^{p}$ is divergence-free, we define the correctors by 
\begin{subequations}\label{estimate 330}
\begin{align}
w_{q+1}^{c} \triangleq \frac{1}{N_{\Lambda}^{2} \lambda_{q+1}^{2}} \sum_{\xi \in \Lambda}& \text{curl} ( \nabla a_{\xi} \times (\phi_{\xi} \Psi_{\xi} \xi)) \nonumber \\
&+ \nabla a_{\xi} \times \text{curl} ( \phi_{\xi} \Psi_{\xi} \xi) + a_{\xi} \nabla \phi_{\xi} \times \text{curl} (\Psi_{\xi} \xi), \label{estimate 167}\\
d_{q+1}^{c} \triangleq \frac{1}{N_{\Lambda}^{2} \lambda_{q+1}^{2}} \sum_{\xi \in \Lambda_{\Xi}}& \text{curl}( \nabla a_{\xi} \times (\phi_{\xi} \Psi_{\xi} \xi_{2})) +\nabla a_{\xi} \times\text{curl}(\phi_{\xi}\Psi_{\xi}\xi_{2}) - a_{\xi} \Delta \phi_{\xi} \Psi_{\xi} \xi_{2}. \label{estimate 168}
\end{align}
\end{subequations} 
Because $\nabla\cdot (\Psi_{\xi} \xi) = \nabla\cdot (\Psi_{\xi} \xi_{2}) = 0$, we see that due to \eqref{estimate 57} 
\begin{equation}\label{estimate 438}
\text{curl curl} (\Psi_{\xi} \xi) = (\lambda_{q+1}^{2} N_{\Lambda}^{2} \varphi_{\xi}) \xi \hspace{1mm} \text{ and } \hspace{1mm} \text{curl curl} (\Psi_{\xi} \xi_{2}) = (\lambda_{q+1}^{2} N_{\Lambda}^{2} \varphi_{\xi}) \xi_{2}
\end{equation} 
which lead to 
\begin{subequations}\label{estimate 193}
\begin{align}
&\frac{1}{N_{\Lambda}^{2} \lambda_{q+1}^{2}} \text{curl curl} \sum_{\xi \in \Lambda} a_{\xi} \phi_{\xi} \Psi_{\xi} \xi = w_{q+1}^{p} + w_{q+1}^{c}, \label{estimate 503}\\
& \frac{1}{N_{\Lambda}^{2} \lambda_{q+1}^{2}} \text{curl curl}\sum_{\xi\in\Lambda_{\Xi}}a_{\xi}\phi_{\xi}\Psi_{\xi}\xi_{2} = d_{q+1}^{p} + d_{q+1}^{c} \label{estimate 504}
\end{align}
\end{subequations} 
via the identities of 
\begin{align*}
\nabla \times \nabla \times A = \nabla (\nabla \cdot A) - \Delta A \hspace{1mm} \text{ and } \hspace{1mm} \text{curl} (A \times B) = A (\nabla \cdot B) - B(\nabla \cdot A) + (B\cdot \nabla)A  - (A\cdot \nabla) B. 
\end{align*}
We note that the difference between $w_{q+1}^{c}$ and $d_{q+1}^{c}$ in \eqref{estimate 330} is due to the fact that upon computing $\frac{1}{N_{\Lambda}^{2} \lambda_{q+1}^{2}} \text{curl curl} \sum_{\xi \in \Lambda} a_{\xi} \phi_{\xi} \Psi_{\xi} \xi$ and $\frac{1}{N_{\Lambda}^{2} \lambda_{q+1}^{2}} \text{curl curl}\sum_{\xi\in\Lambda_{\Xi}}a_{\xi}\phi_{\xi}\Psi_{\xi}\xi_{2}$, one sees that $\nabla \phi_{\xi} \times \Psi_{\xi} \xi = 0$ but $\nabla \phi_{\xi} \times \Psi_{\xi} \xi_{2} \neq 0$.  From \eqref{estimate 193} we see that 
\begin{equation}\label{estimate 405}
\nabla\cdot (w_{q+1}^{p} + w_{q+1}^{c}) = 0 \text{ and } \nabla\cdot (d_{q+1}^{p} + d_{q+1}^{c}) = 0 
\end{equation} 
and $w_{q+1}^{p}+ w_{q+1}^{c}$ and $d_{q+1}^{p} + d_{q+1}^{c}$ are both mean-zero. Next, we define temporal correctors: 
\begin{equation}\label{estimate 334} 
w_{q+1}^{t} \triangleq - \mu^{-} \sum_{\xi \in \Lambda} \mathbb{P} \mathbb{P}_{\neq 0} (a_{\xi}^{2} \mathbb{P}_{\neq 0} (\phi_{\xi}^{2} \varphi_{\xi}^{2} ))\xi \hspace{1mm} \text{ and }  \hspace{1mm} d_{q+1}^{t} \triangleq  \mu^{-1} \sum_{\xi\in\Lambda_{\Xi}}\mathbb{P} \mathbb{P}_{\neq 0} (a_{\xi}^{2}\mathbb{P}_{\neq 0} (\phi_{\xi}^{2}\varphi_{\xi}^{2})) \xi_{2}.
\end{equation}  
It follows from the definitions of $\mathbb{P}$ and $\mathbb{P}_{\neq 0}$ that both $w_{q+1}^{t}$ and $d_{q+1}^{t}$ are divergence-free and mean-zero and from the definition of $\mathbb{P} = \text{Id} - \nabla \Delta^{-1} \text{div}$ that  
\begin{subequations}\label{estimate 217}
\begin{align}
& \partial_{t} w_{q+1}^{t} = - \mu^{-1} \sum_{\xi \in \Lambda} \mathbb{P}_{\neq 0} \partial_{t} (a_{\xi}^{2} \mathbb{P}_{\neq 0} (\phi_{\xi}^{2} \varphi_{\xi}^{2} )) \xi + \mu^{-1} \sum_{\xi \in \Lambda} \nabla \Delta^{-1} \text{div} \partial_{t} (a_{\xi}^{2} \mathbb{P}_{\neq 0} (\phi_{\xi}^{2} \varphi_{\xi}^{2} )) \xi, \label{estimate 214} \\
& \partial_{t} d_{q+1}^{t} = \mu^{-1} \sum_{\xi \in \Lambda_{\Xi}} \mathbb{P}_{\neq 0} \partial_{t} (a_{\xi}^{2} \mathbb{P}_{\neq 0} (\phi_{\xi}^{2} \varphi_{\xi}^{2} )) \xi_{2} - \mu^{-1} \sum_{\xi \in \Lambda_{\Xi}} \nabla \Delta^{-1} \text{div} \partial_{t} (a_{\xi}^{2} \mathbb{P}_{\neq 0} (\phi_{\xi}^{2} \varphi_{\xi}^{2} )) \xi_{2}.\label{estimate 215}
\end{align}
\end{subequations}
At last, we define 
\begin{equation}\label{estimate 176}
w_{q+1} \triangleq w_{q+1}^{p} + w_{q+1}^{c} + w_{q+1}^{t} \hspace{2mm} \text{ and } \hspace{2mm} d_{q+1} \triangleq d_{q+1}^{p} + d_{q+1}^{c}+ d_{q+1}^{t},  
\end{equation} 
which are both mean-zero and divergence-free due to \eqref{estimate 405}; moreover, we define 
\begin{equation}\label{estimate 203} 
v_{q+1} \triangleq v_{l} + w_{q+1} \hspace{2mm} \text{ and } \hspace{2mm}  \Xi_{q+1} \triangleq \Xi_{l} + d_{q+1}.
\end{equation} 
Now we have for all $j \in \mathbb{N}_{0}$, by taking $a \in 2 \mathbb{N}$ sufficiently large, 
\begin{subequations}
\begin{align}
\lVert D^{j} a_{\xi} \rVert_{C_{t} L_{x}^{2}} \overset{\eqref{estimate 146} \eqref{estimate 150}}{\leq} \frac{ \delta_{q+1}^{\frac{1}{2}} M_{0}(t)^{\frac{1}{2}}}{3 C_{\ast} (8\pi^{3})^{\frac{1}{2}} \lvert \Lambda_{\Xi} \rvert} l^{-8j} \hspace{5mm} & \forall \hspace{1mm} \xi \in \Lambda_{\Xi},  \label{estimate 401}\\
\lVert D^{j} a_{\xi} \rVert_{C_{t}L_{x}^{2}} \overset{\eqref{estimate 169} \eqref{estimate 164}}{\leq} \frac{ \delta_{q+1}^{\frac{1}{2}} M_{0}(t)^{\frac{1}{2}}}{3 C_{\ast} (8\pi^{3})^{\frac{1}{2}}\lvert \Lambda_{v} \rvert  } l^{-12j} \hspace{5mm} & \forall \hspace{1mm} \xi \in \Lambda_{v}. \label{estimate 402}
\end{align}
\end{subequations}
Thus, because $\phi_{\xi}$ and $\varphi_{\xi}$ are both $(\mathbb{T}/\lambda_{q+1}\sigma)^{3}$-periodic due to \eqref{estimate 47} and \eqref{estimate 52}, in order to apply Lemma \ref{Lemma 6.2} we set first in case $\xi \in \Lambda_{\Xi}$, ``$f$'' = $a_{\xi}$, ``$g$'' = $\phi_{\xi} \varphi_{\xi}$, ``$\kappa$'' = $\lambda_{q+1} \sigma \in\mathbb{N}$ due to \eqref{estimate 52}, ``$N$'' = 1, ``$p$'' = 2, ``$\zeta$'' = $l^{-8}$, ``$C_{f}$'' = $\frac{\delta_{q+1}^{\frac{1}{2}} M_{0}(t)^{\frac{1}{2}}}{3 C_{\ast} (8\pi^{3})^{\frac{1}{2}} \lvert \Lambda_{\Xi} \rvert}$ while in case $\xi \in \Lambda_{v}$, we set identically with the only exceptions of ``$\zeta$'' = $l^{-12}$ and ``$C_{f}$'' = $\frac{\delta_{q+1}^{\frac{1}{2}} M_{0}(t)^{\frac{1}{2}}}{3C_{\ast} (8\pi^{3})^{\frac{1}{2}} \lvert \Lambda_{v} \rvert}$ for which both conditions in \eqref{estimate 140} may be verified using \eqref{eta}-\eqref{l}. Therefore, 
\begin{subequations}
\begin{align}
& \lVert a_{\xi} \phi_{\xi} \varphi_{\xi} \rVert_{C_{t}L_{x}^{2}} \overset{\eqref{estimate 171}}{\leq} \left( \frac{ \delta_{q+1}^{\frac{1}{2}} M_{0}(t)^{\frac{1}{2}}}{3 C_{\ast} (8 \pi^{3})^{\frac{1}{2}} \lvert \Lambda_{\Xi} \rvert} \right) C_{\ast} \lVert \phi_{\xi} \varphi_{\xi} \rVert_{C_{t}L_{x}^{2}}   \overset{\eqref{estimate 57}}{=} \frac{ \delta_{q+1}^{\frac{1}{2}} M_{0}(t)^{\frac{1}{2}}}{3 \lvert \Lambda_{\Xi} \rvert} \hspace{3mm} \forall \hspace{1mm} \xi \in \Lambda_{\Xi},  \label{estimate 172}\\
& \lVert a_{\xi} \phi_{\xi} \varphi_{\xi} \rVert_{C_{t}L_{x}^{2}} \overset{\eqref{estimate 171} }{\leq} \left( \frac{ \delta_{q+1}^{\frac{1}{2}} M_{0}(t)^{\frac{1}{2}}}{3 C_{\ast} (8 \pi^{3})^{\frac{1}{2}} \lvert \Lambda_{v} \rvert} \right) C_{\ast} \lVert \phi_{\xi} \varphi_{\xi} \rVert_{C_{t}L_{x}^{2}} 
  \overset{\eqref{estimate 57}}{=} \frac{ \delta_{q+1}^{\frac{1}{2}} M_{0}(t)^{\frac{1}{2}}}{3 \lvert \Lambda_{v} \rvert} \hspace{3mm} \forall \hspace{1mm} \xi \in \Lambda_{v}. \label{estimate 173}  
\end{align}
\end{subequations}
It follows that 
\begin{subequations}
\begin{align}
&\lVert d_{q+1}^{p} \rVert_{C_{t}L_{x}^{2}}  \overset{\eqref{estimate 166}}{\leq} \sum_{\xi \in \Lambda_{\Xi}} \lVert a_{\xi} \phi_{\xi} \varphi_{\xi} \rVert_{C_{t}L_{x}^{2}} \overset{\eqref{estimate 172}}{\leq}  \frac{ \delta_{q+1}^{\frac{1}{2}} M_{0}(t)^{\frac{1}{2}}}{3}, \label{estimate 177}\\
&\lVert w_{q+1}^{p} \rVert_{C_{t}L_{x}^{2}} \overset{\eqref{estimate 166} \eqref{estimate 46} }{\leq} \sum_{\xi \in \Lambda_{v}} \lVert a_{\xi} \phi_{\xi} \varphi_{\xi} \rVert_{C_{t}L_{x}^{2}} +  \sum_{\xi \in \Lambda_{\Xi}} \lVert a_{\xi} \phi_{\xi} \varphi_{\xi} \rVert_{C_{t}L_{x}^{2}}  \overset{\eqref{estimate 173}}{\leq}  \frac{2\delta_{q+1}^{\frac{1}{2}} M_{0}(t)^{\frac{1}{2}}}{3}. \label{estimate 178}
\end{align}
\end{subequations} 
Next, using the fact that $\phi_{\xi}$ and $\Psi_{\xi}$ have oscillations in orthogonal directions, we can compute for all $p \in [1,\infty]$ 
\begin{align}
& \lVert d_{q+1}^{c} \rVert_{C_{t}L_{x}^{p}} \label{estimate 179}\\
\lesssim& \lambda_{q+1}^{-2} \sum_{\xi \in \Lambda_{\Xi}} \lVert a_{\xi} \rVert_{C_{t}C_{x}^{2}} \lVert \phi_{\xi} \rVert_{C_{t}L_{x}^{p}} \lVert \Psi_{\xi} \rVert_{L_{x}^{p}} + \lVert a_{\xi} \rVert_{C_{t}C_{x}^{1}} (\lVert \phi_{\xi} \rVert_{C_{t}W_{x}^{1,p}} \lVert\Psi_{\xi} \rVert_{L_{x}^{p}} + \lVert \phi_{\xi} \rVert_{C_{t}L_{x}^{p}} \lVert \Psi_{\xi} \rVert_{W_{x}^{1,p}}) \nonumber\\
& \hspace{10mm} + \lVert a_{\xi} \rVert_{C_{t,x}} \lVert \phi_{\xi}\rVert_{C_{t}W_{x}^{2,p}} \lVert\Psi_{\xi} \rVert_{L_{x}^{p}}  \overset{\eqref{estimate 175} \eqref{estimate 150} \eqref{eta}-\eqref{sigma, r, mu} \eqref{estimate 130} }{\lesssim}\delta_{q+1}^{\frac{1}{2}} M_{0}(t)^{\frac{1}{2}} l^{-2} r^{\frac{1}{p} - \frac{3}{2}} \sigma^{\frac{1}{p} + \frac{1}{2}}.  \nonumber 
\end{align} 
Similarly,
\begin{align}
& \lVert w_{q+1}^{c} \rVert_{C_{t}L_{x}^{p}} \label{estimate 180}\\
\lesssim& \lambda_{q+1}^{-2} (\sum_{\xi \in\Lambda_{v}} + \sum_{\xi \in \Lambda_{\Xi}}) \lVert a_{\xi} \rVert_{C_{t}C_{x}^{2}} \lVert \phi_{\xi} \rVert_{C_{t}L_{x}^{p}} \lVert \Psi_{\xi} \rVert_{L_{x}^{p}} + \lVert a_{\xi} \rVert_{C_{t}C_{x}^{1}} (\lVert \phi_{\xi} \rVert_{C_{t}W_{x}^{1,p}} \lVert \Psi_{\xi} \rVert_{L_{x}^{p}} + \lVert \phi_{\xi} \rVert_{C_{t}L_{x}^{p}} \lVert \Psi_{\xi}\rVert_{W_{x}^{1,p}}) \nonumber\\
& \hspace{11mm} + \lVert a_{\xi} \rVert_{C_{t,x}} \lVert \phi_{\xi} \rVert_{C_{t}W_{x}^{1,p}} \lVert \Psi_{\xi} \rVert_{W_{x}^{1,p}} \overset{\eqref{estimate 175}   \eqref{eta}-\eqref{sigma, r, mu} \eqref{estimate 150} \eqref{estimate 164}}{\lesssim} \delta_{q+1}^{\frac{1}{2}} M_{0}(t)^{\frac{1}{2}} l^{-2} r^{\frac{1}{p} - \frac{3}{2}} \sigma^{\frac{1}{p} + \frac{1}{2}}. \nonumber
\end{align} 
Next, for $p \in (1,\infty)$, we can estimate via \eqref{estimate 60}  
\begin{subequations}\label{estimate 236}
\begin{align}
&  \lVert d_{q+1}^{t} \rVert_{C_{t}L_{x}^{p}} \lesssim \mu^{-1} \sum_{\xi \in \Lambda_{\Xi}} \lVert a_{\xi}\rVert_{C_{t,x}}^{2} \lVert \phi_{\xi} \varphi_{\xi} \rVert_{C_{t}L_{x}^{2p}}^{2} 
\overset{\eqref{estimate 150} }{\lesssim} \mu^{-1} \delta_{q+1} l^{-4} M_{0}(t) r^{\frac{1}{p} - 1} \sigma^{\frac{1}{p} - 1}, \label{estimate 181} \\
& \lVert w_{q+1}^{t} \rVert_{C_{t}L_{x}^{p}} \lesssim \mu^{-1} \sum_{\xi \in \Lambda} \lVert a_{\xi}\rVert_{C_{t,x}}^{2} \lVert \phi_{\xi} \varphi_{\xi} \rVert_{C_{t}L_{x}^{2p}}^{2}   
\overset{\eqref{estimate 150} \eqref{estimate 164} }{\lesssim} \mu^{-1} \delta_{q+1} l^{-4} M_{0}(t) r^{\frac{1}{p} - 1} \sigma^{\frac{1}{p} - 1}. \label{estimate 182} 
\end{align}
\end{subequations}
We are now ready to obtain for $a \in 2 \mathbb{N}$ sufficiently large  
\begin{equation}\label{estimate 202}
\lVert d_{q+1} \rVert_{C_{t}L_{x}^{2}} \leq \frac{\delta_{q+1}^{\frac{1}{2}} M_{0}(t)^{\frac{1}{2}}}{2} \hspace{2mm} \text{ and } \hspace{2mm} \lVert w_{q+1} \rVert_{C_{t}L_{x}^{2}} \leq \frac{ 3  \delta_{q+1}^{\frac{1}{2}} M_{0}(t)^{\frac{1}{2}}}{4}; 
\end{equation} 
e.g., the first is estimated by 
\begin{align}
\lVert d_{q+1} \rVert_{C_{t}L_{x}^{2}} 
\overset{\eqref{estimate 176}}{\leq}& \lVert d_{q+1}^{p} \rVert_{C_{t}L_{x}^{2}} + \lVert d_{q+1}^{c} \rVert_{C_{t}L_{x}^{2}} + \lVert d_{q+1}^{t} \rVert_{C_{t}L_{x}^{2}} \nonumber \\
\overset{\eqref{estimate 177} \eqref{estimate 179}  \eqref{estimate 181}}{\leq}& \frac{ \delta_{q+1}^{\frac{1}{2}} M_{0}(t)^{\frac{1}{2}}}{3} + C [ \delta_{q+1}^{\frac{1}{2}} M_{0}(t)^{\frac{1}{2}} l^{-2} r^{-1} \sigma + \mu^{-1}\delta_{q+1}   l^{-4} M_{0}(t) r^{-\frac{1}{2}} \sigma^{-\frac{1}{2}}] \nonumber \\
\overset{\eqref{estimate 130} \eqref{eta} \eqref{alpha}}{\leq}& \frac{\delta_{q+1}^{\frac{1}{2}} M_{0}(t)^{\frac{1}{2}}}{2}, \label{estimate 340}
\end{align} 
with the second estimate similarly via \eqref{estimate 177} and \eqref{estimate 181} replaced by \eqref{estimate 178} and \eqref{estimate 182}. Next, we estimate for all $p \in [1,\infty]$ 
\begin{subequations}\label{estimate 185}
\begin{align}
&\lVert d_{q+1}^{p} \rVert_{C_{t}L_{x}^{p}} \lesssim \sum_{\xi \in \Lambda_{\Xi}} \lVert a_{\xi} \rVert_{C_{t,x}} \lVert \phi_{\xi} \varphi_{\xi} \rVert_{C_{t}L_{x}^{p}} 
\overset{\eqref{estimate 150} \eqref{estimate 60}}{\lesssim} \delta_{q+1}^{\frac{1}{2}} l^{-2} M_{0}(t)^{\frac{1}{2}} r^{\frac{1}{p} - \frac{1}{2}} \sigma^{\frac{1}{p} - \frac{1}{2}}, \label{estimate 183}\\
& \lVert w_{q+1}^{p} \rVert_{C_{t}L_{x}^{p}} \lesssim \sum_{\xi \in \Lambda} \lVert a_{\xi} \rVert_{C_{t,x}} \lVert \phi_{\xi} \varphi_{\xi} \rVert_{C_{t}L_{x}^{p}} 
\overset{\eqref{estimate 150} \eqref{estimate 164} \eqref{estimate 60}}{\lesssim} \delta_{q+1}^{\frac{1}{2}} l^{-2} M_{0}(t)^{\frac{1}{2}} r^{\frac{1}{p} - \frac{1}{2}} \sigma^{\frac{1}{p} - \frac{1}{2}}.\label{estimate 184}
\end{align}
\end{subequations}
We are now ready to establish the following estimate: for any $p \in (1,\infty)$, 
\begin{align}
& \lVert d_{q+1} \rVert_{C_{t}L_{x}^{p}}  + \lVert w_{q+1} \rVert_{C_{t}L_{x}^{p}} \label{estimate 186}\\
\overset{\eqref{estimate 176}}{\leq}& \lVert d_{q+1}^{p} \rVert_{C_{t}L_{x}^{p}} + \lVert d_{q+1}^{c} \rVert_{C_{t}L_{x}^{p}} + \lVert d_{q+1}^{t} \rVert_{C_{t}L_{x}^{p}} +  \lVert w_{q+1}^{p} \rVert_{C_{t}L_{x}^{p}} + \lVert w_{q+1}^{c} \rVert_{C_{t}L_{x}^{p}} + \lVert w_{q+1}^{t} \rVert_{C_{t}L_{x}^{p}}  \nonumber\\
&\overset{\eqref{estimate 179}-\eqref{estimate 236} \eqref{estimate 185}}{\lesssim} \delta_{q+1}^{\frac{1}{2}} l^{-2} M_{0}(t)^{\frac{1}{2}} r^{\frac{1}{p} - \frac{1}{2}} \sigma^{\frac{1}{p} - \frac{1}{2}} + \delta_{q+1}^{\frac{1}{2}} M_{0}(t)^{\frac{1}{2}} l^{-2} r^{\frac{1}{p} - \frac{3}{2}} \sigma^{\frac{1}{p} + \frac{1}{2}} \nonumber\\
& \hspace{18mm} + \delta_{q+1} \mu^{-1} l^{-4} M_{0}(t) r^{\frac{1}{p} - 1} \sigma^{\frac{1}{p} - 1} \overset{\eqref{estimate 130}}{\lesssim}  \delta_{q+1}^{\frac{1}{2}}  l^{-2} M_{0}(t)^{\frac{1}{2}} r^{\frac{1}{p} - \frac{1}{2}} \sigma^{\frac{1}{p} - \frac{1}{2}}. \nonumber
\end{align}
Next, for all $p \in [1,\infty]$, 
\begin{align}
& \lVert d_{q+1}^{p} \rVert_{C_{t}W_{x}^{1,p}} + \lVert w_{q+1}^{p} \rVert_{C_{t}W_{x}^{1,p}} \lesssim \sum_{\xi \in \Lambda} \lVert a_{\xi} \rVert_{C_{t}C_{x}^{1}} \lVert \phi_{\xi} \varphi_{\xi} \rVert_{C_{t}L_{x}^{p}} + \lVert a_{\xi} \rVert_{C_{t,x}} \lVert \phi_{\xi} \varphi_{\xi} \rVert_{C_{t}W_{x}^{1,p}} \nonumber \\
&\overset{\eqref{estimate 175} \eqref{estimate 150} \eqref{estimate 187}}{\lesssim} (\delta_{q+1}^{\frac{1}{2}} l^{-7} M_{0}(t)^{\frac{1}{2}} + \delta_{q+1}^{\frac{1}{2}} l^{-11} M_{0}(t)^{\frac{1}{2}}) r^{\frac{1}{p} - \frac{1}{2}} \sigma^{\frac{1}{p} -\frac{1}{2}} + \delta_{q+1}^{\frac{1}{2}} l^{-2} M_{0}(t)^{\frac{1}{2}} \lambda_{q+1} r^{\frac{1}{p} - \frac{1}{2}} \sigma^{\frac{1}{p} - \frac{1}{2}} \nonumber \\
& \hspace{50mm} \overset{\eqref{estimate 130}\eqref{alpha}}{\lesssim} \delta_{q+1}^{\frac{1}{2}}  l^{-2} M_{0}(t)^{\frac{1}{2}} r^{\frac{1}{p} - \frac{1}{2}} \sigma^{\frac{1}{p} - \frac{1}{2}} \lambda_{q+1}. \label{estimate 188}
\end{align}
Next, we estimate for any $p \in [1,\infty]$, using again the fact that $\phi_{\xi}$ and $\Psi_{\xi}$ have oscillations  ini orthogonal directions, we deduce  
\begin{align}
&\lVert d_{q+1}^{c} \rVert_{C_{t}W_{x}^{1,p}} \nonumber \\
\lesssim& \lambda_{q+1}^{-2} \sum_{\xi \in \Lambda_{\Xi}} \lVert a_{\xi} \rVert_{C_{t}C_{x}^{3}} \lVert \phi_{\xi} \Psi_{\xi} \rVert_{C_{t}L_{x}^{p}} + \lVert a_{\xi} \rVert_{C_{t}C_{x}^{1}} \lVert \phi_{\xi} \Psi_{\xi} \rVert_{C_{t}W_{x}^{2,p}} + \lVert a_{\xi} \rVert_{C_{t}C_{x}^{2}} \lVert \phi_{\xi} \Psi_{\xi}\rVert_{C_{t}W_{x}^{1,p}} \nonumber\\
& \hspace{2mm} + \lVert a_{\xi} \rVert_{C_{t}C_{x}^{1}} \lVert \phi_{\xi} \rVert_{C_{t}W_{x}^{2,p}} \lVert \Psi_{\xi} \rVert_{L_{x}^{p}} + \lVert a_{\xi} \rVert_{C_{t,x}} (\lVert \phi_{\xi} \rVert_{C_{t}W_{x}^{3,p}} \lVert \Psi_{\xi} \rVert_{L_{x}^{p}} + \lVert \phi_{\xi} \rVert_{C_{t}W_{x}^{2,p}} \lVert \Psi_{\xi} \rVert_{W_{x}^{1,p}})  \nonumber\\
\overset{\eqref{estimate 150}\eqref{estimate 175}}{\lesssim}&  \lambda_{q+1}^{-2} l^{-2} M_{0}(t)^{\frac{1}{2}} \delta_{q+1}^{\frac{1}{2}} [ l^{-15} r^{\frac{1}{p} - \frac{1}{2}} \sigma^{\frac{1}{p} - \frac{1}{2}} + l^{-5} \lambda_{q+1}^{2} r^{\frac{1}{p} - \frac{1}{2}} \sigma^{\frac{1}{p} - \frac{1}{2}} + l^{-10} \lambda_{q+1} r^{\frac{1}{p} - \frac{1}{2}} \sigma^{\frac{1}{p} - \frac{1}{2}} \nonumber\\
& \hspace{30mm} + l^{-5} \lambda_{q+1}^{2} r^{\frac{1}{p} - \frac{5}{2}} \sigma^{\frac{1}{p} + \frac{3}{2}} + \lambda_{q+1}^{3} r^{\frac{1}{p} - \frac{7}{2}} \sigma^{\frac{1}{p} + \frac{5}{2}} + \lambda_{q+1}^{3} r^{\frac{1}{p} - \frac{5}{2}} \sigma^{\frac{1}{p} + \frac{3}{2}}]  \nonumber\\
\overset{\eqref{estimate 130} \eqref{alpha}}{\lesssim}& \delta_{q+1}^{\frac{1}{2}} \lambda_{q+1} l^{-2} M_{0}(t)^{\frac{1}{2}} r^{\frac{1}{p} - \frac{3}{2}} \sigma^{\frac{1}{p} + \frac{1}{2}}. \label{estimate 189}
\end{align} 
Similarly, 
\begin{align}
& \lVert w_{q+1}^{c} \rVert_{C_{t}W_{x}^{1,p}} \lesssim \lambda_{q+1}^{-2} \sum_{\xi \in \Lambda} \lVert a_{\xi} \rVert_{C_{t}C_{x}^{3}} \lVert \phi_{\xi} \Psi_{\xi} \rVert_{C_{t}L_{x}^{p}} + \lVert a_{\xi} \rVert_{C_{t}C_{x}^{1}} \lVert \phi_{\xi} \Psi_{\xi} \rVert_{C_{t}W_{x}^{2,p}} + \lVert a_{\xi} \rVert_{C_{t}C_{x}^{2}} \lVert \phi_{\xi} \Psi_{\xi} \rVert_{C_{t}W_{x}^{1,p}} \nonumber \\
& \hspace{15mm} + \lVert a_{\xi} \rVert_{C_{t}C_{x}^{1}} \lVert \phi_{\xi} \rVert_{C_{t}W_{x}^{1,p}} \lVert \Psi_{\xi} \rVert_{W_{x}^{1,p}} + \lVert a_{\xi} \rVert_{C_{t,x}} ( \lVert \phi_{\xi} \rVert_{C_{t}W_{x}^{2,p}} \lVert \Psi_{\xi} \rVert_{W_{x}^{1,p}} + \lVert \phi_{\xi} \rVert_{C_{t}W_{x}^{1,p}} \lVert \Psi_{\xi} \rVert_{W_{x}^{2,p}}) \nonumber \\
& \hspace{5mm} \overset{\eqref{estimate 175} \eqref{estimate 187}\eqref{estimate 150}}{\lesssim}  \delta_{q+1}^{\frac{1}{2}} \lambda_{q+1}^{-2} M_{0}(t)^{\frac{1}{2}} l^{-2}  \nonumber\\
& \hspace{10mm} \times [ l^{-29} r^{\frac{1}{p} - \frac{1}{2}} \sigma^{\frac{1}{p} - \frac{1}{2}} + l^{-11} \lambda_{q+1}^{2} r^{\frac{1}{p} - \frac{1}{2}} \sigma^{\frac{1}{p} - \frac{1}{2}}  + l^{-20} \lambda_{q+1} r^{\frac{1}{p} - \frac{1}{2}} \sigma^{\frac{1}{p} - \frac{1}{2}}  + l^{-11}\lambda_{q+1}^{2} r^{\frac{1}{p} - \frac{3}{2}} \sigma^{\frac{1}{p} + \frac{1}{2}} \nonumber \\
& \hspace{20mm} + \lambda_{q+1}^{3} r^{\frac{1}{p} - \frac{5}{2}} \sigma^{\frac{1}{p} + \frac{3}{2}} + \lambda_{q+1}^{3} r^{\frac{1}{p} - \frac{3}{2}} \sigma^{\frac{1}{p} + \frac{1}{2}} ] \lesssim \delta_{q+1}^{\frac{1}{2}} \lambda_{q+1} M_{0}(t)^{\frac{1}{2}} l^{-2} r^{\frac{1}{p} - \frac{3}{2}} \sigma^{\frac{1}{p} + \frac{1}{2}}. \label{estimate 190}
\end{align} 
Finally, for all $p \in (1,\infty)$, we can estimate 
\begin{align}
& \lVert d_{q+1}^{t} \rVert_{C_{t}W_{x}^{1,p}}  + \lVert w_{q+1}^{t} \rVert_{C_{t}W_{x}^{1,p}} \nonumber\\
&\lesssim \mu^{-1} (\sum_{\xi \in \Lambda_{v}} + \sum_{\xi \in \Lambda_{\Xi}}) \lVert a_{\xi} \rVert_{C_{t,x}} \lVert a_{\xi} \rVert_{C_{t}C_{x}^{1}} \lVert \phi_{\xi} \varphi_{\xi} \rVert_{C_{t}L_{x}^{2p}}^{2} + \lVert a_{\xi}\rVert_{C_{t,x}}^{2} \lVert \phi_{\xi} \varphi_{\xi} \rVert_{C_{t}L_{x}^{2p}} \lVert \nabla (\phi_{\xi} \varphi_{\xi}) \rVert_{C_{t}L_{x}^{2p}} \nonumber \\
&\overset{\eqref{estimate 175} \eqref{estimate 150} \eqref{estimate 187}}{\lesssim} \delta_{q+1} \mu^{-1} [ l^{-13} M_{0}(t) r^{\frac{1}{p} - 1} \sigma^{\frac{1}{p} - 1} + l^{-4} M_{0}(t) \lambda_{q+1} r^{\frac{1}{p} - 1} \sigma^{\frac{1}{p} - 1}] \nonumber \\
& \hspace{30mm} \overset{\eqref{estimate 130}}{\lesssim}\delta_{q+1} \mu^{-1} l^{-4} M_{0}(t) \lambda_{q+1} r^{\frac{1}{p} -1} \sigma^{\frac{1}{p} -1}. \label{estimate 191}
\end{align}
This leads us to conclude for all $p \in (1,\infty)$, 
\begin{align}
& \lVert d_{q+1} \rVert_{C_{t}W_{x}^{1,p}} + \lVert  w_{q+1} \rVert_{C_{t}W_{x}^{1,p}} \label{estimate 192}\\
\overset{\eqref{estimate 176} \eqref{estimate 188}-\eqref{estimate 191}}{\lesssim}&  l^{-2} M_{0}(t)^{\frac{1}{2}} r^{\frac{1}{p} - \frac{1}{2}} \sigma^{\frac{1}{p} - \frac{1}{2}} \lambda_{q+1} \delta_{q+1}^{\frac{1}{2}} + \lambda_{q+1} l^{-2} M_{0}(t)^{\frac{1}{2}} r^{\frac{1}{p} - \frac{3}{2}} \sigma^{\frac{1}{p} + \frac{1}{2}} \delta_{q+1}^{\frac{1}{2}} \nonumber \\
& \hspace{8mm} + \mu^{-1} l^{-4} M_{0}(t) \lambda_{q+1} r^{\frac{1}{p} - 1} \sigma^{\frac{1}{p} - 1} \delta_{q+1}  \lesssim  \delta_{q+1}^{\frac{1}{2}} l^{-2} M_{0}(t)^{\frac{1}{2}} \lambda_{q+1} r^{\frac{1}{p} - \frac{1}{2}} \sigma^{\frac{1}{p} - \frac{1}{2}}.  \nonumber
\end{align} 
Next, to estimate $\lVert d_{q+1} \rVert_{C_{t,x}^{1}}$ and $\lVert w_{q+1} \rVert_{C_{t,x}^{1}}$, we compute 
\begin{align}
& \lVert d_{q+1}^{p} + d_{q+1}^{c} \rVert_{C_{t,x}^{1}}\nonumber  \\
\overset{\eqref{estimate 193}}{\lesssim}& \lambda_{q+1}^{-2} \sum_{\xi \in \Lambda_{\Xi}} \lVert a_{\xi} \rVert_{C_{t}^{1}C_{x}^{2}} \lVert \phi_{\xi} \Psi_{\xi}\rVert_{C_{t}C_{x}} + \lVert a_{\xi} \rVert_{C_{t,x}} [ \lVert \phi_{\xi} \rVert_{C_{t}^{1}C_{x}^{2}} \lVert \Psi_{\xi} \rVert_{C_{x}} + \lVert \phi_{\xi} \rVert_{C_{t}^{1}C_{x}} \lVert \Psi_{\xi} \rVert_{C_{x}^{2}}] \nonumber \\
& \hspace{10mm} + \lVert a_{\xi} \rVert_{C_{t}C_{x}^{3}} \lVert \phi_{\xi} \Psi_{\xi} \rVert_{C_{t} C_{x}} + \lVert a_{\xi} \rVert_{C_{t,x}} \lVert \phi_{\xi} \Psi_{\xi} \rVert_{C_{t}C_{x}^{3}}   \nonumber \\
\overset{\eqref{estimate 175}\eqref{estimate 150} \eqref{estimate 309} }{\lesssim}& \delta_{q+1}^{\frac{1}{2}}\lambda_{q+1}^{-2} M_{0}(t)^{\frac{1}{2}}[ l^{-2} (\lambda_{q+1}^{3} \sigma^{\frac{5}{2}} r^{-\frac{7}{2}} \mu + \lambda_{q+1}^{3} \sigma^{\frac{1}{2}} r^{-\frac{3}{2}} \mu) \nonumber\\
& \hspace{5mm} + l^{-17} r^{-\frac{1}{2}} \sigma^{-\frac{1}{2}} + l^{-2} \lambda_{q+1}^{3} r^{-\frac{1}{2}} \sigma^{-\frac{1}{2}}] 
\lesssim \delta_{q+1}^{\frac{1}{2}} l^{-2} M_{0}(t)^{\frac{1}{2}} \lambda_{q+1}\sigma^{\frac{1}{2}} r^{-\frac{3}{2}} \mu. \label{estimate 194}
\end{align} 
Similarly, we can estimate 
\begin{align}
& \lVert w_{q+1}^{p}+ w_{q+1}^{c} \rVert_{C_{t,x}^{1}} \nonumber \\
&\overset{\eqref{estimate 193}}{\lesssim} \lambda_{q+1}^{-2} \sum_{\xi \in \Lambda} \lVert a_{\xi} \rVert_{C_{t}^{1}C_{x}^{2}} \lVert \phi_{\xi} \Psi_{\xi} \rVert_{C_{t,x}} + \lVert a_{\xi} \rVert_{C_{t,x}} (\lVert \phi_{\xi} \rVert_{C_{t}^{1}C_{x}^{2}} \lVert \Psi_{\xi} \rVert_{C_{x}} + \lVert \phi_{\xi} \rVert_{C_{t}^{1}C_{x}} \lVert \Psi_{\xi} \rVert_{C_{x}^{2}}) \nonumber \\
& \hspace{10mm} + \lVert a_{\xi} \rVert_{C_{t}C_{x}^{3}} \lVert \phi_{\xi} \Psi_{\xi} \rVert_{C_{t,x}} + \lVert a_{\xi} \rVert_{C_{t,x}} \lVert \phi_{\xi} \Psi_{\xi} \rVert_{C_{t}C_{x}^{3}} \nonumber \\
&\overset{\eqref{estimate 175}\eqref{estimate 187}\eqref{estimate 150}\eqref{estimate 309}}{\lesssim}  \delta_{q+1}^{\frac{1}{2}} \lambda_{q+1}^{-2} M_{0}(t)^{\frac{1}{2}} [ l^{-29} r^{-\frac{1}{2}} \sigma^{-\frac{1}{2}} + l^{-2} (\lambda_{q+1}^{3} \sigma^{\frac{5}{2}} r^{-\frac{7}{2}} \mu  \nonumber \\
& \hspace{20mm} + \lambda_{q+1}^{3} \sigma^{\frac{1}{2}} r^{-\frac{3}{2}} \mu) + l^{-2} \lambda_{q+1}^{3} r^{-\frac{1}{2}} \sigma^{-\frac{1}{2}}]  \lesssim \delta_{q+1}^{\frac{1}{2}} l^{-2} M_{0}(t)^{\frac{1}{2}} \lambda_{q+1} \sigma^{\frac{1}{2}} r^{-\frac{3}{2}} \mu. \label{estimate 195}
\end{align}
Next, we bound $\mathbb{P}\mathbb{P}_{\neq 0}$ in the expense of $\lambda_{q+1}^{\alpha}$ similarly to \cite[Equ. (7.46c)]{BV19b} as follows: for $p \in \mathbb{R}_{+}$ sufficiently large 
\begin{align}
\lVert d_{q+1}^{t} \rVert_{C_{t,x}^{1}} 
\overset{\eqref{estimate 406} \eqref{estimate 309}\eqref{estimate 150} \eqref{estimate 175} \eqref{eta}-\eqref{sigma, r, mu}}{\lesssim} & \delta_{q+1} M_{0}(t) l^{-4} \lambda_{q+1}^{1+ \alpha} \sigma^{\frac{1}{p}} r^{\frac{1}{p} -2} \nonumber\\
& \hspace{10mm}  \overset{\eqref{sigma, r, mu}}{\lesssim} \delta_{q+1} M_{0}(t) l^{-4} \lambda_{q+1}^{1+ \alpha} r^{-2}. \label{estimate 410}
\end{align} 
Similarly, we can compute for $p \in \mathbb{R}_{+}$ sufficiently large 
\begin{align}
\lVert w_{q+1}^{t} \rVert_{C_{t,x}^{1}} \overset{\eqref{estimate 187} \eqref{estimate 413} \eqref{estimate 175}}{\lesssim}&  \delta_{q+1} l^{-4} M_{0}(t) \lambda_{q+1}^{1+ \frac{\alpha}{2-5\eta}} \sigma^{\frac{\alpha}{2-5\eta} + \frac{1}{p}} r^{\frac{1}{p} - 2 - \frac{\alpha}{2-5\eta}} \mu^{\frac{\alpha}{2-5\eta}}  \nonumber\\
& \hspace{20mm} \overset{\eqref{estimate 130}\eqref{eta}-\eqref{sigma, r, mu}}{\lesssim} \delta_{q+1} M_{0}(t) l^{-4}  \lambda_{q+1}^{1+ \alpha} r^{-2}. \label{estimate 197}
\end{align} 
Thus, we are now able to conclude 
\begin{align}
 \lVert w_{q+1} \rVert_{C_{t,x}^{1}} + \lVert d_{q+1} \rVert_{C_{t,x}^{1}}  
\overset{\eqref{estimate 176}\eqref{estimate 194}-\eqref{estimate 197}}{\lesssim}& \delta_{q+1}^{\frac{1}{2}} l^{-2} M_{0}(t)^{\frac{1}{2}} \lambda_{q+1} \sigma^{\frac{1}{2}} r^{-\frac{3}{2}}\mu + \delta_{q+1} M_{0}(t) l^{-4} \lambda_{q+1}^{1+ \alpha} r^{-2} \nonumber\\
& \hspace{2mm} \overset{\eqref{estimate 130} \eqref{alpha} \eqref{eta} }{\lesssim} \delta_{q+1}^{\frac{1}{2}} l^{-2} M_{0}(t)^{\frac{1}{2}}\lambda_{q+1} \sigma^{\frac{1}{2}} r^{-\frac{3}{2}} \mu.  \label{estimate 198}
\end{align}
We are now ready to verify \eqref{estimate 117} as follows:
\begin{subequations}
\begin{align}
&\lVert v_{q+1} (t) - v_{q}(t) \rVert_{L_{x}^{2}} \overset{\eqref{estimate 203}}{\leq} \lVert w_{q+1}(t) \rVert_{L_{x}^{2}} + \lVert v_{l}(t) - v_{q}(t) \rVert_{L_{x}^{2}} \overset{\eqref{estimate 202} \eqref{estimate 199}}{\leq} M_{0}(t)^{\frac{1}{2}}\delta_{q+1}^{\frac{1}{2}},    \\
&\lVert \Xi_{q+1} (t) - \Xi_{q}(t) \rVert_{L_{x}^{2}} \overset{\eqref{estimate 203}}{\leq} \lVert d_{q+1}(t) \rVert_{L_{x}^{2}} + \lVert \Xi_{l}(t) - \Xi_{q}(t) \rVert_{L_{x}^{2}}  \overset{\eqref{estimate 202} \eqref{estimate 199}}{\leq} M_{0}(t)^{\frac{1}{2}}\delta_{q+1}^{\frac{1}{2}}.   
\end{align}
\end{subequations}
We also verify \eqref{estimate 97}-\eqref{estimate 98} at level $q+1$ via definitions from \eqref{estimate 203} as 
\begin{subequations}\label{estimate 425}
\begin{align}
&\lVert v_{q+1} \rVert_{C_{t}L_{x}^{2}} \overset{\eqref{estimate 200}}{\leq} \Vert w_{q+1} \rVert_{C_{t}L_{x}^{2}} + M_{0}(t)^{\frac{1}{2}} (1+ \sum_{1\leq \iota \leq q} \delta_{\iota}^{\frac{1}{2}}) \overset{\eqref{estimate 202}}{\leq} M_{0}(t)^{\frac{1}{2}} (1+ \sum_{1 \leq \iota \leq q+1} \delta_{\iota}^{\frac{1}{2}}), \label{estimate 275}\\
& \lVert \Xi_{q+1} \rVert_{C_{t}L_{x}^{2}} \overset{ \eqref{estimate 200}}{\leq} \Vert d_{q+1} \rVert_{C_{t}L_{x}^{2}} + M_{0}(t)^{\frac{1}{2}} (1+ \sum_{1\leq \iota \leq q} \delta_{\iota}^{\frac{1}{2}}) \overset{\eqref{estimate 202}}{\leq} M_{0}(t)^{\frac{1}{2}} (1+ \sum_{1 \leq \iota \leq q+1} \delta_{\iota}^{\frac{1}{2}}). \label{estimate 276}
\end{align} 
\end{subequations}
Next, we can also compute for $a \in 2 \mathbb{N}$ sufficiently large  
\begin{align*}
\lVert v_{q+1} \rVert_{C_{t,x}^{1}}  + \lVert \Xi_{q+1} \rVert_{C_{t,x}^{1}} &
\overset{\eqref{estimate 203} \eqref{estimate 198}}{\lesssim} \lVert v_{q} \rVert_{C_{t,x}^{1}} + \lVert \Xi_{q} \rVert_{C_{t,x}^{1}} + \delta_{q+1}^{\frac{1}{2}} l^{-2} M_{0}(t)^{\frac{1}{2}} \lambda_{q+1} \sigma^{\frac{1}{2}} r^{-\frac{3}{2}} \mu \\
& \overset{\eqref{estimate 99}}{\lesssim} M_{0}(t)^{\frac{1}{2}} \lambda_{q}^{4} + l^{-2} M_{0}(t)^{\frac{1}{2}} \lambda_{q+1} \sigma^{\frac{1}{2}} r^{-\frac{3}{2}} \mu \overset{\eqref{estimate 130} \eqref{sigma, r, mu} }{\leq} M_{0}(t)^{\frac{1}{2}} \lambda_{q+1}^{4} 
\end{align*} 
which evidently verifies \eqref{estimate 99} at level $q+1$. 

\subsubsection{Reynolds stress} 
In order to determine the magnetic Reynolds stress at level $q+1$, we first write using \eqref{estimate 91}, \eqref{estimate 203}, and \eqref{estimate 204}, 
\begin{align}
 \text{div} \mathring{R}_{q+1}^{\Xi} 
=& - \text{div} ((v_{l} + z_{1,l} ) \otimes (\Xi_{l} + z_{2,l}) - (\Xi_{l} + z_{2,l} ) \otimes (v_{l} + z_{1,l} ) ) + \text{div} (\mathring{R}_{l}^{\Xi} + R_{\text{com1}}^{\Xi}) \nonumber \\
&+ \partial_{t} (d_{q+1}^{p} + d_{q+1}^{c} + d_{q+1}^{t}) + (-\Delta)^{m_{2}} d_{q+1} \nonumber \\
&+ \text{div} ((v_{l} + w_{q+1}) \otimes (\Xi_{l} + d_{q+1}) - (\Xi_{l} + d_{q+1} ) \otimes (v_{l} + w_{q+1})) \nonumber \\
&+ \text{div}  (v_{q+1} \otimes z_{2} + z_{1} \otimes \Xi_{q+1} + z_{1} \otimes z_{2} - \Xi_{q+1} \otimes z_{1} - z_{2} \otimes v_{q+1} - z_{2} \otimes z_{1}) \label{estimate 205}    
\end{align} 
where we may rewrite using \eqref{estimate 176}-\eqref{estimate 203}
\begin{align}
& - (v_{l} + z_{1,l}) \otimes (\Xi_{l} + z_{2,l}) + (\Xi_{l} +z_{2,l} ) \otimes (v_{l} + z_{1,l} ) \nonumber \\
&+ (v_{l} + w_{q+1}) \otimes (\Xi_{l} + d_{q+1}) - (\Xi_{l} + d_{q+1}) \otimes (v_{l} + w_{q+1})  \nonumber \\
=& - v_{q+1} \otimes z_{2,l} - z_{1,l} \otimes \Xi_{q+1} - z_{1,l} \otimes z_{2,l}  + \Xi_{q+1} \otimes z_{1,l} + z_{2,l} \otimes v_{q+1} + z_{2,l} \otimes z_{1,l}  \nonumber \\
&+ (v_{l} + z_{1,l} ) \otimes d_{q+1} + w_{q+1} \otimes (\Xi_{l} + z_{2,l})  \nonumber \\
&+ w_{q+1}^{p} \otimes d_{q+1}^{p} + w_{q+1}^{p} \otimes (d_{q+1}^{c} + d_{q+1}^{t}) + (w_{q+1}^{c} + w_{q+1}^{t}) \otimes d_{q+1}  \nonumber \\
& - (\Xi_{l} + z_{2,l} ) \otimes w_{q+1} - d_{q+1} \otimes (v_{l} + z_{1,l})  \nonumber \\
&  - d_{q+1}^{p} \otimes w_{q+1}^{p} - d_{q+1}^{p} \otimes (w_{q+1}^{c} + w_{q+1}^{t}) - (d_{q+1}^{c} + d_{q+1}^{t}) \otimes w_{q+1}. \label{estimate 206}
\end{align}
Applying \eqref{estimate 206} to \eqref{estimate 205} leads us to 
\begin{align}
& \text{div} \mathring{R}_{q+1}^{\Xi} \label{estimate 208}\\ 
=& \underbrace{(-\Delta)^{m_{2}} d_{q+1} + \partial_{t} (d_{q+1}^{p} + d_{q+1}^{c})}_{\text{Part of div} R_{\text{lin}}^{\Xi}} \nonumber \\
& \underbrace{+ \text{div} ((v_{l} + z_{1,l} ) \otimes d_{q+1} + w_{q+1} \otimes (\Xi_{l} + z_{2,l} ) - (\Xi_{l} + z_{2,l}) \otimes w_{q+1} - d_{q+1} \otimes (v_{l} + z_{1,l} ))}_{\text{Another part of div} R_{\text{lin}}^{\Xi}}  \nonumber \\
& \underbrace{\text{div} ((w_{q+1}^{c} + w_{q+1}^{t}) \otimes d_{q+1} - (d_{q+1}^{c}+  d_{q+1}^{t} ) \otimes w_{q+1}  }_{\text{Part of div} R_{\text{corr}}^{\Xi}} \nonumber \\
& \hspace{20mm} \underbrace{+ w_{q+1}^{p} \otimes (d_{q+1}^{c} + d_{q+1}^{t}) - d_{q+1}^{p} \otimes (w_{q+1}^{c}+ w_{q+1}^{t}))}_{\text{Another part of div} R_{\text{corr}}^{\Xi}} \nonumber    \\
& + \underbrace{ \text{div} (w_{q+1}^{p} \otimes d_{q+1}^{p} - d_{q+1}^{p} \otimes w_{q+1}^{p} + \mathring{R}_{l}^{\Xi} ) + \partial_{t} d_{q+1}^{t}}_{\text{div} R_{\text{osc}}^{\Xi}}  \nonumber \\
&+\underbrace{ \text{div} (v_{q+1} \otimes z_{2} - v_{q+1} \otimes z_{2,l} - \Xi_{q+1} \otimes z_{1} + \Xi_{q+1} \otimes z_{1,l} + z_{1} \otimes \Xi_{q+1} - z_{1,l} \otimes \Xi_{q+1}}_{\text{Part of div} R_{\text{com2}}^{\Xi}} \nonumber \\
& \hspace{10mm} \underbrace{ - z_{2} \otimes v_{q+1} + z_{2,l} \otimes v_{q+1} + z_{1} \otimes z_{2} - z_{1,l} \otimes z_{2,l} - z_{2} \otimes z_{1} + z_{2,l} \otimes z_{1,l})}_{\text{Another part of div} R_{\text{com2}}^{\Xi}} + \text{div} R_{\text{com1}}^{\Xi}. \nonumber 
\end{align}
Similarly, using \eqref{estimate 90}, \eqref{estimate 203}, and \eqref{estimate 204}, we first write 
\begin{align*}
 \text{div} \mathring{R}_{q+1}^{v} - \nabla \pi_{q+1} &= - \text{div} ((v_{l} + z_{1,l} ) \otimes (v_{l} + z_{1,l} ) - (\Xi_{l} + z_{2,l} ) \otimes (\Xi_{l} + z_{2,l} )) - \nabla \pi_{l} \\
& + \text{div} (\mathring{R}_{l}^{v}+ R_{\text{com1}}^{v}) + \partial_{t}(w_{q+1}^{p}+ w_{q+1}^{c} + w_{q+1}^{t}) + (-\Delta)^{m_{1}} w_{q+1} \\
&+ \text{div} ((v_{l} + w_{q+1} ) \otimes (v_{l} + w_{q+1}) - (\Xi_{l} + d_{q+1} ) \otimes (\Xi_{l} + d_{q+1} )) \\
&+ \text{div} (v_{q+1} \otimes z_{1} + z_{1} \otimes v_{q+1} + z_{1} \otimes z_{1} - \Xi_{q+1} \otimes z_{2} - z_{2} \otimes \Xi_{q+1} - z_{2} \otimes z_{2}),
\end{align*}
where due to \eqref{estimate 176} -\eqref{estimate 203}
\begin{align*}
& - (v_{l} + z_{1,l} ) \otimes (v_{l} + z_{1,l} ) + (\Xi_{l} + z_{2,l} ) \otimes (\Xi_{l} + z_{2,l} ) \\
&+ (v_{l} + w_{q+1} ) \otimes (v_{l} + w_{q+1} ) - (\Xi_{l} + d_{q+1}) \otimes (\Xi_{l} + d_{q+1}) \\
=& - v_{q+1} \otimes z_{1,l} - z_{1,l} \otimes v_{q+1} - z_{1,l} \otimes z_{1,l} \\
& + \Xi_{q+1} \otimes z_{2,l} + z_{2,l} \otimes \Xi_{q+1} + z_{2,l} \otimes z_{2,l} \\
& + (v_{l} +z_{1,l}) \otimes w_{q+1}+ w_{q+1}\otimes (v_{l} + z_{1,l}) \\
&+ w_{q+1}^{p} \otimes w_{q+1}^{p} + w_{q+1}^{p} \otimes (w_{q+1}^{c} + w_{q+1}^{t}) + (w_{q+1}^{c} + w_{q+1}^{t}) \otimes w_{q+1}\\
& - (\Xi_{l} + z_{2,l}) \otimes d_{q+1} - d_{q+1} \otimes (\Xi_{l} + z_{2,l}) \\
& - d_{q+1}^{p} \otimes d_{q+1}^{p} - d_{q+1}^{p} \otimes (d_{q+1}^{c} + d_{q+1}^{t}) - (d_{q+1}^{c} + d_{q+1}^{t}) \otimes d_{q+1}, 
\end{align*}
which leads us to 
\begin{align}
& \text{div} \mathring{R}_{q+1}^{v} - \nabla \pi_{q+1} \label{estimate 209}\\
=& \underbrace{(-\Delta)^{m_{1}} w_{q+1} + \partial_{t} (w_{q+1}^{p} + w_{q+1}^{c})}_{\text{Part of (div} (R_{\text{lin}}^{v}) + \nabla \pi_{\text{lin}})} \nonumber \\
&+ \underbrace{\text{div} ((v_{l} + z_{1,l} ) \otimes w_{q+1} + w_{q+1} \otimes (v_{l} + z_{1,l} ) - (\Xi_{l} + z_{2,l} ) \otimes d_{q+1} - d_{q+1} \otimes (\Xi_{l} + z_{2,l}))}_{\text{Another part of (div} (R_{\text{lin}}^{v}) + \nabla \pi_{\text{lin}})}  \nonumber \\
&+ \underbrace{\text{div} ((w_{q+1}^{c}+ w_{q+1}^{t}) \otimes w_{q+1} - (d_{q+1}^{c} + d_{q+1}^{t}) \otimes d_{q+1}}_{\text{Part of (div} (R_{\text{corr}}^{v}) + \nabla \pi_{\text{corr}})} \nonumber \\
& \hspace{30mm} + \underbrace{w_{q+1}^{p} \otimes (w_{q+1}^{c} + w_{q+1}^{t}) - d_{q+1}^{p} \otimes (d_{q+1}^{c} + d_{q+1}^{t}))}_{\text{Another part of (div} (R_{\text{corr}}^{v}) + \nabla \pi_{\text{corr}})}  \nonumber \\
&+ \underbrace{\text{div} (w_{q+1}^{p} \otimes w_{q+1}^{p} - d_{q+1}^{p} \otimes d_{q+1}^{p} + \mathring{R}_{l}^{v}) + \partial_{t} w_{q+1}^{t}}_{\text{div} (R_{\text{osc}}^{v}) + \nabla \pi_{\text{osc}}}  \nonumber \\
&+ \underbrace{\text{div} ( v_{q+1} \otimes z_{1} - v_{q+1} \otimes z_{1,l} - \Xi_{q+1} \otimes z_{2} + \Xi_{q+1} \otimes z_{2,l} + z_{1} \otimes v_{q+1} - z_{1,l} \otimes v_{q+1}}_{\text{Part of (div} (R_{\text{com2}}^{v}) + \nabla \pi_{\text{com2}})}  \nonumber \\
& \hspace{10mm} \underbrace{-z_{2} \otimes \Xi_{q+1} + z_{2,l} \otimes \Xi_{q+1} + z_{1} \otimes z_{1} - z_{1,l} \otimes z_{1,l} - z_{2}\otimes z_{2} + z_{2,l} \otimes z_{2,l})}_{\text{Another part of (div} (R_{\text{com2}}^{v}) + \nabla \pi_{\text{com2}})}  \nonumber \\
&+ R_{\text{com1}}^{v} - \nabla \pi_{l}. \nonumber 
\end{align}
Now let us compute the oscillation terms in more details: 
\begin{align}
\text{div} R_{\text{osc}}^{\Xi} 
\overset{\eqref{estimate 208} \eqref{estimate 46}}{=}& \text{div} ( \sum_{\xi \in \Lambda_{\Xi}} a_{\xi}^{2}\phi_{\xi}^{2} \varphi_{\xi}^{2} (\xi \otimes \xi_{2} - \xi_{2} \otimes \xi) \nonumber  \\
& \hspace{10mm} + \sum_{\xi \in \Lambda, \xi' \in \Lambda_{\Xi}: \xi \neq \xi'} a_{\xi} a_{\xi'} \phi_{\xi} \phi_{\xi'} \varphi_{\xi} \varphi_{\xi'} (\xi \otimes \xi_{2}' - \xi_{2}' \otimes \xi) + \mathring{R}_{l}^{\Xi}) + \partial_{t} d_{q+1}^{t} \nonumber \\
\overset{\eqref{estimate 207}}{=}& \text{div} (\sum_{\xi \in \Lambda_{\Xi}} a_{\xi}^{2} \mathbb{P}_{\geq \frac{\lambda_{q+1} \sigma}{2}} (\phi_{\xi}^{2} \varphi_{\xi}^{2}) (\xi \otimes \xi_{2} - \xi_{2} \otimes \xi) \nonumber \\
& \hspace{10mm} + \sum_{\xi \in \Lambda, \xi' \in \Lambda_{\Xi}: \xi \neq \xi' } a_{\xi} a_{\xi'} \phi_{\xi} \phi_{\xi'} \varphi_{\xi} \varphi_{\xi'} ( \xi \otimes \xi_{2}' - \xi_{2}' \otimes \xi)) + \partial_{t} d_{q+1}^{t} \nonumber \\
=& \text{div} (E_{1}^{\Xi} + E_{2}^{\Xi}) + \partial_{t} d_{q+}^{t}, \label{estimate 363} 
\end{align}
where 
\begin{subequations}\label{estimate 446}
\begin{align}
& E_{1}^{\Xi} \triangleq \sum_{\xi \in \Lambda_{\Xi}} a_{\xi}^{2} \mathbb{P}_{\geq \frac{\lambda_{q+1} \sigma}{2}} (\phi_{\xi}^{2}\varphi_{\xi}^{2}) (\xi \otimes \xi_{2} - \xi_{2} \otimes \xi), \label{estimate 210}\\
& E_{2}^{\Xi} \triangleq \sum_{\xi \in \Lambda, \xi' \in \Lambda_{\Xi}: \xi \neq \xi'} a_{\xi} a_{\xi'} \phi_{\xi} \phi_{\xi'} \varphi_{\xi} \varphi_{\xi'} (\xi \otimes \xi_{2}' - \xi_{2}' \otimes \xi), \label{estimate 211} 
\end{align}
\end{subequations}
and we used the fact that $\phi_{\xi}^{2}\varphi_{\xi}^{2}$ is $(\mathbb{T}/\lambda_{q+1} \sigma)^{3}$-periodic so that minimal active frequency in $\mathbb{P}_{\neq 0} (\phi_{\xi}^{2}\varphi_{\xi}^{2})$ is given by $\lambda_{q+1}\sigma$ (cf. \cite[p. 247]{BV19b}). By definition from \eqref{estimate 334}, $d_{q+1}^{t}$ and hence $\partial_{t} d_{q+1}^{t}$ is divergence-free and mean-zero while it is also clear that $\text{div} E_{1}^{\Xi}$ is mean-zero. Moreover, it can be verified that $\text{div} E_{1}^{\Xi}$ is divergence-free using 
\begin{equation}\label{estimate 212}
\xi_{2} \cdot \nabla \phi_{\xi} = 0, \hspace{2mm} \xi_{2} \cdot \nabla \varphi_{\xi} =0, \hspace{2mm} \xi \cdot \nabla \varphi_{\xi} = 0, 
\end{equation} 
although  
\begin{equation}\label{estimate 213}
\xi \cdot \nabla \phi_{\xi} = \phi_{r}' (\lambda_{q+1} \sigma N_{\Lambda} (\xi \cdot x + \mu t)) \lambda_{q+1} \sigma N_{\Lambda}, 
\end{equation} 
all due to \eqref{estimate 422}. Therefore, $\mathcal{R}^{\Xi} (\text{div} E_{1}^{\Xi} + \partial_{t} d_{q+1}^{t})$ is well-defined (see Lemma \ref{divergence inverse operator}), allowing us to define 
\begin{equation}\label{estimate 423}
R_{\text{osc}}^{\Xi} \triangleq \mathcal{R}^{\Xi} (\text{div} E_{1}^{\Xi} + \partial_{t} d_{q+1}^{t}) + E_{2}^{\Xi}
\end{equation} 
where we further rewrite for subsequent convenience  
\begin{align}
 \text{div} E_{1}^{\Xi} +\partial_{t} d_{q+1}^{t} 
&\overset{\eqref{estimate 210}}{=} \sum_{\xi \in \Lambda_{\Xi}} \text{div} [a_{\xi}^{2} \mathbb{P}_{\geq \frac{\lambda_{q+1} \sigma}{2}} (\phi_{\xi}^{2}\varphi_{\xi}^{2}) (\xi \otimes \xi_{2} - \xi_{2} \otimes \xi) ] + \partial_{t} d_{q+1}^{t} \label{estimate 239}\\
&\overset{\eqref{estimate 212}\eqref{estimate 213}}{=} \sum_{\xi \in \Lambda_{\Xi}} \mathbb{P}_{\neq 0} [\mathbb{P}_{\geq \frac{\lambda_{q+1} \sigma}{2}} (\phi_{\xi}^{2} \varphi_{\xi}^{2}) (\xi \otimes \xi_{2} - \xi_{2} \otimes \xi) \nabla a_{\xi}^{2} \nonumber \\
& \hspace{20mm}  - \xi_{2} a_{\xi}^{2} \mathbb{P}_{\geq \frac{\lambda_{q+1} \sigma}{2}} (\xi \cdot \nabla \phi_{\xi} 2 \phi_{\xi} \varphi_{\xi}^{2} )] + \partial_{t} d_{q+1}^{t} \nonumber\\
&\overset{\eqref{estimate 57}}{=} \sum_{\xi \in \Lambda_{\Xi}} \mathbb{P}_{\neq 0} [ \mathbb{P}_{\geq \frac{\lambda_{q+1} \sigma}{2}} (\phi_{\xi}^{2} \varphi_{\xi}^{2}) (\xi \otimes \xi_{2} - \xi_{2} \otimes \xi) \nabla a_{\xi}^{2}  \nonumber\\
& \hspace{20mm} - \xi_{2} a_{\xi}^{2} \mathbb{P}_{\geq \frac{\lambda_{q+1} \sigma}{2}} (\mu^{-1} \partial_{t} \phi_{\xi} 2 \phi_{\xi} \varphi_{\xi}^{2})] + \partial_{t} d_{q+1}^{t} \nonumber\\
&\overset{\eqref{estimate 215}}{=} \sum_{\xi \in \Lambda_{\Xi}} \mathbb{P}_{\neq 0} [ \mathbb{P}_{\geq \frac{\lambda_{q+1} \sigma}{2}} (\phi_{\xi}^{2} \varphi_{\xi}^{2}) (\xi \otimes \xi_{2} - \xi_{2} \otimes \xi) \nabla a_{\xi}^{2} ] \nonumber\\
& + \mu^{-1} \sum_{\xi \in \Lambda_{\Xi}} \xi_{2} \mathbb{P}_{\neq 0} [\partial_{t} a_{\xi}^{2} \mathbb{P}_{\neq 0} (\phi_{\xi}^{2} \varphi_{\xi}^{2})] - \mu^{-1} \sum_{\xi \in \Lambda_{\Xi}} \nabla \Delta^{-1} \text{div} (\partial_{t} (a_{\xi}^{2} \mathbb{P}_{\neq 0} (\phi_{\xi}^{2} \varphi_{\xi}^{2} ))\xi_{2}). \nonumber
\end{align} 
Next, we compute 
\begin{align}
 \text{div} R_{\text{osc}}^{v}& + \nabla \pi_{\text{osc}}  \overset{\eqref{estimate 209}}{=} \text{div} (w_{q+1}^{p} \otimes w_{q+1}^{p} - d_{q+1}^{p} \otimes d_{q+1}^{p} + \mathring{R}_{l}^{v}) + \partial_{t} w_{q+1}^{t} \label{estimate 263} \\
\overset{\eqref{estimate 57}}{=}& \text{div} ( \sum_{\xi \in \Lambda_{v}} a_{\xi}^{2} \phi_{\xi}^{2} \varphi_{\xi}^{2} (\xi \otimes \xi) + \sum_{\xi \in \Lambda_{\Xi}} a_{\xi}^{2} (\xi \otimes \xi - \xi_{2} \otimes \xi_{2}) \nonumber\\
& \hspace{30mm} + \sum_{\xi \in \Lambda_{\Xi}} a_{\xi}^{2} \mathbb{P}_{\neq 0} (\phi_{\xi}^{2} \varphi_{\xi}^{2}) (\xi \otimes \xi - \xi_{2} \otimes \xi_{2}) + \mathring{R}_{l}^{v}) \nonumber \\
&+ \text{div} (\sum_{\xi, \xi' \in\Lambda: \xi \neq \xi'} a_{\xi} a_{\xi'} \phi_{\xi}\phi_{\xi'} \varphi_{\xi} \varphi_{\xi'} \xi \otimes \xi' - \sum_{\xi,\xi' \in \Lambda_{\Xi}: \xi \neq \xi'} a_{\xi} a_{\xi'} \phi_{\xi} \phi_{\xi'} \varphi_{\xi} \varphi_{\xi'} \xi_{2} \otimes \xi_{2}') + \partial_{t} w_{q+1}^{t} \nonumber 
\end{align}
and use the definition of $\mathring{G}^{\Xi}$ from \eqref{estimate 145} to continue by 
\begin{align}
&\text{div} R_{\text{osc}}^{v} + \nabla \pi_{\text{osc}} \label{estimate 264} \\
\overset{\eqref{estimate 216}}{=}& \nabla \rho_{v} + \text{div} ( \sum_{\xi \in \Lambda_{v}} a_{\xi}^{2} \mathbb{P}_{\neq 0} (\phi_{\xi}^{2} \varphi_{\xi}^{2}) (\xi \otimes \xi))  + \text{div} (\sum_{\xi \in \Lambda_{\Xi}} a_{\xi}^{2} \mathbb{P}_{\geq \frac{\lambda_{q+1} \sigma}{2}} (\phi_{\xi}^{2} \varphi_{\xi}^{2}) (\xi \otimes \xi - \xi_{2} \otimes \xi_{2} ))  \nonumber \\
&+ \text{div} (\sum_{\xi, \xi'\in\Lambda: \xi \neq \xi'} a_{\xi} a_{\xi'} \phi_{\xi} \phi_{\xi'} \varphi_{\xi} \varphi_{\xi'} (\xi \otimes \xi') \nonumber\\
& \hspace{20mm} - \sum_{\xi, \xi' \in \Lambda_{\Xi}: \xi \neq \xi'} a_{\xi} a_{\xi'} \phi_{\xi} \phi_{\xi'} \varphi_{\xi} \varphi_{\xi'} (\xi_{2} \otimes \xi_{2}') ) + \partial_{t} w_{q+1}^{t} \nonumber \\
\overset{\eqref{estimate 217}\eqref{estimate 57}}{=}& \nabla \rho_{v} + \sum_{\xi \in \Lambda} \mathbb{P}_{\neq 0} (\xi (\xi \cdot \nabla a_{\xi}^{2}) \mathbb{P}_{\geq \frac{\lambda_{q+1} \sigma}{2}} (\phi_{\xi}^{2} \varphi_{\xi}^{2}) \nonumber\\
& \hspace{20mm} + \xi a_{\xi}^{2} \mathbb{P}_{\geq \frac{\lambda_{q+1} \sigma}{2}} (2 \phi_{\xi} \xi \cdot \nabla \phi_{\xi} \varphi_{\xi}^{2} + \phi_{\xi}^{2} 2 \varphi_{\xi} \xi \cdot \nabla \varphi_{\xi} ))  \nonumber \\
& - \mu^{-1} \sum_{\xi \in \Lambda} \mathbb{P}_{\neq 0} (\partial_{t} a_{\xi}^{2} \mathbb{P}_{\neq 0} (\phi_{\xi}^{2} \varphi_{\xi}^{2} )) \xi - \sum_{\xi \in \Lambda} \mathbb{P}_{\neq 0} (\xi a_{\xi}^{2} \mathbb{P}_{\geq \frac{\lambda_{q+1} \sigma}{2}}( 2 \phi_{\xi} \xi \cdot \nabla \phi_{\xi} \varphi_{\xi}^{2} ))  \nonumber \\
&+ \mu^{-1} \sum_{\xi \in \Lambda} \nabla \Delta^{-1} \text{div} \partial_{t} (a_{\xi}^{2} \mathbb{P}_{\neq 0} (\phi_{\xi}^{2} \varphi_{\xi}^{2} )) \xi   - \sum_{\xi \in \Lambda_{\Xi}} \mathbb{P}_{\neq 0} (\xi_{2} (\xi_{2} \cdot \nabla a_{\xi}^{2} ) \mathbb{P}_{\geq \frac{\lambda_{q+1} \sigma}{2}} (\phi_{\xi}^{2} \varphi_{\xi}^{2} ))  \nonumber \\
&+ \xi_{2} a_{\xi}^{2} \mathbb{P}_{\geq \frac{\lambda_{q+1} \sigma}{2}} (\xi_{2} \cdot \nabla \phi_{\xi} 2 \phi_{\xi} \varphi_{\xi}^{2} + \xi_{2} \cdot \nabla \varphi_{\xi} 2 \varphi_{\xi} \phi_{\xi}^{2})  \nonumber \\
& + \text{div} (\sum_{\xi, \xi'\in\Lambda: \xi \neq \xi'} a_{\xi} a_{\xi'} \phi_{\xi} \phi_{\xi'} \varphi_{\xi} \varphi_{\xi'} (\xi \otimes \xi') - \sum_{\xi, \xi' \in \Lambda_{\Xi}: \xi \neq \xi'} a_{\xi} a_{\xi'} \phi_{\xi} \phi_{\xi'} \varphi_{\xi} \varphi_{\xi'} (\xi_{2} \otimes \xi_{2}') )  \nonumber \\
\overset{\eqref{estimate 212}}{=}&  \nabla \rho_{v} + \sum_{\xi \in \Lambda} \mathbb{P}_{\neq 0} ( ( \xi \otimes \xi) \nabla a_{\xi}^{2} \mathbb{P}_{\geq \frac{\lambda_{q+1} \sigma}{2}} (\phi_{\xi}^{2} \varphi_{\xi}^{2} ))  \nonumber \\
& - \mu^{-1} \sum_{\xi \in \Lambda} \mathbb{P}_{\neq 0} (\partial_{t} a_{\xi}^{2} \mathbb{P}_{\neq 0} (\phi_{\xi}^{2} \varphi_{\xi}^{2} )) \xi + \mu^{-1} \sum_{\xi \in \Lambda} \nabla \Delta^{-1} \text{div} \partial_{t} (a_{\xi}^{2} \mathbb{P}_{\neq 0} (\phi_{\xi}^{2} \varphi_{\xi}^{2} )) \xi  \nonumber \\
& - \sum_{\xi \in \Lambda_{\Xi}} \mathbb{P}_{\neq 0} (( \xi_{2} \otimes \xi_{2} ) \nabla a_{\xi}^{2} \mathbb{P}_{\geq \frac{\lambda_{q+1} \sigma}{2}} (\phi_{\xi}^{2} \varphi_{\xi}^{2} ) ) \nonumber \\
&+  \text{div} (\sum_{\xi, \xi'\in\Lambda: \xi \neq \xi'} a_{\xi} a_{\xi'} \phi_{\xi} \phi_{\xi'} \varphi_{\xi} \varphi_{\xi'} (\xi \otimes \xi') - \sum_{\xi, \xi' \in \Lambda_{\Xi}: \xi \neq \xi'} a_{\xi} a_{\xi'} \phi_{\xi} \phi_{\xi'} \varphi_{\xi} \varphi_{\xi'} (\xi_{2} \otimes \xi_{2}') ). \nonumber 
\end{align} 
Therefore, we are able to define at last 
\begin{subequations}\label{estimate 427}
\begin{align}
&\mathring{R}_{q+1}^{\Xi} \triangleq R_{\text{lin}}^{\Xi} + R_{\text{corr}}^{\Xi} + R_{\text{osc}}^{\Xi} + R_{\text{com1}}^{\Xi} + R_{\text{com2}}^{\Xi},  \label{estimate 219} \\
& \mathring{R}_{q+1}^{v} \triangleq R_{\text{lin}}^{v} + R_{\text{corr}}^{v} + R_{\text{osc}}^{v} + R_{\text{com1}}^{v} + R_{\text{com2}}^{v},   \label{estimate 218} \\
& \pi_{q+1} \triangleq - \pi_{\text{lin}} - \pi_{\text{corr}}  - \pi_{\text{osc}} - \pi_{\text{com2}} + \pi_{l}, 
\end{align}
\end{subequations} 
 where besides $R_{\text{osc}}^{\Xi}$ in \eqref{estimate 423} and $R_{\text{com1}}^{v}, R_{\text{com1}}^{\Xi}$, and $\pi_{l}$ in \eqref{estimate 235} we define 
\begin{subequations}\label{estimate 426}
\begin{align}
R_{\text{lin}}^{\Xi} \triangleq& \mathcal{R}^{\Xi} ( ( -\Delta)^{m_{2}} d_{q+1} ) +\mathcal{R}^{\Xi} (\partial_{t} (d_{q+1}^{p} + d_{q+1}^{c}))  \label{estimate 220} \\
 & + (v_{l} + z_{1,l} ) \otimes d_{q+1} + w_{q+1} \otimes (\Xi_{l} + z_{2,l}) - (\Xi_{l} + z_{2,l} ) \otimes w_{q+1} - d_{q+1} \otimes (v_{l} + z_{1,l} ),\nonumber  \\
 R_{\text{corr}}^{\Xi} =& (w_{q+1}^{c} + w_{q+1}^{t}) \otimes d_{q+1} - (d_{q+1}^{c}+ d_{q+1}^{t}) \otimes w_{q+1}   \label{estimate 221} \\
 &+ w_{q+1}^{p} \otimes (d_{q+1}^{c} + d_{q+1}^{t}) - d_{q+1}^{p} \otimes (w_{q+1}^{c}+ w_{q+1}^{t}), \nonumber\\
 R_{\text{com2}}^{\Xi} \triangleq& v_{q+1} \otimes (z_{2} - z_{2,l} ) - \Xi_{q+1} \otimes (z_{1} - z_{1,l}) + (z_{1} - z_{1,l}) \otimes \Xi_{q+1} - (z_{2} - z_{2,l}) \otimes v_{q+1} \nonumber \\
 & + (z_{1} - z_{1,l}) \otimes z_{2}+ z_{1,l} \otimes (z_{2} - z_{2,l}) - (z_{2}- z_{2,l}) \otimes z_{1} - z_{2,l} \otimes (z_{1} - z_{1,l}),    \label{estimate 223} \\
 R_{\text{lin}}^{v} \triangleq&  \mathcal{R} ( (-\Delta)^{m_{1}} w_{q+1}) + \mathcal{R}( \partial_{t} (w_{q+1}^{p} + w_{q+1}^{c} ))    \label{estimate 224} \\
 &+ (v_{l} + z_{1,l}) \mathring{\otimes} w_{q+1} + w_{q+1} \mathring{\otimes} (v_{l} + z_{1,l})  - (\Xi_{l} + z_{2,l} ) \mathring{\otimes} d_{q+1} - d_{q+1} \mathring{\otimes} (\Xi_{l} + z_{2,l}), \nonumber \\
 \pi_{\text{lin}} \triangleq& \frac{2}{3} [ (v_{l}+ z_{1,l}) \cdot w_{q+1} - (\Xi_{l} + z_{2,l}) \cdot d_{q+1}],  \label{estimate 225} \\
 R_{\text{corr}}^{v} \triangleq& (w_{q+1}^{c} + w_{q+1}^{t}) \mathring{\otimes} w_{q+1} - (d_{q+1}^{c}+  d_{q+1}^{t}) \mathring{\otimes} d_{q+1}  \label{estimate 226} \\
 &+ w_{q+1}^{p} \mathring{\otimes} (w_{q+1}^{c} + w_{q+1}^{t}) - d_{q+1}^{p} \mathring{\otimes} (d_{q+1}^{c} + d_{q+1}^{t}), \nonumber \\
 \pi_{\text{corr}} \triangleq& \frac{1}{3} [ (w_{q+1}^{c} + w_{q+1}^{t}) \cdot (w_{q+1} + w_{q+1}^{p}) - (d_{q+1}^{c} + d_{q+1}^{t}) \cdot (d_{q+1} + d_{q+1}^{p} )],  \label{estimate 227} \\
 R_{\text{osc}}^{v} \triangleq&   \mathcal{R} [\sum_{\xi \in \Lambda} \mathbb{P}_{\neq 0} ( ( \xi \otimes \xi) \nabla a_{\xi}^{2} \mathbb{P}_{\geq \frac{\lambda_{q+1} \sigma}{2}} (\phi_{\xi}^{2} \varphi_{\xi}^{2} ))  - \mu^{-1} \sum_{\xi \in \Lambda} \mathbb{P}_{\neq 0} (\partial_{t} a_{\xi}^{2} \mathbb{P}_{\neq 0} (\phi_{\xi}^{2} \varphi_{\xi}^{2} )) \xi   \label{estimate 228} \\
& - \sum_{\xi \in \Lambda_{\Xi}} \mathbb{P}_{\neq 0} (( \xi_{2} \otimes \xi_{2} ) \nabla a_{\xi}^{2} \mathbb{P}_{\geq \frac{\lambda_{q+1} \sigma}{2}} (\phi_{\xi}^{2} \varphi_{\xi}^{2} ) )] \nonumber \\
&+  \sum_{\xi, \xi'\in\Lambda: \xi \neq \xi'} a_{\xi} a_{\xi'} \phi_{\xi} \phi_{\xi'} \varphi_{\xi} \varphi_{\xi'} (\xi \mathring{\otimes} \xi') - \sum_{\xi, \xi' \in \Lambda_{\Xi}: \xi \neq \xi'} a_{\xi} a_{\xi'} \phi_{\xi} \phi_{\xi'} \varphi_{\xi} \varphi_{\xi'} (\xi_{2} \mathring{\otimes} \xi_{2}'), \nonumber \\
\pi_{\text{osc}} \triangleq& \rho_{v} + \mu^{-1} \sum_{\xi \in \Lambda} \Delta^{-1} \text{div} \partial_{t} (a_{\xi}^{2} \mathbb{P}_{\neq 0} (\phi_{\xi}^{2} \varphi_{\xi}^{2} )) \xi  \label{estimate 229} \\
& + \frac{1}{3} [\sum_{\xi, \xi' \in \Lambda: \xi\neq \xi'} a_{\xi}  a_{\xi'} \phi_{\xi} \phi_{\xi'} \varphi_{\xi} \varphi_{\xi'} \xi \cdot \xi' - \sum_{\xi, \xi' \in \Lambda_{\Xi}: \xi \neq \xi'} a_{\xi} a_{\xi'} \phi_{\xi} \phi_{\xi'} \varphi_{\xi} \varphi_{\xi'} \xi_{2} \cdot \xi_{2}'], \nonumber \\
R_{\text{com2}}^{v} \triangleq& v_{q+1} \mathring{\otimes} (z_{1} - z_{1,l}) - \Xi_{q+1} \mathring{\otimes} (z_{2} - z_{2,l}) + (z_{1} - z_{1,l} ) \mathring{\otimes} v_{q+1} - (z_{2} - z_{2,l}) \mathring{\otimes} \Xi_{q+1}  \label{estimate 230} \\
& + (z_{1} - z_{1,l}) \mathring{\otimes} z_{1} + z_{1,l} \mathring{\otimes} (z_{1}- z_{1,l}) - (z_{2} - z_{2,l}) \mathring{\otimes} z_{2} - z_{2,l} \mathring{\otimes} (z_{2} - z_{2,l}), \nonumber \\
\pi_{\text{com2}} \triangleq& \frac{2}{3} [v_{q+1} \cdot (z_{1} - z_{1,l} ) - \Xi_{q+1} \cdot (z_{2} - z_{2,l})] + \frac{1}{3}[ \lvert z_{1} \rvert^{2} -\lvert z_{1,l} \rvert^{2} - \lvert z_{2} \rvert^{2} + \lvert z_{2,l} \rvert^{2}].\label{estimate 231} 
\end{align}
\end{subequations} 
Now we estimate each term separately. We fix 
\begin{equation}\label{p ast}
p^{\ast} \triangleq \frac{1- 4 \eta}{1- \frac{9\eta}{2} + 30 \alpha}
\end{equation} 
which may be readily verified to lie in the range of $(1,2)$ using \eqref{eta}-\eqref{alpha}. We first estimate $R_{\text{lin}}^{\Xi}$ from \eqref{estimate 220} and $R_{\text{lin}}^{v}$ from \eqref{estimate 224}. First, concerning the diffusive terms, we can use Gagliardo-Nirenberg's inequality to estimate 
\begin{align}
& \lVert \mathcal{R}^{\Xi} ( (-\Delta)^{m_{2}} d_{q+1}) \rVert_{C_{t}L_{x}^{p^{\ast}}} + \lVert \mathcal{R} ((-\Delta)^{m_{1}} w_{q+1} ) \rVert_{C_{t}L_{x}^{p^{\ast}}} \nonumber\\
\lesssim& \lVert d_{q+1} \rVert_{C_{t}W_{x}^{m_{2}^{\ast}, p^{\ast}}} + \lVert w_{q+1} \rVert_{C_{t}W_{x}^{m_{1}^{\ast}, p^{\ast}}} 
\lesssim \delta_{q+1}^{\frac{1}{2}} l^{-2} M_{0}(t)^{\frac{1}{2}} r^{\frac{1}{p^{\ast}} - \frac{1}{2}} \sigma^{\frac{1}{p^{\ast}} - \frac{1}{2}} \lambda_{q+1}^{\max \{m_{1}^{\ast}, m_{2}^{\ast} \}}. \label{estimate 232}
\end{align}
Next, we estimate temporal derivatives from \eqref{estimate 220} and \eqref{estimate 224} as follows:
\begin{align}
& \lVert \mathcal{R}^{\Xi} (\partial_{t} (d_{q+1}^{p}+ d_{q+1}^{c})) \rVert_{C_{t}L_{x}^{p^{\ast}}} + \lVert \mathcal{R} (\partial_{t} (w_{q+1}^{p} + w_{q+1}^{c})) \rVert_{C_{t}L_{x}^{p^{\ast}}} \label{estimate 233}\\
\overset{\eqref{estimate 193}}{\lesssim}& \lambda_{q+1}^{-2}[ \sum_{\xi \in \Lambda_{\Xi}} \lVert \partial_{t} \text{curl} (a_{\xi} \phi_{\xi} \Psi_{\xi} \xi_{2}) \rVert_{C_{t}L_{x}^{p^{\ast}}} + \sum_{\xi \in \Lambda} \lVert \partial_{t} \text{curl} (a_{\xi} \phi_{\xi} \Psi_{\xi} \xi ) \rVert_{C_{t}L_{x}^{p^{\ast}}} ] \nonumber \\
\lesssim& \lambda_{q+1}^{-2} [ \sum_{\xi \in \Lambda} \lVert a_{\xi} \rVert_{C_{t}^{1}C_{x}^{1}} \lVert \phi_{\xi} \Psi_{\xi} \rVert_{C_{t}L_{x}^{p^{\ast}}} + \lVert a_{\xi} \rVert_{C_{t,x}} \lVert \phi_{\xi} \rVert_{C_{t}^{1} W_{x}^{1, p^{\ast}}} \lVert \Psi_{\xi} \rVert_{L_{x}^{p^{\ast}}} \nonumber \\
& + ( \lVert a_{\xi} \rVert_{C_{t}^{1}C_{x}} \lVert \phi_{\xi} \rVert_{C_{t}L_{x}^{p^{\ast}}} + \lVert a_{\xi} \rVert_{C_{t,x}} \lVert \phi_{\xi} \rVert_{C_{t}^{1}L_{x}^{p^{\ast}}}) \lVert \Psi_{\xi} \rVert_{W_{x}^{1,p^{\ast}}}] \nonumber  \\
&\overset{\eqref{estimate 150} \eqref{estimate 309} \eqref{estimate 187} \eqref{estimate 175}}{\lesssim} \delta_{q+1}^{\frac{1}{2}} \lambda_{q+1}^{-2} M_{0}(t)^{\frac{1}{2}}[ l^{-20} r^{\frac{1}{p^{\ast}} -\frac{1}{2}} \sigma^{\frac{1}{p^{\ast}} -\frac{1}{2}} + l^{-2} \lambda_{q+1}^{2} \sigma^{\frac{1}{p^{\ast}} + \frac{3}{2}} r^{\frac{1}{p^{\ast} }-\frac{5}{2}} \mu \nonumber  \\
& \hspace{6mm} + l^{-11} r^{\frac{1}{p^{\ast}} - \frac{1}{2}} \lambda_{q+1} \sigma^{\frac{1}{p^{\ast}} - \frac{1}{2}} + l^{-2} \lambda_{q+1}^{2} r^{\frac{1}{p^{\ast}} - \frac{3}{2}} \sigma^{\frac{1}{p^{\ast}} + \frac{1}{2}} \mu]  
\lesssim\delta_{q+1}^{\frac{1}{2}} M_{0}(t)^{\frac{1}{2}} l^{-2} r^{\frac{1}{p^{\ast}} - \frac{3}{2}} \sigma^{\frac{1}{p^{\ast}} + \frac{1}{2}} \mu.  \nonumber
\end{align}  
Finally, we estimate the rest of the terms in \eqref{estimate 220} and \eqref{estimate 224} by 
\begin{align}
&  \lVert  (v_{l} + z_{1,l} ) \otimes d_{q+1} + w_{q+1} \otimes (\Xi_{l} + z_{2,l}) - (\Xi_{l} + z_{2,l} ) \otimes w_{q+1} - d_{q+1} \otimes (v_{l} + z_{1,l} ) \rVert_{C_{t}L_{x}^{p^{\ast}}} \nonumber\\
&  + \lVert (v_{l} + z_{1,l}) \mathring{\otimes} w_{q+1} + w_{q+1} \mathring{\otimes} (v_{l} + z_{1,l})  - (\Xi_{l} + z_{2,l} ) \mathring{\otimes} d_{q+1} - d_{q+1} \mathring{\otimes} (\Xi_{l} + z_{2,l}) \rVert_{C_{t}L_{x}^{p^{\ast}}} \nonumber\\
\overset{\eqref{estimate 186}}{\lesssim}& ( l \lVert v_{q} \rVert_{C_{t,x}^{1}} + \lVert z_{1} \rVert_{C_{t}L_{x}^{\infty}} + l \lVert \Xi_{q} \rVert_{C_{t,x}^{1}} + \lVert z_{2} \rVert_{C_{t}L_{x}^{\infty}}) \delta_{q+1}^{\frac{1}{2}} l^{-2} M_{0}(t)^{\frac{1}{2}} r^{\frac{1}{p^{\ast}} - \frac{1}{2}} \sigma^{\frac{1}{p^{\ast}} - \frac{1}{2}} \nonumber\\
\overset{\eqref{estimate 95} \eqref{estimate 99}}{\lesssim}& \delta_{q+1}^{\frac{1}{2}} l^{-1} M_{0}(t) \lambda_{q}^{4} r^{\frac{1}{p^{\ast}} - \frac{1}{2}} \sigma^{\frac{1}{p^{\ast}} - \frac{1}{2}}. \label{estimate 234}
\end{align} 
Therefore, using the fact that due to \eqref{p ast},
\begin{subequations}\label{estimate 237}
\begin{align}
& \frac{27 \alpha}{8} + (8\eta - 2) \frac{1}{p^{\ast}} - 4 \eta + 1 + \max\{m_{1}^{\ast}, m_{2}^{\ast} \} \overset{\eqref{eta}}{\leq} - \frac{453\alpha}{8} - \frac{3}{8} (1- \max\{m_{1}^{\ast}, m_{2}^{\ast}\}),\\
& \frac{27\alpha}{8} + (8\eta- 2) \frac{1}{p^{\ast}} - 9 \eta + 2 = - \frac{453\alpha}{8}, \\
& 2 \alpha + (8\eta -2) \frac{1}{p^{\ast}} - 4 \eta + 1 = -1+ 5\eta - 58 \alpha, 
\end{align}
\end{subequations}
we are now able to conclude from \eqref{estimate 220} and \eqref{estimate 224} for $a \in 2 \mathbb{N}$ sufficiently large 
\begin{align}
& \lVert R_{\text{lin}}^{\Xi} \rVert_{C_{t}L_{x}^{1}} + \lVert R_{\text{lin}}^{v} \rVert_{C_{t}L_{x}^{1}} \label{estimate 249} \\
&\overset{\eqref{estimate 232}-\eqref{estimate 234}}{\lesssim}  \delta_{q+1}^{\frac{1}{2}} l^{-2} M_{0}(t)^{\frac{1}{2}} r^{\frac{1}{p^{\ast}} - \frac{1}{2}} \sigma^{\frac{1}{p^{\ast}} - \frac{1}{2}} \lambda_{q+1}^{\max \{m_{1}^{\ast}, m_{2}^{\ast} \}} + \delta_{q+1}^{\frac{1}{2}} M_{0}(t)^{\frac{1}{2}} l^{-2} r^{\frac{1}{p^{\ast}} - \frac{3}{2}} \sigma^{\frac{1}{p^{\ast}} + \frac{1}{2}} \mu  \nonumber \\
 & \hspace{21mm} + \delta_{q+1}^{\frac{1}{2}} l^{-1} M_{0}(t) \lambda_{q}^{4} r^{\frac{1}{p^{\ast}} - \frac{1}{2}} \sigma^{\frac{1}{p^{\ast}} - \frac{1}{2}} \nonumber\\
 &\overset{\eqref{estimate 130}\eqref{estimate 129}}{\lesssim} M_{0}(t) \delta_{q+2} [ \lambda_{q+1}^{\frac{27 \alpha}{8} + (8\eta -2) \frac{1}{p^{\ast}} - 4 \eta + 1 + \max\{m_{1}^{\ast}, m_{2}^{\ast} \} } + \lambda_{q+1}^{\frac{27 \alpha}{8} + (8\eta -2) \frac{1}{p^{\ast}} - 9 \eta +2}  + \lambda_{q+1}^{2\alpha + (8\eta - 2) \frac{1}{p^{\ast}} - 4 \eta + 1}]    \nonumber\\
&\overset{\eqref{estimate 237}}{\lesssim} M_{0}(t) \delta_{q+2}[ \lambda_{q+1}^{- \frac{453\alpha}{8} - \frac{3}{8} (1- \max\{m_{1}^{\ast}, m_{2}^{\ast} \} )} + \lambda_{q+1}^{- \frac{453\alpha}{8}} + \lambda_{q+1}^{-1 + 5 \eta - 58 \alpha}] \overset{\eqref{eta}}{\ll} M_{0}(t)\delta_{q+2}. \nonumber
\end{align} 
Second, we consider $R_{\text{corr}}^{\Xi}$ from \eqref{estimate 221} and $R_{\text{corr}}^{v}$ from \eqref{estimate 226}. Using the fact that 
\begin{equation}\label{estimate 238}
\frac{79 \alpha}{8} - 11 \eta + 2 + (8\eta -2) \frac{1}{p^{\ast}} \overset{\eqref{p ast}}{=} -\frac{401\alpha}{8} - 2\eta, 
\end{equation} 
we can compute from \eqref{estimate 221} and \eqref{estimate 226} by taking $a \in 2 \mathbb{N}$ sufficiently large 
\begin{align}
&\hspace{5mm}  \lVert R_{\text{corr}}^{\Xi} \rVert_{C_{t}L_{x}^{p^{\ast}}} + \lVert R_{\text{corr}}^{v} \rVert_{C_{t}L_{x}^{p^{\ast}}} \label{estimate 250}\\ 
&\lesssim ( \lVert w_{q+1}^{c} \rVert_{C_{t}L_{x}^{2p^{\ast}}} + \lVert w_{q+1}^{t} \rVert_{C_{t}L_{x}^{2p^{\ast}}} + \lVert d_{q+1}^{c} \rVert_{C_{t}L_{x}^{2p^{\ast}}} + \lVert d_{q+1}^{t} \rVert_{C_{t}L_{x}^{2p^{\ast}}}) \nonumber \\
& \hspace{2mm}  \times (\lVert w_{q+1}^{p} \rVert_{C_{t}L_{x}^{2p^{\ast}}} + \lVert w_{q+1}^{c} \rVert_{C_{t}L_{x}^{2p^{\ast}}} + \lVert w_{q+1}^{t} \rVert_{C_{t}L_{x}^{2p^{\ast}}} + \lVert d_{q+1}^{p} \rVert_{C_{t}L_{x}^{2p^{\ast}}} + \lVert d_{q+1}^{c} \rVert_{C_{t}L_{x}^{2p^{\ast}}} + \lVert d_{q+1}^{t} \rVert_{C_{t}L_{x}^{2p^{\ast}}}) \nonumber \\
&\overset{ \eqref{estimate 179}-\eqref{estimate 236} \eqref{estimate 185}}{\lesssim}  (\delta_{q+1}^{\frac{1}{2}} M_{0}(t)^{\frac{1}{2}} l^{-2} r^{\frac{1}{2p^{\ast}} - \frac{3}{2}} \sigma^{\frac{1}{2p^{\ast}} + \frac{1}{2}} + \delta_{q+1} M_{0}(t) \mu^{-1} l^{-4} r^{\frac{1}{2p^{\ast}} - 1} \sigma^{\frac{1}{2p^{\ast}} - 1}) \nonumber \\
& \times (\delta_{q+1}^{\frac{1}{2}} l^{-2} M_{0}(t)^{\frac{1}{2}} r^{\frac{1}{2p^{\ast}} - \frac{1}{2}} \sigma^{\frac{1}{2p^{\ast}} - \frac{1}{2}} + \delta_{q+1}^{\frac{1}{2}} M_{0}(t)^{\frac{1}{2}} l^{-2} r^{\frac{1}{2p^{\ast}} - \frac{3}{2}} \sigma^{\frac{1}{2p^{\ast}} + \frac{1}{2}} + \delta_{q+1} M_{0}(t) \mu^{-1} l^{-4} r^{\frac{1}{2p^{\ast}} - 1} \sigma^{\frac{1}{2p^{\ast}} - 1}) \nonumber \\
\lesssim& (M_{0}(t) \mu^{-1} l^{-4} r^{\frac{1}{2p^{\ast}} - 1} \sigma^{\frac{1}{2p^{\ast}} - 1})(l^{-2} M_{0}(t)^{\frac{1}{2}} r^{\frac{1}{2p^{\ast}} - \frac{1}{2}} \sigma^{\frac{1}{2p^{\ast}} - \frac{1}{2}})  \nonumber\\
&\hspace{13mm} \overset{\eqref{estimate 130} \eqref{estimate 129} }{\lesssim} M_{0}(t) \delta_{q+2} M_{0}(t)^{\frac{1}{2}} \lambda_{q+1}^{\frac{79 \alpha}{8} - 11 \eta + 2  + (8\eta - 2) \frac{1}{p^{\ast}}} \nonumber\\
& \hspace{9mm} \overset{\eqref{estimate 238}}{\lesssim} M_{0}(t) \delta_{q+2} M_{0}(L)^{\frac{1}{2}} \lambda_{q+1}^{ -\frac{401\alpha}{8} - 2\eta} \ll M_{0}(t) \delta_{q+2}. \nonumber 
\end{align}
Third, we consider $R_{\text{osc}}^{\Xi}$ from \eqref{estimate 423} and $R_{\text{osc}}^{v}$ from \eqref{estimate 228}. Within \eqref{estimate 239} we further rewrite 
\begin{align}
& \nabla \Delta^{-1} \text{div} (\partial_{t} (a_{\xi}^{2} \mathbb{P}_{\neq 0} ( \phi_{\xi}^{2} \varphi_{\xi}^{2})) \xi_{2}) \label{estimate 240}\\
\overset{\eqref{estimate 57}}{=}&  \nabla \Delta^{-1} \text{div} (\partial_{t}a_{\xi}^{2} \mathbb{P}_{\geq \frac{\lambda_{q+1} \sigma}{2}} (\phi_{\xi}^{2} \varphi_{\xi}^{2}) \xi_{2}) + \nabla \Delta^{-1} \text{div} (a_{\xi}^{2} \mathbb{P}_{\geq \frac{\lambda_{q+1} \sigma}{2}} (2 \phi_{\xi} \mu \xi \cdot \nabla \phi_{\xi} \varphi_{\xi}^{2}) \xi_{2})\nonumber \\
=& \nabla \Delta^{-1} \text{div} (\partial_{t} a_{\xi}^{2} \mathbb{P}_{\geq \frac{\lambda_{q+1} \sigma}{2}} (\phi_{\xi}^{2} \varphi_{\xi}^{2}) \xi_{2}) + \mu \nabla \Delta^{-1} (\xi_{2} \cdot \nabla a_{\xi}^{2} \mathbb{P}_{\geq \frac{\lambda_{q+1} \sigma}{2}} (\xi\cdot \nabla \phi_{\xi}^{2} \varphi_{\xi}^{2} )) \nonumber\\
\overset{\eqref{estimate 212}}{=}& \nabla \Delta^{-1} \text{div} (\partial_{t} a_{\xi}^{2} \mathbb{P}_{\geq \frac{\lambda_{q+1} \sigma}{2}} (\phi_{\xi}^{2} \varphi_{\xi}^{2}) \xi_{2}) \nonumber\\
& + \mu\nabla \Delta^{-1} (\xi \cdot \nabla) (\xi_{2} \cdot \nabla a_{\xi}^{2} \mathbb{P}_{\geq \frac{\lambda_{q+1} \sigma}{2}} (\phi_{\xi}^{2} \varphi_{\xi}^{2} )) - \mu \nabla \Delta^{-1} [ (\xi\cdot \nabla) (\xi_{2} \cdot \nabla a_{\xi}^{2}) \mathbb{P}_{\geq \frac{\lambda_{q+1} \sigma}{2}} (\phi_{\xi}^{2} \varphi_{\xi}^{2} )].  \nonumber
\end{align}
Applying \eqref{estimate 240} within \eqref{estimate 239} gives us 
\begin{align}
& \text{div} E_{1}^{\Xi} + \partial_{t} d_{q+1}^{t} \label{extra}\\
=& \sum_{\xi \in \Lambda_{\Xi}} \mathbb{P}_{\neq 0} [ \mathbb{P}_{\geq \frac{\lambda_{q+1} \sigma}{2}} (\phi_{\xi}^{2}\varphi_{\xi}^{2}) (\xi \otimes \xi_{2} - \xi_{2} \otimes \xi) \nabla a_{\xi}^{2} ]  + \mu^{-1} \sum_{\xi \in \Lambda_{\Xi}} \xi_{2} \mathbb{P}_{\neq 0} [ \partial_{t} a_{\xi}^{2} \mathbb{P}_{\neq 0} (\phi_{\xi}^{2} \varphi_{\xi}^{2} )] \nonumber\\
&- \sum_{\xi \in \Lambda_{\Xi}} \mu^{-1} \nabla \Delta^{-1} \text{div} (\partial_{t} a_{\xi}^{2} \mathbb{P}_{\geq \frac{\lambda_{q+1} \sigma}{2}} (\phi_{\xi}^{2} \varphi_{\xi}^{2}) \xi_{2}) \nonumber\\
& \hspace{5mm} + \nabla \Delta^{-1} (\xi \cdot \nabla) (\xi_{2} \cdot \nabla a_{\xi}^{2} \mathbb{P}_{\geq \frac{\lambda_{q+1} \sigma}{2}} (\phi_{\xi}^{2} \varphi_{\xi}^{2} )) - \nabla \Delta^{-1} ( ( \xi \cdot \nabla) (\xi_{2} \cdot \nabla a_{\xi}^{2} ) \mathbb{P}_{\geq \frac{\lambda_{q+1} \sigma}{2}} (\phi_{\xi}^{2} \varphi_{\xi}^{2} )).\nonumber
\end{align}
Therefore, using the fact that
\begin{equation}\label{estimate 242}
\frac{313\alpha}{8} - 2 \eta + (8\eta -2) (\frac{1}{p^{\ast}} - 1) \overset{\eqref{p ast}}{=} - \frac{167\alpha}{8} - \eta
\end{equation} 
and relying on Lemma \ref{Lemma 6.3}, we can estimate from \eqref{extra}
\begin{align}
\lVert \mathcal{R}^{\Xi} (\text{div} E_{1}^{\Xi} + \partial_{t} d_{q+1}^{t} ) &\rVert_{C_{t}L_{x}^{p^{\ast}}}  
\lesssim \sum_{\xi \in \Lambda_{\Xi}} \lVert (-\Delta)^{-\frac{1}{2}} \mathbb{P}_{\neq 0} [ \mathbb{P}_{\geq \frac{\lambda_{q+1} \sigma}{2}} (\phi_{\xi}^{2} \varphi_{\xi}^{2}) \nabla a_{\xi}^{2} ] \rVert_{C_{t}L_{x}^{p^{\ast}}} \label{estimate 244}\\
& + \mu^{-1} \lVert (-\Delta)^{-\frac{1}{2}} \mathbb{P}_{\neq 0} [\partial_{t} a_{\xi}^{2} \mathbb{P}_{\geq \frac{\lambda_{q+1} \sigma}{2}} (\phi_{\xi}^{2} \varphi_{\xi}^{2} )] \rVert_{C_{t}L_{x}^{p^{\ast}}} \nonumber\\
&  + \lVert (-\Delta)^{-\frac{1}{2}} \mathbb{P}_{\neq 0} (\nabla a_{\xi}^{2} \mathbb{P}_{\geq \frac{\lambda_{q+1} \sigma}{2}} (\phi_{\xi}^{2} \varphi_{\xi}^{2} ) ) \rVert_{C_{t}L_{x}^{p^{\ast}}} \nonumber\\
&  + \lVert (-\Delta)^{-\frac{1}{2}} \mathbb{P}_{\neq 0} ( ( \xi \cdot \nabla) (\xi_{2} \cdot \nabla a_{\xi}^{2} ) \mathbb{P}_{\geq \frac{\lambda_{q+1} \sigma}{2}} (\phi_{\xi}^{2} \varphi_{\xi}^{2} )) \rVert_{C_{t}L_{x}^{p^{\ast}}} \nonumber \\
\overset{\eqref{estimate 241}}{\lesssim}& \sum_{\xi \in \Lambda_{\Xi}} (\lambda_{q+1} \sigma)^{-1} [\lVert \nabla a_{\xi}^{2} \rVert_{C_{t}C_{x}^{2}} + \mu^{-1} \lVert \partial_{t} a_{\xi}^{2} \rVert_{C_{t}C_{x}^{2}} + \lVert \nabla^{2} a_{\xi}^{2} \rVert_{C_{t}C_{x}^{2}}] \lVert \phi_{\xi}^{2} \varphi_{\xi}^{2} \rVert_{C_{t}L_{x}^{p^{\ast}}} \nonumber \\
\overset{\eqref{estimate 150}\eqref{estimate 309}\eqref{estimate 175}}{\lesssim}&  \delta_{q+1} \lambda_{q+1}^{-1} \lambda_{q+1}^{1-2\eta} [ l^{-24} M_{0}(t) + \lambda_{q+1}^{\eta-1} l^{-19} M_{0}(t)] (\lambda_{q+1}^{8\eta -2})^{\frac{1}{p^{\ast}} - 1} \nonumber \\
\overset{\eqref{eta} \eqref{estimate 130}\eqref{estimate 129}}{\lesssim}& M_{0}(t) \delta_{q+2} \lambda_{q+1}^{\frac{ 313 \alpha}{8} - 2 \eta + (8\eta -2) (\frac{1}{p^{\ast}} - 1)} \overset{\eqref{estimate 242}}{\approx} M_{0}(t) \delta_{q+2} \lambda_{q+1}^{- \frac{167 \alpha}{8} - \eta} \ll M_{0}(t) \delta_{q+2}. \nonumber 
\end{align} 
On the other hand, it is more straight-forward to estimate $E_{2}^{\Xi}$ from \eqref{estimate 211} by taking $a \in 2 \mathbb{N}$ sufficiently large 
\begin{align}
&\lVert E_{2}^{\Xi} \rVert_{C_{t}L_{x}^{1}} \lesssim \sum_{\xi\in \Lambda, \xi' \in \Lambda_{\Xi}: \xi \neq \xi'} \lVert a_{\xi} \rVert_{C_{t,x}} \lVert a_{\xi'} \rVert_{C_{t,x}} \lVert \phi_{\xi} \phi_{\xi'} \varphi_{\xi} \varphi_{\xi'} \rVert_{C_{t}L_{x}^{1}} \nonumber \\
\overset{\eqref{estimate 187}\eqref{estimate 150}\eqref{estimate 175}}{\lesssim}& \delta_{q+1} l^{-4} M_{0}(t) \sigma r^{-1} 
\overset{\eqref{estimate 130} \eqref{estimate 129}}{\lesssim} M_{0}(t) \delta_{q+2} \lambda_{q+1}^{\frac{53 \alpha}{8} - 4 \eta} \overset{\eqref{eta}}{\ll} M_{0}(t) \delta_{q+2}. \label{estimate 243}
\end{align}
Due to \eqref{estimate 244}-\eqref{estimate 243} applied to \eqref{estimate 423} we conclude 
\begin{equation}\label{estimate 251}
\lVert R_{\text{osc}}^{\Xi} \rVert_{C_{t}L_{x}^{1}} \overset{\eqref{estimate 423}}{=}\lVert  \mathcal{R}^{\Xi} (\text{div} E_{1}^{\Xi} + \partial_{t} d_{q+1}^{t}) + E_{2}^{\Xi} \rVert_{C_{t}L_{x}^{1}} \ll M_{0}(t) \delta_{q+2}. 
\end{equation} 
Next, from \eqref{estimate 228}, using the fact that due to \eqref{p ast}
\begin{equation}\label{estimate 245}
- 2 \eta + \frac{101 \alpha}{2} + (8\eta -2) (\frac{1}{p^{\ast}} - 1) = - \frac{19\alpha}{2} - \eta, 
\end{equation} 
we can compute for $a \in 2 \mathbb{N}$ sufficiently large 
\begin{align}
& \lVert R_{\text{osc}}^{v} \rVert_{C_{t}L_{x}^{1}} \nonumber \\
\lesssim& \sum_{\xi \in \Lambda} \lVert (-\Delta)^{-\frac{1}{2}} \mathbb{P}_{\neq 0} (\nabla a_{\xi}^{2} \mathbb{P}_{\geq \frac{\lambda_{q+1} \sigma}{2}} (\phi_{\xi}^{2} \varphi_{\xi}^{2} )) \rVert_{C_{t}L_{x}^{p^{\ast}}} + \mu^{-1} \lVert (-\Delta)^{-\frac{1}{2}} \mathbb{P}_{\neq 0} (\partial_{t} a_{\xi}^{2} \mathbb{P}_{\geq \frac{\lambda_{q+1} \sigma}{2}} (\phi_{\xi}^{2} \varphi_{\xi}^{2} )) \rVert_{C_{t}L_{x}^{p^{\ast}}} \nonumber\\
&+ \sum_{\xi \in \Lambda_{\Xi}} \lVert (-\Delta)^{-\frac{1}{2}} \mathbb{P}_{\neq 0} (\nabla a_{\xi}^{2} \mathbb{P}_{\geq \frac{\lambda_{q+1} \sigma}{2}} (\phi_{\xi}^{2} \varphi_{\xi}^{2} ))  \rVert_{C_{t}L_{x}^{p^{\ast}}} + \sum_{\xi, \xi' \in \Lambda: \xi \neq \xi'} \lVert a_{\xi} \rVert_{C_{t,x}} \lVert a_{\xi'} \rVert_{C_{t,x}} \lVert \phi_{\xi} \phi_{\xi'} \varphi_{\xi} \varphi_{\xi'} \rVert_{C_{t}L_{x}^{1}} \nonumber\\
&\overset{\eqref{estimate 150} \eqref{estimate 187} \eqref{estimate 241} \eqref{estimate 175}}{\lesssim} \sum_{\xi \in \Lambda} (\lambda_{q+1} \sigma)^{-1} [ \lVert \nabla a_{\xi}^{2} \rVert_{C_{t}C_{x}^{2}} + \mu^{-1} \lVert \partial_{t} a_{\xi}^{2} \rVert_{C_{t}C_{x}^{2}}] \lVert \phi_{\xi}^{2} \varphi_{\xi}^{2} \rVert_{C_{t}L_{x}^{p^{\ast}}}\nonumber \\
& \hspace{25mm} + \sum_{\xi \in \Lambda_{\Xi}} (\lambda_{q+1} \sigma)^{-1} \lVert \nabla a_{\xi}^{2} \rVert_{C_{t}C_{x}^{2}} \lVert \phi_{\xi}^{2} \varphi_{\xi}^{2} \rVert_{C_{t}L_{x}^{p^{\ast}}} + (\delta_{q+1}^{\frac{1}{2}} l^{-2} M_{0}(t)^{\frac{1}{2}})^{2} \sigma r^{-1}  \nonumber\\
&\overset{\eqref{estimate 150}\eqref{estimate 187} \eqref{estimate 175}}{\lesssim}\lambda_{q+1}^{-2\eta} [ \delta_{q+1}^{\frac{1}{2}} l^{-2} M_{0}(t)^{\frac{1}{2}} \delta_{q+1}^{\frac{1}{2}} l^{-29} M_{0}(t)^{\frac{1}{2}} \nonumber\\
& \hspace{10mm} + \lambda_{q+1}^{\eta -1} \delta_{q+1}^{\frac{1}{2}} l^{-2} M_{0}(t)^{\frac{1}{2}} \delta_{q+1}^{\frac{1}{2}} l^{-29} M_{0}(t)^{\frac{1}{2}}] (r^{\frac{1}{2p^{\ast}} - \frac{1}{2}} \sigma^{\frac{1}{2p^{\ast}} - \frac{1}{2}})^{2} \nonumber\\
&\hspace{10mm} +  \lambda_{q+1}^{-2\eta} \delta_{q+1}^{\frac{1}{2}} l^{-2} M_{0}(t)^{\frac{1}{2}} \delta_{q+1}^{\frac{1}{2}} l^{-27} M_{0}(t)^{\frac{1}{2}} (r^{\frac{1}{p^{\ast}} -\frac{1}{2}} \sigma^{\frac{1}{2p^{\ast}} - \frac{1}{2}})^{2} + \delta_{q+1} l^{-4} M_{0}(t) \lambda_{q+1}^{-4\eta}\nonumber \\
&\hspace{10mm} \overset{\eqref{estimate 130}\eqref{estimate 129}}{\lesssim} M_{0}(t) \delta_{q+2} [\lambda_{q+1}^{-2 \eta + \frac{101\alpha}{2} + (8\eta -2) (\frac{1}{p^{\ast}} -1 )} + \lambda_{q+1}^{\frac{53\alpha}{8} - 4 \eta}] \nonumber \\
& \hspace{10mm} \overset{\eqref{estimate 245}}{\approx} M_{0}(t) \delta_{q+2} [\lambda_{q+1}^{- \frac{19\alpha}{2} - \eta} + \lambda_{q+1}^{\frac{53\alpha}{8} - 4 \eta}]\overset{\eqref{eta}}{\ll} M_{0}(t) \delta_{q+2}. \label{estimate 246} 
\end{align}  
Fourth, we consider $R_{\text{com1}}^{v}$ and $R_{\text{com1}}^{\Xi}$ from \eqref{estimate 235}. Continuing from \eqref{estimate 247}-\eqref{estimate 248} we can estimate using the fact that $\delta \in (0, \frac{1}{12})$ so that $\frac{1}{2} - 2 \delta > \frac{1}{3}$, for $a \in 2 \mathbb{N}$ sufficiently large 
\begin{align}
&\lVert R_{\text{com1}}^{\Xi} \rVert_{C_{t}L_{x}^{1}} + \lVert R_{\text{com1}}^{v} \rVert_{C_{t}L_{x}^{1}}  \nonumber\\
\overset{\eqref{estimate 247} \eqref{estimate 248}}{\lesssim}& l ( \lVert v_{q} \rVert_{C_{t,x}^{1}} + \lVert \Xi_{q} \rVert_{C_{t,x}^{1}} + \sum_{k=1}^{2} \lVert z_{k} \rVert_{C_{t}C_{x}^{1}}) ( \lVert v_{q} \rVert_{C_{t}L_{x}^{2}} + \lVert \Xi_{q} \rVert_{C_{t}L_{x}^{2}} + \sum_{k=1}^{2} \lVert z_{k} \rVert_{C_{t,x}}) \nonumber \\
&+ l^{\frac{1}{2} - 2\delta} \sum_{k=1}^{2} \lVert z_{k} \rVert_{C_{t}^{\frac{1}{2} - 2 \delta} C_{x}} ( \sum_{k=1}^{2} \lVert z_{k} \rVert_{C_{t,x}} + \lVert v_{q} \rVert_{C_{t} L_{x}^{2}} + \lVert \Xi_{q} \rVert_{C_{t}L_{x}^{2}}) \nonumber\\
\overset{\eqref{estimate 95} \eqref{estimate 105}}{\lesssim}& l M_{0}(t) \lambda_{q}^{4} + l^{\frac{1}{3}} M_{0}(t) \overset{\eqref{b} \eqref{l} \eqref{estimate 129}}{\lesssim} M_{0}(t) \delta_{q+2} \lambda_{q+1}^{-\frac{3\alpha}{8} - \frac{2}{3b}} \ll M_{0}(t) \delta_{q+2}. \label{estimate 252}
\end{align}
Fifth, we consider $R_{\text{com2}}^{\Xi}$ from \eqref{estimate 223} and $R_{\text{com2}}^{v}$ from \eqref{estimate 230}. We estimate from \eqref{estimate 223} and \eqref{estimate 230} using again the fact that $\delta \in (0,\frac{1}{12})$ so that $\frac{1}{2} - 2 \delta > \frac{1}{3}$, for $a \in 2 \mathbb{N}$ sufficiently large 
\begin{align}
&\lVert R_{\text{com2}}^{\Xi} \rVert_{C_{t}L_{x}^{1}} + \lVert R_{\text{com2}}^{v} \rVert_{C_{t}L_{x}^{1}} \nonumber\\
&\lesssim  \sum_{k=1}^{2} \lVert z_{k} - z_{k,l} \rVert_{C_{t}L_{x}^{\infty} } ( \lVert v_{q+1} \rVert_{C_{t}L_{x}^{2}} + \lVert \Xi_{q+1} \rVert_{C_{t}L_{x}^{2}} + \lVert z_{1} \rVert_{C_{t}L_{x}^{\infty}} + \lVert z_{2} \rVert_{C_{t}L_{x}^{\infty}} + \lVert z_{1,l} \rVert_{C_{t}L_{x}^{\infty}} + \lVert z_{2,l} \rVert_{C_{t}L_{x}^{\infty}})  \nonumber\\
&\overset{\eqref{estimate 95}\eqref{estimate 99}}{\lesssim}  (l L^{\frac{1}{4}} + l^{\frac{1}{2} - 2 \delta} L^{\frac{1}{2}} ) M_{0}(t)^{\frac{1}{2}} \lesssim l^{\frac{1}{3}} M_{0} (t) \overset{\eqref{estimate 129}}{\lesssim} \delta_{q+2} M_{0}(t) \lambda_{q+1}^{\frac{\alpha}{8} - \frac{\alpha}{2} - \frac{2}{3b} } \ll M_{0}(t)\delta_{q+2}. \label{estimate 253}
\end{align}
We now conclude by applying \eqref{estimate 249}, \eqref{estimate 250}, \eqref{estimate 251}, \eqref{estimate 246}, \eqref{estimate 252}-\eqref{estimate 253} to \eqref{estimate 218} -\eqref{estimate 219}  
\begin{equation}
\lVert \mathring{R}_{q+1}^{v} \rVert_{C_{t}L_{x}^{1}} \leq c_{v} M_{0}(t) \delta_{q+2} \hspace{2mm} \text{ and } \hspace{2mm} \lVert \mathring{R}_{q+1}^{\Xi} \rVert_{C_{t}L_{x}^{1}} \leq c_{\Xi} M_{0}(t) \delta_{q+2}. 
\end{equation} 

At last, similarly to previous works (e.g., \cite{HZZ19}), we conclude by commenting on how $(v_{q+1}, \Xi_{q+1}, \mathring{R}_{q+1}^{v}, \mathring{R}_{q+1}^{\Xi})$ are $(\mathcal{F}_{t})_{t\geq 0}$-adapted and that $(v_{q+1}, \Xi_{q+1}, \mathring{R}_{q+1}^{v}, \mathring{R}_{q+1}^{\Xi} )(0,x)$ are all deterministic if $(v_{q}, \Xi_{q}, \mathring{R}_{q}^{v}, \mathring{R}_{q}^{\Xi} )(0,x)$ are deterministic. First, we recall that $z_{1}$ and $z_{2}$ from \eqref{estimate 75} are both $(\mathcal{F}_{t})_{t\geq 0}$-adapted. Due to the compact support of $\vartheta_{l}$ in $\mathbb{R}_{+}$, it follows that $z_{1,l}$ and $z_{2,l}$ are both $(\mathcal{F}_{t})_{t\geq 0}$-adapted. Similarly, because $(v_{q}, \Xi_{q}, \mathring{R}_{q}^{v}, \mathring{R}_{q}^{\Xi})$ are all $(\mathcal{F}_{t})_{t\geq 0}$-adapted by hypothesis, we see that $(v_{l}, \Xi_{l}, \mathring{R}_{l}^{v}, \mathring{R}_{l}^{\Xi}$) are $(\mathcal{F}_{t})_{t\geq 0}$-adapted. Because $M_{0}(t)$ is deterministic, we see that $\rho_{\Xi}$ in \eqref{estimate 133} is $(\mathcal{F}_{t})_{t\geq 0}$-adapted. Due to $\rho_{\Xi}$ and $\mathring{R}_{l}^{\Xi}$ being $(\mathcal{F}_{t})_{t\geq 0}$-adapted, we see that $a_{\xi}$ for $\xi \in \Lambda_{\Xi}$ from \eqref{estimate 138} are $(\mathcal{F}_{t})_{t\geq 0}$-adapted. This leads to $\mathring{G}^{\Xi}$ from \eqref{estimate 145} being $(\mathcal{F}_{t})_{t\geq 0}$-adapted and consequently so is $\rho_{v}$ from \eqref{estimate 152}. Due to $\rho_{v},\mathring{R}_{l}^{v}$, and $\mathring{G}^{\Xi}$ being $(\mathcal{F}_{t})_{t\geq 0}$-adapted, we deduce that $a_{\xi}$ for $\xi \in \Lambda_{v}$ from \eqref{estimate 153} are also $(\mathcal{F}_{t})_{t\geq 0}$-adapted. Due to $\phi_{\xi}, \varphi_{\xi}$, and $\Psi_{\xi}$ from \eqref{estimate 422} being deterministic, we deduce that $w_{q+1}^{p}, d_{q+1}^{p}, w_{q+1}^{c}, d_{q+1}^{c}, w_{q+1}^{t}$, and $d_{q+1}^{t}$ from \eqref{estimate 166}, \eqref{estimate 330}, and \eqref{estimate 334} are all $(\mathcal{F}_{t})_{t\geq 0}$-adapted. This implies that both $w_{q+1}$ and $d_{q+1}$ from \eqref{estimate 176} are $(\mathcal{F}_{t})_{t\geq 0}$-adapted and consequently so are $v_{q+1}$ and $\Xi_{q+1}$ from \eqref{estimate 203}. It follows that all of $R_{\text{lin}}^{v}, R_{\text{corr}}^{v}, R_{\text{osc}}^{v}, R_{\text{com1}}^{v}, R_{\text{com2}}^{v}, R_{\text{lin}}^{\Xi}, R_{\text{corr}}^{\Xi}, R_{\text{osc}}^{\Xi}, R_{\text{com1}}^{\Xi}$, and $R_{\text{com2}}^{\Xi}$ from \eqref{estimate 423}, \eqref{estimate 235}, and \eqref{estimate 426}, are $(\mathcal{F}_{t})_{t\geq 0}$-adapted and consequently so are $\mathring{R}_{q+1}^{v}$ and $\mathring{R}_{q+1}^{\Xi}$ from \eqref{estimate 427}.  The verification that $(v_{q+1}, \Xi_{q+1}, \mathring{R}_{q+1}^{v}, \mathring{R}_{q+1}^{\Xi} )(0,x)$ are all deterministic is similar and thus omitted. 
 
\section{Proofs of Theorems \ref{Theorem 2.3}-\ref{Theorem 2.4}}\label{Section 5}
Within the following definition of a solution, we recall the definitions of $\tilde{U}_{1}, \tilde{U}_{2}, \bar{\Omega},$ and $\bar{\mathcal{B}}_{t}$ from Section \ref{Subsection 3.1}, 
\begin{define}\label{Definition 5.1}
Fix any $\gamma \in (0,1)$. Let $s \geq 0, \xi^{\text{in}} = (\xi_{1}^{\text{in}}, \xi_{2}^{\text{in}}) \in L_{\sigma}^{2} \times L_{\sigma}^{2}$, and $\theta^{\text{in}} = (\theta_{1}^{\text{in}}, \theta_{2}^{\text{in}}) \in \tilde{U}_{1}\times \tilde{U}_{2}$. Then $P \in \mathcal{P} (\bar{\Omega})$ is a probabilistically weak solution to \eqref{stochastic GMHD} with initial condition $(\xi^{\text{in}}, \theta^{\text{in}})$ at initial time $s$ if 
\begin{enumerate}
\item [] (M1) $P ( \{ \xi(t) = \xi^{\text{in}}, \theta(t) = \theta^{\text{in}} \hspace{1mm} \forall \hspace{1mm} t \in [0,s] \}) = 1$ and for all $l \in \mathbb{N}$, 
\begin{equation}
P ( \{  (\xi, \theta) \in \bar{\Omega}: \int_{0}^{l} \sum_{k=1}^{2} \lVert F_{k} (\xi_{k} (r)) \rVert_{L_{2} (U_{k}, L_{\sigma}^{2})}^{2} dr < \infty \}) = 1, 
\end{equation} 
\item [] (M2) under $P$, $\theta$ is a cylindrical $(\bar{\mathcal{B}}_{t})_{t\geq s}$-Wiener process on $U_{1} \times U_{2}$, starting from initial condition $\theta^{\text{in}}$ at initial time $s$ and for every $\psi_{i} = (\psi_{i}^{1}, \psi_{i}^{2}) \in (C^{\infty} (\mathbb{T}^{3}) \cap L_{\sigma}^{2})^{2}$ and $t\geq s$, 
\begin{subequations}\label{estimate 255} 
\begin{align}
&\langle \xi_{1}(t) - \xi_{1}(s), \psi_{i}^{1} \rangle + \int_{s}^{t} \langle \text{div} (\xi_{1} \otimes \xi_{1} - \xi_{2} \otimes \xi_{2})(r) + (-\Delta)^{m_{1}} \xi_{1}(r), \psi_{i}^{1} \rangle dr \nonumber\\
& \hspace{50mm} = \int_{s}^{t} \langle \psi_{i}^{1}, F_{1}(\xi_{1}(r)) d \theta_{1}(r) \rangle, \\
& \langle \xi_{2}(t) - \xi_{2}(s), \psi_{i}^{2} \rangle + \int_{s}^{t} \langle \text{div} (\xi_{1}  \otimes \xi_{2} - \xi_{2} \otimes \xi_{1})(r) + (-\Delta)^{m_{2}} \xi_{2}(r), \psi_{i}^{2} \rangle dr  \nonumber \\
& \hspace{50mm} = \int_{s}^{t} \langle \psi_{i}^{2}, F_{2} (\xi_{2}(r)) d\theta_{2} (r) \rangle, 
\end{align}
\end{subequations}
\item [] (M3) for any $q \in \mathbb{N}$ there exists a function $t \mapsto C_{t,q} \in \mathbb{R}_{+}$ for all $t\geq s$ such that  
\begin{align}
& \mathbb{E}^{P} [ \sup_{r\in [0,t]} \lVert \xi_{1}(r) \rVert_{L_{x}^{2}}^{2q} + \int_{s}^{t} \lVert \xi_{1}(r) \rVert_{\dot{H}_{x}^{\gamma}}^{2} dr \nonumber\\
&+ \sup_{r \in [0,t]} \lVert \xi_{2}(r) \rVert_{L_{x}^{2}}^{2q} + \int_{s}^{t}  \lVert \xi_{2}(r) \rVert_{\dot{H}_{x}^{\gamma}}^{2} dr]\leq C_{t,q} (1+ \lVert \xi_{1}^{\text{in}} \rVert_{L_{x}^{2}}^{2q} + \lVert \xi_{2}^{\text{in}} \rVert_{L_{x}^{2}}^{2q} ).  \label{estimate 254}
\end{align} 
\end{enumerate}
The set of all such probabilistically weak solutions with the same constant $C_{t,q}$ in \eqref{estimate 254} for every $q \in \mathbb{N}$ and $t\geq s$ is denoted by $\mathcal{W} (s, \xi^{\text{in}}, \theta^{\text{in}}, \{C_{t,q} \}_{q \in \mathbb{N}, t \geq s } )$. 
\end{define}
\begin{define}\label{Definition 5.2}
Fix any $\gamma \in (0,1)$. Let $s \geq 0, \xi^{\text{in}} = (\xi_{1}^{\text{in}}, \xi_{2}^{\text{in}}) \in L_{\sigma}^{2} \times L_{\sigma}^{2},\theta^{\text{in}} = (\theta_{1}^{\text{in}}, \theta_{2}^{\text{in}}) \in \tilde{U}_{1}\times \tilde{U}_{2}$, and $\tau \geq s$ be a stopping time of $(\bar{\mathcal{B}}_{t})_{t\geq s}$. We define 
\begin{equation}
\bar{\Omega}_{\tau} \triangleq \{\omega( \cdot \wedge \tau(\omega)): \omega \in \bar{\Omega} \}
\end{equation} 
and $(\bar{\mathcal{B}}_{\tau})$ to be the $\sigma$-field associated to $\tau$. Then $P \in \mathcal{P} (\bar{\Omega}_{\tau})$ is a probabilistically weak solution to \eqref{stochastic GMHD} on $[s, \tau]$ with initial condition $(\xi^{\text{in}}, \theta^{\text{in}})$ at initial time $s$ if 
\begin{enumerate}
\item [] (M1)  $P ( \{ \xi(t) = \xi^{\text{in}}, \theta(t) = \theta^{\text{in}} \hspace{1mm} \forall \hspace{1mm} t \in [0,s] \}) = 1$ and for all $l \in \mathbb{N}$ 
\begin{equation}
P ( \{ (\xi, \theta) \in \bar{\Omega}: \int_{0}^{l \wedge \tau} \sum_{k=1}^{2} \lVert F_{k} (\xi_{k} (r)) \rVert_{L_{2} (U_{k}, L_{\sigma}^{2})}^{2} dr < \infty \}) = 1, 
\end{equation} 
\item [] (M2) under $P$, $\langle \theta(\cdot \wedge \tau), l_{i} \rangle_{U_{1} \times U_{2}}$ where $\{l_{i}\}_{i\in\mathbb{N}} = \{ (l_{i}^{1}, l_{i}^{2}) \}_{i \in \mathbb{N}}$ is an orthonormal basis of $U_{1} \times U_{2}$, is a continuous, square-integrable $(\bar{\mathcal{B}}_{t})_{t\geq s}$-martingale with initial condition $\langle \theta^{\text{in}}, l_{i} \rangle$ at initial time $s$ with its quadratic variation given by $\sum_{k=1}^{2} (t \wedge \tau - s) \lVert l_{i}^{k} \rVert_{U_{k}}^{2}$ and for every $\psi_{i} = (\psi_{i}^{1}, \psi_{i}^{2}) \in (C^{\infty} (\mathbb{T}^{3}) \cap L_{\sigma}^{2})^{2}$ and $t\geq s$, 
\begin{subequations}\label{estimate 256}
\begin{align}
& \langle \xi_{1} (t\wedge \tau) - \xi_{1}(s), \psi_{i}^{1} \rangle + \int_{s}^{t \wedge \tau} \langle \text{div} ( \xi_{1} \otimes \xi_{1} - \xi_{2} \otimes \xi_{2})(r) + (-\Delta)^{m_{1}} \xi_{1} (r), \psi_{i}^{1} \rangle dr \nonumber\\
& \hspace{50mm} = \int_{s}^{t \wedge \tau} \langle \psi_{i}^{1}, F_{1} (\xi_{1}(r)) d\theta_{1} (r) \rangle, \\
& \langle \xi_{2} (t \wedge \tau) - \xi_{2} (s), \psi_{i}^{2} \rangle + \int_{s}^{t \wedge \tau}\langle \text{div} (\xi_{1}\otimes \xi_{2} - \xi_{2} \otimes \xi_{1})(r) + (-\Delta)^{m_{2}} \xi_{2} (r), \psi_{i}^{2} \rangle dr \nonumber\\
& \hspace{50mm} = \int_{s}^{t\wedge \tau} \langle \psi_{i}^{2}, F_{2} (\xi_{2} (r)) d\theta_{2} (r) \rangle, 
\end{align}
\end{subequations}
\item [](M3) for any $q \in \mathbb{N}$ there exists a function $t \mapsto C_{t,q} \in \mathbb{R}_{+}$ for all $t\geq s$ such that  
\begin{align}
& \mathbb{E}^{P} [ \sup_{r \in [0, t \wedge \tau]} \lVert \xi_{1}(r) \rVert_{L_{x}^{2}}^{2q} + \int_{s}^{t \wedge \tau} \lVert \xi_{1}(r) \rVert_{\dot{H}_{x}^{\gamma}}^{2} dr  \nonumber\\
&+ \sup_{r \in [0, t\wedge \tau]} \lVert \xi_{2} (r) \rVert_{L_{x}^{2}}^{2q} + \int_{s}^{t \wedge \tau} \lVert \xi_{2} (r) \rVert_{\dot{H}_{x}^{\gamma}}^{2} dr] \leq C_{t,q} (1+ \lVert \xi_{1}^{\text{in}} \rVert_{L_{x}^{2}}^{2q} + \lVert \xi_{2}^{\text{in}} \rVert_{L_{x}^{2}}^{2q}). 
\end{align}
\end{enumerate} 
\end{define} 
Our first result concerns existence and stability of solution to \eqref{stochastic GMHD} and follows from Proposition \ref{Proposition 4.1} similarly to previous works (e.g., \cite[The. 5.1]{HZZ19} and \cite[Pro. 5.1]{Y21a}). 
\begin{proposition}\label{Proposition 5.1}
For every $(s, \xi^{\text{in}}, \theta^{\text{in}}) \in [0,\infty) \times (L_{\sigma}^{2})^{2} \times (\tilde{U}_{1} \times \tilde{U}_{2})$, there exists a probabilistically weak solution $P \in \mathcal{P} (\bar{\Omega})$ to \eqref{stochastic GMHD} with initial condition $(\xi^{\text{in}},\theta^{\text{in}})$ at initial time $s$ according to Definition \ref{Definition 5.1}. Moreover, if there exists a family $\{(s_{l}, \xi_{l}, \theta_{l} ) \}_{l \in \mathbb{N}} \subset [0,\infty) \times (L_{\sigma}^{2})^{2} \times (\tilde{U}_{1} \times \tilde{U}_{2})$ such that $\lim_{l\to\infty} \lVert (s_{l}, \xi_{l},\theta_{l}) - (s, \xi^{\text{in}},\theta^{\text{in}}) \rVert_{\mathbb{R} \times (L_{\sigma}^{2})^{2} \times (\tilde{U}_{1} \times \tilde{U}_{2})} = 0$ and $P_{l} \in \mathcal{W} (s_{l}, \xi_{l}, \theta_{l}, \{C_{t,q} \}_{q\in \mathbb{N}, t \geq s_{l}} )$, then there exists a subsequence $\{P_{l_{k}} \}_{k \in \mathbb{N}}$ that converges weakly to some $P \in \mathcal{W} (s, \xi^{\text{in}}, \theta^{\text{in}}, \{C_{t,q} \}_{q \in \mathbb{N}, t \geq s } )$. 
\end{proposition}

Proposition \ref{Proposition 5.1} leads to the following results due to \cite[Pro. 5.2 and 5.3]{HZZ19}: 
\begin{lemma}\label{Lemma 5.2}
\rm{(\cite[Pro. 5.2]{HZZ19})} Let $\tau$ be a bounded $(\bar{\mathcal{B}}_{t})_{t \geq 0}$-stopping time. Then, for every $\omega \in \bar{\Omega}$, there exists $Q_{\omega} \in \mathcal{P}(\bar{\Omega})$ such that 
\begin{subequations}
\begin{align}
& Q_{\omega} ( \{ \omega' \in \bar{\Omega}: \hspace{0.5mm}  ( \xi, \theta) (t, \omega') = (\xi, \theta) (t,\omega) \hspace{1mm} \forall \hspace{1mm} t \in [0, \tau(\omega)] \}) = 1, \\
& Q_{\omega} (A) = R_{\tau(\omega), \xi(\tau(\omega), \omega), \theta(\tau(\omega), \omega)} (A) \hspace{1mm} \forall \hspace{1mm} A \in \mathcal{B}^{\tau(\omega)}, 
\end{align} 
\end{subequations}
where $R_{\tau(\omega), \xi(\tau(\omega), \omega), \theta(\tau(\omega), \omega)} \in \mathcal{P} (\bar{\Omega})$ is a probabilistically weak solution to \eqref{stochastic GMHD} with initial condition $(\xi(\tau(\omega), \omega), \theta(\tau(\omega), \omega))$ at initial time $\tau(\omega)$. Moreover, for every $A \in \bar{\mathcal{B}}$, the mapping $\omega \mapsto Q_{\omega}(A)$ is $\bar{\mathcal{B}}_{\tau}$-measurable, where $\bar{\mathcal{B}}$ is the Borel $\sigma$-algebra of $\bar{\Omega}$ from Section \ref{Subsection 3.1}.  
\end{lemma} 

\begin{lemma}\label{Lemma 5.3}
\rm{(\cite[Pro. 5.3]{HZZ19})} Let $\tau$ be a bounded $(\bar{\mathcal{B}}_{t})_{t\geq 0}$-stopping time, $\xi^{\text{in}} = (\xi_{1}^{\text{in}}, \xi_{2}^{\text{in}}) \in L_{\sigma}^{2} \times L_{\sigma}^{2}$, and $P$ be a probabilistically weak solution to \eqref{stochastic GMHD} on $[0,\tau]$ with initial condition $(\xi^{\text{in}}, 0)$ at initial time $0$ according to Definition \ref{Definition 5.2}. Suppose that there exists a Borel set $\mathcal{N} \subset \bar{\Omega}_{\tau}$ such that $P(\mathcal{N}) = 0$ and $Q_{\omega}$ from Lemma \ref{Lemma 5.2} satisfies for every $\omega \in \bar{\Omega}_{\tau} \setminus \mathcal{N}$ 
\begin{equation}
Q_{\omega} (\{ \omega' \in \bar{\Omega}: \hspace{0.5mm}  \tau(\omega') = \tau(\omega) \}) = 1. 
\end{equation} 
Then the probability measure $P\otimes_{\tau}R \in \mathcal{P} (\bar{\Omega})$ defined by 
\begin{equation}\label{estimate 259}
P \otimes_{\tau} R (\cdot) \triangleq \int_{\bar{\Omega}} Q_{\omega} (\cdot) P(d \omega)
\end{equation} 
satisfies $P\otimes_{\tau}R \rvert_{\bar{\Omega}_{\tau}} = P \rvert_{\bar{\Omega}_{\tau}}$ and it is a probabilistically weak solution to \eqref{stochastic GMHD} on $[0,\infty)$ with initial condition $(\xi^{\text{in}}, 0)$ at initial time $0$. 
\end{lemma} 
Now we fix the $\mathbb{R}$-valued Wiener processes $B_{1}$ and $B_{2}$ on $(\Omega, \mathcal{F}, \textbf{P})$ with $(\mathcal{F}_{t})_{t\geq 0}$ as its normal filtration. For $l \in \mathbb{N}, L > 1$, and $\delta \in (0, \frac{1}{12})$ we define 
\begin{align}
&\tau_{L}^{l} (\omega) \triangleq \inf\{t \geq 0: \max_{k=1,2} \lvert \theta_{k} (t,\omega)  \rvert > (L - \frac{1}{l})^{\frac{1}{4}} \}  \nonumber\\
& \hspace{5mm} \wedge \inf\{ t \geq 0: \max_{k=1,2} \lVert \theta_{k} (\omega) \rVert_{C_{t}^{\frac{1}{2} - 2 \delta}} > (L - \frac{1}{l})^{\frac{1}{2}} \wedge L \hspace{1mm} \text{ and } \hspace{1mm} \tau_{L}(\omega) \triangleq \lim_{l\to\infty} \tau_{L}^{l} (\omega). \label{estimate 257}
\end{align}
Comparing \eqref{stochastic GMHD} with \eqref{estimate 255} and \eqref{estimate 256} we see that $F_{k} (\xi_{k}(r)) = \xi_{k}(r)$, $\theta_{k} = B_{k}$ for $k \in \{1,2\}$. As a Brownian path is locally H$\ddot{\mathrm{o}}$lder-continuous with exponent $\alpha \in (0, \frac{1}{2})$, it follows from \cite[Lem. 3.5]{HZZ19} that $\tau_{L}$ is a stopping time of $(\bar{\mathcal{B}}_{t})_{t\geq 0}$. We assume Theorem \ref{Theorem 2.3} and denote by $(u,b)$ the solution constructed therein over time interval $[0, \mathfrak{t}]$ where $\mathfrak{t} = T_{L}$ for $L$ sufficiently large and 
\begin{equation}\label{estimate 258}
T_{L} \triangleq \inf\{t> 0: \max_{k=1,2} \lvert B_{k} (t) \rvert \geq L^{\frac{1}{4}} \} \wedge \inf\{t > 0: \max_{k=1,2} \lVert B_{k} \rVert_{C_{t}^{\frac{1}{2} - 2 \delta}} \geq L^{\frac{1}{2}} \} \wedge L. 
\end{equation} 
With $P$ representing the law of $(u, b, B_{1}, B_{2})$, the following two results also follow immediately from previous works (e.g., \cite[Pro. 5.4 and 5.5]{HZZ19}) making use of the fact that 
\begin{equation}\label{estimate 260}
\theta (t, (u,b, B_{1}, B_{2})) = (B_{1}, B_{2} )(t) \hspace{3mm} \forall \hspace{1mm} t \in [0, T_{L}] \hspace{1mm} \textbf{P}\text{-a.s.}
\end{equation} 

\begin{proposition}\label{Proposition 5.4}
Let $\tau_{L}$ be defined by \eqref{estimate 257}. Then $P = \mathcal{L} (u, b, B_{1}, B_{2})$ is a probabilistically weak solution to \eqref{stochastic GMHD} on $[0, \tau_{L}]$ that satisfies Definition \ref{Definition 5.2}. 
\end{proposition}

\begin{proposition}\label{Proposition 5.5}
Let $\tau_{L}$ be defined by \eqref{estimate 257}. Then $P \otimes_{\tau_{L}} R$ defined by \eqref{estimate 259} is a probabilistically weak solution to \eqref{stochastic GMHD} on $[0,\infty)$ that satisfies Definition \ref{Definition 5.1}. 
\end{proposition}

We are ready to prove Theorem \ref{Theorem 2.4} assuming Theorem \ref{Theorem 2.3}. In fact, this proof is similar to the proof of Theorem \ref{Theorem 2.2} in Section \ref{Section 4}; thus, we leave this in the Appendix. Now we know that $(v, \Xi)$ defined via $\Upsilon_{1}$ and $\Upsilon_{2}$ in \eqref{estimate 40} satisfy \eqref{estimate 22}-\eqref{estimate 23}. This motivates us to pursue $(v_{q}, \Xi_{q}, \mathring{R}_{q}^{v}, \mathring{R}_{q}^{\Xi})$ for $q \in \mathbb{N}_{0}$ that solves 
\begin{subequations}\label{estimate 277}
\begin{align}
&\partial_{t} v_{q} + \frac{1}{2} v_{q} + (-\Delta)^{m_{1}} v_{q} + \text{div} ( \Upsilon_{1} (v_{q} \otimes v_{q}) - \Upsilon_{1}^{-1} \Upsilon_{2}^{2} (\Xi_{q} \otimes \Xi_{q} ))  + \nabla p_{q} = \text{div} \mathring{R}_{q}^{v},\\
&\partial_{t} \Xi_{q} + \frac{1}{2}\Xi_{q} + (-\Delta)^{m_{2}} \Xi_{q} + \Upsilon_{1} \text{div} (v_{q} \otimes \Xi_{q} - \Xi_{q} \otimes v_{q}) = \text{div} \mathring{R}_{q}^{\Xi},  \\
& \nabla\cdot v_{q}  = 0, \hspace{1mm} \nabla\cdot \Xi_{q} = 0,
\end{align}
\end{subequations} 
where $\mathring{R}_{q}^{v}$ is a symmetric trace-free matrix and $\mathring{R}_{q}^{\Xi}$ is a skew-symmetric matrix. We define $\lambda_{q}$ and $\delta_{q}$ identically to \eqref{estimate 93} but introduce a different definition of $M_{0}(t)$ from \eqref{estimate 94} and a new quantity $m_{L}$: 
\begin{equation}\label{estimate 262}
M_{0}(t) \triangleq e^{4L t + 2L} \hspace{2mm} \text{ and } \hspace{2mm} m_{L} \triangleq \sqrt{3} L^{\frac{5}{4}} e^{\frac{5}{2} L^{\frac{1}{4}}}. 
\end{equation} 
Because for $L > 1$ and $t \in [0, T_{L}], \lvert B_{k} (t) \rvert \leq L^{\frac{1}{4}}$ and $\lVert B_{k} \rVert_{C_{t}^{\frac{1}{2} - 2\delta}} \leq L^{\frac{1}{2}}$ due to \eqref{estimate 258}, we have 
\begin{equation}\label{estimate 273}
\lVert \Upsilon_{k} \rVert_{C_{t}^{\frac{1}{2} - 2 \delta}}, \lVert \Upsilon_{k}^{-1} \rVert_{C_{t}^{\frac{1}{2} - 2 \delta}} \leq e^{L^{\frac{1}{4}}} L^{\frac{1}{2}} \leq m_{L}^{\frac{2}{5}} \hspace{1mm} \text{ and } \hspace{1mm} \lVert \Upsilon_{k} \rVert_{C_{t}}, \lVert\Upsilon_{k}^{-1} \rVert_{C_{t}} \leq e^{L^{\frac{1}{4}}} \leq m_{L}^{\frac{2}{5}}. 
\end{equation} 
We assume again that $b \geq 2$ and $a^{\beta b} > (2\pi)^{3}  +1$ as in \eqref{estimate 96} so that $\sum_{1 \leq \iota \leq q} \delta_{\iota}^{\frac{1}{2}} < \frac{1}{2}$ for all $q \in \mathbb{N}$. For the inductive estimates we assume for all $t \in [0, T_{L}]$, 
\begin{subequations}\label{estimate 272}
\begin{align}
& \lVert v_{q} \rVert_{C_{t}L_{x}^{2}} \leq m_{L} M_{0}(t)^{\frac{1}{2}} (1+ \sum_{1 \leq \iota \leq q} \delta_{\iota}^{\frac{1}{2}}) \leq 2 m_{L}M_{0}(t)^{\frac{1}{2}}, \label{estimate 268}\\
& \lVert \Xi_{q} \rVert_{C_{t}L_{x}^{2}} \leq m_{L} M_{0}(t)^{\frac{1}{2}} (1+ \sum_{1 \leq \iota \leq q} \delta_{\iota}^{\frac{1}{2}}) \leq 2 m_{L} M_{0}(t)^{\frac{1}{2}}, \label{estimate 269}\\
& \lVert v_{q} \rVert_{C_{t,x}^{1}} \leq m_{L} M_{0}(t)^{\frac{1}{2}} \lambda_{q}^{4}, \hspace{4mm} \lVert \Xi_{q} \rVert_{C_{t,x}^{1}} \leq m_{L}M_{0}(t)^{\frac{1}{2}} \lambda_{q}^{4}, \label{estimate 270}\\
& \lVert \mathring{R}_{q}^{v} \rVert_{C_{t}L_{x}^{1}} \leq c_{v} M_{0}(t) \delta_{q+1}, \hspace{3mm} \lVert \mathring{R}_{q}^{\Xi} \rVert_{C_{t}L_{x}^{1}} \leq c_{\Xi} M_{0}(t) \delta_{q+1}, \label{estimate 271}
\end{align}
\end{subequations} 
where the second inequalities in \eqref{estimate 268}-\eqref{estimate 269} hold again due to \eqref{estimate 96}.
\begin{remark}\label{Remark 5.1}
The definition of $m_{L}$ in \eqref{estimate 262} is different from those of previous works, e.g., $m_{L} = \sqrt{3} L^{\frac{1}{4}} e^{\frac{1}{2} L^{\frac{1}{4}}}$ in \cite[Equ. (6.5)]{HZZ19}. The reason for this important change is as follows. As we previewed in Remark \ref{Remark 2.1}, although Hofmanov$\acute{\mathrm{a}}$ et al. in \cite{HZZ19} were able to strategically define $\bar{a}_{\xi}$ and thereby $w_{q+1}^{p}$ to reduce the oscillation term in the proof of Theorem \ref{Theorem 2.3} to that of Theorem \ref{Theorem 2.1}, this is impossible for the case of the MHD system. From \eqref{estimate 24}-\eqref{estimate 25} we can see that the oscillation term from the magnetic field equation is more balanced; i.e., both $w_{q+1}^{p} \otimes d_{q+1}^{p}$ and $d_{q+1}^{p} \otimes w_{q+1}^{p}$ are multiplied by $\Upsilon_{1,l}$. Thus, we focus on \eqref{estimate 24}. Let us assume for some $\bar{a}_{\xi}$ to be determined that 
\begin{equation}\label{estimate 428}
w_{q+1}^{p} \triangleq \sum_{\xi \in \Lambda} \bar{a}_{\xi} \phi_{\xi} \varphi_{\xi} \xi \hspace{2mm} \text{ and } \hspace{2mm} d_{q+1}^{p} \triangleq \sum_{\xi \in \Lambda_{\Xi}} \bar{a}_{\xi} \phi_{\xi} \varphi_{\xi} \xi_{2}
\end{equation}
and follow the analogous computations of \eqref{estimate 263}-\eqref{estimate 264}. We see that 
\begin{align}
\text{div} R_{\text{osc}}^{v} +& \nabla p_{\text{osc}} \overset{\eqref{estimate 24}\eqref{estimate 428}}{=} \text{div} ( \Upsilon_{1,l} [ \sum_{\xi \in \Lambda} \bar{a}_{\xi}^{2} \phi_{\xi}^{2} \varphi_{\xi}^{2} \xi \otimes \xi + \sum_{\xi, \xi' \in \Lambda: \xi \neq \xi'} \bar{a}_{\xi} \bar{a}_{\xi'} \phi_{\xi} \phi_{\xi'} \varphi_{\xi} \varphi_{\xi'} \xi \otimes \xi' ] \nonumber\\
& - \Upsilon_{1,l}^{-1} \Upsilon_{2,l}^{2} [ \sum_{\xi \in \Lambda_{\Xi}} \bar{a}_{\xi}^{2} \phi_{\xi}^{2} \varphi_{\xi}^{2} \xi_{2} \otimes \xi_{2} + \sum_{\xi, \xi' \in \Lambda_{\Xi}: \xi \neq \xi'} \bar{a}_{\xi} \bar{a}_{\xi'} \phi_{\xi} \phi_{\xi'} \varphi_{\xi} \varphi_{\xi'} \xi_{2} \otimes \xi_{2}'] + \mathring{R}_{l}^{v}) + \partial_{t} w_{q+1}^{t}  \nonumber\\
&\overset{\eqref{estimate 46}\eqref{estimate 57}}{=} \text{div} (\Upsilon_{1,l} \sum_{\xi \in \Lambda_{v}} \bar{a}_{\xi}^{2} \phi_{\xi}^{2} \varphi_{\xi}^{2} (\xi \otimes \xi) + \sum_{\xi \in \Lambda_{\Xi}} \bar{a}_{\xi}^{2} (\Upsilon_{1,l} \xi \otimes \xi - \Upsilon_{1,l}^{-1} \Upsilon_{2,l}^{2} \xi_{2} \otimes \xi_{2}) \nonumber \\
&+ \sum_{\xi \in \Lambda_{\Xi}} \bar{a}_{\xi}^{2} \mathbb{P}_{\neq 0} (\phi_{\xi}^{2} \varphi_{\xi}^{2}) (\Upsilon_{1,l} \xi \otimes \xi - \Upsilon_{1,l}^{-1} \Upsilon_{2,l}^{2} \xi_{2} \otimes \xi_{2}) \nonumber  \\
&+ \Upsilon_{1,l} \sum_{\xi, \xi' \in \Lambda: \xi \neq \xi'} \bar{a}_{\xi} \bar{a}_{\xi'} \phi_{\xi} \phi_{\xi'} \varphi_{\xi} \varphi_{\xi'} (\xi \otimes \xi') \nonumber\\
&- \Upsilon_{1,l}^{-1} \Upsilon_{2,l}^{2} \sum_{\xi, \xi' \in \Lambda_{\Xi}: \xi \neq \xi'} \bar{a}_{\xi} \bar{a}_{\xi'} \phi_{\xi} \phi_{\xi'} \varphi_{\xi} \varphi_{\xi'} (\xi_{2} \otimes \xi_{2}') + \mathring{R}_{l}^{v}) + \partial_{t} w_{q+1}^{t}. \label{estimate 265} 
\end{align}
In comparison with \eqref{estimate 263}, this motivates us to define differently from \eqref{estimate 145}
\begin{equation}\label{estimate 274}
\mathring{G}^{\Xi} \triangleq \sum_{\xi \in \Lambda_{\Xi}} \bar{a}_{\xi}^{2} (\Upsilon_{1,l} \xi \otimes \xi - \Upsilon_{1,l}^{-1} \Upsilon_{2,l}^{2} \xi_{2} \otimes \xi_{2}) 
\end{equation} 
with $\bar{a}_{\xi}$ for $\xi \in \Lambda_{\Xi}$ still to be determined so that we obtain from \eqref{estimate 265}
\begin{align}
&\text{div} R_{\text{osc}}^{v} + \nabla p_{\text{osc}} \label{estimate 266} \\
=& \text{div} (\Upsilon_{1,l} \sum_{\xi \in \Lambda_{v}} \bar{a}_{\xi}^{2} \phi_{\xi}^{2} \varphi_{\xi}^{2} (\xi \otimes \xi) + \mathring{R}_{l}^{v} + \mathring{G}^{\Xi}) + \text{div} (\sum_{\xi \in \Lambda_{\Xi}} \bar{a}_{\xi}^{2} \mathbb{P}_{\neq 0} (\phi_{\xi}^{2} \varphi_{\xi}^{2}) (\Upsilon_{1,l} \xi \otimes \xi - \Upsilon_{1,l}^{-1} \Upsilon_{2,l}^{2} \xi_{2} \otimes \xi_{2})) \nonumber \\
&+ \text{div} ( \Upsilon_{1,l} \sum_{\xi,\xi' \in \Lambda: \xi \neq \xi'} \bar{a}_{\xi} \bar{a}_{\xi'} \phi_{\xi} \phi_{\xi'} \varphi_{\xi} \varphi_{\xi'} (\xi \otimes \xi')  \nonumber\\
& \hspace{20mm}  - \Upsilon_{1,l}^{-1} \Upsilon_{2,l}^{2} \sum_{\xi, \xi' \in \Lambda_{\Xi}: \xi \neq \xi'} \bar{a}_{\xi} \bar{a}_{\xi'} \phi_{\xi} \phi_{\xi'} \varphi_{\xi} \varphi_{\xi'} (\xi_{2} \otimes \xi_{2}')) + \partial_{t} w_{q+1}^{t}. \nonumber 
\end{align}
In comparison with \eqref{estimate 264} we now realize that if we define the modified velocity amplitude function as 
\begin{equation}\label{estimate 306}
\bar{a}_{\xi} (t,x) \triangleq \Upsilon_{1,l}^{-\frac{1}{2}}(t) a_{\xi}(t,x) \overset{\eqref{estimate 153} }{=} \Upsilon_{1,l}^{-\frac{1}{2}}(t)\rho_{v}^{\frac{1}{2}} (t,x) \gamma_{\xi} ( \text{Id} - \frac{ \mathring{R}_{l}^{v}(t,x) + \mathring{G}^{\Xi} (t,x)}{\rho_{v} (t,x)})  \hspace{2mm} \forall \hspace{1mm} \xi \in \Lambda_{v}
\end{equation}
for $\mathring{G}^{\Xi}$ in \eqref{estimate 274} and $\rho_{v}$ defined identically to \eqref{estimate 152} so that $\Upsilon_{1,l}\bar{a}_{\xi}^{2}= a_{\xi}^{2}$, then we can apply \eqref{estimate 216} to deduce from \eqref{estimate 266}
\begin{align}
&\text{div} R_{\text{osc}}^{v} + \nabla p_{\text{osc}} \label{estimate 267} \\
=& \nabla \rho_{v} + \text{div} ( \sum_{\xi \in \Lambda_{v}} a_{\xi}^{2} \mathbb{P}_{\neq 0} (\phi_{\xi}^{2} \varphi_{\xi}^{2}) (\xi \otimes \xi) ) + \text{div} (\sum_{\xi \in \Lambda_{\Xi}} \bar{a}_{\xi}^{2} \mathbb{P}_{\neq 0} (\phi_{\xi}^{2}\varphi_{\xi}^{2}) (\Upsilon_{1,l} \xi \otimes \xi - \Upsilon_{1,l}^{-1}\Upsilon_{2,l}^{2} \xi_{2} \otimes \xi_{2})) \nonumber\\
&+ \text{div} ( \Upsilon_{1,l} \sum_{\xi,\xi' \in \Lambda: \xi \neq \xi'} \bar{a}_{\xi} \bar{a}_{\xi'} \phi_{\xi} \phi_{\xi'} \varphi_{\xi} \varphi_{\xi'} (\xi \otimes \xi')  \nonumber\\
& \hspace{20mm}  - \Upsilon_{1,l}^{-1} \Upsilon_{2,l}^{2} \sum_{\xi, \xi' \in \Lambda_{\Xi}: \xi \neq \xi'} \bar{a}_{\xi} \bar{a}_{\xi'} \phi_{\xi} \phi_{\xi'} \varphi_{\xi} \varphi_{\xi'} (\xi_{2} \otimes \xi_{2}')) + \partial_{t} w_{q+1}^{t}. \nonumber 
\end{align}
In another comparison with \eqref{estimate 264} we now realize that it is beneficial to define the modified magnetic amplitude function as  
\begin{equation}\label{estimate 362}
\bar{a}_{\xi}(t,x) \triangleq \Upsilon_{1,l}^{-\frac{1}{2}}(t) a_{\xi} (t,x) \overset{\eqref{estimate 138}}{=} \Upsilon_{1,l}^{-\frac{1}{2}}(t) \rho_{\Xi}^{\frac{1}{2}}(t,x) \gamma_{\xi} (- \frac{\mathring{R}_{l}^{\Xi}(t,x) }{\rho_{\Xi}(t,x)} ) \hspace{2mm} \forall \hspace{1mm} \xi \in \Lambda_{\Xi}
\end{equation} 
for $\rho_{\Xi}$ defined identically to \eqref{estimate 133} so that $\bar{a}_{\xi}^{2} \Upsilon_{1,l} = a_{\xi}^{2}$ which will allow us to add together $\text{div}(\sum_{\xi \in \Lambda_{v}} a_{\xi}^{2} \mathbb{P}_{\neq 0} (\phi_{\xi}^{2} \varphi_{\xi}^{2}) (\xi \otimes \xi))$ with $\text{div} \sum_{\xi \in \Lambda_{\Xi}} \bar{a}_{\xi}^{2} \mathbb{P}_{\neq 0} (\phi_{\xi}^{2} \varphi_{\xi}^{2}) (\Upsilon_{1,l} \xi \otimes \xi)$ in \eqref{estimate 267} to deduce 
\begin{align}
&\text{div} R_{\text{osc}}^{v} + \nabla p_{\text{osc}} \nonumber\\
=& \nabla \rho_{v} + \text{div} (\sum_{\xi \in \Lambda} a_{\xi}^{2} \mathbb{P}_{\neq 0} (\phi_{\xi}^{2} \varphi_{\xi}^{2}) (\xi\otimes \xi)) - \text{div} (\sum_{\xi \in \Lambda_{\Xi}} a_{\xi}^{2} \mathbb{P}_{\neq 0} (\phi_{\xi}^{2} \varphi_{\xi}^{2}) \Upsilon_{1,l}^{-2} \Upsilon_{2,l}^{2} (\xi_{2}\otimes \xi_{2} )) \nonumber \\
&+ \text{div} (\sum_{\xi, \xi' \in \Lambda: \xi \neq \xi'} a_{\xi} a_{\xi'} \phi_{\xi} \phi_{\xi'} \varphi_{\xi} \varphi_{\xi'} (\xi\otimes \xi') \nonumber\\
& \hspace{10mm} - \Upsilon_{1,l}^{-2} \Upsilon_{2,l}^{2} \sum_{\xi, \xi' \in \Lambda_{\Xi}: \xi \neq \xi'} a_{\xi} a_{\xi'} \phi_{\xi} \phi_{\xi'} \varphi_{\xi} \varphi_{\xi'} (\xi_{2} \otimes \xi_{2} ')) + \partial_{t} w_{q+1}^{t}. \label{estimate 364}
\end{align}
Therefore, at this point we have found the necessary choices of $\mathring{G}^{\Xi}, \bar{a}_{\xi}$ for both $\xi \in \Lambda_{v}$ and $\Lambda_{\Xi}$, along with those of $\rho_{v}$ and $\rho_{\Xi}$, although they still do not succeed in reducing to the oscillation terms from the proof of Theorem \ref{Theorem 2.1} in contrast to previous works. Unfortunately, 
this choice will still cause a problem if we did not modify the definition of $m_{L}$ in \eqref{estimate 262}. Indeed, if we kept the same definition of $m_{L} = \sqrt{3} L^{\frac{1}{4}} e^{\frac{1}{2} L^{\frac{1}{4}}}$ from previous works, then we would have $\lVert \Upsilon_{k} \rVert_{C_{t}}, \lVert \Upsilon_{k}^{-1} \rVert_{C_{t}} \leq m_{L}^{2}$ instead of $m_{L}^{\frac{2}{5}}$ in \eqref{estimate 273}. In turn, this implies that ``$\Upsilon_{1,l}^{-1} \Upsilon_{2,l}^{2}$'' in \eqref{estimate 274} can be only bounded by $m_{L}^{6}$. Observing \eqref{estimate 169}, this implies that the estimate of $\lVert \bar{a}_{\xi} \rVert_{C_{t}L_{x}^{2}}$ for $\xi \in \Lambda_{v}$ will have an extra factor of at least $m_{L}^{3}$. Observing \eqref{estimate 173} this time, we see that this implies that the estimate of $\lVert \bar{a}_{\xi} \phi_{\xi} \varphi_{\xi} \rVert_{C_{t}L_{x}^{2}}$ for $\xi \in \Lambda_{v}$ will also have an extra factor of at least $m_{L}^{3}$; consequently, observing \eqref{estimate 178} we see that the estimate of $\lVert w_{q+1}^{p} \rVert_{C_{t}L_{x}^{2}}$ will have an extra factor of at least $m_{L}^{3}$. Therefore, through estimate that are analogous to \eqref{estimate 202} and \eqref{estimate 275}, we will not be able to verify the inductive bound \eqref{estimate 268} which only has a factor of $m_{L}$. This is just one example of many issues that will arise unless we revise the definition of $m_{L}$ from previous works to that in \eqref{estimate 262}.
\end{remark} 

\begin{proposition}\label{Proposition 5.6}
Let 
\begin{equation}\label{estimate 429}
v_{0}(t,x) \triangleq \frac{m_{L} e^{2L t + L}}{(2\pi)^{\frac{3}{2}}} 
\begin{pmatrix}
\sin(x^{3}) \\
0\\
0
\end{pmatrix}
\hspace{2mm} \text{ and } \hspace{2mm} 
\Xi_{0} (t,x) \triangleq 
\frac{ m_{L} e^{2L t + L}}{(2\pi)^{3}} 
\begin{pmatrix}
\sin(x^{3}) \\
\cos(x^{3}) \\
0
\end{pmatrix}.
\end{equation}
Then, together with 
\begin{equation}\label{estimate 430}
\mathring{R}_{0}^{v} (t,x) \triangleq \frac{ m_{L }( 2L + \frac{1}{2}) e^{2L t+ L}}{(2\pi)^{\frac{3}{2}}} 
\begin{pmatrix}
0 & 0 & - \cos(x^{3}) \\
0 & 0 & 0 \\
-\cos(x^{3}) & 0 & 0
\end{pmatrix}
+ \mathcal{R} (-\Delta)^{m_{1}} v_{0} (t,x) 
\end{equation}  
and 
\begin{equation}\label{estimate 433}
\mathring{R}_{0}^{\Xi} (t,x) \triangleq \frac{ m_{L} (2L + \frac{1}{2}) e^{2L t + L}}{(2\pi)^{3}} 
\begin{pmatrix}
0 & 0 & -\cos(x^{3}) \\
0 & 0 & \sin(x^{3}) \\
\cos(x^{3}) & -\sin(x^{3}) & 0 
\end{pmatrix}
+ \mathcal{R}^{\Xi} (-\Delta)^{m_{2}} \Xi_{0} (t,x), 
\end{equation} 
$(v_{0}, \Xi_{0})$ satisfy \eqref{estimate 277} and \eqref{estimate 272} at level $q= 0$ provided 
\begin{equation}\label{estimate 279}
\sqrt{3} (( 2\pi)^{3} + 1)^{2} < \sqrt{3} a^{2\beta b} \leq \frac{ \min \{c_{v}, c_{\Xi} \} e^{L - \frac{5}{2} L^{\frac{1}{4}}}}{L^{\frac{5}{4}} [ 8 (2L + \frac{1}{2} ) (2\pi)^{\frac{1}{2}} + 36 \pi^{\frac{3}{2}}]}, \hspace{3mm} L \leq \frac{ (2\pi)^{\frac{3}{2}} a^{4} -2}{2}, 
\end{equation} 
where the first inequality is assumed to justify \eqref{estimate 96}. Finally, $v_{0}(0,x), \Xi_{0} (0,x), \mathring{R}_{0}^{v} (0,x)$, and $\mathring{R}_{0}^{\Xi} (0,x)$ are all deterministic. 
\end{proposition}
 
\begin{proof}[Proof of Proposition \ref{Proposition 5.6} ]
First, $v_{0}$ and $\Xi_{0}$ are both mean-zero and divergence-free; thus, $\mathcal{R} (-\Delta)^{m_{1}} v_{0}$ and $\mathcal{R}^{\Xi} (-\Delta)^{m_{2}} \Xi_{0}$ are well-defined. Moreover, as $\mathcal{R} (-\Delta)^{m_{1}} v_{0}$ is trace-free and symmetric due to Lemma \ref{divergence inverse operator}, we see that $\mathring{R}_{0}^{v}$ is trace-free and symmetric. Similarly, as $\mathcal{R}^{\Xi} (-\Delta)^{m_{2}} \Xi_{0}$ is skew-symmetric, we see that $\mathring{R}_{0}^{\Xi}$ is skew-symmetric. Moreover, it can be readily verified that \eqref{estimate 277} holds with $p_{0} \equiv 0$. Next, let us compute from \eqref{estimate 429}
\begin{equation}\label{estimate 278}
\lVert v_{0} (t) \rVert_{L_{x}^{2}} = \frac{ m_{L} M_{0}(t)^{\frac{1}{2}}}{\sqrt{2}} \leq m_{L} M_{0}(t)^{\frac{1}{2}} \hspace{2mm} \text{ and } \hspace{2mm} \lVert \Xi_{0} (t) \rVert_{L_{x}^{2}}  = \frac{m_{L} M_{0}(t)^{\frac{1}{2}}}{(2\pi)^{\frac{3}{2}}} \leq m_{L} M_{0}(t)^{\frac{1}{2}},
\end{equation} 
which verifies \eqref{estimate 268}-\eqref{estimate 269} at level $q= 0$. We also compute directly from \eqref{estimate 429}  
\begin{subequations}
\begin{align}
&\lVert v_{0} \rVert_{C_{t,x}^{1}}  = \frac{ m_{L} M_{0}(t)^{\frac{1}{2}} (2L + 2)}{(2\pi)^{\frac{3}{2}}} \overset{\eqref{estimate 279}}{\leq} m_{L} M_{0}(t)^{\frac{1}{2}} \lambda_{0}^{4}, \\
&\lVert \Xi_{0} \rVert_{C_{t,x}^{1}} = \frac{ (1+ 2 L + \sqrt{2}) m_{L} M_{0}(t)^{\frac{1}{2}}}{(2\pi)^{3}} \overset{\eqref{estimate 279}}{\leq} m_{L} M_{0}(t)^{\frac{1}{2}} \lambda_{0}^{4}, 
\end{align}
\end{subequations} 
which verifies \eqref{estimate 270} at level $q =0$. Next, we can compute directly from \eqref{estimate 430}  
\begin{equation}\label{estimate 431}
\lVert \mathring{R}_{0}^{v}(t) \rVert_{L_{x}^{1}} \leq 8 m_{L} (2L + \frac{1}{2}) e^{2L t + L} (2\pi)^{\frac{1}{2}} + \lVert \mathcal{R} (-\Delta)^{m_{1}} v_{0}(t) \rVert_{L_{x}^{1}}.
\end{equation} 
We can use the fact that $\Delta v_{0} = - v_{0}$ due to \eqref{estimate 429} leads to $\lVert \mathcal{R} (-\Delta)^{m_{1}} v_{0}(t) \rVert_{L_{x}^{1}} \leq 36 \sqrt{2} \pi^{\frac{3}{2}} \lVert v_{0}(t) \rVert_{L_{x}^{2}}$ from \eqref{estimate 110} to compute 
\begin{equation}
\lVert \mathring{R}_{0}^{v} (t) \rVert_{L_{x}^{1}} 
\overset{\eqref{estimate 431}}{\leq} 8 m_{L} (2L + \frac{1}{2}) e^{2L t + L} (2\pi)^{\frac{1}{2}} + 36\sqrt{2} \pi^{\frac{3}{2}} \lVert v_{0} (t) \rVert_{L_{x}^{2}} 
\overset{\eqref{estimate 278} \eqref{estimate 279}}{\leq} c_{v} M_{0}(t) \delta_{1}. 
\end{equation} 
Similarly, we can estimate 
\begin{equation}\label{estimate 432}
\lVert \mathring{R}_{0}^{\Xi} (t) \rVert_{L_{x}^{1}} \leq \frac{ 8 m_{L} (2L + \frac{1}{2}) M_{0}(t)^{\frac{1}{2}}}{\pi} + \lVert \mathcal{R}^{\Xi} (-\Delta)^{m_{2}} \Xi_{0}(t) \rVert_{L_{x}^{1}}. 
\end{equation} 
We use the fact that $\Delta \Xi_{0} = -\Xi_{0}$ due to \eqref{estimate 429} leads to $\lVert \mathcal{R}^{\Xi} (-\Delta)^{m_{2}} \Xi_{0}(t) \rVert_{L_{x}^{1}} \leq 6 (2\pi)^{\frac{3}{2}} \lVert \Xi_{0}(t) \rVert_{L_{x}^{2}}$ by \eqref{estimate 114} so that
\begin{equation}
\lVert \mathcal{R}_{0}^{\Xi} (t) \rVert_{L_{x}^{1}} 
\overset{\eqref{estimate 432} }{\leq} \frac{ 8 m_{L} (2L + \frac{1}{2}) M_{0}(t)^{\frac{1}{2}}}{\pi} + 6 (2\pi)^{\frac{3}{2}} \lVert \Xi_{0}(t) \rVert_{L_{x}^{2}} \overset{\eqref{estimate 278} \eqref{estimate 279}}{\leq}  c_{\Xi} M_{0}(t) \delta_{1}. 
\end{equation}  
Finally, it is clear that $v_{0}(0,x)$ and $\Xi_{0}(0,x)$ are both deterministic; it follows that $\mathring{R}_{0}^{v}(0,x)$ and $\mathring{R}_{0}^{\Xi} (0,x)$ are also both deterministic. 
\end{proof} 
 
\begin{proposition}\label{Proposition 5.7}
Let $L$ satisfy 
\begin{equation}\label{estimate 280}
\sqrt{3} ((2\pi)^{3} + 1)^{2} < \frac{ \min \{c_{v}, c_{\Xi} \}  e^{L  - \frac{5}{2} L^{\frac{1}{4}}}}{ L^{\frac{5}{4}} [8 (2L + \frac{1}{2}) (2\pi)^{\frac{1}{2}} + 36 \pi^{\frac{3}{2}}]} 
\end{equation} 
and suppose that $(v_{q}, \Xi_{q}, \mathring{R}_{q}^{v}, \mathring{R}_{q}^{\Xi})$ are $(\mathcal{F}_{t})_{t\geq 0}$-adapted processes that solve \eqref{estimate 277} and satisfy \eqref{estimate 272}. Then there exist a choice of parameters $a, b,$ and $\beta$ such that \eqref{estimate 279} is fulfilled and $(\mathcal{F}_{t})_{t\geq 0}$-adapted processes $(v_{q+1}, \Xi_{q+1}, \mathring{R}_{q+1}^{v}, \mathring{R}_{q+1}^{\Xi})$ that solve \eqref{estimate 277} and satisfy \eqref{estimate 272} at level $q+1$, and for all $t \in [0, T_{L}]$, 
\begin{equation}\label{estimate 285}
\lVert v_{q+1}(t) -v_{q}(t) \rVert_{L_{x}^{2}} \leq m_{L} M_{0}(t)^{\frac{1}{2}} \delta_{q+1}^{\frac{1}{2}} \hspace{1mm} \text{ and } \hspace{1mm}  \lVert \Xi_{q+1} (t) - \Xi_{q}(t) \rVert_{L_{x}^{2}} \leq m_{L} M_{0}(t)^{\frac{1}{2}} \delta_{q+1}^{\frac{1}{2}}. 
\end{equation} 
Finally, if $(v_{q}, \Xi_{q}, \mathring{R}_{q}^{v}, \mathring{R}_{q}^{\Xi}) (0,x)$ are deterministic, then so are $(v_{q+1}, \Xi_{q+1}, \mathring{R}_{q+1}^{v},  \mathring{R}_{q+1}^{\Xi}) (0,x)$. 
\end{proposition}  

We are now ready to prove Theorem \ref{Theorem 2.3} assuming Proposition \ref{Proposition 5.7}. 
\begin{proof}[Proof of Theorem \ref{Theorem 2.3}]
Let us fix $T > 0, K > 1$, and $\kappa \in (0,1)$, and then take $L = L(T, K, c_{v}, c_{\Xi}) > 0$ that satisfies \eqref{estimate 280} and enlarge it if necessary to satisfy 
\begin{equation}\label{estimate 281}
\left( \frac{1}{(2\pi)^{\frac{3}{2}}} - \frac{1}{(2\pi)^{3}} \right) e^{2L T} > e^{2L^{\frac{1}{2}}} \left( \frac{2}{(2\pi)^{3}} + \frac{1}{\sqrt{2}} + \frac{1}{(2\pi)^{\frac{3}{2}}} \right) \hspace{1mm} \text{ and } \hspace{1mm} L > [\ln (Ke^{\frac{T}{2}} )]^{2}.
\end{equation} 
Now we can start from $(v_{0}, \Xi_{0}, \mathring{R}_{0}^{v}, \mathring{R}_{0}^{\Xi})$ in Proposition \ref{Proposition 5.6}, and via Proposition \ref{Proposition 5.7} inductively obtain a family $(v_{q}, \Xi_{q}, \mathring{R}_{q}^{v}, \mathring{R}_{q}^{\Xi})$ for all $q \in \mathbb{N}$ that satisfy \eqref{estimate 277},  \eqref{estimate 272}, and \eqref{estimate 285}. Similarly to \eqref{estimate 118}, we can assume that $b\geq 2$ and show for any $\gamma \in (0, \frac{\beta}{4+ \beta})$ and any $t \in [0, T_{L}]$ that via Gagliardo-Nirenberg's inequality, \eqref{estimate 270}, and \eqref{estimate 285} 
\begin{equation}
 \sum_{q\geq 0} \lVert v_{q+1}(t) - v_{q}(t) \rVert_{\dot{H}_{x}^{\gamma}} + \lVert \Xi_{q+1}(t) - \Xi_{q}(t) \rVert_{\dot{H}_{x}^{\gamma}}  \lesssim m_{L} M_{0}(t)^{\frac{1}{2}} \sum_{q\geq 0} \lambda_{q+1}^{-\beta (1-\gamma) + 4 \gamma} \lesssim m_{L} M_{0}(L)^{\frac{1}{2}}.
\end{equation} 
Therefore, $\{v_{q}\}_{q=0}$ and $\{\Xi_{q}\}_{q=0}$ are both Cauchy in $C([0,T_{L}]; \dot{H}^{\gamma} (\mathbb{T}^{3}))$ and thus we deduce the limiting processes $\lim_{q\to\infty} v_{q} \triangleq v$ and $\lim_{q\to\infty} \Xi_{q} \triangleq \Xi$ both in $C([0, T_{L}]; \dot{H}^{\gamma} (\mathbb{T}^{3}))$ for which there exists a deterministic constant $C_{L} > 0$ such that 
\begin{equation}\label{estimate 286}
\lVert v \rVert_{C([0, T_{L} ]; \dot{H}_{x}^{\gamma})} + \lVert \Xi \rVert_{C([0, T_{L} ]; \dot{H}_{x}^{\gamma} )} \leq C_{L}. 
\end{equation} 
Because $(v_{q}, \Xi_{q})$ are $(\mathcal{F}_{t})_{t\geq 0}$-adapted due to Propositions \ref{Proposition 5.6} and \ref{Proposition 5.7}, we see that $(v, \Xi)$ are $(\mathcal{F}_{t})_{t\geq 0}$-adapted. Moreover, because for all $t \in [0, T_{L}]$, due to \eqref{estimate 271} $\lVert \mathring{R}_{q}^{v} \rVert_{C_{t}L_{x}^{1}} + \lVert \mathring{R}_{q}^{\Xi} \rVert_{C_{t}L_{x}^{1}} \to 0$ as $q\to\infty$, $(v, \Xi)$ is a weak solution to  \eqref{estimate 22}-\eqref{estimate 23} on $[0, T_{L}]$ so that $(u,b) = (\Upsilon_{1} v, \Upsilon_{2} \Xi)$ solves \eqref{stochastic GMHD}. Because $\lim_{L\to\infty} T_{L} = + \infty$ $\textbf{P}$-a.s., for the fixed $T> 0$ and $\kappa > 0$, increasing $L$ larger if necessary allows us to obtain \eqref{estimate 19}. Next, as $v$ and $\Xi$ are $(\mathcal{F}_{t})_{t\geq 0}$-adapted, we see that $(u,b)$ are $(\mathcal{F}_{t})_{t\geq 0}$-adapted. Moreover, due to \eqref{estimate 286} and $\lVert \Upsilon_{k} \rVert_{C_{t}} \leq e^{L^{\frac{1}{4}}}$ from \eqref{estimate 273} we obtain \eqref{estimate 20}. Furthermore, we compute for all $t \in [0, T_{L}]$, 
\begin{subequations}\label{estimate 287}
\begin{align}
\lVert v(t) - v_{0}(t) \rVert_{L_{x}^{2}}\overset{\eqref{estimate 285}}{\leq} m_{L} M_{0}(t)^{\frac{1}{2}} \sum_{q\geq 0} \delta_{q+1}^{\frac{1}{2}} \overset{\eqref{estimate 279}}{<} m_{L} M_{0}(t)^{\frac{1}{2}} \frac{1}{(2\pi)^{3}},  \\
\lVert \Xi(t) - \Xi_{0}(t) \rVert_{L_{x}^{2}}\overset{\eqref{estimate 285}}{\leq} m_{L} M_{0}(t)^{\frac{1}{2}} \sum_{q\geq 0} \delta_{q+1}^{\frac{1}{2}} \overset{\eqref{estimate 279}}{<} m_{L} M_{0}(t)^{\frac{1}{2}} \frac{1}{(2\pi)^{3}}. 
\end{align} 
\end{subequations} 
Then it follows that 
\begin{subequations}
\begin{align}
&\lVert v(0) \rVert_{L_{x}^{2}} \leq \lVert v(0) - v_{0} (0) \rVert_{L_{x}^{2}} + \lVert v_{0} (0) \rVert_{L_{x}^{2}} 
\overset{\eqref{estimate 278} \eqref{estimate 287}}{\leq} m_{L} M_{0}(0)^{\frac{1}{2}} (\frac{1}{(2\pi)^{3}} + \frac{1}{\sqrt{2}}), \label{estimate 288} \\
&\lVert \Xi(0) \rVert_{L_{x}^{2}} \leq \lVert \Xi(0) - \Xi_{0} (0) \rVert_{L_{x}^{2}} + \lVert \Xi_{0} (0) \rVert_{L_{x}^{2}} 
\overset{\eqref{estimate 278} \eqref{estimate 287}}{\leq} m_{L} M_{0}(0)^{\frac{1}{2}} (\frac{1}{(2\pi)^{3}} + \frac{1}{(2\pi)^{\frac{3}{2}}}). \label{estimate 289}
\end{align}
\end{subequations}
On the other hand, on a set $\{T_{L} \geq T \}$, 
\begin{align}
\lVert \Xi(T) \rVert_{L_{x}^{2}}   \geq& \lVert \Xi_{0}(T) \rVert_{L_{x}^{2}} - \lVert \Xi(T) - \Xi_{0}(T) \rVert_{L_{x}^{2}} \overset{\eqref{estimate 278}\eqref{estimate 287}}{\geq} \frac{ m_{L} M_{0}(T)^{\frac{1}{2}}}{(2\pi)^{\frac{3}{2}}} - \frac{ m_{L} M_{0}(T)^{\frac{1}{2}}}{(2\pi)^{3}}\nonumber\\
& \hspace{30mm} \overset{\eqref{estimate 281}\eqref{estimate 288} \eqref{estimate 289}}{>} e^{2L^{\frac{1}{2}}} (\lVert v(0) \rVert_{L_{x}^{2}} +  \lVert \Xi(0) \rVert_{L_{x}^{2}}). \label{estimate 290}
\end{align} 
Therefore, on $\{T_{L} \geq T \}$, we obtain \eqref{estimate 26} as follows: 
\begin{align}
\lVert b(T) \rVert_{L_{x}^{2}} \overset{\eqref{estimate 258}}{\geq}&  e^{-L^{\frac{1}{4}}} \lVert \Xi(T) \rVert_{L_{x}^{2}} \nonumber \\
\overset{\eqref{estimate 290}}{\geq}&  e^{L^{\frac{1}{2}}} (\lVert v(0) \rVert_{L_{x}^{2}} +  \lVert \Xi(0) \rVert_{L_{x}^{2}}) 
\overset{\eqref{estimate 281}}{>}  K e^{\frac{T}{2}} ( \lVert u^{\text{in}} \rVert_{L_{x}^{2}} + \lVert b^{\text{in}} \rVert_{L_{x}^{2}}). 
\end{align}
\end{proof}

\subsection{Proof of Proposition \ref{Proposition 5.7}}
\subsubsection{Choice of parameters}  
We fix $L$ sufficiently large so that \eqref{estimate 280} holds. We take the same choices of $m_{1}^{\ast}, m_{2}^{\ast}, \eta, \alpha, \sigma, r, \mu$, and $b$ in \eqref{estimate 128}-\eqref{b} so that all the conditions of \eqref{estimate 50}, \eqref{estimate 53}, and \eqref{estimate 52} are satisfied; in fact, we choose $b$ and $a$ so that $\lambda_{q+1} \sigma \in 2 \mathbb{N}$ again. We choose the same $l$ in \eqref{l} so that estimates in \eqref{estimate 130} hold. By taking $a \in 2\mathbb{N}$ sufficiently large, the second inequality in \eqref{estimate 279} is satisfied. On the other hand, by taking $\beta >0$ sufficiently small, the first inequality of \eqref{estimate 279} is also satisfied because we chose $L$ to satisfy \eqref{estimate 280}. Thus, we consider these parameters fixed, preserving our freedom to take $a \in 2 \mathbb{N}$ as large and $\beta > 0$ as small as we wish, at minimum guaranteeing again \eqref{estimate 129}.

\subsubsection{Mollification}  
We mollify $v_{q}, \Xi_{q}, \mathring{R}_{q}^{v}$, and $\mathring{R}_{q}^{\Xi}$ identically as we did in \eqref{estimate 291} and $\Upsilon_{k}$ as we did in \eqref{estimate 292}. We note that similarly to \eqref{estimate 273}, we have estimates of 
\begin{equation}\label{estimate 296}
\lVert \Upsilon_{k,l} \rVert_{C_{t}}, \lVert \Upsilon_{k,l}^{-1} \rVert_{C_{t}} \overset{\eqref{estimate 258}}{\leq} m_{L}^{\frac{2}{5}}. 
\end{equation} 
Then, the mollified system of  \eqref{estimate 277} is of the form 
\begin{subequations}\label{estimate 358}
\begin{align}
&\partial_{t} v_{l} + \frac{1}{2} v_{l} + (-\Delta)^{m_{1}} v_{l} \nonumber\\
& \hspace{20mm} + \text{div} (\Upsilon_{1,l} (v_{l} \otimes v_{l}) - \Upsilon_{1,l}^{-1} \Upsilon_{2,l}^{2} (\Xi_{l} \otimes \Xi_{l}) ) + \nabla p_{l} = \text{div} (\mathring{R}_{l}^{v}  + R_{\text{com1}}^{v}), \\
& \partial_{t} \Xi_{l} + \frac{1}{2} \Xi_{l} + (-\Delta)^{m_{2}} \Xi_{l} + \Upsilon_{1,l} \text{div} (v_{l} \otimes \Xi_{l} - \Xi_{l} \otimes v_{l}) = \text{div}( \mathring{R}_{l}^{\Xi}  + R_{\text{com1}}^{\Xi}), 
\end{align}
\end{subequations}
if we define 
\begin{subequations}\label{estimate 384}
\begin{align}
p_{l} \triangleq&  (\Upsilon_{1} \frac{ \lvert v_{q} \rvert^{2}}{3})\ast_{x} \varrho_{l} \ast_{t} \vartheta_{l} - (\Upsilon_{1}^{-1} \Upsilon_{2}^{2} \frac{ \lvert \Xi_{q} \rvert^{2}}{3} )\ast_{x} \varrho_{l} \ast_{t} \vartheta_{l} \nonumber\\
& \hspace{20mm} + p_{q} \ast_{x} \varrho_{l} \ast_{t} \vartheta_{l} - \Upsilon_{1,l} \frac{ \lvert v_{l} \rvert^{2}}{3} + \Upsilon_{1,l}^{-1} \Upsilon_{2,l}^{2} \frac{ \lvert \Xi_{l} \rvert^{2}}{3}, \\
R_{\text{com1}}^{v}\triangleq& - (\Upsilon_{1} (v_{q} \mathring{\otimes} v_{q}) ) \ast_{x} \varrho_{l} \ast_{t} \vartheta_{l} + (\Upsilon_{1}^{-1} \Upsilon_{2}^{2} (\Xi_{q} \mathring{\otimes} \Xi_{q} ))\ast_{x} \varrho_{l} \ast_{t} \vartheta_{l} \nonumber\\
& \hspace{20mm} + \Upsilon_{1,l} (v_{l} \mathring{\otimes} v_{l}) - \Upsilon_{1,l}^{-1} \Upsilon_{2,l}^{2} (\Xi_{l} \mathring{\otimes} \Xi_{l} ), \label{estimate 434}\\
R_{\text{com1}}^{\Xi} \triangleq& -(\Upsilon_{1}(v_{q} \otimes \Xi_{q})) \ast_{x} \varrho_{l} \ast_{t} \vartheta_{l} + (\Upsilon_{1} (\Xi_{q} \otimes v_{q}) )\ast_{x} \varrho_{l} \ast_{t} \vartheta_{l} \nonumber \\
& \hspace{20mm} + \Upsilon_{1,l} (v_{l} \otimes \Xi_{l}) - \Upsilon_{1,l} (\Xi_{l} \otimes v_{l}). \label{estimate 435}
\end{align}
\end{subequations} 
We estimate the commutator terms similarly to \eqref{estimate 293}-\eqref{estimate 295}. First, we split from \eqref{estimate 434}
\begin{equation}\label{estimate 297}
\lVert R_{\text{com1}}^{v} \rVert_{C_{t}L_{x}^{1}} \leq A_{1} + A_{2}
\end{equation}
where 
\begin{align*}
A_{1} \triangleq&  \lVert \Upsilon_{1,l} (v_{l} \mathring{\otimes} v_{l} ) - (\Upsilon_{1} (v_{q} \mathring{\otimes} v_{q}) ) \ast_{x} \varrho_{l} \ast_{t} \vartheta_{l} \rVert_{C_{t}L_{x}^{1}},\\
A_{2} \triangleq& \lVert \Upsilon_{1,l}^{-1} \Upsilon_{2,l}^{2} (\Xi_{l} \mathring{\otimes} \Xi_{l}) - (\Upsilon_{1}^{-1} \Upsilon_{2}^{2} (\Xi_{q} \mathring{\otimes} \Xi_{q})) \ast_{x} \varrho_{l} \ast_{t} \vartheta_{l} \rVert_{C_{t}L_{x}^{1}}, 
\end{align*}
for which we can estimate by \eqref{estimate 273} and \eqref{estimate 296}  
\begin{subequations}\label{estimate 298}
\begin{align}
&A_{1} \lesssim lm_{L}^{\frac{2}{5}} \lVert v_{q} \rVert_{C_{t,x}^{1}} \lVert v_{q} \rVert_{C_{t}L_{x}^{2}} + l^{\frac{1}{2} - 2 \delta} m_{L}^{\frac{2}{5}} \lVert v_{q} \rVert_{C_{t}L_{x}^{\infty}} \lVert v_{q} \rVert_{C_{t}L_{x}^{2}}, \label{estimate 299} \\
&A_{2} \lesssim m_{L}^{2} l^{\frac{1}{2} - 2 \delta} \lVert \Xi_{q} \rVert_{C_{t}L_{x}^{\infty}} \lVert \Xi_{q} \rVert_{C_{t}L_{x}^{2}} + m_{L}^{\frac{6}{5}} l \lVert \Xi_{q} \rVert_{C_{t,x}^{1}} \lVert \Xi_{q} \rVert_{C_{t}L_{x}^{2}}; \label{estimate 300}
\end{align} 
\end{subequations}
e.g., \eqref{estimate 299} follows from splitting $A_{1}$ to $A_{1} \leq \sum_{k=1}^{6}A_{1k}$ where 
\begin{align*}
A_{11} \triangleq& \lVert (\Upsilon_{1,l} - \Upsilon_{1}) (v_{q} \ast_{x} \varrho_{l} \ast_{t} \vartheta_{l}) \mathring{\otimes} (v_{q} \ast_{x} \varrho_{l} \ast_{t} \vartheta_{l}) \rVert_{C_{t}L^{1}}, \\
A_{12} \triangleq & \lVert \Upsilon_{1} ( v_{q} \ast_{x} \varrho_{l} \ast_{t} \vartheta_{l} - v_{q} \ast_{x} \varrho_{l}) \mathring{\otimes} (v_{q} \ast_{x} \varrho_{l} \ast_{t} \vartheta_{l}) \rVert_{C_{t}L_{x}^{1}}, \\
A_{13} \triangleq& \lVert \Upsilon_{1} (v_{q} \ast_{x} \varrho_{l} - v_{q}) \mathring{\otimes} (v_{q} \ast_{x} \varrho_{l} \ast_{t} \vartheta_{l}) \rVert_{C_{t}L_{x}^{1}}, \\
A_{14} \triangleq& \lVert \Upsilon_{1} (v_{q} \mathring{\otimes} (v_{q} \ast_{x} \varrho_{l} \ast_{t} \vartheta_{l} - v_{q} \ast_{x} \varrho_{l} )) \rVert_{C_{t}L_{x}^{1}}, \\
A_{15} \triangleq& \lVert \Upsilon_{1} (v_{q} \mathring{\otimes} (v_{q} \ast_{x} \varrho_{l} - v_{q} ) ) \rVert_{C_{t}L_{x}^{1}}, \\
A_{16} \triangleq& \lVert (\Upsilon_{1} v_{q} \mathring{\otimes} v_{q}) - (\Upsilon_{1} v_{q} \mathring{\otimes} v_{q}) \ast_{x} \varrho_{l} \ast_{t} \vartheta_{l} \rVert_{C_{t}L_{x}^{1}},  
\end{align*} 
and then applying standard mollifier estimates. Applying \eqref{estimate 298} to \eqref{estimate 297} leads us to 
\begin{align}
\lVert R_{\text{com1}}^{v} \rVert_{C_{t}L_{x}^{1}} \lesssim& [ lm_{L}^{\frac{6}{5}} (\lVert v_{q} \rVert_{C_{t,x}^{1}} + \lVert \Xi_{q} \rVert_{C_{t,x}^{1}}) \nonumber\\
& + l^{\frac{1}{2} - 2 \delta} m_{L}^{2} (\lVert v_{q} \rVert_{C_{t}L_{x}^{\infty}} + \lVert \Xi_{q} \rVert_{C_{t}L_{x}^{\infty}})] (\lVert v_{q} \rVert_{C_{t}L_{x}^{2}} + \lVert \Xi_{q} \rVert_{C_{t}L_{x}^{2}}). \label{estimate 385}
\end{align} 
Similarly, we can split from \eqref{estimate 435} 
\begin{align*}
\lVert R_{\text{com1}}^{\Xi} \rVert_{C_{t}L_{x}^{1}} \leq& \lVert \Upsilon_{1,l} (v_{l} \otimes \Xi_{l}) - [\Upsilon_{1} (v_{q} \otimes \Xi_{q})] \ast_{x} \varrho_{l} \ast_{t} \vartheta_{l} \rVert_{C_{t}L_{x}^{1}} \\
&+ \lVert \Upsilon_{1,l} (\Xi_{l} \otimes v_{l} ) - [\Upsilon_{1} (\Xi_{q} \otimes v_{q})] \ast_{x} \varrho_{l} \ast_{t} \vartheta_{l} \rVert_{C_{t}L_{x}^{1}}
\end{align*}
and then estimate 
\begin{align}
\lVert R_{\text{com1}}^{\Xi} \rVert_{C_{t}L_{x}^{1}} \lesssim& [ lm_{L}^{\frac{2}{5}} (\lVert v_{q} \rVert_{C_{t,x}^{1}} + \lVert \Xi_{q} \rVert_{C_{t,x}^{1}}) \nonumber\\
& + l^{\frac{1}{2} - 2 \delta} m_{L}^{\frac{2}{5}} (\lVert v_{q} \rVert_{C_{t}L_{x}^{\infty}} + \lVert \Xi_{q} \rVert_{C_{t}L_{x}^{\infty}})] (\lVert v_{q} \rVert_{C_{t}L_{x}^{2}} + \lVert \Xi_{q} \rVert_{C_{t}L_{x}^{2}}). \label{estimate 386}
\end{align} 
Moreover, similarly to \eqref{estimate 301} we can estimate for $a \in 2 \mathbb{N}$ sufficiently large by Young's inequality for convolution and the fact that the mollifiers have unit mass 
\begin{subequations}\label{estimate 305}
\begin{align}
& \lVert v_{q} - v_{l} \rVert_{C_{t}L_{x}^{2}} + \lVert \Xi_{q} - \Xi_{l} \rVert_{C_{t}L_{x}^{2}}\overset{\eqref{estimate 270}}{\lesssim} m_{L} l M_{0}(t)^{\frac{1}{2}} \lambda_{q}^{4}  \overset{\eqref{estimate 130} \eqref{estimate 129}}{\ll} m_{L}M_{0}(t)^{\frac{1}{2}} \delta_{q+1}^{\frac{1}{2}},  \label{estimate 302} \\
& \lVert v_{l} \rVert_{C_{t}L_{x}^{2}} \leq \lVert v_{q} \rVert_{C_{t}L_{x}^{2}} \overset{\eqref{estimate 268}}{\leq} m_{L} M_{0}(t)^{\frac{1}{2}} (1+ \sum_{1 \leq \iota \leq q} \delta_{\iota}^{\frac{1}{2}}), \label{estimate 303}\\
& \lVert \Xi_{l} \rVert_{C_{t}L_{x}^{2}}\leq\lVert \Xi_{q}\rVert_{C_{t}L_{x}^{2}} \overset{\eqref{estimate 269}}{\leq} m_{L} M_{0}(t)^{\frac{1}{2}}(1+\sum_{1\leq \iota\leq q} \delta_{\iota}^{\frac{1}{2}}). \label{estimate 304}
\end{align}
\end{subequations}

\subsubsection{Perturbation}  
We use the same definitions of $\chi$ in \eqref{estimate 131} and $\rho_{\Xi}$ in \eqref{estimate 133} as we decided in Remark \ref{Remark 5.1}; we retain the estimates \eqref{estimate 132} and \eqref{estimate 134} because their verifications depended only on the definition of $\chi$. Furthermore, because the definition of $\rho_{\Xi}$ in \eqref{estimate 133} only involves $\mathring{R}_{l}^{\Xi}$ for which inductive estimates remains same (cf. \eqref{estimate 100} and \eqref{estimate 271}), we still retain \eqref{estimate 135}-\eqref{estimate 136}. Next, although we changed the definitions of $M_{0}(t)$ from \eqref{estimate 94} to \eqref{estimate 262} and the proof of \eqref{estimate 137}-\eqref{estimate 408} relied on the fact that $\partial_{t} M_{0}(t) = 4L M_{0}(t)$, fortunately the $M_{0}(t)$ in \eqref{estimate 262} also satisfies $\partial_{t} M_{0}(t) = 4L M_{0}(t)$; thus, we retain \eqref{estimate 137}-\eqref{estimate 408}. Furthermore, we retain \eqref{estimate 144} and consequently \eqref{estimate 141}-\eqref{estimate 143} because they are immediate consequences of  \eqref{estimate 135}-\eqref{estimate 137}. As we decided in Remark \ref{Remark 5.1}, we strategically define a modified amplitude function $\bar{a}_{\xi}$ for $\xi \in \Lambda_{\Xi}$ as in \eqref{estimate 362}, in which $\gamma_{\xi} (- \frac{\mathring{R}_{l}^{\Xi}}{\rho_{\Xi}})$ is well-defined because we retained \eqref{estimate 132}. This definition of $\bar{a}_{\xi}$ for $\xi \in \Lambda_{\Xi}$ allows us to directly deduce 
\begin{equation*}
 \sum_{\xi \in \Lambda_{\Xi}} \Upsilon_{1,l} \bar{a}_{\xi}^{2}  \mathbb{P}_{=0} (\phi_{\xi}^{2} \varphi_{\xi}^{2}) (\xi \otimes \xi_{2} - \xi_{2} \otimes \xi) 
\overset{\eqref{estimate 362}}{=} \sum_{\xi \in \Lambda_{\Xi}} a_{\xi}^{2} \mathbb{P}_{=0} (\phi_{\xi}^{2} \varphi_{\xi}^{2}) (\xi \otimes \xi_{2} -\xi_{2} \otimes \xi) 
\overset{\eqref{estimate 139}}{=} - \mathring{R}_{l}^{\Xi},  
\end{equation*}
which leads to the following analogous identity to \eqref{estimate 207}:
\begin{equation}\label{estimate 361}
\sum_{\xi \in\Lambda_{\Xi}} \Upsilon_{1,l} \bar{a}_{\xi}^{2} \phi_{\xi}^{2} \varphi_{\xi}^{2} (\xi \otimes \xi_{2} - \xi_{2} \otimes \xi) + \mathring{R}_{l}^{\Xi} = \sum_{\xi \in \Lambda_{\Xi}} \Upsilon_{1,l} \bar{a}_{\xi}^{2} \mathbb{P}_{\neq 0} (\phi_{\xi}^{2} \varphi_{\xi}^{2}) (\xi \otimes \xi_{2} - \xi_{2} \otimes \xi). 
\end{equation} 
By taking $c_{\Xi}$ sufficiently small to satisfy \eqref{estimate 157} again, because we retained the estimates \eqref{estimate 132}, \eqref{estimate 134}, and the inductive bound \eqref{estimate 100} is same as \eqref{estimate 271}, we retain \eqref{estimate 146}. Then  we can estimate for $c_{\Xi}$ sufficiently small that satisfies \eqref{estimate 157}, 
\begin{align}\label{estimate 307}
\lVert \bar{a}_{\xi} \rVert_{C_{t}L_{x}^{2}} 
\overset{\eqref{estimate 362}}{\leq}& \lVert \Upsilon_{1,l}^{-1} \rVert_{C_{t}}^{\frac{1}{2}} \lVert a_{\xi} \rVert_{C_{t}L_{x}^{2}} \nonumber \\
\overset{\eqref{estimate 296} \eqref{estimate 146}}{\leq}&  \min \{ ( \frac{c_{v} }{\lvert \Lambda_{\Xi} \rvert} )^{\frac{1}{2}}, \frac{1}{3 C_{\ast} (8\pi^{3})^{\frac{1}{2}} \lvert \Lambda_{\Xi} \rvert  } \} m_{L}^{\frac{1}{5}} \delta_{q+1}^{\frac{1}{2}} M_{0}(t)^{\frac{1}{2}}  \hspace{2mm} \forall \hspace{1mm} \xi \in \Lambda_{\Xi}. 
\end{align}
Similarly, because we retained \eqref{estimate 421}, \eqref{estimate 132}, \eqref{estimate 144}, and \eqref{estimate 394} we also retain \eqref{estimate 150}  so that 
\begin{subequations}\label{estimate 308}
\begin{align}
&\lVert a_{\xi} \rVert_{C_{t}C_{x}^{j}} \overset{\eqref{estimate 150}}{\lesssim}  l^{-5j - 2} M_{0}(t)^{\frac{1}{2}} \delta_{q+1}^{\frac{1}{2}} \hspace{44mm} \forall \hspace{1mm} j \geq 0,  \hspace{1mm} \xi \in \Lambda_{\Xi}, \label{estimate 436}\\ 
&\lVert \bar{a}_{\xi} \rVert_{C_{t}C_{x}^{j}} \overset{\eqref{estimate 362}}{\leq} \lVert \Upsilon_{1,l} ^{-1} \rVert_{C_{t}}^{\frac{1}{2}} \lVert a_{\xi} \rVert_{C_{t}C_{x}^{j}} \overset{\eqref{estimate 296}\eqref{estimate 150}}{\lesssim} m_{L}^{\frac{1}{5}} l^{-5j - 2} M_{0}(t)^{\frac{1}{2}} \delta_{q+1}^{\frac{1}{2}} \hspace{1mm} \forall \hspace{1mm} j \geq 0, \hspace{1mm}  \xi \in \Lambda_{\Xi}. \label{estimate 437} 
\end{align}
\end{subequations} 
Next, because we retained \eqref{estimate 132}, \eqref{estimate 394}, \eqref{estimate 141}-\eqref{estimate 143}, and \eqref{estimate 144} and \eqref{estimate 100} is same as \eqref{estimate 271}, we can also retain \eqref{estimate 309}-\eqref{estimate 406}. This allows us to directly deduce for $a \in 2 \mathbb{N}$ sufficiently large 
\begin{subequations}\label{estimate 310}
\begin{align}
\lVert a_{\xi} \rVert_{C_{t}^{1}C_{x}^{j}} &\overset{\eqref{estimate 309} }{\lesssim}  l^{-7 - 5j} \delta_{q+1}^{\frac{1}{2}} M_{0}(t)^{\frac{1}{2}} \hspace{36mm} \forall \hspace{1mm} j \in \{0,1,2\}, \hspace{1mm}  \xi \in \Lambda_{\Xi}, \label{estimate 415}\\
\lVert \bar{a}_{\xi} \rVert_{C_{t}^{1}C_{x}^{j}} &\lesssim \lVert \Upsilon_{1,l}^{-1} \rVert_{C_{t}}^{\frac{3}{2}} \lVert \partial_{t} \Upsilon_{1,l} \rVert_{C_{t}} \lVert a_{\xi} \rVert_{C_{t}C_{x}^{j}} + \lVert \Upsilon_{1,l}^{-1} \rVert_{C_{t}}^{\frac{1}{2}} \lVert \partial_{t} a_{\xi} \rVert_{C_{t}C_{x}^{j}} \nonumber\\
& \hspace{5mm} \overset{\eqref{estimate 436} \eqref{estimate 415} \eqref{estimate 296}}{\lesssim}  m_{L}^{\frac{1}{5}} l^{-7 - 5j} \delta_{q+1}^{\frac{1}{2}} M_{0}(t)^{\frac{1}{2}} \hspace{8mm} \forall \hspace{1mm} j \in \{0,1,2\}, \hspace{1mm} \xi \in \Lambda_{\Xi}, \label{estimate 442} \\
\lVert a_{\xi} \rVert_{C_{t}^{2}C_{x}} &\overset{\eqref{estimate 406}}{\lesssim} l^{-12} \delta_{q+1}^{\frac{1}{2}} M_{0}(t)^{\frac{1}{2}} \hspace{57mm} \forall \hspace{1mm} \xi \in \Lambda_{\Xi}, \label{estimate 407}\\
\lVert \bar{a}_{\xi} \rVert_{C_{t}^{2}C_{x}} &\overset{\eqref{estimate 273}  \eqref{estimate 296} \eqref{estimate 415}\eqref{estimate 407} }{\lesssim} m_{L}^{\frac{1}{5}} l^{-12} M_{0}(t)^{\frac{1}{2}} \delta_{q+1}^{\frac{1}{2}} \hspace{27mm} \forall \hspace{1mm} \xi \in \Lambda_{\Xi}. \label{estimate 414}
\end{align}
\end{subequations} 
Now we define $\mathring{G}^{\Xi}$ as planned in \eqref{estimate 274}. It follows that they satisfy the following estimates:
\begin{subequations}\label{estimate 417}
\begin{align}
& \lVert \mathring{G}^{\Xi} \rVert_{C_{t}L_{x}^{1}} \leq 3\sum_{\xi \in \Lambda_{\Xi}} \lVert \bar{a}_{\xi} \rVert_{C_{t}L_{x}^{2}}^{2} (\lVert \Upsilon_{1,l} \rVert_{C_{t}} + \lVert \Upsilon_{1,l}^{-1} \rVert_{C_{t}} \lVert \Upsilon_{2,l} \rVert_{C_{t}}^{2}) 
\overset{ \eqref{estimate 307}}{\leq} 6 m_{L}^{\frac{8}{5}} c_{v} \delta_{q+1} M_{0}(t) , \label{estimate 311}\\
& \lVert \mathring{G}^{\Xi} \rVert_{C_{t}C_{x}^{j}} \lesssim \sum_{\xi \in \Lambda_{\Xi}} \lVert \bar{a}_{\xi}^{2} \rVert_{C_{t}C_{x}^{j}} ( \lVert \Upsilon_{1,l} \rVert_{C_{t}} + \lVert \Upsilon_{1,l}^{-1} \rVert_{C_{t}} \lVert \Upsilon_{2,l} \rVert_{C_{t}}^{2}) \nonumber\\
& \hspace{45mm} \overset{\eqref{estimate 296} \eqref{estimate 308}}{\lesssim} m_{L}^{\frac{8}{5}}  \delta_{q+1} l^{-5j-4} M_{0}(t) \hspace{2mm} \forall \hspace{1mm} j \geq 0, \label{estimate 312}  \\
&\lVert \mathring{G}^{\Xi} \rVert_{C_{t}^{1}C_{x}^{j}} \overset{\eqref{estimate 296}}{\lesssim}  \sum_{\xi \in \Lambda_{\Xi}} \lVert \bar{a}_{\xi}\rVert_{C_{t,x}} \lVert \bar{a}_{\xi} \rVert_{C_{t}^{1}C_{x}^{j}} m_{L}^{\frac{6}{5}}  + \lVert \bar{a}_{\xi} \rVert_{C_{t,x}} \lVert \bar{a}_{\xi} \rVert_{C_{t}C_{x}^{j}} m_{L}^{2} l^{-1} \nonumber\\
& \hspace{45mm} \overset{\eqref{estimate 308} \eqref{estimate 310}}{\lesssim} m_{L}^{\frac{8}{5}}  \delta_{q+1}l^{-9-5j} M_{0}(t) \hspace{2mm} \forall  \hspace{1mm}j \in \{0,1,2\}, \label{estimate 313} \\
& \lVert \mathring{G}^{\Xi} \rVert_{C_{t}^{2}C_{x}} \overset{\eqref{estimate 308} \eqref{estimate 296} \eqref{estimate 414}}{\lesssim} m_{L}^{\frac{8}{5}} l^{-14} M_{0}(t) \delta_{q+1}. \label{estimate 416}
\end{align}
\end{subequations} 
Next, as we decided in Remark \ref{Remark 5.1}, we define $\rho_{v}$ identically to \eqref{estimate 152} but define a modified velocity amplitude $\bar{a}_{\xi}$ for $\xi \in \Lambda_{v}$ in \eqref{estimate 306}. Then we retain the estimates \eqref{estimate 155}-\eqref{estimate 154}, and \eqref{estimate 323} because their verifications only depended on the definition of $\chi$ in \eqref{estimate 131} which we preserved. These lead us to, for $\xi \in \Lambda_{v}$, by taking $c_{v} > 0$ sufficiently small to satisfy \eqref{estimate 156} 
\begin{align}
& \lVert \bar{a}_{\xi}\rVert_{C_{t}L_{x}^{2}} \overset{\eqref{estimate 296} \eqref{estimate 155}}{\leq} m_{L}^{\frac{1}{5}} \lVert \rho_{v} \rVert_{C_{t}L_{x}^{1}}^{\frac{1}{2}} \lVert \gamma_{\xi} \rVert_{C(B_{\epsilon_{v}}(\text{Id} ))} \nonumber \\
& \hspace{20mm} \overset{\eqref{estimate 154}\eqref{estimate 48}}{\leq} m_{L}^{\frac{1}{5}} [8 \epsilon_{v}^{-1} (c_{v} \delta_{q+1} M_{0}(t) 8\pi^{3} + \lVert \mathring{R}_{l}^{v} \rVert_{C_{t}L_{x}^{1}} + \lVert \mathring{G}^{\Xi} \rVert_{C_{t}L_{x}^{1}})]^{\frac{1}{2}} M_{\ast}  \nonumber \\
&\hspace{30mm} \overset{\eqref{estimate 271} \eqref{estimate 311}\eqref{estimate 156}}{\leq} \frac{ m_{L} \delta_{q+1}^{\frac{1}{2}} M_{0}(t)^{\frac{1}{2}}}{3  C_{\ast} (8\pi^{3})^{\frac{1}{2}} \lvert \Lambda_{v} \rvert }. \label{estimate 315}
\end{align}
Moreover, we can estimate by taking $a \in 2 \mathbb{N}$ sufficiently large so that $m_{L}^{\frac{8}{5}} \lesssim l^{-1}$, 
\begin{subequations}\label{estimate 319}
\begin{align}
& \lVert \rho_{v} \rVert_{C_{t,x}}  \overset{\eqref{estimate 154}}{\lesssim} \delta_{q+1} M_{0}(t) + \lVert \mathring{R}_{l}^{v} \rVert_{C_{t}W_{x}^{4,1}} + \lVert \mathring{G}^{\Xi} \rVert_{C_{t}L_{x}^{\infty}} 
\overset{\eqref{estimate 271} \eqref{estimate 312}}{\lesssim} m_{L}^{\frac{8}{5}}  l^{-4} M_{0}(t) \delta_{q+1}, \label{estimate 316} \\
& \lVert \rho_{v} \rVert_{C_{t}C_{x}^{j}}  \overset{\eqref{estimate 312} \eqref{estimate 271}}{\lesssim}  \delta_{q+1} M_{0}(t) l^{-10j} \hspace{37mm} \forall \hspace{1mm} j \geq 1, \label{estimate 317} \\
&  \lVert \rho_{v} \rVert_{C_{t}^{1}C_{x}^{j}} \overset{ \eqref{estimate 316} \eqref{estimate 317}  \eqref{estimate 312}\eqref{estimate 313} \eqref{estimate 271}}{\lesssim} m_{L}^{\frac{8}{5}} l^{-10j - 9}  M_{0}(t) \delta_{q+1} \hspace{2mm} \forall \hspace{1mm} j \in \{0,1,2\}, \label{estimate 318}\\
& \lVert \rho_{v} \rVert_{C_{t}^{2}C_{x}} \overset{\eqref{estimate 318} \eqref{estimate 271} \eqref{estimate 417}}{\lesssim} m_{L}^{\frac{16}{5}} \delta_{q+1}  l^{-18} M_{0}(t), \label{estimate 418}
\end{align}
\end{subequations}
where e.g. to prove \eqref{estimate 318}, one can directly compute $\partial_{t}\rho_{v}$ and then estimate 
\begin{align*}
\lVert \partial_{t} \rho_{v} \rVert_{C_{t,x}} 
\lesssim& L \lVert \rho_{v} \rVert_{C_{t,x}} + ( \lVert \partial_{t} \mathring{R}_{l}^{v} \rVert_{C_{t,x}} + \lVert \partial_{t} \mathring{G}^{\Xi} \rVert_{C_{t,x}} + [ \lVert \mathring{R}_{l}^{v} \rVert_{C_{t,x}} + \lVert \mathring{G}^{\Xi} \rVert_{C_{t,x}}]L ) \\
& \hspace{10mm} \overset{\eqref{estimate 316} \eqref{estimate 312}\eqref{estimate 313}\eqref{estimate 271}}{\lesssim}m_{L}^{\frac{8}{5}} \delta_{q+1} l^{-9}  M_{0}(t)
\end{align*}
while the case $j = 1, 2$ can be estimated via \cite[Equ. (130)]{BDIS15}. Immediate corollaries of \eqref{estimate 319}, using \cite[Equ. (130)]{BDIS15}, consist of 
\begin{subequations}\label{estimate 324}
\begin{align}
& \lVert \rho_{v}^{\frac{1}{2}} \rVert_{C_{t,x}} \overset{\eqref{estimate 316}}{\lesssim} m_{L}^{\frac{4}{5}} l^{-2}  M_{0}(t)^{\frac{1}{2}} \delta_{q+1}^{\frac{1}{2}}, \label{estimate 320} \\
& \lVert \rho_{v}^{\frac{1}{2}} \rVert_{C_{t}C_{x}^{j}} \overset{\eqref{estimate 323} \eqref{estimate 317}}{\lesssim} \delta_{q+1}^{\frac{1}{2}} M_{0}(t)^{\frac{1}{2}} l^{-10j} \hspace{22mm} \forall \hspace{1mm} j \geq 1, \label{estimate 321}\\
& \lVert \rho_{v}^{\frac{1}{2}} \rVert_{C_{t}^{1}C_{x}^{j}} \overset{\eqref{estimate 323}\eqref{estimate 318} \eqref{estimate 317}}{\lesssim} m_{L}^{\frac{8}{5}}\delta_{q+1}^{\frac{1}{2}} M_{0}(t)^{\frac{1}{2}} l^{-10j - 9}  \hspace{5mm} \forall \hspace{1mm} j \in \{0,1,2\}. 
\end{align}
\end{subequations}
Next, we obtain the following estimates; let us emphasize that it will be convenient for us to obtain the estimates on not only $\bar{a}_{\xi}$ but also $a_{\xi}$ for $\xi \in \Lambda_{v}$ here and they differ from those of \eqref{estimate 187}-\eqref{estimate 413} because we changed the definitions of $\mathring{G}^{\Xi}$ within $\rho_{v}$ and hence $a_{\xi}$ for $\xi \in \Lambda_{v}$. We have via \cite[Equ. (130)]{BDIS15}  
\begin{subequations}\label{estimate 420}
\begin{align}
& \lVert a_{\xi} \rVert_{C_{t}C_{x}^{j}} \overset{\eqref{estimate 320}\eqref{estimate 321} \eqref{estimate 317}}{\lesssim} m_{L}^{\frac{4}{5}} M_{0}(t)^{\frac{1}{2}} l^{-10j - 2} \delta_{q+1}^{\frac{1}{2}} \hspace{9mm} \forall \hspace{1mm} j \geq 0, \hspace{9mm}\xi \in \Lambda_{v}, \label{estimate 326}\\
&  \lVert \bar{a}_{\xi} \rVert_{C_{t}C_{x}^{j}} \overset{\eqref{estimate 296} \eqref{estimate 320}\eqref{estimate 321} \eqref{estimate 317}}{\lesssim} m_{L} M_{0}(t)^{\frac{1}{2}} l^{-10j - 2} \delta_{q+1}^{\frac{1}{2}} \hspace{2mm} \forall \hspace{1mm} j \geq 0, \hspace{9mm}\xi \in \Lambda_{v}, \label{estimate 325}\\
& \lVert a_{\xi} \rVert_{C_{t}^{1}C_{x}^{j}}\overset{\eqref{estimate 296} \eqref{estimate 324} \eqref{estimate 417} }{\lesssim} m_{L}^{\frac{12}{5}}  M_{0}(t)^{\frac{1}{2}} l^{-10j - 11} \delta_{q+1}^{\frac{1}{2}} \hspace{12mm} \forall \hspace{1mm} j \in \{0,1,2\},\hspace{1mm} \xi \in \Lambda_{v},\label{estimate 328}\\
& \lVert \bar{a}_{\xi} \rVert_{C_{t}^{1}C_{x}^{j}} \overset{\eqref{estimate 296} \eqref{estimate 324} \eqref{estimate 417}}{\lesssim} m_{L}^{\frac{13}{5}} M_{0}(t)^{\frac{1}{2}} l^{-10j - 11} \delta_{q+1}^{\frac{1}{2}}  \hspace{12mm} \forall \hspace{1mm} j \in \{0,1,2\}, \hspace{1mm}\xi \in \Lambda_{v},\label{estimate 327} \\
& \lVert a_{\xi} \rVert_{C_{t}^{2}C_{x}} \overset{\eqref{estimate 155} \eqref{estimate 323} \eqref{estimate 319} \eqref{estimate 417}}{\lesssim}m_{L}^{\frac{16}{5}} \delta_{q+1}^{\frac{1}{2}} M_{0}(t)^{\frac{1}{2}} l^{-18}  \hspace{27mm} \forall \hspace{1mm} \xi \in \Lambda_{v}, \label{estimate 419}
\end{align}
\end{subequations} 
by taking $ a \in 2 \mathbb{N}$ sufficiently large so that $m_{L}^{\frac{8}{5}} \lesssim l^{-1}$. Next, we have the following identity, analogous to, and a consequence of, \eqref{estimate 216}:
\begin{equation}\label{estimate 329}
\Upsilon_{1,l} \sum_{\xi \in \Lambda_{v}} \bar{a}_{\xi}^{2} \phi_{\xi}^{2} \varphi_{\xi}^{2} (\xi \otimes \xi) + \mathring{R}_{l}^{v} + \mathring{G}^{\Xi}  \overset{\eqref{estimate 306} \eqref{estimate 216}}{=} \rho_{v}\text{Id} + \sum_{\xi \in \Lambda_{v}} a_{\xi}^{2} \mathbb{P}_{\neq 0} (\phi_{\xi}^{2} \varphi_{\xi}^{2}) (\xi \otimes \xi).
\end{equation} 
Now we define $w_{q+1}^{p}, d_{q+1}^{p}, w_{q+1}^{c}$, and $d_{q+1}^{c}$ as follows (cf. \eqref{estimate 166}-\eqref{estimate 330}):
\begin{equation}\label{estimate 331}
w_{q+1}^{p} \triangleq \sum_{\xi \in \Lambda} \bar{a}_{\xi} \phi_{\xi} \varphi_{\xi} \xi, \hspace{3mm} d_{q+1}^{p} \triangleq \sum_{\xi \in \Lambda_{\Xi}} \bar{a}_{\xi} \phi_{\xi} \varphi_{\xi} \xi_{2}, 
\end{equation} 
and  
\begin{subequations}\label{estimate 332}
\begin{align}
w_{q+1}^{c} \triangleq \frac{1}{N_{\Lambda}^{2} \lambda_{q+1}^{2}} \sum_{\xi \in \Lambda}& \text{curl} ( \nabla \bar{a}_{\xi} \times (\phi_{\xi} \Psi_{\xi} \xi)) \nonumber \\
&+ \nabla \bar{a}_{\xi} \times \text{curl} ( \phi_{\xi} \Psi_{\xi} \xi) + \bar{a}_{\xi} \nabla \phi_{\xi} \times \text{curl} (\Psi_{\xi} \xi), \\
d_{q+1}^{c} \triangleq \frac{1}{N_{\Lambda}^{2} \lambda_{q+1}^{2}} \sum_{\xi \in \Lambda_{\Xi}}& \text{curl}( \nabla \bar{a}_{\xi} \times (\phi_{\xi} \Psi_{\xi} \xi_{2})) +\nabla \bar{a}_{\xi} \times\text{curl}(\phi_{\xi}\Psi_{\xi}\xi_{2}) - \bar{a}_{\xi}\Delta \phi_{\xi} \Psi_{\xi}\xi_{2}. 
\end{align}
\end{subequations} 
Due to \eqref{estimate 438} which remains valid because $\nabla\cdot (\Psi_{\xi} \xi ) = \nabla\cdot (\Psi_{\xi} \xi_{2}) = 0$, we have  
\begin{subequations}\label{estimate 380}
\begin{align}
&\frac{1}{N_{\Lambda}^{2} \lambda_{q+1}^{2}} \text{curl curl} \sum_{\xi \in \Lambda} \bar{a}_{\xi} \phi_{\xi} \Psi_{\xi} \xi = w_{q+1}^{p} + w_{q+1}^{c}, \label{estimate 501}\\
& \frac{1}{N_{\Lambda}^{2} \lambda_{q+1}^{2}} \text{curl curl}\sum_{\xi\in\Lambda_{\Xi}}\bar{a}_{\xi}\phi_{\xi}\Psi_{\xi}\xi_{2} = d_{q+1}^{p} + d_{q+1}^{c},  \label{estimate 502} 
\end{align}
\end{subequations} 
similarly to \eqref{estimate 193}. Thus, we see that $w_{q+1}^{p} + w_{q+1}^{c}$ and $d_{q+1}^{p} + d_{q+1}^{c}$ are both divergence-free and mean-zero. We also define the temporal correctors $w_{q+1}^{t}$ and $d_{q+1}^{t}$ identically to \eqref{estimate 334} so that they are mean-zero and divergence-free and the identities in \eqref{estimate 217} remain valid; however, we emphasize the difference due to the change in $\mathring{G}^{\Xi}$ and therefore in $\rho_{v}$ and ultimately in $a_{\xi}$ for $\xi \in \Lambda_{v}$ within the definitions of $w_{q+1}^{t}$ and $d_{q+1}^{t}$ in \eqref{estimate 334}. At last, we define $w_{q+1}, d_{q+1}, v_{q+1}$, and $\Xi_{q+1}$ identically to \eqref{estimate 176}-\eqref{estimate 203} so that $v_{q+1}$ and $\Xi_{q+1}$ are mean-zero and divergence-free. Now we have for all $j \in \mathbb{N}_{0}$, by taking $a \in 2 \mathbb{N}$ sufficiently large, 
\begin{subequations}
\begin{align}
\lVert D^{j} \bar{a}_{\xi} \rVert_{C_{t} L_{x}^{2}} \overset{\eqref{estimate 307}\eqref{estimate 308} }{\leq} \frac{m_{L} \delta_{q+1}^{\frac{1}{2}} M_{0}(t)^{\frac{1}{2}}}{3 C_{\ast} (8\pi^{3})^{\frac{1}{2}} \lvert \Lambda_{\Xi} \rvert} l^{-8j} \hspace{5mm} & \forall \hspace{1mm}  \xi \in \Lambda_{\Xi},  \\
\lVert D^{j} \bar{a}_{\xi} \rVert_{C_{t}L_{x}^{2}} \overset{\eqref{estimate 315}\eqref{estimate 325} }{\leq} \frac{m_{L} \delta_{q+1}^{\frac{1}{2}} M_{0}(t)^{\frac{1}{2}}}{3 C_{\ast} (8\pi^{3})^{\frac{1}{2}}\lvert \Lambda_{v} \rvert } l^{-13j} \hspace{5mm} &  \forall \hspace{1mm} \xi \in \Lambda_{v}. 
\end{align}
\end{subequations}
Thus, similarly to \eqref{estimate 172}-\eqref{estimate 173}, in order to apply Lemma \ref{Lemma 6.2} we set first in case $\xi \in \Lambda_{\Xi}$, ``$f$'' = $\bar{a}_{\xi}$, ``$g$'' = $\phi_{\xi} \varphi_{\xi}$, ``$\kappa$'' = $\lambda_{q+1} \sigma \in\mathbb{N}$ due to \eqref{estimate 52}, ``$N$'' = 1, ``$p$'' = 2, ``$\zeta$'' = $l^{-8}$, ``$C_{f}$'' = $\frac{m_{L}\delta_{q+1}^{\frac{1}{2}} M_{0}(t)^{\frac{1}{2}}  }{3 C_{\ast} (8\pi^{3})^{\frac{1}{2}} \lvert \Lambda_{\Xi} \rvert}$ while in case $\xi \in \Lambda_{v}$, we set identically with the only exceptions of ``$\zeta$'' = $l^{-13}$ and ``$C_{f}$'' = $\frac{m_{L} \delta_{q+1}^{\frac{1}{2}} M_{0}(t)^{\frac{1}{2}}}{3C_{\ast} (8\pi^{3})^{\frac{1}{2}} \lvert \Lambda_{v} \rvert}$ for which both conditions in \eqref{estimate 140} may be verified using \eqref{eta}-\eqref{sigma, r, mu}. Therefore, by \eqref{estimate 171} and \eqref{estimate 57}  
\begin{subequations}\label{estimate 439}
\begin{align}
& \lVert \bar{a}_{\xi} \phi_{\xi} \varphi_{\xi} \rVert_{C_{t}L_{x}^{2}} \leq \left( \frac{ m_{L} \delta_{q+1}^{\frac{1}{2}} M_{0}(t)^{\frac{1}{2}}}{3 C_{\ast} (8 \pi^{3})^{\frac{1}{2}} \lvert \Lambda_{\Xi} \rvert} \right) C_{\ast} \lVert \phi_{\xi} \varphi_{\xi} \rVert_{C_{t}L_{x}^{2}}   = \frac{m_{L}  \delta_{q+1}^{\frac{1}{2}} M_{0}(t)^{\frac{1}{2}}}{3 \lvert \Lambda_{\Xi} \rvert} \hspace{3mm} \forall \hspace{1mm} \xi \in \Lambda_{\Xi}, \\
& \lVert \bar{a}_{\xi} \phi_{\xi} \varphi_{\xi} \rVert_{C_{t}L_{x}^{2}} \leq \left( \frac{m_{L}  \delta_{q+1}^{\frac{1}{2}} M_{0}(t)^{\frac{1}{2}}}{3 C_{\ast} (8 \pi^{3})^{\frac{1}{2}} \lvert \Lambda_{v} \rvert} \right) C_{\ast} \lVert \phi_{\xi} \varphi_{\xi} \rVert_{C_{t}L_{x}^{2}} 
=\frac{ m_{L} \delta_{q+1}^{\frac{1}{2}} M_{0}(t)^{\frac{1}{2}}}{3 \lvert \Lambda_{v} \rvert} \hspace{3mm} \forall \hspace{1mm} \xi \in \Lambda_{v}.
\end{align}
\end{subequations}
It follows that 
\begin{subequations}\label{estimate 343}
\begin{align}
&\lVert d_{q+1}^{p} \rVert_{C_{t}L_{x}^{2}}  \overset{\eqref{estimate 331}}{\leq} \sum_{\xi \in \Lambda_{\Xi}} \lVert \bar{a}_{\xi} \phi_{\xi} \varphi_{\xi} \rVert_{C_{t}L_{x}^{2}} \overset{\eqref{estimate 439}}{\leq} \frac{m_{L}  \delta_{q+1}^{\frac{1}{2}} M_{0}(t)^{\frac{1}{2}}}{3}, \label{estimate 335} \\
&\lVert w_{q+1}^{p} \rVert_{C_{t}L_{x}^{2}} \overset{\eqref{estimate 331}}{\leq} \sum_{\xi \in \Lambda_{v}} \lVert \bar{a}_{\xi} \phi_{\xi} \varphi_{\xi} \rVert_{C_{t}L_{x}^{2}} +  \sum_{\xi \in \Lambda_{\Xi}} \lVert \bar{a}_{\xi} \phi_{\xi} \varphi_{\xi} \rVert_{C_{t}L_{x}^{2}} \overset{\eqref{estimate 439}}{\leq}  \frac{2m_{L} \delta_{q+1}^{\frac{1}{2}} M_{0}(t)^{\frac{1}{2}}}{3}. \label{estimate 336}
\end{align}
\end{subequations} 
Next, analogously to \eqref{estimate 179}, using the fact that $\phi_{\xi}$ and $\Psi_{\xi}$ have oscillations in orthogonal directions, we can compute from \eqref{estimate 332} for all $p \in [1,\infty]$ 
\begin{align}
& \lVert d_{q+1}^{c} \rVert_{C_{t}L_{x}^{p}} \label{estimate 337}\\
\lesssim& \lambda_{q+1}^{-2} \sum_{\xi \in \Lambda_{\Xi}} \lVert \bar{a}_{\xi} \rVert_{C_{t}C_{x}^{2}} \lVert \phi_{\xi} \rVert_{C_{t}L_{x}^{p}} \lVert \Psi_{\xi} \rVert_{L_{x}^{p}} + \lVert \bar{a}_{\xi} \rVert_{C_{t}C_{x}^{1}} (\lVert \phi_{\xi} \rVert_{C_{t}W_{x}^{1,p}} \lVert\Psi_{\xi} \rVert_{L_{x}^{p}} + \lVert \phi_{\xi} \rVert_{C_{t}L_{x}^{p}} \lVert \Psi_{\xi} \rVert_{W_{x}^{1,p}}) \nonumber\\
& \hspace{20mm} + \lVert \bar{a}_{\xi} \rVert_{C_{t,x}} \lVert \phi_{\xi}\rVert_{C_{t}W_{x}^{2,p}} \lVert\Psi_{\xi} \rVert_{L_{x}^{p}}  \overset{\eqref{estimate 175} \eqref{estimate 130} \eqref{estimate 308}}{\lesssim} m_{L}^{\frac{1}{5}} \delta_{q+1}^{\frac{1}{2}} M_{0}(t)^{\frac{1}{2}} l^{-2} r^{\frac{1}{p} - \frac{3}{2}} \sigma^{\frac{1}{p} + \frac{1}{2}}.  \nonumber 
\end{align}  
Similarly, we compute from \eqref{estimate 332}
\begin{align}
\lVert w_{q+1}^{c} \rVert_{C_{t}L_{x}^{p}} 
\lesssim& \lambda_{q+1}^{-2} (\sum_{\xi \in \Lambda_{v}} + \sum_{\xi \in \Lambda_{\Xi}}) \lVert \bar{a}_{\xi} \rVert_{C_{t}C_{x}^{2}} \lVert \phi_{\xi} \rVert_{C_{t}L_{x}^{p}} \lVert \Psi_{\xi} \rVert_{L_{x}^{p}} \label{estimate 338}\\
&+ \lVert \bar{a}_{\xi} \rVert_{C_{t}C_{x}^{1}} ( \lVert \phi_{\xi} \rVert_{C_{t}W_{x}^{1,p}} + \lVert \Psi_{\xi} \rVert_{L_{x}^{p}} + \lVert \phi_{\xi} \rVert_{C_{t}L_{x}^{p}} \lVert \Psi_{\xi} \rVert_{W_{x}^{1,p}}) + \lVert \bar{a}_{\xi} \rVert_{C_{t,x}} \lVert \phi_{\xi} \rVert_{C_{t}W_{x}^{1,p}} \lVert \Psi_{\xi} \rVert_{W_{x}^{1,p}} \nonumber \\
& \hspace{30mm} \overset{\eqref{estimate 308}\eqref{estimate 325}}{\lesssim} m_{L} M_{0}(t)^{\frac{1}{2}} l^{-2} r^{\frac{1}{p} - \frac{3}{2}} \sigma^{\frac{1}{p} + \frac{1}{2}} \delta_{q+1}^{\frac{1}{2}}. \nonumber 
\end{align}
Next, for $p \in (1,\infty)$, we compute $\lVert d_{q+1}^{t} \rVert_{C_{t}L_{x}^{p}}$ and $\lVert w_{q+1}^{t} \rVert_{C_{t}L_{x}^{p}}$. We note that in previous works such as \cite{HZZ19}, it was possible to just use the same estimate in \eqref{estimate 236}; however, we changed the definitions of $\mathring{G}^{\Xi}$ so that $\rho_{v}$ and hence $a_{\xi}$ for $\xi \in \Lambda_{v}$ changed. We compute for $p \in (1,\infty)$ using \eqref{estimate 308} and \eqref{estimate 326}, 
\begin{subequations}\label{estimate 339}
\begin{align}
& \lVert d_{q+1}^{t} \rVert_{C_{t}L_{x}^{p}} \lesssim \mu^{-1} \sum_{\xi \in \Lambda_{\Xi}} \lVert a_{\xi} \rVert_{C_{t,x}}^{2} \lVert \phi_{\xi} \varphi_{\xi} \rVert_{C_{t}L_{x}^{2p}}^{2} 
\overset{\eqref{estimate 175}}{\lesssim} \mu^{-1} \delta_{q+1} l^{-4} M_{0}(t) r^{\frac{1}{p} - 1} \sigma^{\frac{1}{p} -1}, \\
& \lVert w_{q+1}^{t} \rVert_{C_{t}L_{x}^{p}} \lesssim \mu^{-1} \sum_{\xi \in \Lambda} \lVert a_{\xi} \rVert_{C_{t,x}}^{2} \lVert \phi_{\xi} \varphi_{\xi} \rVert_{C_{t}L_{x}^{2p}}^{2} \overset{\eqref{estimate 175}}{\lesssim} m_{L}^{\frac{8}{5}} \mu^{-1} \delta_{q+1}  l^{-4}  M_{0}(t) r^{\frac{1}{p} - 1} \sigma^{\frac{1}{p} - 1}. 
\end{align}
\end{subequations} 
We are now ready to estimate for $a \in 2 \mathbb{N}$ sufficiently large 
\begin{subequations}\label{last}
\begin{align}
& \lVert d_{q+1} \rVert_{C_{t}L_{x}^{2}} \overset{\eqref{estimate 176}\eqref{estimate 343}\eqref{estimate 337}\eqref{estimate 339}}{\leq} m_{L} [ \frac{ \delta_{q+1}^{\frac{1}{2}} M_{0}(t)^{\frac{1}{2}}}{3}    \label{estimate 341} \\
& \hspace{20mm} + C [ M_{0} (t)^{\frac{1}{2}} l^{-2} r^{-1} \sigma \delta_{q+1}^{\frac{1}{2}} + \delta_{q+1} \mu^{-1} l^{-4} M_{0} (t) r^{-\frac{1}{2}} \sigma^{-\frac{1}{2}} ]] \overset{\eqref{estimate 340}}{\leq} \frac{m_{L} M_{0}(t)^{\frac{1}{2}} \delta_{q+1}^{\frac{1}{2}}   }{2}, \nonumber \\
&\lVert w_{q+1} \rVert_{C_{t}L_{x}^{2}} \overset{\eqref{estimate 176}}{\leq} \lVert w_{q+1}^{p} \rVert_{C_{t}L_{x}^{2}} + \lVert w_{q+1}^{c} \rVert_{C_{t}L_{x}^{2}} + \lVert w_{q+1}^{t} \rVert_{C_{t}L_{x}^{2}}  \label{estimate 342}\\
& \hspace{50mm} \overset{\eqref{estimate 343}\eqref{estimate 338} \eqref{estimate 339}}{\leq} \frac{ 3m_{L} M_{0}(t)^{\frac{1}{2}} \delta_{q+1}^{\frac{1}{2}}  }{4}, \nonumber 
\end{align}
\end{subequations}
where \eqref{estimate 341} followed directly from \eqref{estimate 340} whereas \eqref{estimate 342} can be verified using \eqref{estimate 130}. We also estimate for any $p\in [1,\infty]$, from \eqref{estimate 331} 
\begin{subequations}\label{estimate 344}
\begin{align}
& \lVert d_{q+1}^{p} \rVert_{C_{t}L_{x}^{p}} 
\leq \sum_{\xi \in \Lambda_{\Xi}} \lVert \bar{a}_{\xi} \rVert_{C_{t,x}} \lVert \phi_{\xi} \varphi_{\xi} \rVert_{C_{t}L_{x}^{p}}  \overset{\eqref{estimate 308} \eqref{estimate 175}}{\lesssim} m_{L}^{\frac{1}{5}}\delta_{q+1}^{\frac{1}{2}}  M_{0}(t)^{\frac{1}{2}}  l^{-2} r^{\frac{1}{p} - \frac{1}{2}} \sigma^{\frac{1}{p} - \frac{1}{2}},  \\
&\lVert w_{q+1}^{p} \rVert_{C_{t}L_{x}^{p}} 
\leq \sum_{\xi \in \Lambda} \lVert \bar{a}_{\xi} \rVert_{C_{t,x}} \lVert \phi_{\xi} \varphi_{\xi} \rVert_{C_{t}L_{x}^{p}} \overset{\eqref{estimate 308} \eqref{estimate 325} \eqref{estimate 175}}{\lesssim} m_{L}  \delta_{q+1}^{\frac{1}{2}}M_{0}(t)^{\frac{1}{2}} l^{-2} r^{\frac{1}{p} - \frac{1}{2}} \sigma^{\frac{1}{p} - \frac{1}{2}}.
\end{align}
\end{subequations}
We are now ready to estimate for $p \in (1,\infty)$, 
\begin{align}
& \lVert w_{q+1} \rVert_{C_{t}L_{x}^{p}} + \lVert d_{q+1} \rVert_{C_{t}L_{x}^{p}} \label{estimate 345}\\
\overset{\eqref{estimate 344}\eqref{estimate 337}-\eqref{estimate 339}}{\lesssim}& (m_{L} + m_{L}^{\frac{1}{5}}) [ M_{0}(t)^{\frac{1}{2}} l^{-2} \delta_{q+1}^{\frac{1}{2}} r^{\frac{1}{p} - \frac{1}{2}} \sigma^{\frac{1}{p} - \frac{1}{2}} + \delta_{q+1}^{\frac{1}{2}} M_{0}(t)^{\frac{1}{2}} l^{-2} r^{\frac{1}{p} - \frac{3}{2}} \sigma^{\frac{1}{p} + \frac{1}{2}}] \nonumber\\
&+ (m_{L}^{\frac{8}{5}} + 1) \mu^{-1} l^{-4} M_{0}(t) r^{\frac{1}{p} -1} \sigma^{\frac{1}{p} -1} \delta_{q+1} 
\overset{\eqref{estimate 130}}{\lesssim} m_{L} M_{0}(t)^{\frac{1}{2}} l^{-2} r^{\frac{1}{p} - \frac{1}{2}} \sigma^{\frac{1}{p} - \frac{1}{2}} \delta_{q+1}^{\frac{1}{2}}. \nonumber
\end{align} 
Next, for all $p \in[1,\infty]$, we estimate from \eqref{estimate 331} 
\begin{subequations}\label{estimate 440} 
\begin{align}
 \lVert d_{q+1}^{p} \rVert_{C_{t}W_{x}^{1,p}}  \lesssim&\sum_{\xi \in \Lambda_{\Xi}} \lVert \bar{a}_{\xi} \rVert_{C_{t}C_{x}^{1}} \lVert \phi_{\xi} \varphi_{\xi} \rVert_{C_{t}L_{x}^{p}} + \lVert \bar{a}_{\xi} \rVert_{C_{t,x}} \lVert \phi_{\xi} \varphi_{\xi} \rVert_{C_{t}W_{x}^{1,p}} \nonumber\\
 & \hspace{10mm} \overset{\eqref{estimate 175} \eqref{estimate 308}\eqref{estimate 130}}{\lesssim} m_{L}^{\frac{1}{5}} M_{0}(t)^{\frac{1}{2}} l^{-2} \delta_{q+1}^{\frac{1}{2}} r^{\frac{1}{p} - \frac{1}{2}} \sigma^{\frac{1}{p} - \frac{1}{2}} \lambda_{q+1},   \label{estimate 346}\\
\lVert w_{q+1}^{p} \rVert_{C_{t}W_{x}^{1,p}} \lesssim& \sum_{\xi \in \Lambda} \lVert \bar{a}_{\xi} \rVert_{C_{t}C_{x}^{1}} \lVert \phi_{\xi} \varphi_{\xi} \rVert_{C_{t}L_{x}^{p}} + \lVert \bar{a}_{\xi} \rVert_{C_{t,x}} \lVert \phi_{\xi} \varphi_{\xi} \rVert_{C_{t}W_{x}^{1,p}} \nonumber \\
& \hspace{10mm} \overset{\eqref{estimate 175} \eqref{estimate 325}\eqref{estimate 130}}{\lesssim}  m_{L} M_{0}(t)^{\frac{1}{2}} l^{-2} \delta_{q+1}^{\frac{1}{2}} r^{\frac{1}{p} - \frac{1}{2}} \sigma^{\frac{1}{p} - \frac{1}{2}} \lambda_{q+1}. \label{estimate 347}
\end{align}
\end{subequations}
For all $p \in [1, \infty]$, we also deduce the following estimates from \eqref{estimate 189}-\eqref{estimate 190}: 
\begin{subequations}
\begin{align}
\lVert d_{q+1}^{c} \rVert_{C_{t}W_{x}^{1,p}} &\overset{\eqref{estimate 332} \eqref{estimate 308}\eqref{estimate 175}}{\lesssim} m_{L}^{\frac{1}{5}} \lambda_{q+1}^{-2} l^{-2} M_{0}(t)^{\frac{1}{2}}\delta_{q+1}^{\frac{1}{2}} [ l^{-15} r^{\frac{1}{p} - \frac{1}{2}} \sigma^{\frac{1}{p} - \frac{1}{2}} + l^{-5} \lambda_{q+1}^{2} r^{\frac{1}{p} - \frac{1}{2}} \sigma^{\frac{1}{p} - \frac{1}{2}}  \nonumber \\
& + l^{-10} \lambda_{q+1} r^{\frac{1}{p} - \frac{1}{2}} \sigma^{\frac{1}{p} - \frac{1}{2}} + l^{-5} \lambda_{q+1}^{2}  r^{\frac{1}{p} - \frac{5}{2}} \sigma^{\frac{1}{p} + \frac{3}{2}}  + \lambda_{q+1}^{3} r^{\frac{1}{p} - \frac{5}{2}} \sigma^{\frac{1}{p} + \frac{3}{2}} + \lambda_{q+1}^{3} r^{\frac{1}{p} - \frac{3}{2}} \sigma^{\frac{1}{p} + \frac{1}{2}}]\nonumber \\
& \hspace{25mm} \overset{\eqref{estimate 189}}{\lesssim} m_{L}^{\frac{1}{5}} \delta_{q+1}^{\frac{1}{2}} \lambda_{q+1} l^{-2} M_{0}(t)^{\frac{1}{2}} r^{\frac{1}{p} - \frac{3}{2}} \sigma^{\frac{1}{p} + \frac{1}{2}}, \label{estimate 348}\\
 \lVert w_{q+1}^{c} \rVert_{C_{t}W_{x}^{1,p}} &\overset{\eqref{estimate 332} \eqref{estimate 308}\eqref{estimate 325}}{\lesssim} m_{L} \delta_{q+1}^{\frac{1}{2}} \lambda_{q+1}^{-2} M_{0}(t)^{\frac{1}{2}} l^{-2}  [ l^{-30} r^{\frac{1}{p} - \frac{1}{2}} \sigma^{\frac{1}{p} - \frac{1}{2}} + l^{-10} \lambda_{q+1}^{2} r^{\frac{1}{p} - \frac{1}{2}} \sigma^{\frac{1}{p} - \frac{1}{2}} \nonumber\\
&   + l^{-20} \lambda_{q+1} r^{\frac{1}{p} - \frac{1}{2}} \sigma^{\frac{1}{p} - \frac{1}{2}} + l^{-10} \lambda_{q+1}^{2} r^{\frac{1}{p} - \frac{3}{2}} \sigma^{\frac{1}{p} + \frac{1}{2}} + \lambda_{q+1}^{3} r^{\frac{1}{p} - \frac{5}{2}} \sigma^{\frac{1}{p} + \frac{3}{2}} + \lambda_{q+1}^{3} r^{\frac{1}{p} - \frac{3}{2}} \sigma^{\frac{1}{p} + \frac{1}{2}}] \nonumber\\
& \hspace{25mm} \overset{\eqref{estimate 190} \eqref{estimate 130} \eqref{sigma, r, mu} }{\lesssim} m_{L} \delta_{q+1}^{\frac{1}{2}} \lambda_{q+1} M_{0}(t)^{\frac{1}{2}} l^{-2} r^{\frac{1}{p} - \frac{3}{2}} \sigma^{\frac{1}{p} + \frac{1}{2}}. \label{estimate 349}
\end{align}
\end{subequations} 
Finally, for all $p \in (1,\infty)$, we can estimate
\begin{subequations}\label{estimate 441}
\begin{align}
& \lVert d_{q+1}^{t} \rVert_{C_{t}W_{x}^{1,p}} \overset{\eqref{estimate 334} \eqref{estimate 191}}{\lesssim} \delta_{q+1} \mu^{-1} l^{-4} M_{0}(t) \lambda_{q+1} r^{\frac{1}{p} - 1} \sigma^{\frac{1}{p} - 1}, \label{estimate 350}\\
& \lVert w_{q+1}^{t} \rVert_{C_{t}W_{x}^{1,p}} \overset{\eqref{estimate 334}\eqref{estimate 326} \eqref{estimate 310} \eqref{estimate 175}}{\lesssim} m_{L}^{\frac{8}{5}}\delta_{q+1} \mu^{-1}  l^{-4} M_{0}(t) \lambda_{q+1} r^{\frac{1}{p} - 1} \sigma^{\frac{1}{p} -1},  \label{estimate 351}
\end{align}
\end{subequations}
where \eqref{estimate 350} follows from the computations in \eqref{estimate 191} because for all $\xi \in \Lambda_{\Xi}$, our estimates on $\lVert a_{\xi} \rVert_{C_{t}C_{x}^{j}}$ in \eqref{estimate 150} and \eqref{estimate 436} are identical. These lead us to conclude for all $p \in (1,\infty)$, 
\begin{equation}\label{estimate 377}
 \lVert w_{q+1} \rVert_{C_{t} W_{x}^{1,p}} + \lVert d_{q+1} \rVert_{C_{t}W_{x}^{1,p}} \overset{\eqref{estimate 440}-\eqref{estimate 441} \eqref{estimate 130}}{\lesssim} m_{L} M_{0}(t)^{\frac{1}{2}} l^{-2} \delta_{q+1}^{\frac{1}{2}} r^{\frac{1}{p} - \frac{1}{2}} \sigma^{\frac{1}{p} - \frac{1}{2}} \lambda_{q+1}. 
\end{equation}
Next, we estimate using \eqref{estimate 380}  
\begin{subequations}\label{estimate 356}
\begin{align}
&\lVert d_{q+1}^{p} + d_{q+1}^{c}\rVert_{C_{t,x}^{1}} \overset{\eqref{estimate 194}}{\lesssim}m_{L}^{\frac{1}{5}}  \delta_{q+1}^{\frac{1}{2}} l^{-2} M_{0}(t)^{\frac{1}{2}} \lambda_{q+1}\sigma^{\frac{1}{2}} r^{-\frac{3}{2}} \mu, \label{estimate 352} \\
& \lVert w_{q+1}^{p} + w_{q+1}^{c} \rVert_{C_{t,x}^{1}} \lesssim \lambda_{q+1}^{-2} \sum_{\xi \in \Lambda}\lVert \bar{a}_{\xi} \rVert_{C_{t}^{1}C_{x}^{2}} \lVert \phi_{\xi} \Psi_{\xi} \rVert_{C_{t}C_{x}} + \lVert \bar{a}_{\xi} \rVert_{C_{t,x}} (\lVert \phi_{\xi} \rVert_{C_{t}^{1}C_{x}^{2}} \lVert \Psi_{\xi} \rVert_{C_{x}} + \lVert \phi_{\xi} \rVert_{C_{t}^{1}C_{x}} \lVert \Psi_{\xi} \rVert_{C_{x}^{2}}) \nonumber\\
& \hspace{23mm} + \lVert \bar{a}_{\xi} \rVert_{C_{t}C_{x}^{3}} \lVert \phi_{\xi} \Psi_{\xi} \rVert_{C_{t}C_{x}} + \lVert \bar{a}_{\xi} \rVert_{C_{t,x}} \lVert \phi_{\xi} \Psi_{\xi} \rVert_{C_{t}C_{x}^{3}} \nonumber\\
& \hspace{26mm} \overset{\eqref{estimate 420} \eqref{estimate 308} \eqref{estimate 310} \eqref{estimate 175}}{\lesssim}  m_{L} \delta_{q+1}^{\frac{1}{2}} \lambda_{q+1} M_{0}(t)^{\frac{1}{2}} l^{-2} \sigma^{\frac{1}{2}} r^{-\frac{3}{2}} \mu, \label{estimate 353} 
\end{align}
\end{subequations}
where \eqref{estimate 352} follows from \eqref{estimate 194} because the only difference in the estimates of $\lVert a_{\xi} \rVert_{C_{t}C_{x}^{j}}$ and $\lVert a_{\xi} \rVert_{C_{t}^{1}C_{x}^{j}}$ for $\xi \in \Lambda_{\Xi}$ in \eqref{estimate 437} and \eqref{estimate 442} compared to \eqref{estimate 150} and \eqref{estimate 309} is a factor of $m_{L}^{\frac{1}{5}}$. Next, because for $\xi \in \Lambda_{\Xi}$ the bounds of $\lVert a_{\xi} \rVert_{C_{t}C_{x}^{j}}$, $\lVert a_{\xi} \rVert_{C_{t}^{1}C_{x}^{j}}$, and $\lVert a_{\xi} \rVert_{C_{t}^{2}C_{x}}$ for $\xi \in \Lambda_{\Xi}$ in \eqref{estimate 436}, \eqref{estimate 415}, and \eqref{estimate 407} are same as \eqref{estimate 150}, \eqref{estimate 309}, and \eqref{estimate 406}, we can bound $\mathbb{P}\mathbb{P}_{\neq 0}$ in the expense of $\lambda_{q+1}^{\alpha}$ and obtain identically to \eqref{estimate 410} 
\begin{equation}\label{estimate 354}
\lVert d_{q+1}^{t} \rVert_{C_{t,x}^{1}}\lesssim \delta_{q+1} M_{0}(t) l^{-4} \lambda_{q+1}^{1+ \alpha}  r^{-2}.
\end{equation} 
Similarly, we can deduce 
\begin{equation}\label{estimate 355}
\lVert w_{q+1}^{t} \rVert_{C_{t,x}^{1}} \overset{\eqref{estimate 308} \eqref{estimate 310} \eqref{estimate 420} \eqref{estimate 175}}{\lesssim} m_{L}^{\frac{8}{5}} \delta_{q+1} M_{0}(t) l^{-4} \lambda_{q+1}^{1+ \alpha} r^{-2}. 
\end{equation} 
We are now ready to conclude via \eqref{estimate 176} 
\begin{align}
 \lVert w_{q+1} \rVert_{C_{t,x}^{1}} + \lVert d_{q+1} \rVert_{C_{t,x}^{1}} \overset{\eqref{estimate 356}-\eqref{estimate 355}}{\lesssim}& m_{L} l^{-2} \delta_{q+1}^{\frac{1}{2}} M_{0}(t)^{\frac{1}{2}} \lambda_{q+1} \sigma^{\frac{1}{2}} r^{-\frac{3}{2}} \mu + m_{L}^{\frac{8}{5}} M_{0}(t) l^{-4} \lambda_{q+1}^{1+ \alpha} r^{-2} \delta_{q+1}  \nonumber\\
\overset{\eqref{sigma, r, mu} \eqref{estimate 130}}{\lesssim}& m_{L} l^{-2} \delta_{q+1}^{\frac{1}{2}} M_{0}(t)^{\frac{1}{2}} \lambda_{q+1} \sigma^{\frac{1}{2}} r^{-\frac{3}{2}} \mu.\label{estimate 357}
\end{align}
We can verify \eqref{estimate 285} via \eqref{estimate 203} as follows:
\begin{subequations}
\begin{align}
&\lVert v_{q+1} (t) - v_{q}(t) \rVert_{L_{x}^{2}} \leq \lVert w_{q+1}(t) \rVert_{L_{x}^{2}} + \lVert v_{l}(t) - v_{q}(t) \rVert_{L_{x}^{2}} \overset{\eqref{estimate 342} \eqref{estimate 302}}{\leq} m_{L} M_{0}(t)^{\frac{1}{2}}\delta_{q+1}^{\frac{1}{2}},    \\
&\lVert \Xi_{q+1} (t) - \Xi_{q}(t) \rVert_{L_{x}^{2}} \leq \lVert d_{q+1}(t) \rVert_{L_{x}^{2}} + \lVert \Xi_{l}(t) - \Xi_{q}(t) \rVert_{L_{x}^{2}}  \overset{\eqref{estimate 341} \eqref{estimate 302}}{\leq} m_{L}M_{0}(t)^{\frac{1}{2}} \delta_{q+1}^{\frac{1}{2}} .   
\end{align}
\end{subequations}
Additionally, we can verify \eqref{estimate 268}-\eqref{estimate 269} at level $q+1$ via \eqref{estimate 203} as 
\begin{subequations}\label{estimate 448}
\begin{align}
&\lVert v_{q+1} \rVert_{C_{t}L_{x}^{2}} \leq\Vert w_{q+1} \rVert_{C_{t}L_{x}^{2}} + m_{L} M_{0}(t)^{\frac{1}{2}} (1+ \sum_{1\leq \iota \leq q} \delta_{\iota}^{\frac{1}{2}}) \leq m_{L}M_{0}(t)^{\frac{1}{2}} (1+ \sum_{1 \leq \iota \leq q+1} \delta_{\iota}^{\frac{1}{2}}), \\
& \lVert \Xi_{q+1} \rVert_{C_{t}L_{x}^{2}} \leq \Vert d_{q+1} \rVert_{C_{t}L_{x}^{2}} + m_{L}M_{0}(t)^{\frac{1}{2}} (1+ \sum_{1\leq \iota \leq q} \delta_{\iota}^{\frac{1}{2}}) \leq m_{L}M_{0}(t)^{\frac{1}{2}} (1+ \sum_{1 \leq \iota \leq q+1} \delta_{\iota}^{\frac{1}{2}})
\end{align} 
\end{subequations}
due to \eqref{estimate 303}-\eqref{estimate 304}, and \eqref{estimate 341}-\eqref{estimate 342}. At last, we verify \eqref{estimate 270} at level $q+1$ as follows: by taking $a \in 2 \mathbb{N}$ sufficiently large, 
\begin{align*}
\lVert v_{q+1} \rVert_{C_{t,x}^{1}}  + \lVert \Xi_{q+1} \rVert_{C_{t,x}^{1}} &
\overset{\eqref{estimate 203} \eqref{estimate 357}}{\lesssim} \lVert v_{q} \rVert_{C_{t,x}^{1}} + \lVert \Xi_{q} \rVert_{C_{t,x}^{1}} + m_{L} \delta_{q+1}^{\frac{1}{2}} l^{-2} M_{0}(t)^{\frac{1}{2}} \lambda_{q+1} \sigma^{\frac{1}{2}} r^{-\frac{3}{2}} \mu \\
& \overset{\eqref{estimate 270}}{\lesssim} m_{L}M_{0}(t)^{\frac{1}{2}} \lambda_{q}^{4} + m_{L}l^{-2} M_{0}(t)^{\frac{1}{2}} \lambda_{q+1} \sigma^{\frac{1}{2}} r^{-\frac{3}{2}} \mu \overset{\eqref{estimate 130}}{\leq} m_{L}M_{0}(t)^{\frac{1}{2}} \lambda_{q+1}^{4}. 
\end{align*} 

\subsubsection{Reynolds stress}  
We first compute using \eqref{estimate 277}, \eqref{estimate 358}, \eqref{estimate 176}, and \eqref{estimate 203}, 
\begin{align}
\text{div} \mathring{R}_{q+1}^{\Xi} &= \partial_{t} d_{q+1}^{p} + \partial_{t} d_{q+1}^{c} + \partial_{t} d_{q+1}^{t} + \frac{1}{2} d_{q+1} + (-\Delta)^{m_{2}} d_{q+1} \label{estimate 443}  \\
& - \Upsilon_{1,l} \text{div} (v_{l} \otimes \Xi_{l} - \Xi_{l} \otimes v_{l}) + \text{div} \mathring{R}_{l}^{\Xi} + \text{div} R_{\text{com1}}^{\Xi} + \Upsilon_{1} \text{div} (v_{q+1} \otimes \Xi_{q+1} - \Xi_{q+1} \otimes v_{q+1}) \nonumber 
\end{align}
where we may further rewrite via \eqref{estimate 176}-\eqref{estimate 203} 
\begin{align}
& - \Upsilon_{1,l} \text{div} (v_{l} \otimes \Xi_{l} - \Xi_{l} \otimes v_{l}) + \Upsilon_{1} \text{div} (v_{q+1} \otimes \Xi_{q+1} - \Xi_{q+1} \otimes v_{q+1}) \nonumber \\
=&  \Upsilon_{1,l} \text{div} (v_{l} \otimes d_{q+1} + (w_{q+1}^{c} + w_{q+1}^{t}) \otimes d_{q+1} + w_{q+1}^{p} \otimes d_{q+1}^{p} \nonumber \\
& \hspace{5mm} + w_{q+1}^{p} \otimes (d_{q+1}^{c} + d_{q+1}^{t}) + w_{q+1} \otimes \Xi_{l} \nonumber \\
& \hspace{5mm} - \Xi_{l} \otimes w_{q+1} - (d_{q+1}^{c} + d_{q+1}^{t}) \otimes w_{q+1} - d_{q+1}^{p} \otimes w_{q+1}^{p} - d_{q+1}^{p} \otimes (w_{q+1}^{c} + w_{q+1}^{t}) - d_{q+1} \otimes v_{l}) \nonumber \\
&+ (\Upsilon_{1} - \Upsilon_{1,l}) \text{div} (v_{q+1} \otimes \Xi_{q+1} - \Xi_{q+1} \otimes v_{q+1}). \label{estimate 444} 
\end{align}
Applying \eqref{estimate 444}  to \eqref{estimate 443} gives us  
\begin{align}
 \text{div} \mathring{R}_{q+1}^{\Xi} =& \underbrace{\frac{1}{2} d_{q+1} + (-\Delta)^{m_{2}} d_{q+1} + \partial_{t} d_{q+1}^{p} + \partial_{t} d_{q+1}^{c}}_{\text{Part of } \text{div} R_{\text{lin}}^{\Xi}} \nonumber \\
& \underbrace{ + \text{div} \Upsilon_{1,l} ( w_{q+1} \otimes \Xi_{l} + v_{l} \otimes d_{q+1} - d_{q+1} \otimes v_{l} - \Xi_{l} \otimes w_{q+1})}_{\text{Another part of } \text{div} R_{\text{lin}}^{\Xi}} \nonumber \\
&\underbrace{+ \text{div} \Upsilon_{1,l} ( ( w_{q+1}^{c} + w_{q+1}^{t} ) \otimes d_{q+1} + w_{q+1}^{p} \otimes (d_{q+1}^{c} + d_{q+1}^{t})}_{\text{Part of div} R_{\text{corr}}^{\Xi}} \nonumber  \\
& \hspace{5mm}  \underbrace{ - (d_{q+1}^{c} + d_{q+1}^{t}) \otimes w_{q+1} - d_{q+1}^{p} \otimes (w_{q+1}^{c} + w_{q+1}^{t})}_{\text{Another part of div} R_{\text{corr}}^{\Xi}} \nonumber \\
&\underbrace{+ \text{div} (\Upsilon_{1,l} w_{q+1}^{p} \otimes d_{q+1}^{p} - \Upsilon_{1,l} d_{q+1}^{p} \otimes w_{q+1}^{p} + \mathring{R}_{l}^{\Xi} ) + \partial_{t} d_{q+1}^{t}}_{\text{div} R_{\text{osc}}^{\Xi}} \nonumber\\
&\underbrace{+ (\Upsilon_{1} - \Upsilon_{1,l} ) \text{div} (v_{q+1} \otimes \Xi_{q+1} - \Xi_{q+1} \otimes v_{q+1})}_{\text{div} R_{\text{com2}}^{\Xi}} + \text{div} R_{\text{com1}}^{\Xi}.  \label{estimate 396}
\end{align}
Similarly, we can write using \eqref{estimate 277}, \eqref{estimate 358}, \eqref{estimate 176}, and \eqref{estimate 203}, 
\begin{align}
\text{div} \mathring{R}_{q+1}^{v} - \nabla p_{q+1} 
&=\partial_{t} w_{q+1}^{p} + \partial_{t} w_{q+1}^{c} + \partial_{t} w_{q+1}^{t} + \frac{1}{2} w_{q+1} + (-\Delta)^{m_{1}} w_{q+1}  \nonumber \\
&-\text{div} (\Upsilon_{1,l} (v_{l} \otimes v_{l} ) - \Upsilon_{1,l}^{-1} \Upsilon_{2,l}^{2} (\Xi_{l} \otimes \Xi_{l} )) -\nabla p_{l} + \text{div} \mathring{R}_{l}^{v} + \text{div} R_{\text{com1}}^{v} \nonumber \\
&+ \text{div} ( \Upsilon_{1} (v_{q+1} \otimes v_{q+1} ) - \Upsilon_{1}^{-1} \Upsilon_{2}^{2} (\Xi_{q+1} \otimes \Xi_{q+1} )), \label{estimate 359} 
\end{align}
where we rewrite the nonlinear terms within \eqref{estimate 359} as follows, although it is complicated due to $\Upsilon_{1,l}^{-1} \Upsilon_{2,l}^{2}$; first, we write 
\begin{align}
& - \Upsilon_{1}^{-1} \Upsilon_{2}^{2} (\Xi_{q+1} \otimes \Xi_{q+1}) +\Upsilon_{1,l}^{-1} \Upsilon_{2,l}^{2} (\Xi_{l} \otimes \Xi_{l}) = - (\Upsilon_{1}^{-1} - \Upsilon_{1,l}^{-1})\Upsilon_{2}^{2} (\Xi_{q+1} \otimes \Xi_{q+1}) \nonumber \\
& \hspace{7mm} - \Upsilon_{1,l}^{-1} (\Upsilon_{2} - \Upsilon_{2,l}) (\Upsilon_{2} + \Upsilon_{2,l}) (\Xi_{q+1} \otimes \Xi_{q+1})  - \Upsilon_{1,l}^{-1} \Upsilon_{2,l}^{2} (\Xi_{q+1} \otimes \Xi_{q+1} - \Xi_{l} \otimes \Xi_{l})  \label{estimate 445}
\end{align}
so that via \eqref{estimate 176}-\eqref{estimate 203} 
\begin{align}
& - \text{div} (\Upsilon_{1,l} (v_{l} \otimes v_{l}) - \Upsilon_{1,l}^{-1} \Upsilon_{2,l}^{2} (\Xi_{l} \otimes \Xi_{l} )) + \text{div} (\Upsilon_{1} (v_{q+1}\otimes v_{q+1} ) - \Upsilon_{1}^{-1}\Upsilon_{2}^{2}(\Xi_{q+1} \otimes \Xi_{q+1} ))  \nonumber\\
=& \Upsilon_{1,l} \text{div} (v_{l} \otimes w_{q+1} + (w_{q+1}^{c} + w_{q+1}^{t}) \otimes w_{q+1} + w_{q+1}^{p} \otimes w_{q+1}^{p}  \nonumber\\
& \hspace{40mm}  + w_{q+1}^{p} \otimes (w_{q+1}^{c}+  w_{q+1}^{t}) + w_{q+1} \otimes v_{l})  \nonumber\\
& - \Upsilon_{1,l}^{-1} \Upsilon_{2,l}^{2} \text{div} (\Xi_{l} \otimes d_{q+1} + (d_{q+1}^{c} + d_{q+1}^{t}) \otimes d_{q+1} + d_{q+1}^{p} \otimes d_{q+1}^{p}   \nonumber\\
& \hspace{40mm} + d_{q+1}^{p} \otimes (d_{q+1}^{c} + d_{q+1}^{t}) + d_{q+1} \otimes \Xi_{l})  \nonumber\\
&+ (\Upsilon_{1} - \Upsilon_{1,l}) \text{div} (v_{q+1} \otimes v_{q+1})   - (\Upsilon_{1}^{-1} - \Upsilon_{1,l}^{-1}) \Upsilon_{2}^{2} \text{div} (\Xi_{q+1} \otimes \Xi_{q+1}) \nonumber \\
& -\Upsilon_{1,l}^{-1} (\Upsilon_{2} - \Upsilon_{2,l}) (\Upsilon_{2} + \Upsilon_{2,l}) \text{div} (\Xi_{q+1} \otimes \Xi_{q+1}).  \label{estimate 360}
\end{align}
Therefore, applying \eqref{estimate 360} to \eqref{estimate 359} gives us 
\begin{align}
& \text{div} \mathring{R}_{q+1}^{v} - \nabla p_{q+1} \label{estimate 397}\\
=& \underbrace{\frac{1}{2} w_{q+1} + (-\Delta)^{m_{1}} w_{q+1} + \partial_{t} w_{q+1}^{p} + \partial_{t} w_{q+1}^{c}}_{\text{ Part of (div}R_{\text{lin}}^{v} +  \nabla p_{\text{lin}}^{v})} \nonumber\\
& + \underbrace{ \text{div} (\Upsilon_{1,l} w_{q+1} \otimes v_{l} + \Upsilon_{1,l} v_{l} \otimes w_{q+1} - \Upsilon_{1,l}^{-1} \Upsilon_{2,l}^{2} d_{q+1} \otimes \Xi_{l} - \Upsilon_{1,l}^{-1} \Upsilon_{2,l}^{2} \Xi_{l} \otimes d_{q+1})}_{\text{Another part of (div} R_{\text{lin}}^{v} + \nabla p_{\text{lin}}^{v})} \nonumber\\ 
& + \underbrace{ \text{div}(\Upsilon_{1,l} (w_{q+1}^{c} + w_{q+1}^{t}) \otimes w_{q+1} + \Upsilon_{1,l} w_{q+1}^{p} \otimes (w_{q+1}^{c} + w_{q+1}^{t})}_{\text{Part of (div}R_{\text{corr}}^{v} + \nabla p_{\text{corr}}^{v})}  \nonumber \\
& \hspace{5mm} \underbrace{-\Upsilon_{1,l}^{-1} \Upsilon_{2,l}^{2} (d_{q+1}^{c} + d_{q+1}^{t}) \otimes d_{q+1} - \Upsilon_{1,l}^{-1} \Upsilon_{2,l}^{2} d_{q+1}^{p} \otimes (d_{q+1}^{c} + d_{q+1}^{t} ) )}_{\text{Another part of (div}R_{\text{corr}}^{v} + \nabla p_{\text{corr}}^{v})} \nonumber\\
&+  \underbrace{ \text{div} (\Upsilon_{1,l} w_{q+1}^{p} \otimes w_{q+1}^{p} - \Upsilon_{1,l}^{-1} \Upsilon_{2,l}^{2} d_{q+1}^{p} \otimes d_{q+1}^{p} + \mathring{R}_{l}^{v}) + \partial_{t} w_{q+1}^{t}}_{\text{div} R_{\text{osc}}^{v} + \nabla p_{\text{osc}}^{v}} \nonumber\\
&+ \underbrace{ (\Upsilon_{1} - \Upsilon_{1,l}) \text{div} (v_{q+1} \otimes v_{q+1} ) - (\Upsilon_{1}^{-1} - \Upsilon_{1,l}^{-1} )\Upsilon_{2}^{2} \text{ div } (\Xi_{q+1} \otimes \Xi_{q+1})}_{\text{Part of (div} R_{\text{com2}}^{v} + \nabla p_{\text{com2}})} \nonumber \\
&  \underbrace{ - \Upsilon_{1,l}^{-1} (\Upsilon_{2} - \Upsilon_{2,l}) (\Upsilon_{2} + \Upsilon_{2,l}) \text{div} (\Xi_{q+1} \otimes \Xi_{q+1} )}_{\text{Another parat of (div} R_{\text{com2}}^{v} + \nabla p_{\text{com2}})} + \text{div} R_{\text{com1}}^{v} - \nabla p_{l}. \nonumber 
\end{align}
Let us now compute from \eqref{estimate 396}  
\begin{align} 
\text{div} R_{\text{osc}}^{\Xi} 
=& \text{div} ( \Upsilon_{1,l} \sum_{\xi \in \Lambda_{\Xi}} \bar{a}_{\xi}^{2} \phi_{\xi}^{2} \varphi_{\xi}^{2} (\xi \otimes \xi_{2} - \xi_{2} \otimes \xi)  \\
& \hspace{5mm} + \Upsilon_{1,l} \sum_{\xi \in \Lambda, \xi' \in \Lambda_{\Xi}: \xi \neq \xi'} \bar{a}_{\xi} \bar{a}_{\xi'} \phi_{\xi} \phi_{\xi'} \varphi_{\xi} \varphi_{\xi'} (\xi \otimes \xi_{2}' - \xi_{2}' \otimes \xi) + \mathring{R}_{l}^{\Xi}) + \partial_{t} d_{q+1}^{t}  \nonumber \\
\overset{\eqref{estimate 361}}{=}& \text{div} (\sum_{\xi \in \Lambda_{\Xi}} \Upsilon_{1,l} \bar{a}_{\xi}^{2} \mathbb{P}_{\neq 0} (\phi_{\xi}^{2} \varphi_{\xi}^{2}) (\xi \otimes \xi_{2} - \xi_{2} \otimes \xi)  \nonumber \\
& \hspace{5mm}+ \Upsilon_{1,l} \sum_{\xi \in \Lambda, \xi'\in \Lambda_{\Xi}: \xi \neq \xi'} \bar{a}_{\xi} \bar{a}_{\xi'} \phi_{\xi} \phi_{\xi'} \varphi_{\xi} \varphi_{\xi'} (\xi \otimes \xi_{2}' - \xi_{2}' \otimes \xi) ) + \partial_{t} d_{q+1}^{t}  \nonumber \\
\overset{\eqref{estimate 306} \eqref{estimate 362}}{=}& \text{div} (\sum_{\xi \in \Lambda_{\Xi}}   a_{\xi}^{2} \mathbb{P}_{\neq \frac{ \lambda_{q+1} \sigma}{2}} (\phi_{\xi}^{2} \varphi_{\xi}^{2}) (\xi \otimes \xi_{2} - \xi_{2} \otimes \xi)  \nonumber \\
& \hspace{5mm}+  \sum_{\xi \in \Lambda, \xi'\in \Lambda_{\Xi}: \xi \neq \xi'} a_{\xi} a_{\xi'} \phi_{\xi} \phi_{\xi'} \varphi_{\xi} \varphi_{\xi'} (\xi \otimes \xi_{2}' - \xi_{2}' \otimes \xi) ) + \partial_{t} d_{q+1}^{t}. \nonumber 
\end{align} 
This is precisely $\text{div} (E_{1}^{\Xi} + E_{2}^{\Xi}) + \partial_{t} d_{q+1}^{t}$ in \eqref{estimate 363}-\eqref{estimate 446}.  Therefore, the definition of $R_{\text{osc}}^{\Xi}$ in our current case is same as that of \eqref{estimate 423}  in the proof of Proposition \ref{Proposition 4.8}. Moreover, because the estimates of $\lVert a_{\xi} \rVert_{C_{t}C_{x}^{j}}$ and $\lVert a_{\xi} \rVert_{C_{t}^{1}C_{x}^{j}}$ for $\xi \in \Lambda_{\Xi}$ in \eqref{estimate 150} and \eqref{estimate 309} remained same in \eqref{estimate 308}-\eqref{estimate 310}, the estimate $\lVert \mathcal{R}^{\Xi} (\text{div} E_{1}^{\Xi} + \partial_{t} d_{q+1}^{t} ) \rVert_{C_{t}L_{x}^{p^{\ast}}}  \ll M_{0}(t) \delta_{q+2}$ in \eqref{estimate 244} remains valid in our current case.  On the other hand, for $\text{div}R_{\text{osc}}^{v} + \nabla p_{\text{osc}}$, we continue from \eqref{estimate 364} as 
\begin{align}
&\text{div} R_{\text{osc}}^{v} + \nabla p_{\text{osc}}  \\
\overset{\eqref{estimate 364}}{=}& \nabla \rho_{v} + \text{div} (\sum_{\xi \in \Lambda} a_{\xi}^{2} \mathbb{P}_{\neq 0} (\phi_{\xi}^{2} \varphi_{\xi}^{2}) (\xi\otimes \xi)) - \text{div} (\sum_{\xi \in \Lambda_{\Xi}} a_{\xi}^{2} \mathbb{P}_{\neq 0} (\phi_{\xi}^{2} \varphi_{\xi}^{2}) \Upsilon_{1,l}^{-2} \Upsilon_{2,l}^{2} (\xi_{2}\otimes \xi_{2} )) \nonumber \\
&+ \text{div} (\sum_{\xi, \xi' \in \Lambda: \xi \neq \xi'} a_{\xi} a_{\xi'} \phi_{\xi} \phi_{\xi'} \varphi_{\xi} \varphi_{\xi'} (\xi\otimes \xi') \nonumber\\
& \hspace{10mm} - \Upsilon_{1,l}^{-2} \Upsilon_{2,l}^{2} \sum_{\xi, \xi' \in \Lambda_{\Xi}: \xi \neq \xi'} a_{\xi} a_{\xi'} \phi_{\xi} \phi_{\xi'} \varphi_{\xi} \varphi_{\xi'} (\xi_{2} \otimes \xi_{2} ')) + \partial_{t} w_{q+1}^{t}  \nonumber\\
&\overset{\eqref{estimate 57}\eqref{estimate 212} \eqref{estimate 214} }{=} \nabla \rho_{v} + \sum_{\xi \in\Lambda} \mathbb{P}_{\neq 0} (( \xi \otimes \xi) \nabla a_{\xi}^{2} \mathbb{P}_{\geq \frac{\lambda_{q+1} \sigma}{2}} (\phi_{\xi}^{2}\varphi_{\xi}^{2} )) \nonumber\\
& - \mu^{-1} \sum_{\xi \in \Lambda} \mathbb{P}_{\neq 0} (\partial_{t} a_{\xi}^{2} \mathbb{P}_{\neq 0} (\phi_{\xi}^{2} \varphi_{\xi}^{2} )) \xi + \mu^{-1} \sum_{\xi \in \Lambda} \nabla \Delta^{-1} \text{div} \partial_{t} (a_{\xi}^{2} \mathbb{P}_{\neq 0} (\phi_{\xi}^{2} \varphi_{\xi}^{2} )) \xi  \nonumber\\
& - \Upsilon_{1,l}^{-2} \Upsilon_{2,l}^{2} \sum_{\xi \in \Lambda_{\Xi}} \mathbb{P}_{\neq 0} (( \xi_{2} \otimes \xi_{2}) \nabla a_{\xi}^{2} \mathbb{P}_{\geq \frac{\lambda_{q+1} \sigma}{2}} (\phi_{\xi}^{2} \varphi_{\xi}^{2}))  \nonumber\\
& + \text{div} (\sum_{\xi,\xi' \in\Lambda: \xi \neq \xi'} a_{\xi}a_{\xi'} \phi_{\xi}\phi_{\xi'} \varphi_{\xi}\varphi_{\xi'} \xi \otimes \xi' - \Upsilon_{1,l}^{-2} \Upsilon_{2,l}^{2}\sum_{\xi, \xi' \in\Lambda_{\Xi}: \xi \neq \xi'}  a_{\xi} a_{\xi'} \phi_{\xi}\phi_{\xi'} \varphi_{\xi}\varphi_{\xi'} \xi_{2}\otimes \xi_{2}').  \nonumber
\end{align}
Therefore, at last we define again $\mathring{R}_{q+1}^{v}$ and $\mathring{R}_{q+1}^{\Xi}$ identically to \eqref{estimate 218}-\eqref{estimate 219} while 
\begin{equation}
p_{q+1} \triangleq - p_{\text{lin}} - p_{\text{corr}} - p_{\text{osc}} - p_{\text{com2}} + p_{l}
\end{equation} 
and besides $R_{\text{osc}}^{\Xi}$ in \eqref{estimate 423}, and $R_{\text{com1}}^{v}, R_{\text{com1}}^{\Xi}$, and $p_{l}$ in \eqref{estimate 384} we define 
\begin{subequations}\label{estimate 449}
\begin{align}
R_{\text{lin}}^{\Xi} \triangleq& \mathcal{R}^{\Xi} ( \frac{1}{2} d_{q+1} + (-\Delta)^{m_{2}} d_{q+1}) + \mathcal{R}^{\Xi} ( \partial_{t} d_{q+1}^{p} + \partial_{t} d_{q+1}^{c})  \nonumber \\
&+ \Upsilon_{1,l} (w_{q+1} \otimes \Xi_{l} + v_{l} \otimes d_{q+1} - d_{q+1} \otimes v_{l} - \Xi_{l} \otimes w_{q+1}), \label{estimate 365}\\
R_{\text{corr}}^{\Xi} \triangleq& \Upsilon_{1,l} ( ( w_{q+1}^{c} + w_{q+1}^{t} ) \otimes d_{q+1} + w_{q+1}^{p} \otimes (d_{q+1}^{c} + d_{q+1}^{t} ) \nonumber\\
& \hspace{10mm} - (d_{q+1}^{c} + d_{q+1}^{t}) \otimes w_{q+1}  - d_{q+1}^{p} \otimes (w_{q+1}^{c} + w_{q+1}^{t} )), \label{estimate 366}\\ 
R_{\text{com2}}^{\Xi} \triangleq& (\Upsilon_{1} - \Upsilon_{1,l}) (v_{q+1} \otimes \Xi_{q+1} - \Xi_{q+1} \otimes v_{q+1}), \label{estimate 368}\\
R_{\text{lin}}^{v} \triangleq& \mathcal{R} ( \frac{1}{2} w_{q+1} + (-\Delta)^{m_{1}} w_{q+1}) + \mathcal{R} (\partial_{t} w_{q+1}^{p} + \partial_{t} w_{q+1}^{c})) \nonumber \\
& + \Upsilon_{1,l} (w_{q+1} \mathring{\otimes} v_{l} + v_{l} \mathring{\otimes} w_{q+1}) - \Upsilon_{1,l}^{-1} \Upsilon_{2,l}^{2} (d_{q+1} \mathring{\otimes} \Xi_{l} + \Xi_{l} \mathring{\otimes} d_{q+1}), \label{estimate 369} \\
p_{\text{lin}} \triangleq& \frac{2}{3} [ \Upsilon_{1,l} (w_{q+1} \cdot v_{l}) - \Upsilon_{1,l}^{-1} \Upsilon_{2,l}^{2} (d_{q+1} \cdot \Xi_{l}) ], \label{estimate 370}\\
R_{\text{corr}}^{v} \triangleq& \Upsilon_{1,l} [ (w_{q+1}^{c} + w_{q+1}^{t}) \mathring{\otimes} w_{q+1} + w_{q+1}^{p} \mathring{\otimes} (w_{q+1}^{c}+ w_{q+1}^{t}) ] \nonumber \\
& - \Upsilon_{1,l}^{-1} \Upsilon_{2,l}^{2} [(d_{q+1}^{c} + d_{q+1}^{t}) \mathring{\otimes} d_{q+1} + d_{q+1}^{p} \mathring{\otimes} (d_{q+1}^{c} + d_{q+1}^{t} )], \label{estimate 371}\\
p_{\text{corr}} \triangleq& \frac{1}{3} [ \Upsilon_{1,l}  ( w_{q+1}^{c} + w_{q+1}^{t}) \cdot (w_{q+1} + w_{q+1}^{p}) \nonumber\\
& \hspace{10mm} - \Upsilon_{1,l}^{-1} \Upsilon_{2,l}^{2} (d_{q+1}^{c} + d_{q+1}^{t}) \cdot (d_{q+1} + d_{q+1}^{p})],  \label{estimate 372} \\
R_{\text{osc}}^{v} \triangleq& \mathcal{R} [ \sum_{\xi \in \Lambda} \mathbb{P}_{\neq 0} ( ( \xi \otimes \xi) \nabla a_{\xi}^{2} \mathbb{P}_{\geq \frac{\lambda_{q+1} \sigma}{2}} (\phi_{\xi}^{2} \varphi_{\xi}^{2} )) - \mu^{-1} \sum_{\xi \in \Lambda} \mathbb{P}_{\neq 0} (\partial_{t} a_{\xi}^{2} \mathbb{P}_{\neq 0} (\phi_{\xi}^{2} \varphi_{\xi}^{2} )) \xi \label{estimate 373} \\
& \hspace{20mm} - \Upsilon_{1,l}^{-2} \Upsilon_{2,l}^{2} \sum_{\xi \in \Lambda_{\Xi}} \mathbb{P}_{\neq 0} (( \xi_{2} \otimes \xi_{2}) \nabla a_{\xi}^{2} \mathbb{P}_{\geq \frac{\lambda_{q+1} \sigma}{2}} (\phi_{\xi}^{2} \varphi_{\xi}^{2} ))] \nonumber  \\
&+ \sum_{\xi, \xi' \in \Lambda: \xi \neq \xi'} a_{\xi} a_{\xi'} \phi_{\xi} \phi_{\xi'} \varphi_{\xi} \varphi_{\xi'} (\xi \mathring{\otimes} \xi') - \Upsilon_{1,l}^{-2} \Upsilon_{2,l}^{2} \sum_{\xi, \xi' \in \Lambda_{\Xi}: \xi \neq \xi'} a_{\xi} a_{\xi'} \phi_{\xi} \phi_{\xi'} \varphi_{\xi} \varphi_{\xi'} (\xi_{2} \mathring{\otimes} \xi_{2}'),  \nonumber \\
p_{\text{osc}} \triangleq& \rho_{v} + \mu^{-1} \sum_{\xi \in \Lambda} \Delta^{-1} \text{div} \partial_{t} (a_{\xi}^{2} \mathbb{P}_{\neq 0} (\phi_{\xi}^{2} \varphi_{\xi}^{2} )) \xi \label{estimate 374}\\
&+ \frac{1}{3}[ \sum_{\xi, \xi' \in \Lambda: \xi \neq \xi'} a_{\xi}a_{\xi'} \phi_{\xi}\phi_{\xi'} \varphi_{\xi}\varphi_{\xi'} \xi \cdot \xi' - \Upsilon_{1,l}^{-2} \Upsilon_{2,l}^{2} \sum_{\xi, \xi' \in \Lambda_{\Xi}: \xi \neq \xi'} a_{\xi}a_{\xi'} \phi_{\xi}\phi_{\xi'} \varphi_{\xi}\varphi_{\xi'} \xi_{2} \cdot \xi_{2}'], \nonumber \\
R_{\text{com2}}^{v} \triangleq& (\Upsilon_{1} - \Upsilon_{1,l}) (v_{q+1} \mathring{\otimes} v_{q+1}) - (\Upsilon_{1}^{-1} - \Upsilon_{1,l}^{-1}) \Upsilon_{2}^{2} (\Xi_{q+1} \mathring{\otimes} \Xi_{q+1}) \nonumber \\
& \hspace{20mm}  - \Upsilon_{1,l}^{-1} (\Upsilon_{2} -\Upsilon_{2,l}) (\Upsilon_{2} + \Upsilon_{2,l}) (\Xi_{q+1} \mathring{\otimes} \Xi_{q+1}), \label{estimate 375}\\
p_{\text{com2}} \triangleq& \frac{1}{3} [ (\Upsilon_{1} - \Upsilon_{1,l}) \lvert v_{q+1} \rvert^{2} - (\Upsilon_{1}^{-1} - \Upsilon_{1,l}^{-1}) \Upsilon_{2}^{2} \lvert \Xi_{q+1} \rvert^{2}  \nonumber\\
& \hspace{20mm} - \Upsilon_{1,l}^{-1} (\Upsilon_{2} - \Upsilon_{2,l}) (\Upsilon_{2} + \Upsilon_{2,l}) \lvert \Xi_{q+1}\rvert^{2}]. \label{estimate 376} 
\end{align}
\end{subequations} 
We choose the same $p^{\ast}$ from \eqref{p ast} and first estimate $R_{\text{lin}}^{\Xi}$ from \eqref{estimate 365} and $R_{\text{lin}}^{v}$ from \eqref{estimate 369}. We immediately have  
\begin{equation}\label{estimate 378}
\lVert \mathcal{R}^{\Xi} ( \frac{1}{2} d_{q+1} ) \rVert_{C_{t}L_{x}^{p^{\ast}}} + \lVert \mathcal{R} (\frac{1}{2} w_{q+1}) \rVert_{C_{t}L_{x}^{p^{\ast}}} 
\overset{\eqref{estimate 345}}{\lesssim} m_{L} M_{0}(t)^{\frac{1}{2}} l^{-2} r^{\frac{1}{p^{\ast}} - \frac{1}{2}} \sigma^{\frac{1}{p^{\ast}} - \frac{1}{2}} \delta_{q+1}^{\frac{1}{2}}. 
\end{equation} 
Concerning the diffusive terms within \eqref{estimate 365} and \eqref{estimate 369} we estimate  by \eqref{estimate 345} and \eqref{estimate 377} 
\begin{align}
& \lVert \mathcal{R}^{\Xi} ((-\Delta)^{m_{2}} d_{q+1}) \rVert_{C_{t}L_{x}^{p^{\ast}}} + \lVert \mathcal{R} ((-\Delta)^{m_{1}} w_{q+1}) \rVert_{C_{t}L_{x}^{p^{\ast}}}  \nonumber \\
& \hspace{30mm} \lesssim  m_{L} M_{0}(t)^{\frac{1}{2}} l^{-2} r^{\frac{1}{p^{\ast}} - \frac{1}{2}} \sigma^{\frac{1}{p^{\ast}} - \frac{1}{2}} \delta_{q+1}^{\frac{1}{2}} \lambda_{q+1}^{\max\{m_{1}^{\ast},m_{2}^{\ast} \}}.\label{estimate 379}
\end{align} 
Concerning the temporal derivatives in \eqref{estimate 365} and \eqref{estimate 369} we estimate 
\begin{align}
& \lVert \mathcal{R}^{\Xi} (\partial_{t} d_{q+1}^{p} + \partial_{t} d_{q+1}^{c}) \rVert_{C_{t}L_{x}^{p^{\ast}}} + \lVert \mathcal{R} (\partial_{t} w_{q+1}^{p} + \partial_{t} w_{q+1}^{c} ) \rVert_{C_{t}L_{x}^{p^{\ast}}} \nonumber \\
\overset{\eqref{estimate 380}}{\lesssim}& \lambda_{q+1}^{-2} [\sum_{\xi \in\Lambda_{\Xi}} \lVert \partial_{t} \text{curl} (\bar{a}_{\xi} \phi_{\xi} \Psi_{\xi} \xi_{2}) \rVert_{C_{t}L_{x}^{p^{\ast}}} + \sum_{\xi \in \Lambda} \lVert \partial_{t} \text{curl} (\bar{a}_{\xi} \phi_{\xi} \Psi_{\xi} \xi) \rVert_{C_{t}L_{x}^{p^{\ast}}}] \nonumber \\
&\overset{ \eqref{estimate 308} \eqref{estimate 310}\eqref{estimate 420}  \eqref{estimate 175}}{\lesssim} m_{L} M_{0}(t)^{\frac{1}{2}} l^{-2} \delta_{q+1}^{\frac{1}{2}} r^{\frac{1}{p^{\ast}} - \frac{3}{2}} \sigma^{\frac{1}{p^{\ast}} + \frac{1}{2}} \mu. \label{estimate 381}
\end{align} 
Via  \eqref{estimate 296}, \eqref{estimate 345}, and \eqref{estimate 270}, we estimate the rest of the terms in \eqref{estimate 365} and \eqref{estimate 369} by 
\begin{align}
& \lVert \Upsilon_{1,l} (w_{q+1} \otimes \Xi_{l} + v_{l} \otimes d_{q+1} - d_{q+1} \otimes v_{l} -\Xi_{l}\otimes w_{q+1}) \rVert_{C_{t}L_{x}^{p^{\ast}}} \label{estimate 382} \\
&+ \lVert \Upsilon_{1,l} (w_{q+1}\mathring{\otimes} v_{l} + v_{l} \mathring{\otimes }w_{q+1}) - \Upsilon_{1,l}^{-1} \Upsilon_{2,l}^{2} (d_{q+1} \mathring{\otimes} \Xi_{l} + \Xi_{l} \mathring{\otimes} d_{q+1}) \rVert_{C_{t}L_{x}^{p^{\ast}}} \nonumber \\
\lesssim& m_{L}^{\frac{6}{5}}l (\lVert v_{q} \rVert_{C_{t,x}^{1}} + \lVert \Xi_{q} \rVert_{C_{t,x}^{1}}) m_{L} M_{0}(t)^{\frac{1}{2}} l^{-2}\delta_{q+1}^{\frac{1}{2}} r^{\frac{1}{p^{\ast}} - \frac{1}{2}} \sigma^{\frac{1}{p^{\ast}} - \frac{1}{2}} 
\lesssim m_{L}^{\frac{16}{5}} l^{-1} \delta_{q+1}^{\frac{1}{2}} M_{0}(t) \lambda_{q}^{4} r^{\frac{1}{p^{\ast}} -\frac{1}{2}} \sigma^{\frac{1}{p^{\ast}} - \frac{1}{2}}.\nonumber 
\end{align}
Therefore, starting from \eqref{estimate 365} and \eqref{estimate 369}, we can apply \eqref{estimate 378}-\eqref{estimate 382} and actually the desired estimate follows from the computations within \eqref{estimate 249} from the proof of Proposition \ref{Proposition 4.8} as follows: for $a \in 2 \mathbb{N}$ sufficiently large  
\begin{align}
& \lVert R_{\text{lin}}^{\Xi} \rVert_{C_{t}L_{x}^{1}} + \lVert R_{\text{lin}}^{v} \rVert_{C_{t}L_{x}^{1}} \label{estimate 383}\\
\overset{\eqref{estimate 378}-\eqref{estimate 382}}{\lesssim}& m_{L}^{\frac{16}{5}}[ M_{0}(t)^{\frac{1}{2}} l^{-2} r^{\frac{1}{p^{\ast}} - \frac{1}{2}} \sigma^{\frac{1}{p^{\ast}} - \frac{1}{2}} \delta_{q+1}^{\frac{1}{2}} \lambda_{q+1}^{\max\{m_{1}^{\ast},m_{2}^{\ast} \}} \nonumber\\
&+ M_{0}(t)^{\frac{1}{2}} l^{-2} \delta_{q+1}^{\frac{1}{2}} r^{\frac{1}{p^{\ast}} - \frac{3}{2}} \sigma^{\frac{1}{p^{\ast}} + \frac{1}{2}} \mu +  l^{-1} \delta_{q+1}^{\frac{1}{2}} M_{0}(t) \lambda_{q}^{4} r^{\frac{1}{p^{\ast}} -\frac{1}{2}} \sigma^{\frac{1}{p^{\ast}} - \frac{1}{2}} ] \overset{\eqref{estimate 249}}{\ll} M_{0}(t) \delta_{q+2}.\nonumber 
\end{align}
Second, the desired estimate on $R_{\text{corr}}^{\Xi}$ from \eqref{estimate 366} and $R_{\text{corr}}^{v}$ from \eqref{estimate 371} also follows from the computations within \eqref{estimate 250} within the proof of Proposition \ref{Proposition 4.8} as follows: for $a \in 2 \mathbb{N}$ sufficiently large 
\begin{align}
& \lVert R_{\text{corr}}^{\Xi} \rVert_{C_{t}L_{x}^{p^{\ast}}} + \lVert R_{\text{corr}}^{v} \rVert_{C_{t}L_{x}^{p^{\ast}}} \label{estimate 387}\\
 \overset{\eqref{estimate 296}}{\lesssim}& m_{L}^{\frac{6}{5}}  (\lVert w_{q+1}^{c} \rVert_{C_{t}L_{x}^{2p^{\ast}}} + \lVert w_{q+1}^{t} \rVert_{C_{t}L_{x}^{2p^{\ast}}} + \lVert d_{q+1}^{c} \rVert_{C_{t}L_{x}^{2p^{\ast}}} + \lVert d_{q+1}^{t} \rVert_{C_{t}L_{x}^{2p^{\ast}}}) \nonumber\\
\times& (\lVert w_{q+1}^{p} \rVert_{C_{t}L_{x}^{2p^{\ast}}} + \lVert w_{q+1}^{c} \rVert_{C_{t}L_{x}^{2p^{\ast}}} + \lVert w_{q+1}^{t} \rVert_{C_{t}L_{x}^{2p^{\ast}}} + \lVert d_{q+1}^{p} \rVert_{C_{t}L_{x}^{2p^{\ast}}} + \lVert d_{q+1}^{c} \rVert_{C_{t}L_{x}^{2p^{\ast}}} + \lVert d_{q+1}^{t} \rVert_{C_{t}L_{x}^{2p^{\ast}}}) \nonumber \\
&\overset{\eqref{estimate 337}-\eqref{estimate 339} \eqref{estimate 344}}{\lesssim} m_{L}^{\frac{22}{5}} (M_{0}(t)^{\frac{1}{2}} \delta_{q+1}^{\frac{1}{2}} l^{-2} r^{\frac{1}{2p^{\ast}} - \frac{3}{2}} \sigma^{\frac{1}{2p^{\ast}} + \frac{1}{2}} +  \mu^{-1} \delta_{q+1} l^{-4} M_{0}(t) r^{\frac{1}{2p^{\ast}} -1} \sigma^{\frac{1}{2p^{\ast}} -1}) \nonumber\\
&\times (\delta_{q+1}^{\frac{1}{2}} M_{0}(t)^{\frac{1}{2}} l^{-2} r^{\frac{1}{2p^{\ast}} - \frac{1}{2}} \sigma^{\frac{1}{2p^{\ast}} - \frac{1}{2}} +  M_{0}(t)^{\frac{1}{2}} \delta_{q+1}^{\frac{1}{2}} l^{-2} r^{\frac{1}{2p^{\ast}} - \frac{3}{2}} \sigma^{\frac{1}{2p^{\ast}} + \frac{1}{2}} \nonumber\\
& \hspace{40mm} + \mu^{-1} \delta_{q+1} l^{-4} M_{0}(t) r^{\frac{1}{2p^{\ast}} -1} \sigma^{\frac{1}{2p^{\ast}} -1})  \overset{\eqref{estimate 250}}{\ll} M_{0}(t) \delta_{q+2}. \nonumber
\end{align}
Third, concerning $R_{\text{osc}}^{\Xi}$ from \eqref{estimate 423}, as we discussed already, the estimate $\lVert \mathcal{R}^{\Xi} (\text{div} E_{1}^{\Xi} + \partial_{t} d_{q+1}^{t} ) \rVert_{C_{t}L_{x}^{p^{\ast}}}  \ll M_{0}(t) \delta_{q+2}$ in \eqref{estimate 244} remains valid and thus we only need to estimate $\lVert E_{2}^{\Xi} \rVert_{C_{t}L_{x}^{1}}$ which also follows from the computations within \eqref{estimate 243} in the proof of Proposition \ref{Proposition 4.8} as follows: for $a \in 2 \mathbb{N}$ sufficiently large 
\begin{align}
\lVert E_{2}^{\Xi} \rVert_{C_{t}L_{x}^{1}} 
\overset{\eqref{estimate 211}}{\lesssim}& \sum_{\xi \in\Lambda, \xi'\in\Lambda_{\Xi}: \xi \neq \xi'} \lVert a_{\xi} \rVert_{C_{t,x}} \lVert a_{\xi'} \rVert_{C_{t,x}} \lVert \phi_{\xi} \phi_{\xi'} \varphi_{\xi} \varphi_{\xi'} \rVert_{C_{t}L_{x}^{1}} \nonumber\\
&\overset{\eqref{estimate 326} \eqref{estimate 308}\eqref{estimate 175}}{\lesssim} m_{L}^{\frac{4}{5}} M_{0}(t) \delta_{q+1} l^{-4} \sigma r^{-1} 
\overset{\eqref{estimate 243}}{\ll} M_{0}(t) \delta_{q+2}. \label{estimate 389}
\end{align}
Therefore, we conclude
\begin{equation}\label{estimate 388}
\lVert R_{\text{osc}}^{\Xi} \rVert_{C_{t}L_{x}^{1}} \overset{\eqref{estimate 244}\eqref{estimate 389}}{\ll} M_{0}(t) \delta_{q+2}. 
\end{equation} 
Moreover, using 
\begin{equation}\label{estimate 447} 
-2\eta + \frac{221\alpha}{4} + (8\eta -2) (\frac{1}{p^{\ast}} - 1) \overset{\eqref{p ast}}{=} - \frac{19\alpha}{4} - \eta, 
\end{equation} 
we can compute from \eqref{estimate 373}
\begin{align}
& \lVert R_{\text{osc}}^{v} \rVert_{C_{t}L_{x}^{1}} \nonumber \\
\overset{\eqref{estimate 296}}{\lesssim}& \sum_{\xi \in \Lambda} \lVert (-\Delta)^{-\frac{1}{2}} \mathbb{P}_{\neq 0} ( \nabla a_{\xi}^{2} \mathbb{P}_{\geq \frac{\lambda_{q+1} \sigma}{2}} (\phi_{\xi}^{2} \varphi_{\xi}^{2} )) \rVert_{C_{t}L_{x}^{p^{\ast}}} + \mu^{-1} \lVert (-\Delta)^{-\frac{1}{2}} \mathbb{P}_{\neq 0} ( \partial_{t} a_{\xi}^{2} \mathbb{P}_{\geq \frac{\lambda_{q+1} \sigma}{2}} (\phi_{\xi}^{2} \varphi_{\xi}^{2})) \rVert_{C_{t}L_{x}^{p^{\ast}}} \nonumber \\
&+ m_{L}^{\frac{8}{5}} \sum_{\xi \in \Lambda_{\Xi}} \lVert (-\Delta)^{-\frac{1}{2}} \mathbb{P}_{\neq 0} (\nabla a_{\xi}^{2} \mathbb{P}_{\geq \frac{\lambda_{q+1} \sigma}{2}} (\phi_{\xi}^{2} \varphi_{\xi}^{2})) \rVert_{C_{t}L_{x}^{p^{\ast}}} + \sum_{\xi,\xi'\in \Lambda: \xi \neq \xi'} \lVert a_{\xi} \rVert_{C_{t,x}} \lVert a_{\xi'} \rVert_{C_{t,x}} \lVert \phi_{\xi}\phi_{\xi'} \varphi_{\xi}\varphi_{\xi'} \rVert_{C_{t}L_{x}^{1}}\nonumber \\
&+ m_{L}^{\frac{8}{5}} \sum_{\xi, \xi' \in\Lambda_{\Xi}: \xi \neq \xi'} \lVert a_{\xi} \rVert_{C_{t,x}} \lVert a_{\xi'} \rVert_{C_{t,x}} \lVert \phi_{\xi} \phi_{\xi'} \varphi_{\xi} \varphi_{\xi'} \rVert_{C_{t}L_{x}^{1}} \nonumber \\
&\overset{\eqref{estimate 241} \eqref{estimate 308}\eqref{estimate 326}\eqref{estimate 175}}{\lesssim} \lambda_{q+1}^{-2\eta} [m_{L}^{\frac{4}{5}} l^{-2} \delta_{q+1}^{\frac{1}{2}} M_{0}(t)^{\frac{1}{2}} m_{L}^{\frac{4}{5}} l^{-32} \delta_{q+1}^{\frac{1}{2}} M_{0}(t)^{\frac{1}{2}} \nonumber \\
& \hspace{20mm}+ \lambda_{q+1}^{\eta -1} m_{L}^{\frac{4}{5}} \delta_{q+1}^{\frac{1}{2}} l^{-2} M_{0}(t)^{\frac{1}{2}} m_{L}^{\frac{12}{5}} M_{0}(t)^{\frac{1}{2}} \delta_{q+1}^{\frac{1}{2}} l^{-31} ] (r^{\frac{1}{2p^{\ast}} - \frac{1}{2}} \sigma^{\frac{1}{2p^{\ast}} - \frac{1}{2}})^{2} \nonumber \\
&+ m_{L}^{\frac{8}{5}} \sum_{\xi \in \Lambda_{\Xi}} \lambda_{q+1}^{-2\eta} l^{-2} \delta_{q+1}^{\frac{1}{2}} M_{0}(t)^{\frac{1}{2}} \delta_{q+1}^{\frac{1}{2}} l^{-17} M_{0}(t)^{\frac{1}{2}} (r^{\frac{1}{2p^{\ast}} - \frac{1}{2}} \sigma^{\frac{1}{2p^{\ast}} - \frac{1}{2}})^{2} + m_{L}^{\frac{8}{5}} l^{-4} \delta_{q+1} M_{0}(t) \lambda_{q+1}^{-4\eta} \nonumber \\
\overset{\eqref{estimate 130}}{\lesssim}& m_{L}^{\frac{8}{5}} \delta_{q+2} M_{0}(t) [ \lambda_{q+1}^{-2 \eta + \frac{221 \alpha}{4} + 2 \beta b} (\lambda_{q+1}^{8\eta -2})^{\frac{1}{p^{\ast}} - 1} + \lambda_{q+1}^{\frac{13\alpha}{2} - 4 \eta + 2 \beta b}] \nonumber \\ 
& \hspace{10mm} \overset{ \eqref{estimate 447}\eqref{estimate 129}}{\lesssim} m_{L}^{\frac{8}{5}} \delta_{q+2} M_{0}(t) [ \lambda_{q+1}^{- \frac{37 \alpha}{8} - \eta} + \lambda_{q+1}^{\frac{53 \alpha}{8} - 4 \eta}] \ll M_{0}(t) \delta_{q+2}. \label{estimate 390}
\end{align}
Fourth, we estimate $R_{\text{com1}}^{v}$ and $R_{\text{com1}}^{\Xi}$ from \eqref{estimate 384} by continuing from \eqref{estimate 385}-\eqref{estimate 386} and making use of $\delta \in (0, \frac{1}{12})$, for $a \in 2 \mathbb{N}$ sufficiently large 
\begin{align}
 \lVert R_{\text{com1}}^{v} \rVert_{C_{t}L_{x}^{1}}+ \lVert R_{\text{com1}}^{\Xi} \rVert_{C_{t}L_{x}^{1}} &
\overset{\eqref{estimate 385}\eqref{estimate 386}}{\lesssim} [l m_{L}^{\frac{6}{5}} m_{L} M_{0}(t)^{\frac{1}{2}} \lambda_{q}^{4} + l^{\frac{1}{2} - 2 \delta} m_{L}^{2} (m_{L} M_{0}(t)^{\frac{1}{2}} \lambda_{q}^{4} )] m_{L} M_{0}(t)^{\frac{1}{2}} \nonumber\\
\lesssim& l^{\frac{1}{3}} m_{L}^{4} M_{0}(t) \lambda_{q}^{4} 
\overset{\eqref{b}\eqref{estimate 129}}{\lesssim}  M_{0}(t) \delta_{q+2} [\lambda_{q+1}^{-\frac{\alpha}{8}} m_{L}^{4}] \ll M_{0}(t) \delta_{q+2}. \label{estimate 391}
\end{align}
Finally, we estimate $R_{\text{com2}}^{\Xi}$ from \eqref{estimate 368} and $R_{\text{com2}}^{v}$ from \eqref{estimate 375} by relying on \eqref{estimate 268}-\eqref{estimate 269} at level $q+1$ that we verified in \eqref{estimate 448} and the fact that $\delta \in (0, \frac{1}{12})$, as well as taking $a \in 2 \mathbb{N}$ sufficiently large 
\begin{align}
&\lVert R_{\text{com2}}^{v} \rVert_{C_{t}L_{x}^{1}} + \lVert R_{\text{com2}}^{\Xi} \rVert_{C_{t}L_{x}^{1}} \nonumber\\
\overset{\eqref{estimate 273} \eqref{estimate 296} }{\lesssim}& l^{\frac{1}{2} - 2 \delta} \lVert \Upsilon_{1} \rVert_{C_{t}^{\frac{1}{2} - 2 \delta}} \lVert v_{q+1} \rVert_{C_{t}L_{x}^{2}}^{2} + \lVert \Upsilon_{1}^{-1} \Upsilon_{1,l}^{-1} (\Upsilon_{1,l} - \Upsilon_{1}) \rVert_{C_{t}} m_{L}^{\frac{4}{5}} \lVert \Xi_{q+1} \rVert_{C_{t}L_{x}^{2}}^{2} \nonumber\\
&+ l^{\frac{1}{2} - 2 \delta} \lVert \Upsilon_{2} \rVert_{C_{t}^{\frac{1}{2} - 2 \delta}} m_{L}^{\frac{4}{5}} \lVert \Xi_{q+1} \rVert_{C_{t}L_{x}^{2}}^{2} + l^{\frac{1}{2} - 2 \delta} \lVert \Upsilon_{1} \rVert_{C_{t}^{\frac{1}{2} - 2 \delta} } \lVert v_{q+1} \rVert_{C_{t}L_{x}^{2}} \lVert \Xi_{q+1} \rVert_{C_{t}L_{x}^{2}} \nonumber\\
\overset{\eqref{estimate 273} \eqref{estimate 448}}{\lesssim}& M_{0}(t) \delta_{q+2}[\lambda_{q+1}^{\frac{\alpha}{8} - \frac{\alpha}{2}} m_{L}^{4}]  \overset{\eqref{estimate 129}}{\ll} M_{0}(t) \delta_{q+2}. \label{estimate 392}
\end{align}
In sum of \eqref{estimate 383}, \eqref{estimate 387}, \eqref{estimate 388}, \eqref{estimate 390}-\eqref{estimate 392}, we conclude 
\begin{equation}
\lVert \mathring{R}_{q+1}^{v} \rVert_{C_{t}L_{x}^{1}}  + \lVert \mathring{R}_{q+1}^{\Xi} \rVert_{C_{t}L_{x}^{1}}  \ll M_{0}(t) \delta_{q+2}, 
\end{equation} 
which implies \eqref{estimate 271} at level $q+1$. 

Finally, the verification that $(v_{q+1}, \Xi_{q+1}, \mathring{R}_{q+1}^{v}$, $\mathring{R}_{q+1}^{\Xi})$ are $(\mathcal{F}_{t})_{t\geq 0}$-adapted and that $(v_{q+1},$ $\Xi_{q+1}$, $\mathring{R}_{q+1}^{v}$, $\mathring{R}_{q+1}^{\Xi})(0,x)$ are deterministic if $(v_{q}, \Xi_{q}, \mathring{R}_{q}^{v}, \mathring{R}_{q}^{\Xi})(0,x)$ are deterministic is very similar to the proof of Proposition \ref{Proposition 4.8} and thus omitted. 

\section{Appendix}\label{Appendix}
\subsection{Further preliminaries}
\begin{lemma}\label{divergence inverse operator}
\rm{(\cite[Equ. (5.34)]{BV19b} and \cite[Sec. 6.1-6.2]{BBV21})} For any $f \in C^{\infty}(\mathbb{T}^{3})$ that is mean-zero, define 
\begin{equation}\label{estimate 107}
(\mathcal{R}f)^{kl} \triangleq ( \partial^{k}\Delta^{-1} f^{l} + \partial^{l} \Delta^{-1} f^{k}) - \frac{1}{2} (\delta^{kl} + \partial^{k} \partial^{l} \Delta^{-1}) \text{div} \Delta^{-1} f
\end{equation} 
for $k, l \in \{1,2,3\}$. Then $\mathcal{R} f(x)$ is a symmetric trace-free matrix for each $x \in \mathbb{T}^{3}$, and  satisfies $\text{div} (\mathcal{R} f) = f$. Moreover, $\mathcal{R}$ satisfies the classical Calder$\acute{\mathrm{o}}$n-Zygmund and Schauder estimates: $\lVert (-\Delta)^{\frac{1}{2}} \mathcal{R} \rVert_{L_{x}^{p} \mapsto L_{x}^{p}} + \lVert \mathcal{R} \rVert_{L_{x}^{p} \mapsto L_{x}^{p}}  + \lVert \mathcal{R} \rVert_{C_{x} \mapsto C_{x}} \lesssim 1$ for all $p \in (1, \infty)$. Additionally, we define for $f: \mathbb{R}^{3}\mapsto \mathbb{R}^{3}$ such that $\nabla\cdot f = 0$,  
\begin{equation}\label{estimate 112}
( \mathcal{R}^{\Xi} f)^{ij} \triangleq \epsilon_{ijk} (-\Delta)^{-1} (\nabla \times f)^{k}
\end{equation} 
where $\epsilon_{ijk}$ is the Levi-Civita tensor. Then $\mathcal{R}^{\Xi} (f) = f, \mathcal{R}^{\Xi} (f) = - (\mathcal{R}^{\Xi} (f))^{T}$, and $(-\Delta)^{\frac{1}{2}}\mathcal{R}^{\Xi}$ is a Calderon-Zygmund operator again.  
\end{lemma} 
\begin{lemma}\label{Lemma 6.2}
\rm{(\cite[Lem. 3.7]{BV19a}, \cite[Lem. 5.4]{BBV21})} Fix integers $N, \kappa \geq 1$, and let $\zeta > 1$ satisfy 
\begin{equation}\label{estimate 140}
\frac{2 \pi \sqrt{3} \zeta}{\kappa} \leq \frac{1}{3} \hspace{2mm} \text{ and } \hspace{2mm} \zeta^{4} \frac{ (2\pi\sqrt{3} \zeta)^{N}}{\kappa^{N}} \leq 1. 
\end{equation} 
Let $p \in \{1,2\}$ and $f$ be a $\mathbb{T}^{3}$-periodic function such that there exists a constant $C_{f}  > 0$ such that 
\begin{equation}
\lVert D^{j} f \rVert_{L_{x}^{p}} \leq C_{f} \zeta^{j}  \hspace{3mm} \forall \hspace{1mm} 0 \leq j \leq N + 4. 
\end{equation} 
Additionally, let $g$ be a $(\mathbb{T}/\kappa)^{3}$-periodic function. Then for a universal constant $C_{\ast}$, 
\begin{equation}\label{estimate 171}
\lVert fg \rVert_{L_{x}^{p}} \leq C_{f} C_{\ast} \lVert g \rVert_{L_{x}^{p}}.
\end{equation} 
\end{lemma}

\begin{lemma}\label{Lemma 6.3}
\rm{(\cite[Lem. 4.1]{CL21}, cf. \cite[Lem. 7.4]{LQ20})} Let $g \in C^{2} (\mathbb{T}^{3})$ and $k \in \mathbb{N}$. Then for any $p \in (1,\infty)$ and any $f \in L^{p} (\mathbb{T}^{3})$, 
\begin{equation}\label{estimate 241}
\lVert (-\Delta)^{-\frac{1}{2}} \mathbb{P}_{\neq 0} (g \mathbb{P}_{\geq k} f ) \rVert_{L_{x}^{p}} \lesssim k^{-1} \lVert g \rVert_{C_{x}^{2}} \lVert f \rVert_{L_{x}^{p}}. 
\end{equation} 
\end{lemma}

\subsection{Sketch of proof on extending convex integration scheme in \cite{BBV21} to the case $\nu_{1}, \nu_{2} > 0$ and $m_{1}, m_{2} \in (0, \frac{3}{4})$}
We explain how one can extend the convex integration scheme in \cite{BBV21} from the ideal case of $\nu_{1} = \nu_{2} = 0$ to the case $\nu_{1}, \nu_{2} > 0$, and $m_{1}, m_{2} \in (0, \frac{3}{4})$. The actual proof is lengthy and has many parts similar to the proofs of Theorems \ref{Theorem 2.1}-\ref{Theorem 2.4}; thus, we only sketch the main ideas here. We also only consider the case of an additive noise so that $F_{k} \equiv 1$ and both $B_{k}$ for $k \in \{1,2\}$ are $G_{k}G_{k}^{\ast}$-Wiener processes as in the setup of Theorems \ref{Theorem 2.1}-\ref{Theorem 2.2}. In short, Beekie et al. in \cite[p. 11]{BBV21} defines a small parameter $r$ and subsequently sets it to be $\lambda_{q+1}^{-\frac{3}{4}}$ on \cite[p. 14]{BBV21}. Essentially, one can keep the freedom to choose this ``$r$,'' choose $l \triangleq \lambda_{q+1}^{-\frac{3\alpha}{2}}$ in contrast to \eqref{l} for simplicity  and proceed similarly to the proof of \cite[The. 1.4]{BBV21}. Then, at the very end of the convex integration when one must verify the final inductive estimate \eqref{estimate 100} at level $q+1$, 
\begin{equation}\label{estimate 14}
\lVert \mathring{R}_{q+1}^{v} \rVert_{C_{t}L_{x}^{1}} \leq c_{v} M_{0}(t) \delta_{q+2} \hspace{2mm} \text{ and } \hspace{2mm} \lVert \mathring{R}_{q+1}^{\Xi} \rVert_{C_{t}L_{x}^{1}} \leq c_{\Xi} M_{0}(t) \delta_{q+2}, 
\end{equation}
by choosing various parameters such as $\alpha$ very carefully, one arrives at 
\begin{subequations}
\begin{align}
\lVert  \mathring{R}_{q+1}^{\Xi} \rVert_{C_{t}L_{x}^{1}} \lesssim&   l^{- \frac{3}{2}} M_{0}(t)^{\frac{1}{2}} r^{\frac{1}{p^{\ast}} - \frac{1}{2}} \lambda_{q+1}^{2m_{2} -1} + l^{- \frac{7}{2}} M_{0}(t) r^{\frac{1}{2}}  + \delta_{q+2} M_{0}(t) a^{b^{q+1} [ \frac{139\alpha}{8} - 1]}  \label{estimate 13}\\
&+ l^{-9} M_{0}(t) \lambda_{q+1}^{-1} r^{-2 + \frac{1}{p^{\ast}}} + l^{-3} M_{0}(t) r  + \delta_{q+2} M_{0}(t) a^{b^{q} (-4)} + \delta_{q+2} M_{0}(t) a^{b^{q+1} (-\frac{3\alpha}{8})},  \nonumber   \\
\lVert \mathring{R}_{q+1}^{v} \rVert_{C_{t}L_{x}^{1}} \lesssim& l^{- \frac{3}{2}} M_{0}(t)^{\frac{1}{2}} r^{\frac{1}{p^{\ast}} - \frac{1}{2}} \lambda_{q+1}^{2m_{1} -1} + l^{- \frac{7}{2}} M_{0}(t) r^{\frac{1}{2}}  + \delta_{q+2} M_{0}(t) a^{b^{q+1} [ \frac{139\alpha}{8} - 1]}  \label{estimate 12} \\
&+ l^{-13} M_{0}(t) \lambda_{q+1}^{-1} r^{-2 + \frac{1}{p^{\ast}}} + l^{-3} M_{0}(t) r  + \delta_{q+2} M_{0}(t) a^{b^{q} (-4)} + \delta_{q+2} M_{0}(t) a^{b^{q+1} (-\frac{3\alpha}{8})},  \nonumber
\end{align} 
\end{subequations} 
where $l^{- \frac{3}{2}} M_{0}(t)^{\frac{1}{2}} r^{\frac{1}{p^{\ast}} - \frac{1}{2}} \lambda_{q+1}^{2m_{1} -1}$ in \eqref{estimate 12} and $l^{- \frac{3}{2}} M_{0}(t)^{\frac{1}{2}} r^{\frac{1}{p^{\ast}} - \frac{1}{2}} \lambda_{q+1}^{2m_{2} -1}$ in \eqref{estimate 13} come respectively from viscous and magnetic diffusive terms. For brevity we omit details. Next, we consider different cases and choose $r$ appropriately. 

First, let us consider the case $m_{1} \in (0, \frac{1}{2}]$ and $m_{2} \in (0, \frac{1}{2}]$. Then we simply have $\lambda_{q+1}^{2m_{1} - 1} \lesssim 1$ and $\lambda_{q+1}^{2m_{2} - 1} \lesssim 1$ so that we can define $r \triangleq \lambda_{q+1}^{-\frac{3}{4}}$ identically to \cite{BBV21}. Then by taking $\alpha > 0$ and $p^{\ast} > 1$ sufficiently small so that $\frac{9\alpha}{4} + (-\frac{3}{4}) (\frac{1}{p^{\ast}} - \frac{1}{2}) < 0$ and $\beta > 0$ sufficiently small, we can deduce 
\begin{equation*}
 l^{- \frac{3}{2}} M_{0}(t)^{\frac{1}{2}} r^{\frac{1}{p^{\ast}} - \frac{1}{2}} \lambda_{q+1}^{2m_{2} -1} \lesssim \lambda_{q+1}^{\frac{9\alpha}{4}} M_{0}(t)^{\frac{1}{2}} (\lambda_{q+1}^{-\frac{3}{4}} )^{\frac{1}{p^{\ast}} - \frac{1}{2}} \ll c_{\Xi} M_{0}(t) \delta_{q+2}
\end{equation*} 
for $a \in 2\mathbb{N}$ sufficiently large. Other terms may be estimated similarly by taking $\alpha > 0$ and $p^{\ast} > 1$ sufficiently small, $\beta > 0$ sufficiently small and $a \in 2\mathbb{N}$ sufficiently large to attain \eqref{estimate 14} as desired. 

Second, let us consider the case $m_{1} \in (0, \frac{1}{2}]$ and $m_{2} \in (\frac{1}{2}, \frac{3}{4})$. Two key terms in \eqref{estimate 13} to handle are 
\begin{equation}\label{estimate 17}
l^{- \frac{3}{2}} M_{0}(t)^{\frac{1}{2}} r^{\frac{1}{p^{\ast}} - \frac{1}{2}} \lambda_{q+1}^{2m_{2} -1} \text{ from magnetic diffusion and }  l^{-9} M_{0}(t) \lambda_{q+1}^{-1} r^{-2 + \frac{1}{p^{\ast}}}.
\end{equation}  
Planning ahead to take $\alpha > 0$ and $p^{\ast} > 1$ arbitrarily small, these two conditions simplify to, if we denote $r = \lambda_{q+1}^{\overline{m}}$ temporarily, 
\begin{equation}\label{estimate 15}
\overline{m} (\frac{1}{1}- \frac{1}{2}) + 2m_{2} - 1 < 0 \hspace{1mm}   \text{ and } \hspace{1mm}   -1 + \overline{m}(-2 + \frac{1}{1}) < 0 \hspace{1mm}  \text{ or } \hspace{1mm}   -1 < \overline{m} < 2 (1-2m_{2})
\end{equation}
and we realize that $(-1, 2 (1-2m_{2} )) \neq \emptyset$ precisely because $m_{2} < \frac{3}{4}$. Therefore, fixing such    
\begin{equation}\label{estimate 16}
r = \lambda_{q+1}^{\overline{m}} \text{ for any } \overline{m} \in (-1, 2 (1-2m_{2})), 
\end{equation} 
we can select $\alpha = \alpha (\overline{m}) > 0$ and then $p^{\ast} = p^{\ast} (m_{2}, \overline{m}, \alpha) \in (1,2)$ as needed to deduce the necessary estimates. E.g., concerning the first key term in \eqref{estimate 17}, we can compute  
\begin{equation*}
l^{-\frac{3}{2}} M_{0}(t)^{\frac{1}{2}}r^{\frac{1}{p^{\ast}} - \frac{1}{2}}  \lambda_{q+1}^{2m_{2} -1} = \lambda_{q+1}^{\frac{9\alpha}{4}} M_{0}(t)^{\frac{1}{2}} \lambda_{q+1}^{\overline{m} (\frac{1}{p^{\ast}} - \frac{1}{2})} \lambda_{q+1}^{2m_{2} -1} \ll c_{\Xi} M_{0}(t) \delta_{q+2} 
\end{equation*} 
by requiring $\alpha < \frac{4}{9} (1- 2m_{2} - \frac{\overline{m}}{2})$ and then $p^{\ast} < \frac{ \overline{m}}{ - \frac{9\alpha}{4} + 1 - 2 m_{2} + \frac{\overline{m}}{2}}$, as well as $\beta > 0$ sufficiently small and $a \in 2\mathbb{N}$ sufficiently large where one can use \eqref{estimate 16} to show that $\frac{4}{9} (1- 2m_{2} - \frac{\overline{m}}{2}) > 0$ which leads to $\frac{\overline{m}}{- \frac{9\alpha}{4} + 1- 2m_{2} + \frac{\overline{m}}{2}} > 1$. Concerning the second key term in \eqref{estimate 17}, we can compute  
\begin{equation}
l^{-9}M_{0}(t) \lambda_{q+1}^{-1} r^{-2 + \frac{1}{p^{\ast}}} = \lambda_{q+1}^{\frac{27\alpha}{2}}M_{0}(t) \lambda_{q+1}^{-1} \lambda_{q+1}^{\overline{m} (-2 + \frac{1}{p^{\ast}})} \ll c_{\Xi} M_{0}(t) \delta_{q+2} 
\end{equation} 
by additionally requiring 
\begin{equation*}
\begin{cases}
\alpha < \frac{2}{27} (1+ \overline{m}),  \hspace{2mm} p^{\ast} < \frac{\overline{m}}{- \frac{27\alpha}{2} + 1 + 2 \overline{m}} &\text{ if } \overline{m} \in (-1, -\frac{1}{2}],\\
\alpha < \frac{2}{27} (1+ 2\overline{m}), p^{\ast} > \frac{\overline{m}}{- \frac{27\alpha}{2} + 1 + 2\overline{m}}  &\text{ if } \overline{m} \in (- \frac{1}{2}, 2 (1-2m_{2})), 
\end{cases}
\end{equation*} 
as well as $\beta > 0$ sufficiently small and $a \in 2 \mathbb{N}$ sufficiently large. The other terms may be estimated more easily by adding requirements on $\alpha > 0$ and $p^{\ast} > 1$ as needed. The third case $m_{1} \in (0, \frac{1}{2}]$ and $m_{2} \in (\frac{1}{2}, \frac{3}{4})$ can clearly be achieved similarly to the second case. 

Finally, we consider the case $m_{1} \in (\frac{1}{2}, \frac{3}{4})$ and $m_{2} \in (\frac{1}{2}, \frac{3}{4})$. Similarly to \eqref{estimate 17}-\eqref{estimate 15} in the second case, four key terms in \eqref{estimate 13}-\eqref{estimate 12} to handle are 
\begin{align*}
& l^{- \frac{3}{2}} M_{0}(t)^{\frac{1}{2}} r^{\frac{1}{p^{\ast}} - \frac{1}{2}} \lambda_{q+1}^{2m_{2} -1}, \hspace{2mm} l^{-9} M_{0}(t) \lambda_{q+1}^{-1} r^{-2 + \frac{1}{p^{\ast}}}, \\
& l^{- \frac{3}{2}} M_{0}(t)^{\frac{1}{2}} r^{\frac{1}{p^{\ast}} - \frac{1}{2}} \lambda_{q+1}^{2m_{1} -1}, \hspace{2mm} l^{-13} M_{0}(t) \lambda_{q+1}^{-1} r^{-2 + \frac{1}{p^{\ast}}}, 
\end{align*} 
so that planning ahead to take $\alpha > 0$ and $p^{\ast} > 1$ sufficiently small, replacing $r$ by $\lambda_{q+1}^{\overline{m}}$, we see that we certainly need 
\begin{align*}
\frac{\overline{m}}{2} + 2m_{2} - 1< 0, -1 - \overline{m} < 0, \text{ and } \frac{\overline{m}}{2} + 2m_{1} - 1< 0, \text{ or } -1 < \overline{m} < \min_{k=1,2} 2(1-2m_{k})
\end{align*} 
and $(-1, \min_{k=1,2} 2(1-2m_{k})) \neq \emptyset$ precisely because $m_{1}, m_{2} < \frac{3}{4}$. Therefore, we can fix such 
\begin{equation}\label{estimate 393} 
r = \lambda_{q+1}^{\overline{m}} \text{ for any } \overline{m} \in (-1, \min_{k=1,2} 2 (1-2m_{k})), 
\end{equation} 
and obtain necessary bounds by taking $\alpha = \alpha(\overline{m}) > 0$ sufficiently small and then $p^{\ast} = p^{\ast} (\alpha, m_{1}, m_{2}, \overline{m}) \in (1,2)$ as needed similarly to the second case. We omit further details. 

\subsection{Proof of Theorem \ref{Theorem 2.4} assuming Theorem \ref{Theorem 2.3}}
We fix $T > 0$ arbitrary, $K > 1$, and $\kappa \in (0,1)$ such that $\kappa K^{2} \geq 1$. The probability measure $P \otimes_{\tau_{L}} R$ from Proposition \ref{Proposition 5.5} satisfies
\begin{equation}\label{estimate 261}
P \otimes_{\tau_{L}} R ( \{ \tau_{L} \geq T \}) \overset{\eqref{estimate 259}}{=} P ( \{ \tau_{L} \geq T \} ) \overset{\eqref{estimate 257}-\eqref{estimate 260}}{=}  \textbf{P} ( \{ T_{L} \geq T \}) \overset{\eqref{estimate 19}}{>} \kappa.
\end{equation} 
This leads to 
\begin{equation}
\mathbb{E}^{P \otimes_{\tau_{L} } R} [ \lVert \xi_{2}(T) \rVert_{L_{x}^{2}}^{2}]  \geq \mathbb{E}^{P \otimes_{\tau_{L}} R} [ 1_{\{ \tau_{L} \geq T \}} \lVert \xi_{2}(T) \rVert_{L_{x}^{2}}^{2}] 
\overset{\eqref{estimate 26} \eqref{estimate 261}}{>} \kappa K^{2} e^{T} [\lVert u^{\text{in}} \rVert_{L_{x}^{2}}^{2} + \lVert b^{\text{in}} \rVert_{L_{x}^{2}}^{2}].
\end{equation} 
On the other hand, classical Galerkin approximation gives us another solution $\mathcal{Q}$ that admits $\mathbb{E}^{\mathcal{Q}} [\lVert \xi(T) \rVert_{L_{x}^{2}}^{2}] \leq e^{T} [\lVert \xi_{1}^{\text{in}} \rVert_{L_{x}^{2}}^{2} + \lVert \xi_{2}^{\text{in}} \rVert_{L_{x}^{2}}^{2}].$ This implies a lack of joint uniqueness in law and consequently non-uniqueness in aw for \eqref{stochastic GMHD} due to Cherny's theorem (\cite[Lem. C.1]{HZZ19}). 
 
\section*{Acknowledgements}
The author expresses deep gratitude to Prof. Mimi Dai for some stimulating discussions during the 5-day workshop at the American Institute of Mathematics in April 2021.

\end{document}